\documentclass[a4paper,11pt,twoside]{article}

\author{
\normalsize Fran\c{c}ois Delarue \\[8pt]
	\small Laboratoire J.A.Dieudonn\'e \\
	\small UMR CNRS-UNS No 7351 \\
	\small Universit\'e de Nice Sophia-Antipolis \\
	\small Parc Valrose \\
    \small France-06108 NICE Cedex 2\\
	\small delarue@unice.fr
\and
\normalsize William Salkeld \\[8pt]
	\small Laboratoire J.A.Dieudonn\'e \\
	\small UMR CNRS-UNS No 7351 \\
	\small Universit\'e de Nice Sophia-Antipolis \\
	\small Parc Valrose \\
       \small France-06108 NICE Cedex 2\\
	\small  salkeld@unice.fr
}
\usepackage[nodayofweek]{datetime}
 
\usepackage[utf8]{inputenc}
\usepackage[T1]{fontenc}
\usepackage[english]{babel}
\usepackage{charter}

\usepackage{xcolor}
\usepackage{verbatim}
\usepackage{hyperref}

\usepackage[top=3.cm, bottom=4.0cm, left=2.2cm, right=2.2cm]{geometry}
\usepackage{amsmath}
\usepackage{amsthm}	
\usepackage{amsfonts}	
\usepackage{amssymb,bbm}
\usepackage{shuffle}
\usepackage{mathrsfs}
\usepackage{pifont}

\usepackage{enumerate}
\usepackage{enumitem}
\usepackage[nobysame, alphabetic,initials]{amsrefs}

\usepackage{stmaryrd}
\usepackage{appendix}
\usepackage{graphicx}
\usepackage{ulem} 

\numberwithin{equation}{section}
\theoremstyle{plain}
\newtheorem{theorem}{Theorem}[section]
\newtheorem{lemma}[theorem]{Lemma}
\newtheorem{proposition}[theorem]{Proposition}
\newtheorem{corollary}[theorem]{Corollary}

\newtheorem{definition}[theorem]{Definition}

\newtheorem{remark}[theorem]{Remark}
\newtheorem{example}[theorem]{Example}

\allowdisplaybreaks

\newcommand{\bE}{\mathbb{E}}

\newcommand{\bN}{\mathbb{N}}
\newcommand{\bP}{\mathbb{P}}

\newcommand{\bR}{\mathbb{R}}

\newcommand{\bW}{\mathbb{W}}
\newcommand{\bX}{\mathbb{X}}
\newcommand{\bY}{\mathbb{Y}}
\newcommand{\bZ}{\mathbb{Z}}

\newcommand{\cA}{\mathcal{A}}
\newcommand{\cB}{\mathcal{B}}

\newcommand{\cD}{\mathcal{D}}
\newcommand{\cE}{\mathcal{E}}
\newcommand{\cF}{\mathcal{F}}
\newcommand{\cG}{\mathcal{G}}

\newcommand{\cL}{\mathcal{L}}

\newcommand{\cN}{\mathcal{N}}
\newcommand{\cO}{\mathcal{O}}
\newcommand{\cP}{\mathcal{P}}

\newcommand{\fJ}{\mathfrak{J}}

\newcommand{\fP}{\mathfrak{P}}

\newcommand{\fR}{\mathfrak{R}}

\newcommand{\fV}{\mathfrak{V}}

\newcommand{\scA}{\mathscr{A}}
\newcommand{\scC}{\mathscr{C}}
\newcommand{\scF}{\mathscr{F}}
\newcommand{\scG}{\mathscr{G}}
\newcommand{\scH}{\mathscr{H}}
\newcommand{\scI}{\mathscr{I}}

\newcommand{\scN}{\mathscr{N}}
\newcommand{\scP}{\mathscr{P}}

\newcommand{\scS}{\mathscr{S}}
\newcommand{\scT}{\mathscr{T}}

\newcommand{\rD}{\mathbf{D}}

\newcommand{\rp}{\mathbf{p}}

\newcommand{\rv}{\mathbf{V}}
\newcommand{\rw}{\mathbf{W}}

\newcommand{\rId}{\mathbf{1}}

\newcommand{\vertiii}{{\vert\kern-0.25ex \vert\kern-0.25ex \vert}}

\DeclareMathOperator{\lip}{Lip}

\DeclareMathOperator{\lin}{Lin}

\newcommand{\A}[1]{A^{(0)}_{#1}}

\usepackage{pgf,tikz}
\usetikzlibrary{arrows}

\tikzstyle{vertex} = [fill, shape=circle,inner sep=2pt,]
\tikzstyle{edge} = [fill, line width = 0.5pt]
\tikzstyle{zhyedge1} = [opacity=.5,fill opacity=.5, line cap=round, line join=round, line width=27pt]
\tikzstyle{zhyedge2} = [opacity=.5,fill opacity=.5, line cap=round, line join=round, line width=25pt,color=white]
\tikzstyle{hyedge1} = [opacity=.5,fill opacity=.5, line cap=round, line join=round, line width=27pt]
\tikzstyle{hyedge2} = [opacity=.5,fill opacity=.5, line cap=round, line join=round, line width=25pt, color=white]

\tikzstyle{vertexS} = [fill, shape=circle,inner sep=1pt,]
\tikzstyle{edgeS} = [fill, line width = 0.25pt]
\tikzstyle{zhyedge1S} = [opacity=.5,fill opacity=.5, line cap=round, line join=round, line width=12pt,color=black]
\tikzstyle{zhyedge2S} = [opacity=.5,fill opacity=.5, line cap=round, line join=round, line width=10pt,color=white]
\tikzstyle{hyedge1S} = [opacity=.5,fill opacity=.5, line cap=round, line join=round, line width=12pt]
\tikzstyle{hyedge2S} = [opacity=.5,fill opacity=.5, line cap=round, line join=round, line width=10pt, color=white]
\pgfdeclarelayer{background}
\pgfsetlayers{background,main}

\title{Probabilistic rough paths II \\Lions-Taylor expansions and Random controlled rough paths}

\begin{document}

\maketitle

\begin{abstract} 
In line with the notion of probabilistic rough paths introduced in the previous contribution \cite{salkeld2021Probabilistic}, we address corresponding random controlled rough paths (first introduced in \cite{2019arXiv180205882.2B}), the structure of which is indexed by Lions forests. These are statistical distributions over the space of paths described by the combination of a jet on the underlying probabilistic rough path and a remainder term. The regularity of the latter facilitates the definition of a rough integral. 

We establish closedness and stability of two key operators on random controlled rough paths: rough integration and composition by a smooth function on the Wasserstein space. These are important results towards a complete theory of rough McKean-Vlasov equations that is still in gestation. The proof goes through a higher-order Taylor expansion for the Lions derivative which we rigorously expound. 

The coupled Hopf algebra structure (see \cite{salkeld2021Probabilistic}) and the Lions-Taylor expansion (established in Section \ref{section:TaylorExpansions}) introduce a number of additional challenges which mean these results are not simply a natural extension of classical theory. We dedicate this work to pursuing these details. 

\end{abstract} 


{\bf Keywords:} Probabilistic rough paths, Random controlled rough paths, McKean-Vlasov equations

\vspace{0.3cm}

\noindent
{\bf MSC2020 Mathematics Subject Classification System:}\\
Primary: 60L30	
,
41A58  

Secondary: 
60G07  
, 
46G05  


\setcounter{tocdepth}{2}
\tableofcontents

\section{Introduction}

The main purpose of this paper is to elaborate on the notion of \emph{probabilistic rough paths} introduced in \cite{salkeld2021Probabilistic}. As such, it is the second of a series of papers whose aim is to provide a global theory for rough mean-field equations. 

\subsection{A short review of mean-field models and rough path theory}

The relevance of our program is evident from the many developments in the theories of rough paths and of dynamic mean-field  models. An overview has been given in \cite{salkeld2021Probabilistic}, but we consider it useful to summarise some of the results here. On the one hand, the theory of rough paths, introduced in \cite{lyons1998differential}, gives meaning to differential systems driven by irregular noises. Although the theory is in itself deterministic, it covers many probabilistic examples, including of course Brownian motion, where the associated stochastic integral is understood in the It\^o or Stratonovich sense. In many cases, the random structure helps in the construction of the stack of iterated integrals, whose collection forms the signature of the signal and which plays a crucial role in the analysis of the associated differential systems. 

Lyons' pioneering work was developed further in the article of Gubinelli \cite{gubinelli2010ramification} which identifies iterated integrals of the driving signal with trees. A vector space spanned by the trees paired with a product and coproduct generates a Hopf algebra, called the Connes-Kreimer algebra. It is now well understood that a rough path is a trajectory, with values in the set of characters of a Hopf algebra, which is additive for the convolution product. This additivity constraint is often referred to as the Chen relation. In particular, this concept is reflected in the theory of regularity structures, developed by Hairer \cite{hairer2014theory} to study singular stochastic partial differential equations. In this respect, the systematic use of algebraic structures to encode the underlying derivation and integration operations plays a decisive role in the associated renormalisation steps. Inspired by this, our previous contribution \cite{salkeld2021Probabilistic} describes an alternative representation built around the Lions derivative which introduces paths on the so-called McKean-Vlasov group of characters. 

On the other hand, the theory of mean-field models, originating from statistical mechanics (\cites{kac1956foundations, McKean1966}), has blossomed after numerous developments in the line of the theory of stochastic processes (\cites{kurtz1999particle, meleard1996asymptotic, Sznitman}). This was driven by a new revival of interest in connection with large population optimisation problems, studied for example in the framework of mean-field game and mean-field control theories, see for example \cites{lasry2007mean, Huang2006Large}, the books \cites{CarmonaDelarue2017book1, CarmonaDelarue2017book2} as well as the literature cited within. 

\subsection{Probabilistic rough paths and beyond}

In most studied examples, mean-field models are driven by Markovian noise. Nevertheless, the question of extending the theory of mean-field models to systems driven by general rough signals is natural and was posed long before our first work \cite{salkeld2021Probabilistic}. The first paper in this direction is due to Cass and Lyons (\cite{CassLyonsEvolving}). This was followed by \cite{Bailleul2015Flows} and \cite{deuschel2017enhanced}. In all these works, the mean-field interaction appears explicitly only in the transport term (which is absolutely continuous) and not in the coefficients (sometimes called volatility) driving the rough signals. The extension to models where the volatility is truly mean-field raises a conceptual difficulty, which has been systematically addressed in \cite{2019arXiv180205882.2B} for noises whose H\"older exponent is between $1/3$ and $1/2$. 

The main idea of \cite{2019arXiv180205882.2B} is to treat the mean-field volatility as a function of a random variable (as an object of infinite dimension) and not directly as a function of a measure. This approach is inspired by Lions' interpretation of the derivative (also known as Wasserstein's derivative) on the space of probability measures in the form of a Fr\'echet derivative on the Hilbert space of random variables. This perspective is particularly well adapted to particle descriptions of mean-field models in which an observer follows the evolution of a tagged particle within a population. The lifting procedure from the space of probability measures to the space of random variables, which is used systematically in Lions' approach to the Wasserstein derivative, also plays a fundamental role in our first work \cite{salkeld2021Probabilistic}. 

In a schematic way, \cite{salkeld2021Probabilistic} proposes a general definition of a probabilistic rough path with a given law $\mu$ over path space. The iterated integrals are constructed on probability spaces of increasing size, obtained by successive tensorization of the Wiener space, with each new copy of the Wiener space carrying a new sample from the law $\mu$ (independent of the previous ones). 

The thrust of \cite{salkeld2021Probabilistic} is to provide an algebraic structure for encoding the stack of these iterated integrals. For instance, one could naively think of labelling all the underlying samples and then equipping the trees from the Connes-Kreimer algebra with those labels, but this would lead to costly and useless repetitions due to the exchangeable structure of the model. The construction of \cite{salkeld2021Probabilistic} proceeds efficiently as it does not rely on the explicit values of the labels that could be assigned to the realisations. Intuitively, only the clusters formed by the nodes equipped with the same labels are used. This gives rise to trees equipped with a partition of the tree nodes, called hyperedges, instead of trees equipped with labels. However, a striking fact of the theory initiated in \cite{salkeld2021Probabilistic} is  that not all partitions are relevant. In the mean-field setting, only certain types of partitions (a rigorous description of which can be found in Definition \ref{definition:Forests} below) suffice: we called the trees equipped with such partitions Lions trees. 
 
In this new contribution, we make another step forward toward a complete theory of rough mean-field models. Whilst \cite{salkeld2021Probabilistic} addressed the definition of the signature and the construction of an algebraic structure associated with it, we did not address the dynamics driven by these signals. The goal of this article is to explore this question. 

In order to proceed, we adopt the perspective of \cite{gubinelli2004controlling} and define a general notion of \emph{random controlled rough paths} that is consistent with the coupled Hopf algebras introduced in  \cite{salkeld2021Probabilistic}. Our results include a stability property for mean-field controlled rough paths under composition by smooth functions depending on the time marginal of the paths. This is a milestone in our program since similar stability properties play a crucial role in the analysis of related rough differential equations in the classical (Lyons-Gubinelli) setting. Noticeably, the derivation of this stability property in the mean-field setting goes through a generalisation of the notion of higher-order derivatives for functionals defined on the space of probability measures and is, in turn, based on a proper form of higher order Taylor expansion. To the best of our knowledge, this Taylor expansion for the Lions derivative is new (at least in this form, see \cite{TseHigher2021} for another formulation), and we strongly believe it to have its own interest beside the specific application that we address here. In this regard, a striking fact in our analysis is the form of the expansion itself: it is encoded by means of partition sequences that are used to encode grafting operations of Lions trees. 

\subsection{From elementary differentials of a McKean-Vlasov equation to random controlled paths}

In order to provide a meaningful motivation for the central results proved in this manuscript, we must discuss the contraction operators used for solving rough differential equations and classical mean-field equations by means of a Banach fixed point theorem. To ensure that there is no confusion over this point, these contraction operators will not be addressed directly in this work. However, we hope that the reader well versed in rough path and regularity structure literature will see that multiple contraction operators chosen on appropriate spaces would ensure (provided they indeed exist) the existence and uniqueness of a rough differential equation of the form
\begin{equation}
\label{eq:meanfield:equation}
dX_{t} = f\Big(X_{t}, \cL^{X}_{t} \Big) dW_{t}. 
\end{equation}
The results of this paper could serve to establish key stability results that would allow for a well structured and concise description of such contraction operators. We hope that a sceptical reader will delay any expectations for such contraction operators until a future paper and will approach this work as a demonstration of the power of the higher order Lions-Taylor expansions that we prove in Section \ref{section:TaylorExpansions} and as a direct sequel to \cite{salkeld2021Probabilistic} that proves how these Lions-Taylor expansions are inter-weaved with Lions trees and coupled Hopf algebras. 

In the classical theory of rough paths, solutions of rough differential equations may be locally expanded as a series of elementary differentials acted on by components of the rough path corresponding to iterated integrals of the driving signal. By collecting the elementary differentials of a solution together, we obtain an element of a Hopf algebra which characterises the solution and which provides a reformulation of the corresponding differential equation as an equation on the elementary differentials. In fact, this idea was pushed further by Gubinelli: By gathering together all paths on the Hopf algebra that satisfy a necessary regularity condition, Gubinelli was able to describe the concept of a controlled rough path, see \cite{gubinelli2004controlling}. The advantage of this notion, which encompasses elementary differentials themselves, is that the space of controlled rough paths is a Banach space, unlike the collection of rough paths. Hence, techniques from functional analysis such as constructing a contraction operator to prove the existence and uniqueness of an element with particular properties can still be used (although the obtention of a contraction is a separate problem). This concept was further generalised in \cite{hairer2014theory} to describe the concept of modelled distributions and is now widely recognised as the correct approach for approximating the dynamics of singular stochastic partial differential equations. 

We saw in \cite{salkeld2021Probabilistic} that a probabilistic rough path takes its values on the aforementioned McKean-Vlasov group of characters, a subgroup of the characters that satisfies an additional identity relating the probability distributions of random variables associated to Lions trees with similar but differently tagged hyperedges. Indeed, a key part of the theory from \cite{salkeld2021Probabilistic} is that one (possibly empty) hyperedge of a Lions tree is identified as being distinct from all other hyperedges, referred to as the $0$-hyperedge. Intuitively, this $0$-tag means that the corresponding probability space carries a tagged particle. In this framework, the actions of McKean-Vlasov characters on a tree with a $0$-hyperedge and on the same tree but with the $0$-hyperedge being untagged are strongly connected. Here, the McKean-Vlasov characters should be thought of as functionals on a \textit{coupled} Hopf algebra in \cite{salkeld2021Probabilistic}, the elements of which can be decomposed in random variables indexed by Lions trees. This Hopf algebra like structure is said to be coupled because coupling operators are necessary to explain properly the statistical correlations that do exist between random variables indexed by two different Lions trees. Hence, we want to find a collection of path-valued random variables that together form an element of this coupled Hopf algebra and satisfy favourable regularity properties. Such a collection of paths (see Definition \ref{definition:RandomControlledRP}), endowed with a complete topology (see Theorem \ref{theorem:Banachspace}), would be invaluable in describing the microscopic dynamics of equations of the form \eqref{eq:meanfield:equation}. A similar concept has already been explored in  \cite{2019arXiv180205882.2B} and are referred to as \emph{random controlled rough paths}. 

\subsection{Contributions of this paper}

As previously announced, the central contribution of this paper is the extension of the concept of a modelled distribution from regularity structures with respect to probabilistic rough paths as developed in \cite{salkeld2021Probabilistic} which we refer to as random controlled rough paths. The main statement in this regard is Definition \ref{definition:RandomControlledRP}, which can be summarized as follows. If $\rw$ is a probabilistic rough path on a coupled Hopf algebra $\scH$, a path $\bX$ from $[0,1]$ to $\scH$ is called a \emph{Random Controlled Rough Path} (RCRP) controlled by $\rw$ if $\forall s, t\in[0,1]$ with $s<t$, any component of the random controlled rough path evaluated at $t$ can be expressed in terms of a jet containing terms from the RCRP evaluated at $s$ and increments of the probabilistic rough path plus an additional remainder term dependent on $s$ and $t$ which has higher regularity. In comparison with the classical (non mean-field) framework, the key point here is that all the aforementioned terms are random variables constructed on different spaces. 

Further, we should say that the coupled Hopf algebra $\scH$ is graded, with the grade of a tree depending on the number of nodes in the $0$-hyperedge and the number of nodes in detagged hyperedges. In turn, the probability spaces on which the terms of an RCRP are constructed get larger and larger with the number of hyperedges indexing these terms. In fact, those spaces are products $\Omega^{\times m}$ of a common path space $\Omega$, with each new copy of $\Omega$ accounting for a new hyperedge in the tree indexing the corresponding random variables. Hence, part of the challenge in this construction is to bring back all the terms involved in the expansion of a component of $\bX$ onto a common probability space. This is achieved by taking conditional expectations appropriately, which thus requires suitable integrability properties of all the random variables in hand. We refer to Equation \eqref{eq:definition:RandomControlledRP1} and \eqref{eq:definition:RandomControlledRP2} for the complete form of these conditional expectations. Moreover, due to the coupled coproduct equipping the Hopf algebra (which implicitly induces statistical correlations between the random variables belonging to the coupled Hopf algebra), we need to additionally include coupling functions to formalise the various conditionings, see Definition \ref{definition:CouplingFunctions}. This necessarily introduces new challenges and significantly increases  the difficulty of otherwise standard results. Hence the contributions of this paper are much more challenging than simple consequences of the results proved in \cite{hairer2014theory} or in earlier works in rough path theory.

As with classical theory, one needs to measure the regularity of the various terms underpinning a random controlled rough path in a convenient way. Following the standard framework, one then considers the regularity of each remainder term individually in order to define a norm on random controlled rough paths. However, there are more subtleties because of the random nature each term and some care should be taken to evaluate each probability space appropriately and precisely: in particular, the detagged probability spaces should be integrated over whereas the tagged probability space should be evaluated path by path. This is reflected in Equation \eqref{eq:definition:RandomControlledRP4}. 

In the end, we claim that random controlled rough paths provide the ideal framework for giving meaning to rough integrals with mean-field coefficients. The intuition for this is the same as in the standard theory: the increments of the random controlled rough path are determined uniquely up to a ``smooth'' remainder term by the increments of a probabilistic rough path and there is a canonical way of describing the integral of any coefficient of a probabilistic rough path by the probabilistic rough path. Hence, we can define a mean-field rough integral in terms of an extended Riemann sum comprising the jet of a random controlled rough path integrated with respect to the probabilistic rough path locally.

Our study comprises further critical results. As an application of the Lions-Taylor expansion obtained in  Section \ref{section:TaylorExpansions}, we are additionally able to prove that the smooth image of a random controlled rough path is also a random controlled rough path (see Theorem \ref{theorem:ContinIm-RCRPs}). More specifically, for any function $f \in C_b^{n, (n)}\big( \bR^e \times \cP_2(\bR^e); \lin(\bR^d, \bR^e) \big)$ and any pair $(\bX, \bY)$ of random controlled rough paths (controlled by a common probabilistic rough path $\rw$), there is a random controlled rough path $\bZ$ (controlled by $\rw$) that satisfies
$$
\Big\langle \bZ_t, \rId \Big\rangle(\omega_0) = f\Big(  \big\langle \bX_t, \rId \big\rangle(\omega_0), \cL^{\langle \bY_t, \rId\rangle} \Big),
$$
where $(\langle \bZ_t, \rId \rangle)_{t \in [0,1]}$ (and similarly for 
$(\langle \bX_t, \rId \rangle)_{t \in [0,1]}$ and $(\langle \bY_t, \rId \rangle)_{t \in [0,1]}$) is a standard notation for the path, taking values in the physical space $\bR^e$, that lives below the random controlled rough path $\bZ$. Equivalently, the local increments of 
$(\langle \bZ_t, \rId \rangle)_{t \in [0,1]}$ are precisely described by means of $\bZ$. The proof of this chain rule is by no means classical and we highlight the most noticeable difference in Equation \eqref{eq:theorem:ContinIm-RCRPs}: again, it requires the management of the statistical correlations that exist between the various random variables that enter the chain rule; here, this is achieved for a given Lions tree by systematically distinguishing a collection of hyperedges that are ``lost'' via the application of the decoupling operation that determine the various conditionings appearing in the terms of $\bZ$. These specific hyperedges are called ghost hyperedges, see Definition \ref{definition:Z-set}. They provide an efficient way of describing all of the combinatorics associated with mean-field contributions that arise from pathwise dependencies within mean-field coefficients. As a simple example, they capture the mean-field contributions that arise from a pathwise dependency within a mean-field coefficient. Such inter-dependencies are not naturally explainable with the English language, but we found our approach to be productive and accurate. 

To conclude this work, we consider the local stability results for random controlled rough paths in Section \ref{section:Stability-RCRP}. When two random controlled rough paths are controlled by the same path, they coexist in a linear space of random variables and we can define a (complete, separable) norm using Lebesgue integrals. However, two random controlled rough paths that are not controlled by the same path exist on totally different spaces. Nonetheless, we can define a pseudo-metric between two random controlled rough paths by solving the optimal transport for the cost function that compares the initial conditions and the regularity of the remainder terms for all coefficients of the random controlled rough path, see Definition \ref{definition:InhomogeneousMetric-RCRPs}. This pseudo-metric has many similarities with the Wasserstein distance, but a point worth emphasising is that, in this pseudo-metric, the minimum is taken over couplings between marginal laws equipping the path space $\Omega$ and not between laws equipping products of the type $\Omega^{\times n}$. In other words, couplings are constructed between each probability space associated to each (detagged) hyperedge of each Lions tree and the pseudo-metric then finds the optimal coupling that minimizes a single transport cost that accounts for all Lions trees at once. In this approach, the tagged probability space (which implicitly carries out the realisation of the tagged particle) can be treated differently depending on the context corresponding to whether one compares the distribution or the paths of two random controlled rough paths. To make it clear, one may or may not freeze the realisation of the tagged particle in 
the definition of the pseudo-metric.

Using this pseudo-metric, we are able to establish the locally pathwise stability result for the integral of a random controlled rough path (see Theorem \ref{theorem:Stability-RoughInt}) and the smooth image of random controlled rough paths (see Theorem \ref{theorem:Stability-ContinImage}). These estimates are pathwise locally Lipschitz, but under appropriate localisation assumptions can be strengthened to Lipschitz. 

Importantly, this paper also contains several contributions that individually have the potential to create new directions of research in the study of numerical approximations for mean-field dynamics. In particular, we highlight Theorem \ref{theorem:LionsTaylor2} which provides a generalised Taylor expansion of a function of two variables, a spacial variable and a variable on the Wasserstein space of measures. The derivatives used are in the sense of Lions calculus of variations and we express the iterated sequences of derivatives in spacial, measure and free variables using partition sequences. We hope to address the possible numerical applications of it in a future work. 

\subsection{Organisation}

In Section \ref{section:TaylorExpansions}, we introduce a new notation for iterative Lions derivatives, demonstrate the link between these derivatives and partitions and prove a Taylor expansion for functions of measures, see Theorem \ref{theorem:LionsTaylor1}. These results are then extended to the multivariate case in Subsection \ref{subsect:Multivariate_Lions-Taylor} and we also provide a Schwarz Theorem (see Theorem \ref{thm:Schwarz-Lions}) which identifies partition sequences whose Lions derivatives are equal. These results were first stated in \cite{salkeld2021Probabilistic}, but we choose to prove them in this sequel as they are fundamental to the results of Section \ref{section:ModelledDistributions} but were only motivational in the former. That said, the notion of a partition sequence and their link with iterative Lions derivatives is critical to the conceptual development of the Lions tree so it was also necessary to include these results in our earlier work. We hope the reader will forgive us for this. 

The key contributions of Section \ref{section:ModelledDistributions} is the definition of a random controlled rough path and the proof that two key operators are closed on the space of random controlled rough paths (see Theorem \ref{theorem:Reconstruction} and Theorem \ref{theorem:ContinIm-RCRPs}). 

Finally, Section \ref{section:Stability-RCRP} introduces a notion of metric between two random controlled rough paths (controlled by different probabilistic rough paths). These are used to prove two local stability estimates, see Theorem \ref{theorem:Stability-RoughInt} and Theorem \ref{theorem:Stability-ContinImage} which will turn out to be fundamental in establishing existence and uniqueness of a solution to rough differential equations in future works. 

\subsection{Notations and notions from standard rough path theory}
\label{subsection:Notation}

Let ${\mathbb N}$ be the set of positive integers and ${\mathbb N}_{0}={\mathbb N} \cup \{0\}$. Let $\bR$ be the field of real numbers and for $d \in \bN$, let $\bR^d$ be the $d$-dimensional vector space over the field $\bR$. For vector spaces $U$ and $V$, we define $\lin(U, V)$ to be the collection of linear operators from $U$ to $V$. Let $\oplus$ and $\otimes$ be the direct sum and tensor product operations. 

On the tensor space $T(V, U) = \bigoplus_{n=0}^\infty \lin( V^{\otimes n}, U)$, we define the ring multiplication 
$$
\otimes : \lin(V^{\otimes m}, U) \times \lin(V^{\otimes n}, U) \to \lin(V^{\otimes (m+n)}, U).
$$
In particular, when $m,n=0$, $\otimes$ defines a coordinatewise product $U\times U \to U$. 

For a vector space $U$, let $\cB(U)$ be the Borel $\sigma$-algebra. Let $(\Omega, \cF, \bP)$ be a probability space. For $p \in [1, \infty)$, let $L^p(\Omega, \cF, \bP; U)$ be the space of $p$-integrable random variables taking values in $U$. When the $\sigma$-algebra is not ambiguous, we will simply write $L^p(\Omega, \bP; U)$. Further, let $L^0(\Omega, \bP; U)$ be the space of measurable mappings $(\Omega, \cF) \mapsto (U, \cB(U))$. 

For a set $N$, we call $2^N$ the collection of subsets of $N$ and $\scP(N)$ the set of all partitions of the set $N$. This means $\scP(N) \subseteq 2^{2^N}$. A partition $P \in \scP(N)$
if and only if the following three properties are satisfied:
$$
\forall x \in N, \quad \exists p \in P: x \in p; 
\qquad
\forall p, q \in P, \quad p \cap q = \emptyset; 
\qquad
\emptyset \notin P.
$$

\subsubsection*{Hopf algebras}
\label{subsubsection:Connes-Kreimer}

For a modern compendium on Hopf algebras, we direct the reader to \cite{cartier2021hopf}. A Hopf algebra $(\scH, \odot, \triangle, \scS, \rId, \epsilon)$ is a module over a ring $R$ such that simultaneously $(\scH, \odot, \rId)$ is a unital associative algebra and $(\scH, \triangle, \epsilon)$ is a counital coassociative coalgebra and further $\scS: \scH \to \scH$ is an antiautomorphism that satisfies
\begin{equation}
\label{eq:HopfAlgebraAntipode}
\odot \circ \scS \otimes I \circ \triangle = \odot \circ I \otimes \scS \circ \triangle = \rId \epsilon. 
\end{equation}
A Hopf algebra is described as graded if there exists a monoid $(\scN, +)$ such that $\scH$ can be represented of the form
$$
\scH = \bigoplus_{n \in \scN} \scH_n, 
\quad
\scH_m \odot \scH_n \subseteq \scH_{m+n}, 
\quad
\triangle\big[ \scH_n\big] \subseteq \bigoplus_{p+q=n} \scH_p \otimes \scH_q, 
\quad
\scS[ \scH_n ] \subseteq \scH_n
$$
Further, a graded Hopf algebra is said to be connected if $\scH_0 = R$. 

Let $\scI=\{N+1, N+2, ...\}$ be an ideal of the monoid $\scN=\bN_0$ and denote $\tilde{\scN} = \scN / \scI$. We define
$$
\scH^N:= \bigoplus_{n \in \tilde{\scN}} \scH_n. 
$$
Then $\big( \scH^N, \Delta, \epsilon\big)$ is a counital subcoalgebra of $\big( \scH, \Delta, \epsilon\big)$ and $\big( \scH^N, \odot, \rId\big)$ is a quotient algebra of $\big( \scH, \odot, \rId\big)$. 

Let $G(\scH)$ be the set of characters of the Hopf algebra, linear functionals $f: \scH \to R$ that satisfy the identity
$$
f \otimes f = f \circ \odot. 
$$
The coproduct of $\scH$ induces the convolution product $\ast: G(\scH) \otimes G(\scH) \to G(\scH)$ defined by 
$$
f \ast g = f \otimes g \circ \Delta, 
$$
rendering $G(\scH)$ a group with inverse $f^{-1} = f\circ \scS$. 

Hopf algebras are one of the more ubiquitous structures in mathematics with applications in quantum field theory, condensed-matter physics, algebraic topology and deformation theory. For the purpose of this work, we draw the readers attention to the use of Hopf algebras in regularity structures, see \cite{hairer2014theory}. 

\subsubsection*{Rough paths}

The theory of rough paths, first proposed in \cite{lyons1998differential}, is now a wide ranging, multi-disciplined field of research. Over the last twenty years, the field has developed and there are now many different approaches to defining what a rough path is with differing levels of abstraction. The concept of a branched rough paths was first introduced in \cite{gubinelli2010ramification}. However, here we use a definition closer to that of the recent work \cite{tapia2020geometry}. 

Let $\big(\scH, \odot, \triangle, \rId, \epsilon, \scS\big)$ be a connected, graded Hopf algebra over a ring $R$ with basis $\scF_0$ (with the convention $\scF = \scF_0\backslash\{\rId\}$) and grading $\scG: \scF_0 \to \bN_0$ such that $\forall n\in \scN$, the module $\scH_n$ is normed. Let $\scI=\{N+1, N+2, ...\}$ be the monoid ideal of $(\bN_0, +)$ and let $\big( \scH^N, \odot, \triangle, \rId, \epsilon, \scS\big)$ be the Hopf algebra with finite grading and basis $\scF^N$. 

Let $\alpha>0$ and set $N = \big\lfloor \tfrac{1}{\alpha}\big\rfloor$. We say that $\rw:[0,1] \to G(\scH^N)$ is a $(\scH^N, \alpha)$-rough path if it satisfies that $\forall s,t,u\in [0,1]$, 
\begin{equation}
\label{eq:BranchedRP1}
\rw_{s, t} = (\rw_s)^{-1} \ast \rw_t, 
\qquad
\rw_{s,t} \ast \rw_{t,u} = \rw_{s,u}
\end{equation}
and $\forall \tau\in \scF_0^N$
\begin{equation}
\label{eq:BranchedRP2}
\Big\| \big\langle \rw_{s, t}, \tau \big\rangle \Big\|_{\scH_{\scG[\tau]}} \lesssim |t-s|^{\alpha \cdot \scG[\tau]} . 
\end{equation}

When the Hopf algebra is chosen to be the tensor shuffle algebra with the deconcatenation coproduct $\big( T(\bR^d, \bR^e), \shuffle, \triangle, \rId, \epsilon, \scS \big)$ (see \cite{reutenauer2003free}), one obtains the so-called \emph{weak geometric rough paths}. Alternatively, by choosing the Connes-Kreimer Hopf algebra (see \cite{connes1999hopf}), one obtains the so-called \emph{Branched rough paths}. 

In \cite{hairer2014theory}, this concept was generalised to solve singular stochastic partial differential equations. Regularity structures use an abstract Taylor expansion that best approximates the solution to determine the relevant necessary information about the driving noise to solve an equation. Although the first ideas of regularity structures were described a decade ago, the fundamental philosophies can be traced back many years before. 

\subsubsection*{Controlled rough paths}

The central ambition of the theory of controlled rough paths is to describe a collection of paths that are rough enough to allow for a rich solution theory for rough differential equations but structured enough that one can ``\emph{integrate}'' them with respect to some rough signal. The collection of controlled rough paths is favourable to work with due to its linearity and that one can integrate with respect to a fixed rough signal without worrying about the existence of the resulting integral. However, they do not provide a natural approximation theory and the associated \emph{Gubinelli derivatives} are often not uniquely defined. 

Let $\rw:[0,1] \to G(\scH^N)$ be a $(\scH^N, \alpha)$-rough path. We say that $\bX:[0,1] \to \scH^{N-1}$ is a \emph{controlled rough path} (controlled by $\rw$) if $\forall s, t \in [0,1]$ and $\forall \rho \in \scF_0^{N-1}$
\begin{equation}
\label{eq:Controlled_RP1}
\begin{split}
\big\langle \bX_{s, t}, \rId\big\rangle =& \sum_{\tau \in \scF^{N-}} \big\langle \bX_s, \tau \big\rangle \cdot \big\langle \rw_{s, t}, \tau \big\rangle + \big\langle \bX_{s, t}^{\sharp}, \rId\big\rangle
\\
\big\langle \bX_{s, t}, \rho \big\rangle =& \sum_{\varsigma, \tau \in \scF^{N-}} c'\big( \varsigma, \tau, \rho\big) \cdot \big\langle \bX_s, \varsigma \big\rangle \cdot \big\langle \rw_{s, t}, \tau \big\rangle + \big\langle \bX_{s, t}^{\sharp}, \rho \big\rangle
\end{split}
\end{equation}
where $c': \scF_0 \times \scF_0 \times \scF_0 \to \bN_0$ is the reduced coproduct counting function and
\begin{equation}
\label{eq:Controlled_RP2}
\sup_{s, t\in [0,1]} \frac{ \Big| \big\langle \bX_{s, t}^{\sharp}, \rId\big\rangle \Big|}{|t-s|^{N \alpha}} < \infty, 
\qquad
\sup_{s, t \in [0,1]} \frac{\Big| \big\langle \bX_{s, t}^{\sharp}, \rho \big\rangle \Big|}{|t-s|^{(N - \scG[\rho])\cdot \alpha}}< \infty. 
\end{equation}

It is now well known that given a controlled rough path $\bX$ controlled by $\rw$, we can define the rough integral
$$
\int_0^1 \bX_r d\rw_r = \lim_{|D_n| \to 0} \sum_{[u, v] \in D_n} \sum_{T \in \scF_0^{N-1}} \big\langle \bX_u, \tau \big\rangle \cdot \big\langle \rw_{u, v}, \lfloor \tau\rfloor \big\rangle. 
$$
Further, given a $N$-times differentiable function $f:R \to R$ and a controlled rough path $\bX$ controlled by a rough path $\rw$, we can find another controlled rough path $\bZ$ (controlled by $\rw$) that satisfies
$$
\big\langle \bZ_s, \rId \big\rangle = f\Big( \big\langle \bX_s, \rId\big\rangle \Big), 
\quad 
\big\langle \bZ_s, \tau \big\rangle = \sum_{i=1}^{N-1} \frac{\nabla^i f\Big( \big\langle \bX_s, \rId\big\rangle \Big)}{i!} \cdot \Bigg[ \sum_{\substack{\tau_1, ..., \tau_i\in \scF \\ \odot_{j=1}^i\tau_j = \tau }} \bigotimes_{j=1}^i \big\langle \bX_s, \tau_j\big\rangle \Bigg]. 
$$


\newpage
\section{Taylor expansions over the Wasserstein space}
\label{section:TaylorExpansions}

Motivated by the differential equation \eqref{eq:meanfield:equation}, we want to consider some Taylor expansion for a function
$$
f: \bR^e \times \cP_2(\bR^e) \to \lin(\bR^d, \bR^e)
$$
where $e$ is the dimension of the solution process and $d$ is the dimension of the driving signal. However, to streamline the notation somewhat for the reader, in this section we will simply consider 
$$
f: \bR^d \times \cP_2(\bR^d) \to \bR^e. 
$$
We emphasise that this does not change the mathematics beyond the dimension of the associated vector spaces. 

Taylor's Theorem is a well-known result that states that for a function $f$ that is $n$ times differentiable, we have
$$
f(x) - f(y) = \sum_{k=1}^n \frac{\nabla^k f(y)}{n!} [ (x-y)^{\otimes k}] + O\Big( |x-y|^{n+1} \Big). 
$$
Our objective is to address a similar version for functionals depending on a measure argument. 

Throughout this paper, the derivatives we consider are thus constructed on the space $\cP_{p}(\bR^d)$ for $p=2$, the so-called `Wasserstein space' of probability measures with a finite second moment. For any $p \geq 1$, $\cP_p(\bR^d)$ can be equipped with the $\bW^{(p)}$-Wasserstein distance, defined by:
\begin{equation}
\label{eq:WassersteinDistance}
\bW^{(p)}(\mu,\nu) = \inf_{\Pi \in \cP_{p}(\bR^d \times \bR^d)}
\biggl( \int_{\bR^d \times \bR^d}
| x-y |^p
d \Pi(x,y) \biggr)^{1/p},
\end{equation}
the infimum being taken with respect to all the probability measures $\Pi$ on the product space $\bR^d \times \bR^d$ with $\mu$ and $\nu$ as respective $d$-dimensional marginal laws.

Higher order Taylor expansions are central to approximation techniques throughout the mathematical sciences. Therefore, it is perfectly natural to desire a differential calculus on the space of measures when attempting to approximate the dynamics of large populations and their associated mean-field limits. The origins of this philosophy can be found in \cite{Jordan1998variation} where the connection between the Fokker-Planck equations and gradient flows on the Wasserstein space is first established. We refer the reader to the monograph \cite{Ambrosio2008Gradient} for a complete overview of the subject. 

\subsection{Integer-valued sequences with a 1-Lip sup envelope}
\label{subsection:1-Lip_sup_envelope}

We build a series of differential operators on the $2$-Wasserstein space. For a function $f:\cP_2(\bR^d) \to \bR^e$, we consider the canonical lift $F: L^2(\Omega, \cF, \bP; \bR^d) \to \bR^e$ defined by $F(X) = f( \bP\circ X^{-1})$. We say that $f$ is $L$-differentiable at $\mu$ if $F$ is Fr\'echet differentiable at some point $X$ such that $\mu = \bP\circ X^{-1}$. Denoting the Fr\'echet derivative by $DF$, it is now well known (see for instance \cite{GangboDifferentiability2019} that $DF$ is a $\sigma(X)$-measurable random variable of the form $DF(\mu, \cdot):\bR^d \to \lin(\bR^d, \bR^e)$ depending on the law of $X$ and satisfying $DF(\mu, \cdot) \in L^2\big( \bR^d, \cB(\bR^d), \mu; \lin(\bR^d, \bR^e) \big)$. We denote the $L$-derivative of $f$ at $\mu$ by the mapping $\partial_\mu f(\mu)(\cdot): \bR^d \ni v \to \partial_\mu f(\mu)(v) \in \lin(\bR^d, \bR^e)$ satisfying $DF(\mu, X) = \partial_\mu f(\mu)(X)$. This derivative is known to coincide with the so-called Wasserstein derivative, as defined in for instance \cite{Ambrosio2008Gradient}, \cite{CarmonaDelarue2017book1} and \cite{GangboDifferentiability2019}. As we explained in the introduction, Lions' approach is well-fitted to probabilistic approaches for mean-field models since, very frequently, we have  a `canonical' random 
variable $X$ for representing the law of a given probability measure $\mu$. 

The second order derivatives are obtained by differentiating $\partial_{\mu} f$ with respect to $v$ (in the standard Euclidean sense) and $\mu$ (in the same Lions' sense). The two derivatives $\nabla_{v} \partial_{\mu} f$ and $\partial_{\mu} \partial_{\mu} f$ are thus very different functions: The first one is defined on ${\cP}_{2}({\bR}^d) \times {\bR}^d$ and writes $(\mu,v) \mapsto \nabla_{v} \partial_{\mu} f(\mu,v)$ whilst the second one is defined on ${\cP}_{2}({\bR}^d) \times {\bR}^d \times {\bR}^d$ and writes $(\mu,v,v') \mapsto \partial_{\mu} \partial_{\mu} f(\mu,v,v')$. The $d$-dimensional entries of $\nabla_{v} \partial_{\mu} f$ and $\partial_{\mu} \partial_{\mu} f$ are called here the \textit{free} variables, since they are integrated with respect to the measure $\mu$ itself. In words, $v$ is the free variable of 
$\partial_{\mu} f$ and $(v,v')$ are the free variables of $\partial_{\mu} \partial_{\mu} f$. Accordingly, the quadratic form $L^2(\Omega,\cF,\bP; \bR^d)$ associated on with respect to these two second-order derivatives is
\begin{align*}
L^2(\Omega,\cF,\bP; \bR^d) \ni X \mapsto &\bE^1 \Big[ \nabla_{v} \partial_{\mu} f \big( \mu, X(\omega_1) \big) \cdot X(\omega_1) \otimes X(\omega_1) \Big] 
\\
&+ \bE^{1,2} \Big[ \partial_{\mu} \partial_{\mu}f \big( \mu, X(\omega_1), X(\omega_2) \big) \cdot X(\omega_1) \otimes X(\omega_2) \Big]. 
\end{align*}
In the first term of the right-hand side, the expectation makes sense if 
$$
\nabla_{v} \partial_{\mu}f(\mu,X) \in L^\infty(\Omega,\cF,\bP; \bR^d),
$$
which is the case if $\partial_{\mu} f$ is Lipschitz continuous in $v$.

Despite the obvious differences between the two second order derivatives, both capture necessary information for the Taylor expansion and we want to find a common system of notation that easily extends to higher order derivatives. This leads us to Definition \ref{def:a} below, the principle of which can be stated as follows for the first and second order derivatives: The derivative symbol $\partial_{\mu}$ can be denoted by $\partial_{1}$ and then the two derivative symbols $\nabla_{v} \partial_{\mu}$ and $\partial_{\mu} \partial_{\mu}$ can be respectively denoted by $\partial_{(1,1)}$ and $\partial_{(1,2)}$. In the first case, the length of the index is 1, hence indicating that the derivative is of order 1. In the other two cases, the length of the vector-valued index is 2, indicating that the derivative is of order 2. Also, in the notation $\partial_{(1,1)}$, the repetition of the index $1$ indicates that we use the same free variable for the second order derivative, or equivalently that the second derivative has to be $\nabla_{v}$. In the notation $\partial_{(1,2)}$, the fact that the second index (in $(1,2)$) is different from the first one says that we use a new free variable for the second order derivative, which, in turn, must be $\partial_{\mu} \partial_{\mu}$. 

\begin{definition}
\label{def:a}
The sup-envelope of an integer-valued sequence $(a_{k})_{k=1,...,n}$ of length $n$ is the non-decreasing sequence $(\max_{l=1,...,k} a_{l})_{k=1,...,n}$. The sup-envelope is said to be $1$-Lipschitz (or just $1$-Lip) if, for any $k \in \{2,...,n\}$, 
\begin{equation*}
\max_{l=1,...,k} a_{l} \leq 1+ \max_{l=1,...,k-1} a_{l}.
\end{equation*}
We call $A_{n}$ the collection of all ${\bN}$-valued sequences of length $n$, with $a_{1}=1$ as initial value and with a 1-Lip sup-envelope. Thus $A_n$ is the collection of all sequences $(a_k)_{k=1, ..., n} \in A_n$ taking values on $\{1, ..., n\}$ such that 
$$
a_1 = 1, \quad a_k \in \Big\{1, ..., 1+ \max_{l=1, ..., k-1} a_l \Big\}. 
$$
We refer to $A_n$ as the collection of \emph{partition sequences}. The length $n$ of $a$ is denoted by $|a|$ and the maximum $\max_{l=1,...,n} a_{l}$ is denoted by $m[a]$. Moreover, we let $l[a] \in \bN^{m[a]}$ be the $m[a]$-tuple $\big( l[a]_{i} = \sum_{k=1}^n {\mathbf 1}_{i}(a_{k}) \big)_{i=1, ..., m[a]}$, i.e., for $i \in \{1, ..., m[a]\}$, $l[a]_i$ is the number of entries `$i$' in the sequence $a$. 
\end{definition}

\begin{example}
We have
\begin{align*}
A_0 =& \emptyset, \quad A_1 = \Big\{ (1)\Big\}, \quad A_2 = \Big\{ (1,1), (1,2) \Big\}, 
\\
A_3 =& \Big\{ (1,1,1), (1,1,2), (1,2,1), (1,2,2), (1,2,3) \Big\}, 
\\
A_4 =& \Big\{ (1,1,1,1), (1,1,1,2), (1,1,2,1), (1,1,2,2), (1,1,2,3), (1,2,1,1), (1,2,1,2), 
\\
& \quad (1,2,1,3), (1,2,2,1), (1,2,2,2), (1,2,2,3), (1,2,3,1), (1,2,3,2), (1,2,3,3), (1,2,3,4) \Big\}, 
\end{align*}
and so on. 
\end{example}

Observe that by construction, an element $a=(a_{k})_{k=1,...,n}$ is 
a surjective mapping from $\{1,...,n\}$ onto $\{1,...,\max[a]\}$. However, it must be stressed that the representation of the arrival set does not matter so much for our purpose. In short, any other arrival set of cardinality $m[a]$ could be used in our analysis. In fact, what really matters in our definition of a sequence $a=(a_{k})_{k=1,...,n} \in A_{n}$ are the repetitions, since they permit us to distinguish between \textit{already used} free variables and \textit{new} free variables below. This idea may be formalised by identifying integer-valued sequences that can be labelled by the same element of $A_{n}$. For $a \in A_{n}$, we can indeed call $\llbracket a\rrbracket$ the collection of all sequences $(b_{1},...,b_{n})$ of length $n$ taking values in $\bN$ that take the form
$$
(b_{i})_{i=1,...,n} = (k_{a_i})_{i=1, ..., n} \, ; \, k_1,..., k_{\max a} \in \bN, k_i \neq k_j.
$$
For example
$$
(1, 2), \  (2, 1), \  (4, 5) \in \bigl\llbracket (1,2)\bigr\rrbracket.  
$$
For two sequences $(b_{1},...,b_{n})$ and 
$(b_{1}',...,b_{n}')$, we write $(b_{1},...,b_{n}) \equiv 
(b_{1}',...,b_{n}')$ if they belong to the same $\llbracket a\rrbracket$.

\subsection{Lions-Taylor expansion}

We now have all the ingredients needed to define the symbols associated with higher-order Lions' derivatives. Indeed, for $n\in \bN$ and $a\in A_n$, we are willing to define $\partial_a$ inductively by
\begin{align*}
\partial_{(1)} =& \partial_\mu, 
\\
\partial_{(a_1, ..., a_{k-1}, a_k)} =&
\begin{cases} 
\nabla_{v_{a_k}} \cdot \partial_{(a_1, ..., a_{k-1})} & \quad a_k \leq \max \{a_1, ..., a_{k-1}\}, 
\\
\partial_\mu \cdot \partial_{(a_1, ..., a_{k-1})} & \quad a_k > \max \{a_1, ..., a_{k-1}\}. 
\end{cases}
\end{align*}
To be consistent with our discussion in the previous subsection, we start with the following reminder, taken for instance from \cites{buckdahn2017mean, chassagneux2014classical, CarmonaDelarue2017book1}:

\begin{definition}
\label{definition:Twice-Different}
We say that a function $f:\cP_2(\bR^d) \to \bR^e$ is in $C_b^{(2)}\big( \cP_2(\bR^d); \bR^e \big)$ if 
\begin{itemize}
\item $f$ is continuously Lions-differentiable with Lions derivative $\partial_\mu f:\cP_2(\bR^d) \times \bR^d \to \lin( \bR^d, \bR^e)$ . 
\item For every $\mu \in \cP_2(\bR^d)$, the $\mu$-measurable function $\partial_\mu f(\mu, \cdot)$ is differentiable with bounded and Lipschitz derivative $\nabla_v \partial_\mu f$ (with the Lipschitz property with respect to $\mu$ being for $\bW^{(1)}$) that satisfies
$$
\nabla_v \partial_\mu f: \cP_2(\bR^d) \times \bR^d \to \lin\big( (\bR^d)^{\otimes 2}, \bR^e\big). 
$$ 
\item For every $v\in \bR^d$, the function $\partial_\mu f(\cdot, v)$ has Lions derivative 
$$
\partial_\mu \partial_\mu f: \cP_2(\bR^d) \times \bR^d \times \bR^d \to \lin\big( (\bR^d)^{\otimes 2} , \bR^e\big)
$$
which is bounded and Lipschitz (with the Lipschitz property with respect to $\mu$ being for $\bW^{(1)}$). 
\end{itemize}
\end{definition}
When there is no ambiguity, we will often drop the output space and write $C_b^{(2)}\big( \cP_2(\bR^d) \big)$. Before we go on into the generalisation of the above definition, we feel useful to make the following remarks:
\begin{remark}
\label{remark:WassersteinRemark1}
As pointed out in the definition, the Lipschitz property with respect to the measure argument in Definition \ref{definition:Twice-Different} is understood as being for the aforementioned $\bW^{(1)}$-distance. The Lipschitz properties on the product spaces $\cP_{2}(\bR^d) \times \bR^d$ and $\cP_{2}(\bR^d) \times \bR^d \times \bR^d$ are understood for the corresponding product distances, $\bR^d$ being equipped with the Euclidean norm. Our choice to impose Lipschitz continuity with respect to the $\bW^{(1)}$-distance, which is obviously coarser than $\bW^{(2)}$, is explained in Remark \ref{remark:WassersteinRemark3} below. 
\end{remark}

\begin{remark}
\label{remark:WassersteinRemark2}
The requirement to have joint continuity with respect to all the arguments is in fact a strong requirement, which is known to be suboptimal in practical applications. Indeed, Lions' derivative $\partial_{\mu} u(\mu,v)$ is typically `well-defined' at elements $v \in {\mathbb R}^d$ that belong to the support of $\mu$. Put differently, the definition of the derivative outside the support of $\mu$ is somewhat arbitrary in the sense that any choice outside the support leads to a convenient  derivative. However, things become much more rigid when global continuity is imposed, as is the case here. In this case, the values of $\partial_{\mu} u(\mu,v)$ for $v$ outside the support are necessarily prescribed since we can always write $\partial_{\mu} u(\mu,v)=\lim_{n \rightarrow \infty} \partial_{\mu} u(\mu_{n},v)$, where $(\mu_{n})_{n \geq 1}$ is a sequence of fully supported probability measures that converges in $\cP_2(\bR^d)$ towards $\mu$. 

Thus, there is a slight loss of generality in our definition. Actually, the same restriction is imposed in \cites{buckdahn2017mean, CarmonaDelarue2017book2, 2019arXiv180205882.2B}. Handling the general case leads to many technicalities, even when the derivatives that are studied are of order 2, see \cite{chassagneux2014classical} together with \cite{CarmonaDelarue2017book1}. 
\end{remark}

\begin{remark}
\label{remark:WassersteinRemark3}
The boundedness requirements on $\partial_{\mu} f$ and $\partial_{\mu} \partial_{\mu} f$ are also more demanding than what the general theory could allow. Typically, the Lions derivative of a function that is continuously differentiable and Lipschitz continuous is bounded in $L^2$, i.e., $\sup_{\mu \in \cP_{2}(\bR^d)} \int_{\bR^d} |v|^2 d\mu(v) < \infty$, and not globally in $L^\infty$, as we require here. 

In fact, it is pretty easy to see that requiring the derivative to be globally bounded imposes the function $f$ to be globally Lipschitz for the $\bW^{(1)}$-distance, which is obviously stronger. Once again, similar restrictions are imposed in \cites{buckdahn2017mean, CarmonaDelarue2017book2, 2019arXiv180205882.2B}, and handling the general case leads to cumbersome technicalities. 

In the end, this explains why in Remark \ref{remark:WassersteinRemark1} we decided to require Lipschitz continuity for $\bW^{(1)}$. 
\end{remark}

In particular, a function $f\in C^{(2)}_b\big( \cP_2(\bR^d) \big)$ satisfies
$$
\partial_\mu f(\mu, \cdot)\in C_b^1\Big(\bR^d; \lin\big( \bR^d, \bR^d\big) \Big), 
$$
that is; it is both bounded, $\mu$-measurable and differentiable.  

\begin{definition}
\label{def:general:Lions:derivative}
We say that a function $f:\cP_2(\bR^d) \to \bR^e$ belongs to $C_b^{(n)}\big( \cP_2(\bR^d); \bR^e \big)$ if there exists a collection of functions $(\partial_{a} f)_{a \in \cup_{k=1}^n A_{k}}$ such that:

\begin{enumerate}
\item For any $k \in \{0, 1,.., n\}$, for any $a \in A_{k}$
\begin{align*}
\partial_{a} f : \cP_{2}( \bR^d) \times ( \bR^d)^{\times m[a]} &\rightarrow
\lin\Big( (\bR^d)^{\otimes k}, \bR^e\Big)
\\
\big( \mu, v_{1},...,v_{m[a]} \big) &\mapsto \partial_{a} f(\mu,v_{1}, ..., v_{m[a]}).
\end{align*}

\item For any $k \in \{1,...,n\}$ and any $a \in A_{k}$,  the function 
$\partial_{a} f$ is bounded and Lipschitz continuous on $\cP_{2}( \bR^d) \times (\bR^d)^{\times m[a]}$, the first space being equipped with the $1$-Wasserstein distance.

\item For any $k \in \{1,...,n-1\}$ and any $a \in A_{k}$, the function 
$\partial_{a} f$ is differentiable with respect to $(v_{1},...,v_{m[a]})$ and 
$$
\partial_{v_{j}} \partial_{a} f = \partial_{(a_{1},\cdots,a_{k},j)} f.
$$

\item For any $k \in \{1,\cdots,n-1\}$ and any $a \in A_{k}$, the function 
$\partial_{a} f$ is differentiable with respect to $\mu$ and
\begin{equation*}
\partial_{\mu} \partial_{a} f = \partial_{(a_{1},\cdots,a_{k},m[a] + 1)} f.
\end{equation*}
\end{enumerate}
\end{definition}

As with Remark \ref{remark:WassersteinRemark3}, we directly require all the derivatives $\partial_{a} f$ to be Lipschitz continuous with respect to $\bW^{(1)}$. In fact, this only makes a difference for the derivatives $\partial_{a} f$, with $|a|=n$: These derivatives should just be required to be $\bW^{(2)}$-Lipschitz continuous if we wanted to fit the standard construction of the Lions' derivative. Whenever $|a| \leq n-1$, we have by assumption that $\partial_{\mu} \partial_{a} f$ is bounded which, by the same third item, implies that $\partial_{a} f$ is necessarily $\bW^{(1)}$-Lipschitz continuous.

On the road to a general Lions-Taylor expansion, we recall the first-order expansion, which underpins the very definition of the Lions derivative. For any two $\mu$ and $\nu$ in $\cP_{2}(\bR^d)$, let $\Pi^{\mu, \nu}$ be a measure on $(\bR^d)^{\oplus 2}$ with marginal distributions $\mu$ and $\nu$. Then, for a continuously Lions differentiable function $f : \cP_{2}(\bR^d) \rightarrow \bR^e$,
\begin{equation}
\label{eq:n=1:ito:lions}
f(\nu) - f(\mu) = \int_{(\bR^d)^{\oplus 2}} \partial_{\mu} f(\mu,u) \cdot (v-u) d \Pi^{\mu,\nu}(u,v) 
+ o \biggl[ \biggl( 
\int_{(\bR^d)^{\oplus 2}}
|u-v|^2 d\Pi^{\mu,\nu}(u,v)
\biggr)^{1/2} \biggr]. 
\end{equation}
In fact, the remainder can be explicitly written out:
\begin{align}
\nonumber
o \biggl[& \biggl( \int_{(\bR^d)^{\oplus 2}} \vert u-v \vert^2 d\Pi^{\mu,\nu}(u,v) \biggr)^{1/2} \biggr]
\\
\label{eq:n=1:expression:remainder}
&=  
\int_{0}^1 \bigg( \partial_{\mu} f \Big( \Pi_{\xi}^{\mu,\nu}, u + \xi (v-u) \Big) - \partial_{\mu} f \Big( \mu, u \Big) \bigg) \cdot (v-u) d \Pi^{\mu,\nu}(u,v),
\end{align}
with the notation
\begin{equation}
\label{eq:Pixi}
\Pi_\xi^{\mu, \nu} = \Pi^{\mu, \nu} \circ \Big( u + \xi(v-u)\Big)^{-1}. 
\end{equation}
In particular, when $f$ is in $C^{(1)}(\cP_{2}(\bR^d))$ in the sense of 
Definition \ref{def:general:Lions:derivative}, the Landau symbol in \eqref{eq:n=1:ito:lions} can be easily upper bounded by
\begin{equation}
\label{eq:n=1:bound:remainder}
\begin{split}
o \Bigg( \biggl( \int_{(\bR^d)^{\oplus 2}} \vert u-v \vert^2 d\Pi^{\mu,\nu}(u,v) \biggr)^{1/2} \Bigg)
&\leq O \Bigl( \bW^{(2)}(\mu,\nu)^2 \Bigr).
\end{split}
\end{equation}
Finally, observe that there is no other constraint on the probability measure $\Pi^{\mu,\nu}$ than it being a \textit{coupling} of $\mu$ and $\nu$. There is no need to require any optimality (say for instance in the sense of Equation \eqref{eq:WassersteinDistance}) in the choice of the coupling. In fact, the possible accuracy of the coupling (for the $L^2$ norm) reads not only in the first term in the right-hand side but also in the second term.

In order to generalise \eqref{eq:n=1:ito:lions}, we define, for $a\in A_n$,
the corresponding differential operator, which is acting on elements $f \in C_b^{(n)}\big( \cP_2(\bR^d); \bR^e \big)$ in the following way:
\begin{align}
\nonumber
\rD^a &f(\mu)[\Pi^{\mu, \nu}] 
\\
\label{eq:D^a:without:0}
&= \underbrace{\int_{(\bR^d)^{\oplus 2}} ... \int_{(\bR^d)^{\oplus 2}} }_{\times m[a]} \partial_a f\Big( \mu, u_1, ..., u_{m[a]} \Big) \cdot \bigotimes_{i=1}^n ( v_{a_i} - u_{a_i}) d\Pi^{\mu, \nu}(u_1, v_1)  ... d\Pi^{\mu, \nu}(u_{m[a]}, v_{m[a]} ) . 
\end{align}

\begin{theorem}[Lions-Taylor Theorem]
\label{theorem:LionsTaylor1}
Let $n \in \bN$ and let $f\in C_b^{(n)} \big( \cP_2(\bR^d); \bR^e \big)$. Then for any $\mu, \nu \in \cP_{n+1} (\bR^d)$ with joint distribution $\Pi^{\mu, \nu}$, we have that
\begin{align}
\label{eq:TaylorExpansion}
f(\nu) - f(\mu) &= \sum_{k=1}^n \sum_{a\in A_k} \frac{\rD^a f(\mu)[\Pi^{\mu, \nu}] }{k!} + R_n^{a, \Pi^{\mu, \nu}}(f), 
\end{align}
where
\begin{align*}
&R_n^{a, \Pi^{\mu, \nu}}(f) 
\\
&= \frac{1}{(n-1)!} \sum_{a\in A_{n}} \underbrace{\int_{(\bR^d)^{\oplus 2}} ... \int_{(\bR^d)^{\oplus 2}} }_{\times m[a]} f^{a, \mu, \nu} \cdot \bigotimes_{i=1}^n ( v_{a_i} - u_{a_i}) d\Pi^{\mu, \nu}(u_1, v_1)  ... d\Pi^{\mu, \nu}(u_{m[a]}, v_{m[a]} ), 
\end{align*}
and
$$
f^{a, \mu, \nu} = \int_0^1 \Big( \partial_a f(\Pi^{\mu, \nu}_\xi, u_1 + \xi(v_1 - u_1), ..., u_{m[a]} + \xi(v_{m[a]} - u_{m[a]} )) - \partial_a f(\mu, u_1, ..., u_{m[a]} )\Big) (1-\xi)^{n-1} d\xi. 
$$
The probability measure $\Pi^{\mu,\nu}_{\xi}$ is defined as in \eqref{eq:Pixi}. Further, the remainder term on the second line of \eqref{eq:TaylorExpansion} can be upper bounded by:
\begin{equation}
\label{eq:bound:remainder:TaylorExpansion}
\bigl\vert 
R_n^{a, \Pi^{\mu, \nu}}(f)
\bigr\vert
\leq C \int_{(\bR^d)^{\oplus 2}} |u-v|^{n+1} d\Pi^{\mu, \nu}(u, v),
\end{equation}
for a constant $C$ depending only on the bounds for $f$ and its derivatives (including the Lipschitz bounds).
\end{theorem}

\begin{remark}
Notice that the integrability property at order $n+1$ (imposed on $\mu$ and $\nu$) explicitly appears in the bound \eqref{eq:bound:remainder:TaylorExpansion}. This is consistent with \eqref{eq:n=1:bound:remainder}, which corresponds to $n=1$.  
\end{remark}

\begin{proof}
Firstly, we observe that \eqref{eq:n=1:ito:lions}-\eqref{eq:n=1:expression:remainder} can be rewritten in the form:
\begin{align*}
f(\nu) - f(\mu) =\rD^{(1)} f(\mu)[ \Pi^{\mu, \nu}] + \int_{(\bR^d)^{\oplus 2}} f^{(1), \mu, \nu}\cdot (v_1-u_1) d\Pi^{\mu, \nu}(u_1, v_1). 
\end{align*}
Together with \eqref{eq:n=1:bound:remainder}, this gives the result when $n=1$.

Now we proceed by induction on $n$. Consider an integer $n \geq 2$ such that the conclusion of the statement holds true for any $a\in A_{n-1}$ and any $f \in C_b^{(n-1)} \big( \cP_2(\bR^d); \bR^e \big)$. In turn, for $f \in C_b^{(n)} \big( \cP_2(\bR^d);\bR^e \big)$ and  $a \in A_{n-1}$, we have that $\partial_a f: \cP_2(\bR^d) \times (\bR^d)^{\times m[a]} \to \lin\big( (\bR^d)^{\otimes |a|}, \bR^e \big)$ is differentiable in all variables and the derivatives are bounded and Lipschitz. Hence
\begin{align*}
&f^{a, \mu, \nu} = \int_0^1 \int_0^\xi \frac{d\Big( \partial_a f\big(\Pi^{\mu, \nu}_\theta, u_1 + \theta(v_1 - u_1), ..., u_{m[a]} + \theta(v_{m[a]} - u_{m[a]} )\big) \Big)}{d\theta} d\theta (1-\xi)^{n-2} d\xi
\\
&=\int_0^1 \int_0^\xi \bigg( \int_{(\bR^d)^{\oplus 2}} \partial_\mu \partial_a f\big( \Pi_\theta^{\mu, \nu}, u_1 + \theta(v_1-u_1), ..., u_{m[a]+1} + \theta(v_{m[a]+1} - u_{m[a]+1} )\big) 
\\
& \hspace{40pt} \cdot (v_{m[a]+1} - u_{m[a]+1}) d\Pi^{\mu, \nu} (u_{m[a]+1}, v_{m[a]+1}) \bigg) d\theta (1-\xi)^{n-2} d\xi 
\\
&+\int_0^1 \int_0^\xi \sum_{i=1}^{m[a]} \nabla_{v_i} \partial_a f\Big( \Pi_\theta^{\mu, \nu}, u_1+\theta(v_1- u_1), ..., u_{m[a]} + \theta(v_{m[a]} - u_{m[a]}) \Big) \cdot (v_i - u_i) d\theta (1-\xi)^{n-2} d\xi
\\
&=\tfrac{1}{n(n-1)} \int_{(\bR^d)^{\oplus 2}} \partial_\mu \partial_a f\Big( \mu, u_1, ..., u_{m[a]+1}\Big) \cdot (v_{m[a]+1} - u_{m[a]+1})d\Pi^{\mu, \nu}(u_{m[a]+1}, v_{m[a]+1}) 
\\
&+\tfrac{1}{n(n-1)} \sum_{i=1}^{m[a]} \nabla_{v_i} \partial_a f\Big( \mu, u_1, ..., u_{m[a]} \Big) \cdot (v_i - u_i) 
\\
&+ \tfrac{1}{n-1} \int_0^1 \bigg(  \int_{(\bR^d)^{\oplus 2}} \partial_\mu \partial_a f\Big( \Pi_\xi^{\mu, \nu}, ..., u_{m[a]+1}+ \xi( v_{m[a]+1} - u_{m[a]+1}) \Big) - \partial_\mu \partial_a f\Big(\mu, ..., u_{m[a]+1} \Big) 
\\
&\hspace{40pt} \cdot ( v_{m[a]+1} - u_{m[a]+1}) d\Pi^{\mu, \nu}(u_{m[a]+1}, v_{m[a]+1})\bigg) (1-\xi)^{n-1} d\xi
\\
&+\tfrac{1}{n-1} \int_0^1 \sum_{i=1}^{m[a]} \bigg( \nabla_{v_i} \partial_a f\Big( \Pi_\xi^{\mu, \nu}, ..., u_{m[a]} + \xi( v_{m[a]} - u_{m[a]}) \Big) - \nabla_{v_i}\partial_a f\Big( \mu, ..., u_{m[a]}\Big) \bigg) 
\\
&\hspace{40pt} \cdot (v_i - u_i) (1-\xi)^{n-1} d\xi. 
\end{align*}
Substituting this into Equation \eqref{eq:TaylorExpansion} at rank $n-1$, we get the same expansion at rank $n$.

As for the estimate \eqref{eq:bound:remainder:TaylorExpansion} of the remainder, we have that the error term $R_n^{a, \mu, \nu}(f)$ satisfies that
\begin{align*}
R_n^{a, \Pi^{\mu, \nu}}(f)
=&
O \Bigg( \bW^{(2)} (\mu, \nu) \cdot \prod_{j=1}^{m[a]} \int_{(\bR^d)^{\oplus 2}} |u-v|^{l[a]_j} d\Pi^{\mu, \nu}(u, v)  \Bigg)
\\
&+ \sum_{i=1}^{m[a]} O\Bigg( \prod_{j=1}^{m[a]} \int_{(\bR^d)^{\oplus 2}} |u-v|^{l[a]_j + \delta_{i,j}} d\Pi^{\mu, \nu}(u, v) \Bigg)
\\
=& O\Bigg( \int_{(\bR^d)^{\oplus 2}} |u-v|^{n+1} d\Pi^{\mu, \nu}(u, v) \Bigg), 
\end{align*}
where we used H\"older's inequality together with the fact that $l[a]_{1}+\cdots+l[a]_{m[a]}=n$.
\end{proof}

\subsection{Multivariate Lions-Taylor expansion}
\label{subsect:Multivariate_Lions-Taylor}

The Lions-Taylor expansion given in the statement of Theorem \ref{theorem:LionsTaylor1} cannot suffice for the study of mean-field equations of the form \eqref{eq:meanfield:equation}, as we need to consider functionals depending on both a Euclidean variable $x$ and 
a measure argument $\mu$. To address this increase in complexity, we must revisit the framework introduced in Section \ref{subsection:1-Lip_sup_envelope}, in order to have a convenient system of notation for the mixed derivatives with respect to $x$ and $\mu$. 

Indeed, unlike the derivative $\partial_{a} f$ in item $(i)$ of Definition \ref{def:general:Lions:derivative}, in which the $m^{\rm th}$ variable $v_{m}$ can only appear if $\partial_{a} f$ contains at least $m$ derivatives with respect to $\mu$, the variable $x$ now appears 
in any derivatives of $f$ whenever $f$ is a function of the form $f(x,\mu)$.  

We clarify this in the next definition. Intuitively, derivatives with respect to the $x$-component are encoded in the corresponding sequence $a$ through insertions of a `$0$'. Repeated $0$'s thus account for repeated derivatives in the direction of $x$. There is no constraint on the way that those $0$'s may appear in the corresponding $a$. 

\begin{definition}
Let $k, n\in \bN_0$ and denote $(a\cdot b)$ to be the concatenation of two sequences $a$ and $b$. Let $A_{k,n}$ be the collection of all sequences 
$a' = (a_i')_{i=1, ..., k+n}$ of length $k+n$ taking values in $\{0, ..., n\}$ of the form $a' = \sigma( (0) \cdot a)$, where $(0) = (0)_{i=1, ..., k}$ is the sequence of length $k$ with all entries $0$, $a\in A_n$ and $\sigma$ is a $(k,n)$-shuffle, i.e., a permutation of $\{1, ..., n+k\}$ such that $\sigma(1) < ... < \sigma(k)$ and $\sigma(k+1)<... < \sigma(n+k)$. 
\end{definition}

The rationale for requiring $\sigma$ to be a $(k,n)$-suffle is that the positive entries of $a$, which encode the derivatives with respect to the measure argument $\mu$, obey the prescriptions of Definition  \ref{def:a}. $k$ is the number of $0$ in the sequence and $n$ is the number of non-zero values in the sequence. 

\begin{example}
We have
\begin{align*}
A_{1,1} =& \Big\{ (0,1), (1,0) \Big\}, 
\quad 
A_{1,2} = \Big\{ (0,1,1), (0,1,2), (1,0,1), (1,0,2), (1,1,0), (1,2,0) \Big\}, 
\\
A_{2,1} =& \Big\{ (0,0,1), (0,1,0), (1,0,0) \Big\}, 
\quad
A_{0,3} = \Big\{ (1,1,1), (1,1,2), (1,2,1), (1,2,2), (1,2,3) \Big\}
\end{align*}
and so on. 
\end{example}

Moreover, for $a\in A_{k,n}$, we call $\llbracket a \rrbracket$ the collection of all sequences $(p_{1},...,p_{n+k})$ of length $k+n$ taking values on $\bN_0$ that take the form
$$
(p_{a_i})_{i=1, ..., n}; \ p_0 = 0, \ p_1,..., p_{\max a} \in \bN, \ p_i\neq p_j \ (\textrm{\rm for} \ i \not = j). 
$$

\begin{definition}
\label{definition:A_0^n}
For a given $n \in {\mathbb N}$, we let 
\begin{equation*}
\A{n}= \bigcup_{k=0}^n A_{k,n-k}.
\end{equation*}
Then, for $a \in \bigcup_{n \in \bN} \A{n}$, we call $|a|$ the length of $a$ (i.e., $|a|=n$ if and only if $a \in \A{n}$), $m[a] := \max_{i=1, ...,|a|} a_i$ and $l[a]\in \bN_0^{\{0,\cdots,m[a]\}}$ such that $l[a]_i = |\{ j: a_j = i\}|$, for $i \in \{0,\cdots,m[a]\}$.
\end{definition}

In the next lemma, we provide another interpretation of the set $\A{n}$, as we prove it to be in bijection with the set of partitions $\scP \big( \{0,1, ..., n\}\big)$ of $\{0,1, ..., n\}$. It turns out that the set of partitions is key to understanding the Lions derivative. Intuitively, the partition will indicate how the free variables generated by the iterated Lions derivatives will interact with each other, thus providing us with information on which probability space the Lions derivatives should be considered. 

\begin{lemma}
\label{Lemma:Bijection-Partition}
For $n\in \bN$, there exists a bijection between the set $\A{n}$ and $\scP\big( \{0,1, ..., n\}\big)$. 
\end{lemma}

Intuitively, the partition associated with an element $a \in A_{n}^{(0)}$ is obtained by gathering (in a common element of the partition) the indices $i$ of $\{0, ..., n\}$ that have the same value $a_{i}$
in the sequence $a$, with the convention that $a_{0}=0$. The reader may skip ahead on an initial reading.

\begin{proof}
For $a \in \A{n}$, we associate the collection of sets
\begin{equation*}
p_{k}:=a^{-1}(\{k\})=\bigl\{i \in \{0,\cdots,n\} : a_{i} = k \bigr\}, \quad k \in \{0,\dots,\max(a)\},
\end{equation*}
with the convention that $a_{0}=0$. Then the collection of sets $P=\{ p_{0}, ..., p_{m[a]}\}$ is a partition of $\{0,\cdots,n\}$. This creates a mapping ${\mathfrak m}$ from $\A{n}$ into $\scP( \{0,1, ..., n\})$. This mapping is injective, since any two $a$ and $b$ in $\A{n}$ such that ${\mathfrak m}[a]={\mathfrak m}[b]$ have the same pre-images $(a^{-1}(\{k\}))_{k =0, ..., n}$ and $(b^{-1}(\{k\}))_{k=0, ..., n}$ and hence coincide. 

It thus remains to prove that ${\mathfrak m}$ is surjective onto $\scP(\{0,1,...,n\})$. Let $P \in \scP(\{0,...,n\})$ so that $|P| \leq n+1$. Then there exists $p_0 \in P$ such that $0 \in p_0$. Thus $P \backslash \{p_0\} \in \scP( \{0, 1, ..., n\} \backslash p_0)$ and the set $\{0, 1, ... n\}\backslash p_0$ is ordered. In turn, there exists $x\in \{0, 1, ..., n\} \backslash p_0$ such that $x = \min \{0, 1, ..., n\} \backslash p_0$, which allows us to  call $p_1$ the unique element of $P$ such that  $p_{1} \ni x$. Continuing in this fashion, we obtain an enumeration of $P$ in the form $P= \{p_0, p_1, ...,p_{m} \}$, with $m \leq n+1$. 

Define the sequence $(a_i)_{i=1, ..., n}$ by $a_i = j$ if and only if $i \in p_j$. We verify that $a \in A_n$. Firstly, either $1\in p_0$ or $1 \notin p_0$. If $1 \in p_0$ then $a_1 = 0$. If $1 \notin p_0$ then $1 \in \{0, 1, ..., n\} \backslash p_0$ and $0 \notin \{0, 1, ..., n\} \backslash p_0$ so that $1 = \min \{0, 1, ..., n\} \backslash p_0$. Thus $1 \in p_1$ and then $a_1 =1$. 

Next suppose for $k <n$ that $\tilde{a} = (a_i)_{i=1, ..., k} \in \A{k}$. We must prove that $m[\tilde{a}]:=\max_{l =1,...,k} a_{l} \leq 
\max_{l =1,...,k+1} a_{l}$. We already have a sequence of sets $p_0, ..., p_{m[\tilde{a}]} \in P$ such that $1, ..., k \in \bigcup_{i=0}^{m[\tilde{a}]} p_i$. Then either $k+1 \in \bigcup_{i=0}^{m[\tilde{a}]} p_i$ or $k+1 \in \{0, 1, ..., n\} \backslash ( \bigcup_{i=0}^{m[\tilde{a}]} p_i )$. If $k+1 \in \bigcup_{i=0}^{m[\tilde{a}]} p_i$, then there exists $j \in \{0, 1, ..., m[\tilde{a}]\}$ such that $k+1 \in p_j$ so that $a_{k+1} =j \leq \max_{l=1,...,k} a_{l}$. On the other hand if $k+1 \notin  \bigcup_{i=0}^{m[\tilde{a}]} p_i$, then $k+1 = \min \{ 0, 1, ..., n\} \backslash  (\bigcup_{i=0}^{m[\tilde{a}]} p_i)$, so that $k+1 \in p_{m[\tilde{a}] +1}$ and $a_{k+1} = m[\tilde{a}]+1$.  We conclude by induction. 
\end{proof}

Following our notation $\partial_{a}$, $a \in A_{n}$, for the iterated Lions' derivative defined by $a$, we now want to define
$a\in A_n^{(0)}$, for $n\in \bN$, according to the following induction:
\begin{align*}
\partial_{(1)} =& \partial_\mu, \quad \partial_{(0)} = \nabla_x
\\
\partial_{(a_1, ..., a_{k-1}, a_k)} =&
\begin{cases} 
\nabla_x \cdot \partial_{(a_1, ..., a_{k-1})} & \quad a_k =0, 
\\
\nabla_{v_{a_k}} \cdot \partial_{(a_1, ..., a_{k-1})} & \quad 0 < a_k \leq \max \{ a_1, ..., a_{k-1}\}, 
\\
\partial_\mu \cdot \partial_{(a_1, ..., a_{k-1})} & \quad a_k > \max \{a_1, ..., a_{k-1}\}.  
\end{cases}
\end{align*}

\begin{definition}
\label{def:general:Lions-spacial:derivative}
For any $n \in \bN$, we say that a function $f: \bR^d \times \cP_2(\bR^d) \to \bR^e$ belongs to $C_{b}^{n, (n)}\big( \bR^d \times \cP_2(\bR^d); \bR^e \big)$ if there exists a collection of functions $\big\{ \partial_a f: a \in \bigcup_{i=1}^n \A{i} \big\}$ such that:
\begin{enumerate}
\item For $i=1, ..., n$ and for all $a\in \A{i}$, there exists a function
\begin{align*}
\partial_a f: \bR^d \times \cP_2(\bR^d) \times (\bR^d)^{\times m[a]} &\to \lin\big( (\bR^d)^{\otimes i}, \bR^e\big)
\\
\bigl(x,\mu,(v_{1},\cdots,v_{m[a]}) \bigr) &\mapsto \partial_{a} f(x,\mu,v_{1},\cdots,v_{m[a]}). 
\end{align*}

\item For all $a\in \bigcup_{i=1}^n \A{i}$, the functions $\partial_a f$ are bounded and Lipschitz continuous on $\bR^d \times \cP_2(\bR^d) \times (\bR^d)^{\times m[a]}$, the second space being equipped with the $\bW^{(1)}$-distance. 

\item For any $a\in \A{n-1}$, the function $\partial_{a} f$ is differentiable with respect to $x$, and 
$$
\partial_{x} \partial_{a} f = \partial_{(a,0)} f.
$$
\item For any $a \in \A{n-1}$, the function $\partial_a f$ is differentiable with respect to $\mu$, and
$$
\partial_\mu \partial_a f = \partial_{(a, m[a]+1)} f
$$

\item For any $a \in \A{n-1}$, the function $\partial_{a} f$ is differentiable with respect to $(v_1, ..., v_{m[a]})$ and for any $p \in \{1, ..., m[a]\}$, 
\begin{equation*}
\partial_{v_p} \partial_{a} f = \partial_{(a, p)} f.
\end{equation*}
\end{enumerate}
\end{definition}

We can now extend the differential operator defined in \eqref{eq:D^a:without:0} to the multivariate case. Let $x_0, y_0\in \bR^d$ and let $\Pi^{\mu, \nu}$ be a measure on $(\bR^d)^{\oplus 2}$ with marginal distribution $\mu, \nu \in \cP_{n+1}$. For $a\in A_{k,n}$, we define the operator
\begin{align}
\label{eq:rDa}
\rD^a& f(x_0, \mu)[ y_0-x_0, \Pi^{\mu, \nu}]
\\
\nonumber
&= \underbrace{\int_{(\bR^d)^{\oplus 2}} ... \int_{(\bR^d)^{\oplus 2}} }_{\times m[a]} \partial_a f\Big( x_0, \mu, x_1, ..., x_{m[a]} \Big) \cdot \bigotimes_{i=1}^{k+n} ( y_{a_i} - x_{a_i}) d\Pi^{\mu, \nu}(x_1, y_1)  ... d\Pi^{\mu, \nu}(x_{m[a]}, y_{m[a]} ) 
\end{align}

Then Theorem \ref{theorem:LionsTaylor1} admits the following generalisation:

\begin{theorem}
\label{theorem:LionsTaylor2}
Let $n \in \bN$ and let $f\in C_b^{n,(n)} \big( \bR^d \times \cP_2(\bR^d); \bR^e \big)$. Then for any $\mu, \nu \in \cP_{n+1} (\bR^d)$ with joint distribution $\Pi^{\mu, \nu}$ and any $x_0, y_0\in \bR^d$ we have that
\begin{align}
\label{eq:FullTaylorExpansion}
f(y_0, \nu) -& f(x_0, \mu) 
= \sum_{i=1}^n \sum_{a\in \A{i}} \frac{\rD^a f(x_0, \mu)[y_0 - x_0, \Pi^{\mu, \nu}] }{i!} 
+ R_{n}^{(x_0, y_0), \Pi^{\mu, \nu}}(f) 
\end{align}
where
\begin{align}
\nonumber
R_{n}^{(x_0, y_0), \Pi^{\mu, \nu}}(f)
= 
\frac{1}{(n-1)!}\sum_{a\in \A{n}} & \underbrace{\int_{(\bR^d)^{\oplus 2}} ... \int_{(\bR^d)^{\oplus 2}} }_{\times m[a]} f^{a, (x_0, y_0), \Pi^{\mu, \nu}} 
\\
\label{eq:theorem:LionsTaylor2_remainder1}
&\cdot \bigotimes_{p=1}^{n} ( y_{a_p} - x_{a_p}) d\Pi^{\mu, \nu}(x_1, y_1)  ... d\Pi^{\mu, \nu}(x_{m[a]}, y_{m[a]} ), 
\end{align}
and
\begin{align}
\nonumber
f^{a, (x_0, y_0), \Pi^{\mu, \nu}} = \int_0^1& \Big( \partial_a f \big(x_0 + \xi(y_0 - x_0), \Pi^{\mu, \nu}_\xi, x_1 + \xi(y_1 - x_1), ..., x_{m[a]} + \xi(y_{m[a]} - x_{m[a]} )\big)
\\
\label{eq:theorem:LionsTaylor2_remainder1_1}
 - \partial_a f\big(x_0, \mu,& x_1, ..., x_{m[a]} \big)\Big) (1-\xi)^{n-1} d\xi. 
\end{align}
The probability measure $\Pi^{\mu,\nu}_{\xi}$ is defined as in \eqref{eq:Pixi}. 

As before, for $a \in A_{i, n-i}$ the error term satisfies
\begin{align}
\nonumber
\underbrace{\int_{(\bR^d)^{\oplus 2}} ... \int_{(\bR^d)^{\oplus 2}} }_{\times m[a]} & f^{a, (x_0, y_0), \Pi^{\mu, \nu}} \cdot \bigotimes_{p=1}^{n} ( y_{a_p} - x_{a_p}) d\Pi^{\mu, \nu}(x_1, y_1)  ... d\Pi^{\mu, \nu}(x_{m[a]}, y_{m[a]} ) \\
\nonumber
=&O\Bigg( |y_0 - x_0|^{i+1} \cdot \prod_{p=1}^{m[a]} \int_{(\bR^d)^{\oplus 2}} |y_p - x_p|^{l[a]_p} d\Pi^{\mu, \nu}(x_p, y_p) \Bigg)
\\
\nonumber
&+ O\Bigg( \bW^{(1)} (\mu,\nu) \cdot |y_0 - x_0|^{i} \cdot \prod_{p=1}^{m[a]} \int_{(\bR^d)^{\oplus 2}} |y_p - x_p|^{l[a]_p} d\Pi^{\mu, \nu}(x_p, y_p) \Bigg)
\\
\label{eq:theorem:LionsTaylor2_Rem1}
&+ \sum_{q=1}^{m[a]} O \Bigg( |y_0 - x_0|^{i} \cdot \prod_{p=1}^{m[a]} \int_{(\bR^d)^{\oplus 2}} |y_q - x_q|^{l[a]_q + \delta_{p,q}} d\Pi^{\mu, \nu}(x_q, y_q) \Bigg). 
\end{align}
\end{theorem}

\begin{proof}
The proof is a multivariate adaption of the proof of Theorem \ref{theorem:LionsTaylor1} and is left as an exercise for those who are curious. The authors stress there is no extra novelty in including this proof. 
\end{proof}

\begin{remark}
We can sum up all the remainder terms of \eqref{eq:FullTaylorExpansion}. An application of the H\"older inequality to Equation  \eqref{eq:theorem:LionsTaylor2_Rem1} (in the same fashion as in Equation \eqref{eq:bound:remainder:TaylorExpansion}) and summing over all the terms of Equation \eqref{eq:theorem:LionsTaylor2_remainder1} gives that
\begin{align*}
R_{n}^{(x_0, y_0), \Pi^{\mu, \nu}}(f) \lesssim& \sum_{j=0}^n O\bigg( |x-y|^{j+1} \cdot \int |v-u|^{n-j} d\Pi^{\mu, \nu}(u, v) \bigg), 
\\
&+ \sum_{j=0}^n O\bigg( |x-y|^{j} \cdot \int |v-u|^{n-j+1} d\Pi^{\mu, \nu}(u, v) \bigg). 
\end{align*}
\end{remark}

\subsubsection{Differentiability of Multivariate derivatives}

Here, we provide a Taylor expansion for $\partial_{a} f$ itself, assuming that the derivative is sufficiently differentiable. This requires that we first introduce some suitable notation for denoting a sequence $\overline a$ whose restriction on $\{1, ..., |a| \}$ coincides
with $a$ itself, in which case $\overline a$ reads as an extension of $a$. 

For $i,j\in \bN_0$, let $a\in A_{i,j}$ and for integers $\tilde{i}\geq i, \tilde{j}\geq j$ such that $\tilde{i}+\tilde{j}>i+j$, define 
\begin{align*}
&A_{\tilde{i},\tilde{j} |a}:= \Big\{ \overline{a} \in A_{\tilde{i},\tilde{j}}: \forall p=1, ..., i+j, \quad \overline{a}_p = a_p \Big\}. 
\end{align*}
Finally, for another integer $n>i+j$, we let
$$
\A{n|a}:= \bigcup_{k=i}^{n-j} A_{k, n-k |a}. 
$$

Then, for $a\in A_{i,j}$ and $\overline{a}\in A_{\tilde{i},\tilde{j}|a}$, we define the operator
\begin{align*}
\rD^{\overline{a}|a}& \partial_a f(x_0, \mu, x_1, ..., x_{m[a]} )\Big[ y_0-x_0, \Pi^{\mu, \nu}, y_1-x_1, ..., y_{m[a]} - x_{m[a]} \Big]
\\
&= \underbrace{\int_{(\bR^d)^{\oplus 2}} ... \int_{(\bR^d)^{\oplus 2}} }_{\times (m[\overline{a}] - m[a])} \partial_{\overline{a}} f\Big( x_0, \mu, x_1, ..., x_{m[\overline{a}]}\Big) \cdot \bigotimes_{p=i+j+1}^{\tilde{i}+\tilde{j}} ( y_{\overline{a}_p} - x_{\overline{a}_p}) 
\\
&\qquad \qquad d\Pi^{\mu, \nu}(x_{m[a]+1}, y_{m[a]+1})  ... d\Pi^{\mu, \nu}(x_{m[\overline{a}]}, y_{m[\overline{a}]} ). 
\end{align*}

Theorem \ref{theorem:LionsTaylor2} admits the following useful Corollary:

\begin{corollary}
\label{corollary:LionsTaylor2}
Let $n \in \bN$ and let $f\in C_b^{n, (n)}\big( \bR^d \times \cP_2(\bR^d); \bR^e \big)$. Let $(i, j)\in \bN_0 \times \bN_0$ such that $0<i+j<n$ and let $a \in A_{i, j}$. 

For any $x_0, y_0, x_1, y_1, ..., x_{m[a]}, y_{m[a]} \in \bR^d$ and for any $\mu, \nu \in \cP_{n+1}(\bR^d)$ with joint distribution $\Pi^{\mu, \nu}$ we have that
\begin{align*}
\partial_a &f(y_0, \nu, y_1, ..., y_{m[a]}) - \partial_a f(x_0, \mu, x_1, ..., x_{m[a]}) 
\\
=& \sum_{k=i+j+1}^n \sum_{\overline{a}\in \A{k|a}} \frac{\rD^{\overline{a}|a} \partial_a f(x_0, \mu, x_1, ..., x_{m[a]} )}{ \big( k-(i+j) \big)!} \Big[ y_0-x_0, \Pi^{\mu, \nu}, y_1-x_1, ..., y_{m[a]} - x_{m[a]} \Big]
\\
&+R_{n|a}^{(x_0, y_0), \Pi^{\mu, \nu}} \Big( \partial_a f \Big), 
\end{align*}
where
\begin{align*}
R_{n|a}^{(x_0, y_0), \Pi^{\mu, \nu}}\Big( \partial_a f \Big)
=&
\frac{1}{(n - (i+j)-1)!} \sum_{\overline{a}\in \A{n|a}} \underbrace{\int_{(\bR^d)^{\oplus 2}} ... \int_{(\bR^d)^{\oplus 2}} }_{\times m[\overline{a}] - m[a]} (\partial_a f)^{\overline{a}, (x_0, y_0), \Pi^{\mu, \nu}}
\\
& \cdot \bigotimes_{r=i+j+1}^{n} ( y_{\overline{a}_r} - x_{\overline{a}_r}) d\Pi^{\mu, \nu}(x_{m[a]+1}, y_{m[a]+1})  ... d\Pi^{\mu, \nu}(x_{m[\overline{a}]}, y_{m[\overline{a}]} )
\end{align*}
and
\begin{align}
\nonumber
(\partial_a f&)^{\overline{a}, (x_0, y_0), \Pi^{\mu, \nu}}
\\
\nonumber
&=\int_0^1 \Big( \partial_{\overline{a}} f \big(x_0 + \xi(y_0 - x_0), \Pi^{\mu, \nu}_\xi, x_1 + \xi(y_1 - x_1), ..., x_{m[\overline{a}]} + \xi(y_{m[\overline{a}]} - x_{m[\overline{a}]} ) \big)
\\
\label{eq:corollary:LionsTaylor2_remainder1_1}
&\quad - \partial_{\overline{a}} f(x_0, \mu, x_1, ..., x_{m[\overline{a}]} ) \big) \Big) (1-\xi)^{(n-(i+j) -1} d\xi. 
\end{align}
\end{corollary}

\subsubsection{Schwarz Theorem for Lions derivatives}

Ascending sequences of elements $a\in A_{k,n}$ expose how new derivative terms are added to the Taylor expansion. For each term in the remainder of the Taylor expansion, we perform the Mean Value Theorem on every variable which yields a new derivative associated to all the possible sequences $(a, i) \in A_{k+1,n} \cup A_{k,n+1}$ where $i\in \{0, 1, ..., m[a]+1\}$. However, what this representation does not convey is that many  derivatives $\partial_a f$ in this expansion are equal. 
	
\begin{theorem}[Schwarz Theorem for Lions Derivatives]
\label{thm:Schwarz-Lions}
Let $k,n\in \bN$ such that $k+n\geq 2$. Let $a, b\in A_{k,n}$ such that $\forall i\in \big\{ 1, ..., m[a] \big\}$, $l[a]_i = l[b]_i$. Then for any $f\in C_{b}^{n,(n)} \big( \bR^d \times \cP_2(\bR^d) \big)$, 
$$
\partial_a f = \partial_b f,
$$
in the sense that, for any $e_{1},... ,e_{n} \in \bR^d$ and any permutation $\sigma$ from $\{1, ..., n\}$ into itself such that $a_{\sigma(i)}=b_{i}$ for all $i \in \{1, ..., n\}$, and all $\mu \in \cP_{2}(\bR^d)$ and $v_{1},... ,v_{m} \in \bR^d$,
$$
\partial_{a} f(\mu,v_{1},...,v_{m}) \cdot \otimes_{i=1}^n e_{i}
=
\partial_{b} f(\mu,v_{1},...,v_{m}) \cdot \otimes_{i=1}^n e_{\sigma(i)}.
$$
\end{theorem}

\begin{proof}
By proceeding by induction on $n$ and by denoting $a^{-n}$ and $b^{-n}$, we can assume that the vectors $(a_{1},...,a_{n-1})$ and $(b_{1},...,b_{n-1})$ are non-decreasing. Without any loss of generality, we can assume that $a$ itself is non-decreasing, so that the only thing we have to do is to prove that we can ``shift'' $b_{n}$ so that $b$ becomes itself non-decreasing. 

In fact, there are two cases. If $b_{n} \geq b_{n-1}$, then there is nothing to do and the proof is over! So, we can assume that $b_{n}<b_{n-1}$, which means that there is another index $i<n-1$ such that $b_{n}=b_{i}$. We then impose $i$ to be the largest of all the integers in $\{1,...,n-1\}$ such that $b_{n}=b_{i}$. Then, we can let $c=(b_{1},...,b_{i})$. We can regard the function $\partial_{c} f \cdot e_{1} \otimes ... \otimes e_{i} $ as a function defined on $\cP(\bR^d) \times (\bR^d)^{\times p}$, with $p=\max(c)$. Then, by \cite{CarmonaDelarue2017book2}*{Remark 4.16}, for any $v_{p+1} \in \bR^d$, 
\begin{align*}
\partial_{v_{b_{n}}} &\partial_{\mu} \bigl[ \partial_{c} f(\mu,v_{1},...,v_{p}) \cdot e_{1} \otimes ... \otimes e_{i} \bigr](v_{p+1}) \cdot e_{i+1} \otimes e_{{n}} 
\\
&=
\partial_{\mu} \partial_{v_{b_{n}}}  \bigl[ \partial_{c} f(\mu,v_{1},\cdots,v_{p}) ,  e_{1} \otimes ... \otimes e_{i} \bigr] (v_{p+1} ) \cdot e_{{n}} \otimes e_{i+1},
\end{align*}
which yields (the identity below being obviously true if $b_{i+1}$ coincides with $b_j$ for some $j=1,...,i$)
\begin{align*}
&\partial_{(c,b_{i+1},b_{n})} f(\mu)(v_{1},...,v_{p},v_{p+1}) \cdot e_{1} \otimes ... \otimes e_{i} \otimes e_{i+1} \otimes e_{{n}}
\\
&= \partial_{(c,b_{n},b_{i+1})} f(\mu)(v_{1},...,v_{p},v_{p+1}) \cdot e_{1} \otimes ... \otimes e_{i} \otimes e_{{n}}\otimes e_{i+1}.
\end{align*}
For $j$ such that $b_{i+1}=b_{i+2}=\cdots=b_{j}$, we get
\begin{align*}
\partial_{v_{p+1}}^{j-(i+1)} &\Bigl[ \partial_{(c,b_{i+1},b_{n})} f(\mu,v_{1},...,v_{p},v_{p+1}) \cdot e_{1} \otimes ... \otimes e_{i} \otimes e_{i+1} \otimes e_{{n}} \Bigr] \cdot e_{i+2} \otimes ... \otimes e_{j} 
\\
&=
\partial_{v_{p+1}}^{j-(i+1)} \Bigl[ \partial_{(c,b_{n},b_{i+1})} f(\mu,v_{1},...,v_{p},v_{p+1}) \cdot e_{1} \otimes ... \otimes e_{i} \otimes e_{{n}}\otimes e_{i+1} \Bigr] \cdot e_{i+2} \otimes ... \otimes e_{j} 
\\
&= \partial_{(c,b_{n},b_{i+1},b_{i+2},...,b_{j})} f(\mu,v_{1},...,v_{p},v_{p+1}) \cdot e_{1} \otimes ... \otimes e_{i} \otimes e_{{n}}\otimes e_{i+1} \otimes e_{i+2} \otimes ... \otimes e_{j} .
\end{align*}
Now,
\begin{align*}
\partial_{v_{p+1}}^{j-(i+1)} &\Bigl[ \partial_{(c,b_{i+1},b_{n})} f(\mu,v_{1},...,v_{p},v_{p+1}) \cdot e_{1} \otimes ... \otimes e_{i} \otimes e_{i+1} \otimes e_{{n}} \Bigr] \cdot e_{i+2} \otimes ... \otimes e_{j} 
\\
&=
\partial_{v_{p+1}}^{j-(i+1)} \partial_{v_{b_{n}}} \Bigl[ \partial_{(c,b_{i+1})} f(\mu,v_{1},...,v_{p},v_{p+1}) \cdot e_{1} \otimes ... \otimes e_{i} \otimes e_{i+1} \Bigr] \cdot e_{{n}} \otimes e_{i+2} \otimes ... \otimes e_{j} 
\\
&=\partial_{v_{b_{n}}} \partial_{v_{p+1}}^{j-(i+1)} \Bigl[ \partial_{(c,b_{i+1})} f(\mu,v_{1},...,v_{p},v_{p+1}) \cdot e_{1} \otimes ... \otimes e_{i} \otimes e_{i+1} \Bigr] \cdot e_{i+2} \otimes ... \otimes e_{j} \otimes e_{{n}} 
\\
&=
\partial_{(c,b_{i+1},...,b_{j},b_{n})} f(\mu,v_{1},...,v_{p},v_{p+1}) \cdot e_{1} \otimes ...\otimes e_{i} \otimes e_{i+1} \otimes e_{i+2} \otimes ... \otimes e_{j} \otimes e_{{n}},
\end{align*}
from which we deduce that 
\begin{align*}
&\partial_{(c,b_{n},b_{i+1},b_{i+2},...,b_{j})}  f(\mu,v_{1},...,v_{p},v_{p+1}) \cdot e_{1} \otimes ... \otimes e_{i} \otimes e_{{n}}\otimes e_{i+1} \otimes e_{i+2} \otimes ... \otimes e_{j} 
\\
&=
\partial_{(c,b_{i+1},\cdots,b_{j},b_{n})} f(\mu)(v_{1},\cdots,v_{p},v_{p+1}) \cdot e_{1} \otimes \cdots \otimes e_{i} \otimes e_{i+1} 
\otimes e_{i+2} \otimes ... \otimes e_{j} \otimes e_{{n}},
\end{align*}
In fact, we can iterate this argument as long as $b_{j} \leq b_{i+1}$ (even though equality does not hold). Assume now that $b_{j+1}>b_{i+1}$, which means that, when adding $b_{j+1}$ to $(c,b_{i+1},...,b_{j},b_{n})$, we take a new Lions derivative. Then, proceeding as before, but with $c$ being replaced by $(c,b_{i+1},...,b_{j})$, we deduce that, for any $v_{p+2} \in \bR^d$, 
\begin{align*}
&\partial_{(c,b_{n},b_{i+1},b_{i+2},...,b_{j},b_{j+1})} f(\mu,v_{1},...,v_{p+2}) \cdot e_{1} \otimes ... \otimes e_{i} \otimes e_{{n}}\otimes e_{i+1} \otimes e_{i+2} \otimes ... \otimes e_{j}\otimes e_{j+1} 
\\
&=
\partial_{(c,b_{i+1},...,b_{j+1},b_{n})} f(\mu,v_{1},...,v_{p+2}) \cdot e_{1} \otimes ... \otimes e_{i} \otimes e_{i+1} \otimes e_{i+2} \otimes ... \otimes e_{j} \otimes e_{j+1} \otimes e_{{n}},
\end{align*}
and, by iteration, we complete the proof. 
\end{proof}

\begin{remark}
We leave it as an exercise to the reader to find a more compact statement of Equation \eqref{theorem:LionsTaylor2} that additionally captures the equalities proved in \ref{thm:Schwarz-Lions}. 
\end{remark}

\subsection{Norms on differentiable functions}

Finally, we introduce a norm on the collection of functions described in Definition \ref{def:general:Lions-spacial:derivative}:

\begin{definition}
\label{definition:FunctionNorm}
For $a\in \A{n}$, we denote
\begin{equation}
\label{eq:definition:FunctionNorm1}
\| \partial_a f\|_\infty:= \sup_{x_0 \in \bR^d} \sup_{\mu \in \cP_2(\bR^d)}  \sup_{(x_1, ..., x_{m[a]}) \in (\bR^d)^{\times m[a]}} \big| \partial_a f(x_0, \mu, x_1, ..., x_{m[a]}) \big|. 
\end{equation}

Further, we denote
\begin{equation}
\label{eq:definition:FunctionNorm2}
\left.\begin{aligned}
\big\| \partial_a f \big\|_{\lip, 0}
&:=
\sup_{\substack{x_0, y_0\in \bR^d \\ \mu \in \cP_1(\bR^d) \\ x_1, ..., x_{m[a]} \in \bR^d}} \frac{\big| \partial_a f(x_0, \mu, ..., x_{m[a]}) - \partial_af(y_0, \mu, ..., x_{m[a]}) \big|}{|x_0 - y_0|}, 
\\
\big\| \partial_a f \big\|_{\lip, \mu}
&:=
\sup_{\substack{ \mu, \nu\in \cP_1(\bR^d) \\ x_0, x_1, ..., x_{m[a]} \in \bR^d}} \frac{\big| \partial_a f(x_0, \mu, ..., x_{m[a]}) - \partial_af(x_0, \nu, ..., x_{m[a]}) \big|}{\bW^{(1)}(\mu, \nu)}, 
\\
\big\| \partial_a f \big\|_{\lip, j}
&:= 
\sup_{\substack{x_0,x_1, ..., x_{m[a]} \in \bR^d \\ y_j \in \bR^d \\ \mu \in \cP_1(\bR^d)}} \frac{\big| \partial_a f(x_0, \mu, ..., x_j, ..., x_{m[a]}) - \partial_af(x_0, \mu, ..., y_j, ..., x_{m[a]}) \big|}{|x_j - y_j|}. 
\end{aligned}\right\rbrace
\end{equation}

We define
\begin{equation}
\label{eq:definition:FunctionNorm}
\big\| f \big\|_{C_b^{n, (n)}} := \sum_{ a\in \bigcup_{k=1}^{n} \A{k}} \big\| \partial_a f \big\|_\infty
+
\sum_{a\in \A{n}} \bigg( \big\| \partial_a f \big\|_{\lip, 0} + \big\| \partial_a f \big\|_{\lip, \mu} + \sum_{j=1}^{m[a]} \big\| \partial_a f \big\|_{\lip, j} \bigg). 
\end{equation}
\end{definition}
	
\newpage

\section{Probabilistic rough paths}
\label{section:ProbabilisticRoughPaths}

The goal of this section is to provide a (very) short introduction to the key mathematical ideas and notation of \cite{salkeld2021Probabilistic} that will be relevant. 

\subsection{Coupled Hopf algebra}
\label{subsection:CoupledHopf}

One of the main contributions of \cite{salkeld2021Probabilistic} is to introduce a new form of labelled tree that is tailor-made to the analysis of mean-field equations, see Definition \ref{definition:Forests} below. Such trees are called Lions' trees because they are connected to Lions' derivatives through the notion of sup-envelope of an integer-valued sequence as introduced in Definition \ref{def:a}. 

The tree is paired with a collection of hyperedges that form a partition of the nodes of the tree. Each partition element contains a collection of nodes that correspond to instances of a driving signal that are tagged together but are independent of all other instances. An identified hyperedge, denoted by a `0', refers to the instance of the driving signal that corresponds to the solution of the mean-field equation. 

\begin{definition}
\label{definition:Forests}
Let $N$ be a non-empty set containing a finite number of elements. Let $E\subset N\times N$ such that $(x,y)\in E \implies (y,x)\notin E$. Let $h_0 \subseteq N$ and $H\in \scP(N\backslash h_0)$ is a partition of $N\backslash h_0$. We denote $H' = H \cup \{h_0\}$. Let $L:N \to \{1, ..., d\}$. We refer to $T=(N, E, h_0, H, L)$ as a \emph{Lions forest} if
\begin{enumerate}
\item $(N, E, L)$ is a non-planar, labelled, directed forest with partial ordering $<$ (obtained by comparing the distance to roots)
\item $(N, H')$ is a 1-regular, matching hypergraph where the hyperedges $h_0, h \in H'$ satisfy:
\begin{itemize}
\item If $h_0 \neq \emptyset$, then $\exists x\in h_0$ such that $\forall y\in N$, $x\leq y$. 
\item For $h_i \in H'$, suppose $x, y\in h_i$ and $x<y$. Then $\exists z \in h_i$ such that $(y, z)\in E$. 
\item For $h_i \in H'$, suppose $x_1, y_1 \in h_i$, $x_1 \leq \geq y_1$, $(x_1, x_2), (y_1, y_2) \in E$ and $x_2 \neq y_2$. Then $x_2, y_2 \in h_i$. 
\end{itemize}
\end{enumerate}
The collection of all such Lions forests is denoted by $\scF$. When a Lions forest $T=(N,E,h_0, H,L)$ satisfies that $\exists! x_0 \in N$ such that $\forall y\in N$  $\exists! (y_i)_{i=1, ..., n}$ such that $y_1=y$, $y_n=x_0$ and $(y_i,y_{i+1})\in E$, it is referred to as a Lions tree and the collection of all lions trees is denoted $\scT$. 

We define $\scF_0= \scF\cup\{\rId\}$ where $\rId=(\emptyset, \emptyset, \emptyset, \emptyset,L)$  is the empty tree. On a separate note, we define $\scT_0 = \{T \in \scT: h_0^T \neq \emptyset\}$. 
\end{definition}

We use the terminology `hyperedge' introduced above to denote the elements $h$ of $H$. The next Definition \ref{definition:Lionsproduct} describes operations on Lions trees, including operations on the hyperedges themselves:

\begin{definition}
\label{definition:Lionsproduct}
We define $\circledast: \scF_0 \times \scF_0 \to \scF_0$ so that for two Lions forests $T_1=(N^1, E^1, h_0^1, H^1, L^1)$ and $T_2 = (N^2, E^2, h_0^2, H^2, L^2)$, we have $T_1 \circledast T_2 = (\tilde{N}, \tilde{E}, \tilde{h_0}, \tilde{H}, \tilde{L})$ such that 
\begin{align*}
&\tilde{N} = N^1 \cup N^2,
\quad
\tilde{E} = E^1 \cup E^2, 
\quad
\tilde{h}_0 = h_0^1 \cup h_0^2 
\quad
\tilde{H} = H^1 \cup H^2,  
\\
&\tilde{L}: \tilde{N} \to \{1, ..., d\} \quad \mbox{ such that } \quad \tilde{L}|_{N^i} = L^i. 
\end{align*}

We define $\cE: \scF_0 \to \scF_0$ for $T=(N, E, h_0, H, L)$, we have 
$$
\cE[T] = (N, E, \emptyset, H', L).
$$ 

For $i\in\{1, ..., d\}$, we define the operator $\lfloor \cdot \rfloor_i: \scF_0 \to \scT$ for $T=(N, E, h_0, H, L)$, $\lfloor T\rfloor_i = (\tilde{N}, \tilde{E}, \tilde{h_0}, H, \tilde{L})$ where
\begin{align*}
\mbox{For} \quad x_0 \notin N,
\quad
\tilde{N} = N \cup \{ x_0\}, 
\quad&
\tilde{h_0} = h_0 \cup\{x_0\}, 
\\
\mbox{For} \quad \{x_1, ..., x_n\}\subseteq N 
\mbox{ such that } \quad&
\forall j=1,...,n \mbox{ and } \forall y\in N, x_j\leq y. \quad\mbox{Then }
\\
\tilde{E} = E \cup& \{ (x_1, x_0), ..., (x_n, x_0) \}.
\end{align*}
\end{definition}

The $\circledast$ operation is an associative and commutative product on the collection of forests with unit $\rId$. The $\cE$ operation transforms the $h_0$-hyperedge of a forest into an $H$-hyperedge, leaving the $h_0$-hyperedge of the resulting forest empty. Meanwhile, $\lfloor \cdot \rfloor_i$ is a grafting operation where the new root is placed within the 0-hyperedge. 

\subsubsection{Coupled algebras}

Throughout, we use a probability space $(\Omega, \cF, \bP)$ and we use the generic notation $L^0$ for the corresponding collection of (possibly vector valued) measurable functions (without any further integrability constraints). For any $m,n\in \bN$, we have that
$$
L^0\Big(\Omega \times \Omega^{\times m}, \bP \times \bP^{\times m}; \lin\big( (\bR^d)^{\otimes n}, \bR^e\big) \Big)
$$
is a module with respect to the ring $L^0(\Omega, \bP; \bR^e)$. 

In the next definition, we let this module act on Lions trees equipped with $n$ nodes and $m+1$ hyperedges (including the $h_{0}$-hyperedge). 

\begin{definition}
\label{definition:AlgebraRVs}
Let $\scH$ be the module spanned by the forests $\scF_0$ multiplied by measurable functions over product probability spaces, 
\begin{equation}
\label{eq:definition:AlgebraRVs}
\scH = \bigoplus_{T\in \scF_0} L^0 \Big( \Omega \times \Omega^{\times |H^T|}, \bP \times (\bP)^{\times |H^T|}; \lin \big( (\bR^d)^{\otimes |N^T|}, \bR^e\big) \Big) \cdot T
\end{equation}
$\scH$ is a module over the ring $L^0(\Omega, \bP; \bR^e)$. 

Let $T = (N, E, h_0, H, L)\in \scF$. Then for $X\in \scH$, we have
$$
X = \sum_{T\in \scF_0} \langle X, T\rangle \cdot T. 
$$
where
\begin{equation}
\label{eq:definition:AlgebraRVs:<X,T>}
\langle X, T\rangle \in L^{0} \Big( \Omega \times \Omega^{\times |H|}; \bP \times (\bP)^{\times |H|} ; \lin\big( (\bR^d)^{\otimes |N|}, \bR^e\big) \Big). 
\end{equation}
By extending $\circledast$ to be bilinear, the triple $(\scH, \circledast, \rId)$ becomes a commutative algebra over the ring $L^0(\Omega, \bP; \bR^e)$ since
\begin{equation}
\label{eq:definition:AlgebraRVs:Product}
X_1 \circledast X_2 = \sum_{T\in \scF_0} \langle X_1 \circledast X_2, T\rangle \cdot T = \sum_{T\in \scF_0} \Big( \sum_{\substack{T_1, T_2\in \scF_0\\ T_1\circledast T_2 = T }} \langle X_1, T_1\rangle \otimes \langle X_2, T_2\rangle \Big) \cdot T
\end{equation}
where
\begin{align*}
\Big\langle X_1, T_1\Big\rangle& (\omega_0, \omega_{H^{T_1}}) \otimes \Big\langle X_2, T_2\Big\rangle (\omega_0, \omega_{H^{T_2}}) 
\\
\in& L^{0}\Big( \Omega \times \Omega^{\times |H^{T_1\circledast T_2}|}, \bP \times \bP^{\times |H^{T_1\circledast T_2}|} ; \lin\big( (\bR^d)^{\otimes |N^{T_1\circledast T_2}|}, \bR^e\big) \Big). 
\end{align*}
The operator $\cE$ can be extended to a linear operator $\cE: \scH \to \scH$. For $X\in \scH$ we have
$$
\cE[ X ] = \sum_{T\in \scF_0} \langle \cE[X], T \rangle \cdot T = \sum_{T\in \cE[\scF_0]} \Big( \sum_{T' \in \cE^{-1}[T]} \langle X, T'\rangle \Big) \cdot T. 
$$
Let $a \in \A{n}$ and $m=m[a]$. Let $T_1, ..., T_n\in \scF_0$ be a sequence of Lions forests. We define the operator $\cE^a: \scF_0^{\times n} \to \scF_0$ by
\begin{equation}
\label{eq:definition:E^a-notation}
\cE^a\Big[ T_1, ..., T_n\Big] = \Big[ \underset{\substack{i:\\ a_i=0}}{\circledast} T_i \Big] \circledast \cE\Big[ \underset{\substack{i:\\ a_i=1}}{\circledast} T_i \Big] \circledast ... \circledast \cE\Big[ \underset{\substack{i:\\ a_i=m[a]}}{\circledast} T_i \Big] . 
\end{equation}
Intuitively, the partition sequence $a$ groups the sequence of trees into $m[a]+1$ partition sets. 
\end{definition}

\subsubsection{Coupled coalgebras}

A coupling between two partitions $P$ and $Q$ is a way of describing a bijective mapping between a subset of $P$ and a subset of $Q$. 

\begin{definition}
\label{definition:psi-Mappings}
Let $M$ and $N$ be two non-empty, disjoint, finite sets and let $P$ and $Q$ be partitions of $M$ and $N$ respectively. We define
$$
P \tilde{\cup} Q = \Big\{ G \in \scP(M\cup N): \{ g\cap M: g\in G\}\backslash \emptyset = P, \{ g\cap N: g\in G\}\backslash \emptyset = Q\Big\}. 
$$
We refer to $P \tilde{\cup} Q$ as the set of coupled partitions. For every $G\in P\tilde{\cup} Q$, we denote the two injective mappings $\psi^{P,G}: P \to G$ and $\psi^{Q,G}: Q \to G$ such that $\forall p\in P$ and $\forall q\in Q$, 
\begin{equation}
\label{eq:definition:psi}
\psi^{P,G}[ p] \cap M = p, \quad \psi^{Q,G}[ q] \cap N = q. 
\end{equation}
\end{definition}

$P \tilde{\cup} Q $ is the collection of partitions of the set $M\cup N$ that agree with the partition $P$ when restricting to $M$ and agrees with the partition $Q$ when restricting to $N$. The operator $\psi^{P, G}$ maps the set $p$, an element of the partition $P$, to the unique element of the partition $G$ that contains the set $p$. 

Suppose that we have random variables $X \in L^0(\Omega^{\times |P|}, \bP^{\times |P|}; U)$ and $Y\in L^0(\Omega^{\times |Q|}, \bP^{\times |Q|}; V)$. Then, for $G\in P\tilde{\cup}Q$, we define $X\otimes^G Y \in L^0(\Omega^{\times |G|}, \bP^{\times |G|}; U \otimes V)$ defined by
\begin{equation}
\label{eq:CoupledTensor_rvs}
X\otimes^G Y(\underbrace{\omega_g, ...}_{g\in G}) = X(\underbrace{\omega_{\psi^{P,G}[p]}, ...}_{p \in P}) \otimes Y(\underbrace{\omega_{\psi^{Q,G}[q]}, ...}_{q \in Q}). 
\end{equation}
For each choice of $G$, there is a product of random variable $X\otimes^G Y$ constructed on the top of $X$ and $Y$, where the transport plan between $X$ and $Y$ is specified through the partition $G$.

This motivates the following definition:

\begin{definition}
\label{definition:coupled_trees}
We define the set
$$
\scF_0 \tilde{\times} \scF_0:=\Big\{ T_1 \times^G T_2: \quad T_1, T_2\in \scF_0,\quad G\in H^1\tilde{\cup} H^2\Big\} = \bigsqcup_{T_1 \times T_2 \in \scF_0 \times \scF_0} H^1 \tilde{\cup} H^2. 
$$
We refer to the element $T_1 \times^G T_2$ as a coupled pair. We define the $L^0(\Omega, \bP; \bR^e)$-module
$$
\scH \tilde{\otimes} \scH = \bigoplus_{T_1 \times^G T_2 \in \scF_0\tilde{\times} \scF_0 } L^0 \Big( \Omega \times \Omega^{\times |G|}, \bP \times (\bP)^{\times |G|}; \lin \big( (\bR^d)^{\otimes ( |N^{T_1}| + |N^{T_2}|)}, \bR^e\big) \Big) \cdot T_1 \times^G T_2. 
$$
\end{definition}

The choice of symbol $\tilde{\otimes}$ is pertinent here since $\scH\tilde{\otimes}\scH$ is not a tensor product of modules but rather a coupled tensor module. 

An admissible cut is a way of dividing a tree into two subtrees, one a rooted tree (referred to as the root) and one a forest (referred to as the prune), see for instance \cite{connes1999hopf}. In our context, a cut removes edges from a graph, but does not alter hyperedges. Thus, a cut takes a tree to a coupled pair of forests contained in $\scF_0 \tilde{\times} \scF_0$ rather than simply $\scF_0 \times \scF_0$. 

\begin{definition}
\label{definition:CutsCoproduct}
Let $T=(N, E, h_0, H, L)$ be a rooted tree. A subset $c\subseteq E$ is called an admissible cut if $\forall y\in N$, the unique path $(e_i)_{i=1, ..., n}$ from $y$ to the root $x$ satisfies that if $e_i\in c \implies \forall j\neq i, e_j\notin c$. The set of admissible cuts for the tree $T$ is denoted $C(T)$. 

We call the root of the cut $c$ the rooted tree $T_c^R = (N_c^R, E_c^R, h_0 \cap N_c^R, H_c^R, L_c^R)$ where 
\begin{align*}
N_c^R:&= \big\{ y\in N: \exists (y_i)_{i=1, ..., n}\in N, y_1 = y, y_n=x, (y_i, y_{i+1})\in E\backslash c \big\}, 
\\
E_c^R:&= \big\{ (y,z)\in E: y,z\in N_c^R \big\}, 
\qquad
H_c^R:= \big\{ h\cap N_c^R: h\in H \big\}\backslash \{ \emptyset \}, 
\\
L_c^R:& N_c^R \to \{1, ..., d\}, \quad L_c^R = L|_{N_c^R}
\end{align*} 

The prune of the cut $c$ is the rooted forest $(N_c^P, E_c^P, h_0\cap N_c^P, H_c^p, L_c^P)$ where $N_c^P = N\backslash N_c^R$ and $E_c^P = E\backslash (E_c^R \cup c)$, $H_c^P:= \{ h\cap N_c^P: h\in H \}\backslash \emptyset$ and $L_c^P = L|_{N_c^R}$. 
\end{definition}

Following the ideas of rough paths, we wish to construct a coproduct over the module that will encode the incremental relationships of paths (in the sense of Chen). However, we need to frame this concept within the context of the couplings we have developed.  

\begin{definition}
\label{definition:coproduct}
Let $\Delta: \scH \to \scH \tilde{\otimes} \scH$ be the linear operator such that for $T=(N,E,h_0,H,L) \in \scF_0$, 
\begin{equation}
\label{eq:lemma:Coproduct-Equivalence}
\Delta[T] = \rId \times^H T + T \times^H \rId + \sum_{c\in C(T)} T_c^P \times^{H^T} T_c^R
\end{equation}
and extended to $\scF_0$ using $\Delta[T_1 \circledast T_2] = \circledast^{(2)}\big[ \Delta[T_1], \Delta[T_2]\big]$. Further, $\Delta[\rId] = \rId \times^{\emptyset} \rId$ and for the tree $T_1=(\{x_0\}, \emptyset, (\{x_0\}, \emptyset), L)$ with a single node contained in $h_0$, we have $\Delta[T_1]= T_1 \times^{\emptyset} \rId + \rId \times^{\emptyset} T_1$. 

Then $\Delta$ satisfies
\begin{equation}
\label{eq:definition:coproduct}
\begin{split}
\Delta\Big[ \lfloor T\rfloor_i \Big] = \lfloor T\rfloor_i \times^H &\rId + \Big( I \times^H \lfloor \cdot \rfloor_i\Big) \Delta[T], 
\\
\Delta\Big[ T_1 \circledast  T_2\Big] = \circledast^{(2)} \Big[ \Delta[T_1] , \Delta[T_2] \Big], &
\qquad
\Delta\Big[ \cE[T] \Big] = (\cE \tilde{\times} \cE) \Big[ \Delta [T] \Big]. 
\end{split}
\end{equation}

Thus for $X\in \scH$, we have
$$
\Delta\Big[ X \Big] = \sum_{T \in \scF_0} \Big\langle X, T \Big\rangle \cdot \Delta \Big[ T \Big]
$$
and we denote
\begin{equation}
\label{eq:Coproduct-counting}
c\Big( T_1, T_2, T_3 \Big) := \Big\langle \Delta[T_1], T_2 \times^{H^{T_1}} T_3 \Big\rangle. 
\end{equation}
We pair the operator $\Delta$ with the linear functional $\epsilon: \scH \to L^0(\Omega, \bP; \bR^e)$ which satisfies $\epsilon(\rId) = 1$, and $\forall T\in \scF$, $\epsilon(T) = 0$. The operator $\epsilon$ is the counit of $\Delta$. 
\end{definition}

\subsubsection{Grading}

\begin{definition}
Let $\scG:\scF_0 \to \bN_0^2$ such that for $T=(N^T, E^T, h_0^T, H^T, L^T)$, 
\begin{equation}
\label{eq:lemma:grading}
\scG[T]:= \Big( |h_0^T|, |N^T\backslash h_0^T| \Big). 
\end{equation}
\end{definition}

The grading is given by the number of nodes contained in the $0$-hyperedge and by the number of nodes contained in all the other hyperedges. We will often write $\scG_{\alpha, \beta}[T] = \alpha|h_0^T| + \beta|N^T\backslash h_0^T|$ for $\alpha,\beta \in \bR^+$ and henceforth use the notation
$$
\scF^{\gamma, \alpha, \beta} = \Big\{ T \in \scF: \scG_{\alpha, \beta}[T] \leq \gamma \Big\}, 
\quad
\scT^{\gamma, \alpha, \beta} = \Big\{ T \in \scT: \scG_{\alpha, \beta}[T] \leq \gamma \Big\}, 
$$
for some $\gamma\geq \alpha \wedge \beta$. 

Then we can prove that the module
\begin{equation}
\label{eq:Coupled-HopfAlgebra}
\scH^{\gamma, \alpha, \beta}:= \bigoplus_{\substack{T \in \scF_0 \\ \scG_{\alpha \beta}[T]\leq \gamma}} L^0 \Big( \Omega \times \Omega^{\times |H^T|}, \bP \times (\bP)^{\times |H^T|}; \lin \big( (\bR^d)^{\otimes |N^T|}, \bR^e\big) \Big) \cdot T
\end{equation}
satisfies that $(\scH^{\gamma, \alpha, \beta}, \Delta, \epsilon)$ is a coupled counital subcoalgebra of $(\scH, \Delta, \epsilon)$ and $(\scH^{\gamma, \alpha, \beta}, \circledast, \rId)$ is a quotient algebra of $(\scH, \circledast, \rId)$. Thus $(\scH^{\gamma, \alpha, \beta}, \circledast, \Delta, \rId, \epsilon)$ is also a coupled Bialgebra. 

\subsection{Group structures of a coupled Hopf algebra}

The coupled Bialgebra algebra $\scH^{\gamma, \alpha, \beta}$ does not have a canonical topology so there is no sense of analytic dual space. However, this detail is addressed with the convention
$$
\Big(\scH^{\gamma, \alpha, \beta} \Big)^\dagger =
\bigoplus_{T\in \scF_0^{\gamma, \alpha, \beta}} L^{0}\Big(\Omega \times \Omega^{ \times |H^T|}, \bP \times \bP^{ \times |H^T|}; (\bR^d)^{\otimes |N^T|} \Big) \cdot T. 
$$

Let $(\cA, m_{\cA})$ be a commutative algebra over $L^0(\Omega, \bP; \bR^d)$. For two mappings $f,g\in \lin(\scH, \cA)$, we define the convolution product $\ast: \lin(\scH, \cA) \times \lin(\scH, \cA) \to \lin(\scH, \cA)$ by
$$
f \ast g = m_{\cA} \circ f\tilde{\otimes} g \circ \Delta. 
$$
We define the set of all characters $\cG(\scH, \cA)$ as the collection of algebra homomorphisms from $\scH$ to $\cA$ that satisfy
$$
\cG(\scH, \cA) = \Big\{ f\in \lin(\scH, \cA): f\circ \circledast = m_{\cA} \circ f \otimes f \Big\}. 
$$
In practice, we choose $\cA = L^0(\Omega, \bP; \bR^d)$. Then $f$ is a character if and only if $f \in (\scH^{\gamma, \alpha, \beta})^\dagger$ and satisfies that $\langle f, \rId\rangle = 1$, $\forall T_1, T_2 \in \scF_0$
\begin{equation}
\label{eq:definition:DualModule1}
\Big\langle f, T_2 \circledast T_2 \Big\rangle = \Big\langle f, T_1 \Big\rangle \otimes \Big\langle f, T_2 \Big\rangle
\end{equation}

\subsubsection{Antipode and McKean-Vlasov characters}

The final step of our construction is to equip the coupled Bialgebra $(\scH, \circledast, \rId, \Delta, \epsilon)$ with a Hopf structure:
\begin{definition}
\label{definition:Antipode-Inverse}
Let $\scS: \scH^\dagger \to \scH^\dagger$ be defined to be the inverse of the operation $I: \scH^\dagger \to \scH^\dagger$ with respect to the convolution product. That is, 
$$
\scS:= \rId \epsilon + \sum_{i=1}^\infty ( \rId \epsilon - I)^{\ast i}
$$
for any value of $\scH^\dagger$ for which this series converges. Then $\scS$ satisfies the two inductive relationships
\begin{equation}
\label{eq:definition:Antipode-Inverse}
\begin{split}
\Big\langle \scS[X], T \Big\rangle &= -\Big\langle X, T\Big\rangle - \sum_{c\in C(T)} \Big\langle \scS[X], T_c^P\Big\rangle \otimes^{H^T} \Big\langle X, T_c^R\Big\rangle
\\
&= - \Big\langle X, T\Big\rangle - \sum_{c\in C(T)} \Big\langle X, T_c^P\Big\rangle \otimes^{H^T} \Big\langle \scS[X], T_c^R\Big\rangle. 
\end{split}
\end{equation}
\end{definition}

Thus we have that 
$$
\scS \ast I = I \ast \scS = \epsilon
$$

The next definition is a key step in the definition of probabilistic rough paths.

\begin{definition}
\label{definition:DualModule}
Let $G\big( \scH^{\gamma, \alpha, \beta}, L^0(\Omega, \bP; \bR^e)\big) \subseteq \cG\big( \scH^{\gamma, \alpha, \beta}, L^0(\Omega, \bP; \bR^e)\big)$ such that $\forall T \in \scF_0^{\gamma, \alpha, \beta}$ such that $h_0^T \neq \emptyset$, $\exists \mathcal{N}_T \subset \Omega\times \Omega^{\times |H^{\cE[T]}|}$ such that $\bP \times \bP^{\times |H^{\cE[T]}|} \big[ \mathcal{N}_T\big] = 1$ and $\forall (\omega_0, \omega_{h_0^T}, \omega_{H^T}) \in \mathcal{N}_T$, 
\begin{equation}
\label{eq:definition:DualModule2}
\Big\langle f, T \Big\rangle(\omega_{h_0^T}, \omega_{H^T}) = \Big\langle f, \cE[T] \Big\rangle (\omega_0, \omega_{h_0^T}, \omega_{H^T}). 
\end{equation}
When there is no ambiguity in the choice of coupled Hopf module $\scH^{\gamma, \alpha, \beta}$, we will denote 
$$
G\big( \scH^{\gamma, \alpha, \beta}, L^0(\Omega, \bP; \bR^e)\big) = G^{\gamma, \alpha, \beta}\big( L^0(\Omega, \bP; \bR^e)\big).
$$ 
We refer to $G^{\gamma, \alpha, \beta}\big( L^0(\Omega, \bP; \bR^e)\big)$ as the group of McKean-Vlasov characters with the group operation
\begin{equation}
\label{eq:definition:DualModule_ConvolProd}
f_1 \ast f_2 = \circledast \circ f_1 \tilde{\otimes} f_2 \circ \Delta
\end{equation}
and unit $\epsilon$. 
\end{definition}
Equation \eqref{eq:definition:DualModule1} is the canonical property that defines a character for a classical Hopf algebra. For a coupled Hopf algebra, we additionally need Equation \eqref{eq:definition:DualModule2}, which captures the idea of mapping an observation to the distribution of observations. 

As upshot of Equation \eqref{eq:definition:DualModule2} is that for any $f \in G^{\gamma, \alpha, \beta}\big( L^0(\Omega, \bP; \bR^e)\big)$, $X\in \scH^{\gamma, \alpha, \beta}$ and $T\in \scF_0$ such that $h_0^T \neq \emptyset$, 
\begin{align}
\nonumber
\bE^0\bigg[& \bE^{H^{\cE[T]}} \Big[ \big\langle f, \cE[T] \big\rangle (\omega_0, \omega_{h_0}, \omega_{H^T}) \cdot \big\langle X, \cE[T] \big\rangle(\omega_0, \omega_{h_0}, \omega_{H^T}) \Big] \bigg] 
\\
\label{eq:proposition:omegaIndependence}
&= 
\bE^{(H^{T})'} \bigg[ \big\langle f, T \big\rangle( \omega_{h_0}, \omega_{H^T}) \cdot \bE^0 \Big[ \big\langle X, \cE[T] \big\rangle (\omega_0, \omega_{h_0}, \omega_{H^T}) \Big] \bigg]. 
\end{align}

\subsubsection{Ghost hyperedges}

In our subsequent analysis of probabilistic rough paths, we often let characters act onto forests $T$ of the form $T= \cE^a[ T_1, ..., T_n]$, with the latter notation being defined in \eqref{eq:definition:E^a-notation}. This requires us to identify carefully the hyperedges that equip $T$ in terms of the hyperedges that equip each of the forests $T_1,...,T_n$:

\begin{definition}
\label{definition:Z-set}
Let $a\in \A{n}$ and let $T_1, ... T_n \in \scF_0$. For $j=1, ... m[a]$, let $\tilde{h}_j^{\cE^a[T_1, ..., T_n]}$ be a collection of disjoint, non-empty sets such that
\begin{align*}
&\bigcup_{\substack{k, \\a_k = j}} h_0^{T_k} \neq \emptyset 
\quad \implies \quad
\tilde{h}_j^{\cE^a[T_1, ..., T_n]} = \bigcup_{\substack{k, \\a_k = j}} h_0^{T_k}, 
\\
&Z^a[T_1, ..., T_n]: = \Big\{ \tilde{h}_j^{\cE^a[T_1, ..., T_n]}: j=1, ..., m[a], \bigcup_{\substack{k, \\a_k = j}} h_0^{T_k} = \emptyset \Big\}. 
\end{align*}
The elements of $Z^a[T_1, ..., T_n]$ are referred to as \emph{ghost hyperedges}, although we should emphasise that they not strictly speaking hyperedges of any hypergraph. We denote the set of ghost hyperedges
$$
\tilde{H}^{\cE^a[T_1, ..., T_n]} = \big\{ \tilde{h}_1^{\cE^a[T_1, ..., T_n]}, ..., \tilde{h}_{m[a]}^{\cE^a[T_1, ..., T_n]} \big\}. 
$$

Then
\begin{equation}
\label{eq:definition:Z-set-hyperedges}
\tilde{H}^{\cE^a[T_1, ..., T_n]} \cup \bigcup_{i=1}^n H^{T_i} = Z^a[T_1, ..., T_n] \cup H^{\cE^a[T_1, ..., T_n]}. 
\end{equation}
\end{definition}

The set $Z^a[T_1, ..., T_n]$ contains a collection of sets that can be identified with all the tagged hyperedges of the forests $T_1$ through to $T_n$ which are not included in $\cE^a\big[ T_1, ..., T_n \big]$, whence the terminology `ghost hyperedges'. These ghost hyperedges are disjoint, non-empty sets and we do not specify their content.

The requirement for ghost hyperedges is elucidated in Proposition \ref{proposition:EAction} below. The intuitive reason why we need ghost hyperedges is that, even though the $0$-hyperedge $h_{0}^T$ of a forest $T$ may be empty, the elements $\langle X,T\rangle$, for $X \in \scH$, may depend on the variable $\omega_{h_{0}^T}$. Accordingly, it is necessary to have a consistent notation to keep track of these variables when performing the operation ${\mathcal E}^a[T_{1},...,T_{n}]$. 

\begin{lemma}
Let $n\in \bN$, let $a\in \A{n}$ and let $T_1, ..., T_n, \Upsilon_1, ..., \Upsilon_n, Y_1, ...Y_n\in \scF$. Suppose that for all $i=1, ..., n$, 
$$
c' \Big( T_i, \Upsilon_i, Y_i \Big) >0. 
$$
Then
$$
Z^a\big[ T_1, ..., T_n \big] = \Big\{ \tilde{h}_j^{\cE^a[T_1, ..., T_n]}: 
\tilde{h}_j^{\cE^a[\Upsilon_1, ..., \Upsilon_n]} \in Z^a[\Upsilon_1, ..., \Upsilon_n] 
\quad \mbox{and} \quad
\tilde{h}_j^{\cE^a[Y_1, ..., Y_n]} \in Z^a[Y_1, ..., Y_n] \bigg\}
$$
\end{lemma}

\begin{proof}
By assumption, for each $i=1, ..., n$ we have that $c'(T_i, \Upsilon_i, Y_i)>0$ so that
$$
h_0^{T_i} = h_0^{\Upsilon_i} \cup h_0^{Y_i}
\quad \mbox{and} \quad
H^{T_i} \in H^{\Upsilon_i} \tilde{\cup} H^{Y_i}. 
$$
Hence
\begin{align*}
\tilde{h}_j^{\cE^{T_1, ..., T_n]}} \in Z^a[T_1, ..., T_n] \iff& \bigcup_{\substack{i: \\ a_i=j}} h_0^{T_i} = \emptyset
\\
\iff& \bigcup_{\substack{i: \\ a_i=j}} h_0^{\Upsilon_i} = \emptyset
\quad \mbox{and} \quad
\bigcup_{\substack{i: \\ a_i=j}} h_0^{Y_i} = \emptyset
\\
\iff& \tilde{h}_j^{\cE^{\Upsilon_1, ..., \Upsilon_n]}} \in Z^a[\Upsilon_1, ..., \Upsilon_n]
\quad \mbox{and} \quad
\tilde{h}_j^{\cE^{Y_1, ..., Y_n]}} \in Z^a[Y_1, ..., Y_n]. 
\end{align*}
\end{proof}

Ghost hyperedges are a way to capture the contributions for the tagged probability space of some mean-field term that are viewed as a mean-field contribution. We do not see such contributions from the probabilistic rough path due to Equation \eqref{eq:definition:DualModule2} which states that mean-field terms are independent of the tagged probability space (equivalently, the distribution of the driving signal is independent of the sample). 

However, we do see these contributions outside of the McKean-Vlasov group of characters and these will be critical in Section \ref{section:ModelledDistributions} below. 

\begin{proposition}
\label{proposition:EAction}
Let $n\in\bN$, let $T_1, ..., T_n \in \scF_0^{\gamma, \alpha, \beta}$ and let $a\in \A{n}$. Let $f\in G^{\gamma, \alpha, \beta} \big( L^0(\Omega, \bP; \bR^e) \big)$ and $X\in \scH^{\gamma, \alpha, \beta}$. Suppose that
$$
\bE^{1, ..., m[a]}\bigg[ \Big| g(\omega_0, ..., \omega_{m[a]}) \cdot \bigotimes_{i=1}^n \bE^{H^{T_i}}\Big[ \big\langle X, T_i \big\rangle(\omega_{a_i}, \omega_{H^{T_i}}) \cdot \big\langle f, T_i \big\rangle(\omega_{a_i}, \omega_{H^{T_i}}) \Big] \Big| \bigg] <\infty
$$
For brevity, denote $T = \cE^a[ T_1, ..., T_n]$. Then we can equivalently state
\begin{align}
\nonumber
&\bE^{1, ..., m[a]}\bigg[ g(\omega_0, ..., \omega_{m[a]}) \cdot \bigotimes_{i=1}^n \bE^{H^{T_i}}\Big[ \big\langle X, T_i \big\rangle (\omega_{a_i}, \omega_{H^{T_i}}) \cdot \big\langle f, T_i \big\rangle(\omega_{a_i}, \omega_{H^{T_i}}) \Big] \bigg] 
\\
\label{eq:proposition:EAction}
&= 
\bE^{H^T}\bigg[ \bE^{Z^a[T_1, ..., T_n]}\Big[ g( \omega_{0}, \omega_{\tilde{h}_1^{T}}, ..., \omega_{\tilde{h}_{m[a]}^{T}} ) \cdot \bigotimes_{i=1}^n \big\langle X, T_i \big\rangle ( \omega_{\tilde{h}_{a_i}^{T}}, \omega_{H^{T_i}}) \Big] \cdot \big\langle f, T \big\rangle (\omega_{0}, \omega_{H^T}) \bigg]  
\end{align}
where we exchangeably denote $\omega_0$ and $\omega_{\tilde{h}_0^T}$. 
\end{proposition}

\begin{proof}
Let $a\in \A{n}$. Following Equation \eqref{eq:definition:E^a-notation}, 
$$
T = \cE^a\Big[ T_1, ..., T_n\Big] = \Big[ \underset{\substack{i:\\ a_i=0}}{\circledast} T_i \Big] \circledast \cE\Big[ \underset{\substack{i:\\ a_i=1}}{\circledast} T_i \Big] \circledast ... \circledast \cE\Big[ \underset{\substack{i:\\ a_i=m[a]}}{\circledast} T_i \Big]. 
$$
By assumption, we can swap the order of integration and relabel the probability spaces $\Omega_1 \times ... \times \Omega_{m[a]}$ by $\Omega_{\tilde{h}_1^T} \times ... \times \Omega_{\tilde{h}_{m[a]}^T}$ to get
\begin{align*}
&\bE^{1, ..., m[a]}\bigg[ g(\omega_0, ..., \omega_{m[a]}) \cdot \bigotimes_{i=1}^n \bE^{H^{T_i}} \Big[ \big\langle f, T_i \big\rangle(\omega_{a_i}, \omega_{H^{T_i}}) \cdot \big\langle X, T_i \big\rangle(\omega_{a_i}, \omega_{H^{T_i}}) \Big] \bigg] 
\\
&=\bE^{\bigcup_{i=1}^n H^{T_i}} \Bigg[ \bE^{\{\tilde{h}_1^{T}, ..., \tilde{h}_{m[a]}^{T} \}}\Big[ g(\omega_{0}, \omega_{\tilde{h}_1^T} ..., \omega_{\tilde{h}_{m[a]}^{T}}) \cdot \bigotimes_{i=1}^n \big\langle X, T_i \big\rangle (\omega_{\tilde{h}_{a_i}^{T}}, \omega_{H^{T_i}}) 
\\
&\qquad \cdot \bigotimes_{j=0}^{m[a]} \big\langle f, \underset{\substack{i:\\ a_i=j}}{\circledast} T_i \big\rangle (\omega_{\tilde{h}_j^{T}}, \omega_{H^{T_i}}) \Big] \Bigg]. 
\end{align*}
Then using Equation \eqref{eq:definition:Z-set-hyperedges} allows us to state this in the same form as Equation \eqref{eq:proposition:EAction}. 

For brevity, we denote for $j=0, ..., m[a]$, the Lions forest $\tilde{T}_j= \underset{\substack{i:\\ a_i=j}}{\circledast} T_i$. 

Suppose that for some $j=1, ..., m[a]$, 
$$
\bigcup_{\substack{i:\\a_i=j}} h_0^{T_i} \neq \emptyset
$$
so that $\tilde{h}_j^{T} \in H^{T}$. Since we are integrating over $\omega_{\tilde{h}_j^{T}}$, we could replace $\tilde{T_j}$ by $\cE\big[ \tilde{T_j}\big]$. Further, the set $\tilde{h}_j^{T}$ will be included in the collection of sets $H^T$ and not in $Z^a[T_1, ..., T_n]$. 

On the other hand, if 
$$
\bigcup_{\substack{i:\\a_i=j}} h_0^{T_i} = \emptyset,
$$
then the random variable $\big\langle f, \tilde{T}_j \big\rangle$ is constant in $\omega_{\tilde{h}_j^{T}}$ by Equation \eqref{eq:proposition:omegaIndependence} so we can move this term outside the expectation over this variable. We also retain equality by rewriting this as $\big\langle f, \cE\big[ \tilde{T}_j \big]\big\rangle $. 

To conclude, we apply \eqref{eq:definition:DualModule1} from the definition of a character.
\end{proof}

The following remark is borrowed directly from \cite{salkeld2021Probabilistic}. We feel it useful for the reader.

\begin{remark}
The reader should observe that, on the first line of \eqref{eq:proposition:EAction}, the variable $\omega_{a_{1}}, ...,\omega_{a_{n}}$ that appear in the tensorial product are the same
as the variables $\omega_{0}, ..., \omega_{m[a]}$ that appear inside the function $g$. This comes from the very definition of $a$.

The second line is more subtle since the expectation is split into two subsequent expectation symbols. Chiefly, the character $f$ is independent of the variables $\omega_{\tilde{h}_{j}^T}$, for $\tilde{h}_{j}^{T} \in Z^a[T_{1}, ..., T_{n}]$. As we explained above, these variables correspond to empty $0$-hyperedges in the collection of forests $\{T_1, ..., T_n\}$. Those variables appear as arguments of the function $g$ (but not all the arguments of $g$ are indexed by those variables) and $\Big\langle X, T_i\Big\rangle$ for appropriate choices of $i=1, ..., n$. Notice that, in this scenario, the hyperedge $\tilde h_{0}^{T_{i}}$ is not a hyperedge of $T$ and so will not be included in the expansion on the left hand side of Equation \eqref{eq:proposition:EAction}. This motivates why we integrate over all hyperedges contained in $Z^a[T_1, ..., T_n]$ on the right hand side of Equation \eqref{eq:proposition:EAction}. We refer to \textit{ghost variables} as the variables $\omega_{\tilde{h}_{i}^T}$ labelled by ghost hyperedges $\tilde{h}_{i}^{T} \in Z^a[T_{1},... ,T_{n}]$.

Lastly, from Definition \ref{definition:Z-set} we emphasise that the sets $\tilde{h}_j^T$ are disjoint and thus the variables $\omega_{\tilde{h}_i^T}$ and $\omega_{\tilde{h}_j^T}$ cannot be identified.
\end{remark}

\subsection{Graded norms and metrics}
\label{subsection:Integrability-Characters}

We now introduce a class of norms on the coupled Hopf algebra $\scH^{\gamma, \alpha, \beta}$. These norms are graded, with the grading relying on an auxiliary integrability function $p$. 

\begin{definition}
\label{definition:Graded-norm}
Let $\alpha, \beta>0$, $\gamma\geq\alpha\wedge\beta$. Let $p: \scF_0^{\gamma, \alpha, \beta} \to [1, \infty)$. We refer to $p$ as an \emph{integrability functional}. Let 
$$
\fV = \bigg\{ L^0\Big( \Omega \times \Omega^{\times |H^T|}, \bP \times (\bP)^{\times |H^T|}; \lin\big( (\bR^d)^{\otimes |N^T|}, \bR^e\big) \Big): (N^T, E^T, h_0^T, H^T, L^T) \in \scF_0^{\gamma, \alpha, \beta} \bigg\}. 
$$
Given an integrability functional $p$, we define a graded norm $\rp$ of $\scH^{\gamma, \alpha, \beta}$ by
$$
\| X \|_{\rp} = \sum_{T \in \scF^{\gamma, \alpha, \beta}} \bE^{(H^T)'} \bigg[ \Big| \langle X, T\rangle(\omega_{h_0}, \omega_{H^T}) \Big|^{p[T]}  \bigg]^{\tfrac{1}{p[T]}}. 
$$
\end{definition}

In general, we will be less interested in norms on $\scH^{\gamma,\alpha, \beta}$ and more interested in dual norms on $(\scH^{\gamma,\alpha, \beta})^*$ with certain integrability conditions that we will need for our rough paths. 

\begin{definition}
\label{definition:(dual)_integrab-functional}
Let $\alpha, \beta>0$, $\gamma\geq \alpha\wedge\beta$. Let $p:\scF_0^{\gamma, \alpha, \beta} \to [1, \infty)$ be an integrability functional. Let $q: \scF^{\gamma, \alpha, \beta} \to [1, \infty)$ such that $\forall T \in \scF^{\gamma, \alpha, \beta}$, 
\begin{equation}
\label{eq:IntegrabilityP1}
\frac{1}{p[\rId]} = \frac{1}{p[T]} + \frac{1}{q[T]}
\end{equation}
and $\forall T_1, T_2, T_3 \in \scF^{\gamma, \alpha, \beta}$ such that $c(T_1, T_2, T_3) > 0$, 
\begin{equation}
\label{eq:IntegrabilityP2}
\frac{1}{q[T_1]} \geq \frac{1}{q[T_2]} + \frac{1}{q[T_3]}, 
\quad 
q\big[ \cE[T_1] \big] = q\big[ T_1 \big]
\quad \mbox{and} \quad
\frac{1}{q[T_1 \circledast T_2]} = \frac{1}{q[T_1]} + \frac{1}{q[T_2]}. 
\end{equation}
Here $c:\scF_0 \times \scF_0 \times \scF_0 \to \bN_0$ is the coproduct counting function (see Equation \eqref{eq:Coproduct-counting}). 

Then we say that $q$ is the dual integrability functional of $p$. We refer to the pair $(p, q)$ as a dual pair of integrability functionals. 
\end{definition}

\begin{lemma}
\label{lemma:(dual)_integrab-functional}
Let $\alpha, \beta>0$, $\gamma>\alpha\wedge\beta$ and let $(p, q)$ be a dual pair of integrability functionals. Then $\forall T, \Upsilon, Y\in \scF$ such that $c'(T, \Upsilon, Y)>0$, 
$$
\frac{1}{p[Y]} \geq \frac{1}{p[T]} + \frac{1}{q[\Upsilon]}. 
$$
\end{lemma}

\begin{proof}
Thanks to Equation \eqref{eq:IntegrabilityP1} and Equation \eqref{eq:IntegrabilityP2}, we have that
\begin{align*}
\frac{1}{p[Y]} =& \frac{1}{p[\rId]} - \frac{1}{q[Y]}
\\
=& \Big( \frac{1}{p[T]} + \frac{1}{q[T]} \Big) - \frac{1}{q[Y]}
\geq
\frac{1}{p[T]} + \frac{1}{q[\Upsilon]}. 
\end{align*}
\end{proof}

\subsection{Probabilistic rough paths}
\label{subsection:ProbabilisticRPsDefs}

With a sense of integrability and a metric over our group of characters, we are able to introduce a definition for the titular probabilistic rough path:

\begin{definition}
\label{definition:ProbabilisticRoughPaths}
Let $\alpha, \beta>0$ and let $\gamma:= \inf\{ \scG_{\alpha, \beta}[T]: T \in \scF, \scG_{\alpha, \beta}[T]>1-\alpha \}$. Let $(p,q)$ be a dual pair of integrability functionals. 

We say that $\rw:[0,1]^2 \to G^{\gamma, \alpha, \beta}\big( L^{p[\rId]}(\Omega, \bP; \bR^e)\big)$ is a $(\scH^{\gamma, \alpha, \beta}, p,q)$-probabilistic rough path if 
\begin{align}
\label{eq:definition:ProbabilisticRoughPaths1}
&\rw_{s,t} \ast \rw_{t,u} = \rw_{s,u}
\end{align}
and $\forall T \in \scF^{\gamma, \alpha, \beta}$, 
\begin{align}
\label{eq:definition:ProbabilisticRoughPaths2}
\bE^0 \Bigg[ \sup_{s, t\in[0,1]} \frac{\bE^{H^T}\bigg[ \Big| \big\langle \rw_{s, t} , T\big\rangle(\omega_{0}, \omega_{H^T}) \Big|^{q[T]} \bigg]}{\big| t-s \big|^{q[T]\scG_{\alpha, \beta}[T]}} \Bigg]^{\tfrac{1}{q[T]}} < \infty. 
\end{align}
We denote the set of probabilistic rough paths by $\scC(\scH^{\gamma, \alpha, \beta}, p,q)$. 

We define the inhomogeneous (random) pseudo-metric 
$$
\rho_{(\alpha, \beta, p, q), 0}: \scC(\scH^{\gamma, \alpha, \beta}, p, q) \times \scC(\scH^{\gamma, \alpha, \beta}, p, q) \to L^{p[\rId]}( \Omega, \bP; \bR^+)
$$
by
\begin{align}
\nonumber
\rho_{(\alpha, \beta, p, q), 0}\Big( \rv, \rw\Big)(\omega_0) 
=& 
\inf_{\Pi} \sum_{T \in \scF^{\gamma, \alpha, \beta}} \sup_{s, t\in[0,1]} \frac{1}{|t-s|^{\scG_{\alpha, \beta}[T]}} \bigg( \int_{( \hat{\Omega}\times \overline{\Omega})^{\times |H^T|}} \Big| \big\langle \rv_{s, t}, T\big\rangle(\omega_0, \hat{\omega}_{H^T}) 
\\
\label{eq:definition:ProbabilisticRoughPaths3}
& - \big\langle \rw_{s, t}, T\big\rangle(\omega_0, \overline{\omega}_{H^T})\Big|^{q[T]} d\Big( \Pi(\hat{\omega}, \overline{\omega}) \Big)^{\times|H^T|} \bigg)^{\tfrac{1}{q[T]}} 
\end{align}
where $\inf_{\Pi}$ runs over all probability measures on $\big( \hat{\Omega}\times \overline{\Omega}, \hat{\cF}\otimes \overline{\cF} \big)$ with left and right marginals $\bP$ and $\bP$ respectively. 
\end{definition}

Using \eqref{eq:definition:DualModule2}, the reader may observe that for $T\in \scF$ such that $h_0^T \neq \emptyset$, the random variable 
$$
\Big\langle \rw_{s, t} , T\Big\rangle(\omega_{h_0}, \omega_{H^T}) = \Big\langle \rw_{s, t}, \cE[T] \Big\rangle(\omega_0, \omega_{h_0}, \omega_{H^T})
$$
so that $\langle \rw_{s, t}, \cE[T] \rangle(\omega_0, \omega_{h_0}, \omega_{H^T})$ is constant in $\omega_0$ and Equation \eqref{eq:definition:ProbabilisticRoughPaths2} can be restated as
$$
\bE^{H^{\cE[T]}}\bigg[ \Big| \big\langle \rw_{s, t} , \cE[T] \big\rangle(\omega_0, \omega_{h_0}, \omega_{H^T}) \Big|^{q[T]} \bigg]^{\tfrac{1}{q[T]}} \leq C_{\cE[T]}\cdot \big| t-s \big|^{\scG_{\alpha, \beta}\big[ \cE[T] \big]}. 
$$


\newpage
\section{Random controlled rough paths and couplings}
\label{section:ModelledDistributions}

Now that we have briefly laid out the definition of probabilistic rough paths first introduced in \cite{salkeld2021Probabilistic}, we want to introduce the related notion of paths controlled by a probabilistic rough path following following the ideas of \cite{gubinelli2004controlling}. To this end, the goal of this section is to define and prove some key properties of \emph{random controlled rough paths}. This idea was first introduced in \cite{2019arXiv180205882.2B}, although the definition we use in this work builds on the key philosophy that a probabilistic rough path is a path of the characters of a coupled Hopf algebra along with a choice of graded norm. 

Throughout this section, we will use that convention that $T$ is a tree with prunes $\Upsilon$ and roots $Y$. This will aid distinction and identification for the reader when there are multiple prunes and roots within an expression. Further, following on from Definition \ref{definition:coproduct} we will be using the reduced coproduct counting function $c':\scF \times \scF \times \scF\to \bN$ defined by
\begin{align*}
c' \Big( T, \Upsilon, Y \Big) = \Big\langle \Delta'\big[ T \big] , \Upsilon \times^{H^T} Y \Big\rangle,
\\
\Delta'\big[ T \big] = \Delta[T] - \rId\times^{H^T} T - T \times^{H^T} \rId. 
\end{align*}

For a set $\scA \subseteq \scF_0$, we will henceforward denote 
$$
\sum_{T\in\scA}^{\gamma, \alpha, \beta}
:= 
\sum_{\substack{T\in \scA\\ \scG_{\alpha, \beta}[T]\leq \gamma}}, 
\quad
\sum_{T\in\scA}^{\gamma-, \alpha, \beta}
:= 
\sum_{\substack{T\in \scA\\ \scG_{\alpha, \beta}[T] < \gamma}}. 
$$ 

This section is central to our paper. In Subsection \ref{subse:def:RCRP}, we provide and describe the Definition of a random controlled rough path, see Definition \ref{definition:RandomControlledRP}. In Subsection \ref{subsec:Int_MF_ElemDiff}, we provide some motivation for this choice by considering elementary differentials of a McKean-Vlasov equation. Subsection \ref{subsec:Operat_RCRPs} contains two key theorems of this paper along with the necessary notational explanations. Subsections \ref{subsec:CoupledCoproductId} and \ref{subsec:DualityId} include a collection of definitions and results that will be key to working with random controlled rough paths. The reader should be able to take these results on trust during the first read through. Subsection \ref{subsec:RCRPs} addresses some of the technical details of working with random controlled rough paths, including a proof that the space of random controlled rough paths is a Banach space (see Theorem \ref{theorem:Banachspace}). Finally, Subsection \ref{subsection:Proofs-Sec4} provides the proofs to the results of Subsection \ref{subsec:Operat_RCRPs}. 

\subsection{Definition of a random controlled rough path}
\label{subse:def:RCRP}

We first provide the definition of a random controlled rough path, which is the key notion in our contribution. The reader will notice some obvious similarities with the definition of a standard controlled rough path in \eqref{eq:Controlled_RP1}. However, a modicum of care is needed as the random feature of the paths create a lot of additional technicalities, which represent the novelty and interest of this definition. In particular, the encoding of the jets in the expansion of a random controlled rough path requires a proper identification of the connections between the hyperedges of a Lions forest $T$ and the hyperedges of the prune and the root for some cut of $T$. Formally, this identification goes through Definitions \ref{definition:CouplingFunctions} and \ref{definition:E-Set}. The reader can study these two definitions directly, but, to ease the reading, we have chosen to provide here a more intuitive primer of them 
at this stage of the exposition.

For $T,\Upsilon,Y \in \scF$ such that $c'(T,\Upsilon,Y) >0$ and for a non-0 hyperedge $h^\Upsilon$ of $\Upsilon$ (i.e., $h^\Upsilon \in H^{\Upsilon}$), $h^\Upsilon$ can be regarded as the restriction to the nodes of $\Upsilon$ of a wider hyperedge of $T$, which we denote in Definition \ref{definition:CouplingFunctions} by $\psi^{\Upsilon,T}(h^\Upsilon)$. Here, we use the simpler notation $h^T$ for $\psi^{\Upsilon,T}(h^\Upsilon)$. Then we say that $h^{\Upsilon}$ is not connected to $Y$ (with respect to the coupling $H^T$) if $h^{T}$ does not contain any node of $Y$. We call $E^{T,\Upsilon,Y}$ the collection of hyperedges of $\Upsilon$ that are not connected to $Y$, which implicitly means that the $0$-hyperedge of $\Upsilon$ is always regarded as being connected. When $h^{\Upsilon}$ is connected to $Y$, we can associate with it an hyperedge of $Y$, namely an element of $H^Y$, by considering $h^Y := h^T \cap N^Y$ where $N^Y$ denotes the collection of nodes of $Y$. In the following definition, $h^Y$ is denoted by $\varphi^{T,\Upsilon,Y}[h^\Upsilon]$. 

Conversely, if we are given first $h^T$ an hyperedge of $T$, we may divide it according to the cut that leads to $\Upsilon$ and $Y$ and then retain one piece only. If $h^T$ does not contain any node of $Y$, then we regard $h^T$ as an hyperedge of $\Upsilon$, which we retain as a result of the cut. If $h^T$ has a non-empty intersection with $N^Y$, then $h^T\cap N^Y$ is an hyperedge of $Y$ and we retain it as a result of the cut. We let $h^\Upsilon:=h^T$ in the first case, with $h^\Upsilon$ being formally regarded as an hyperedge of $\Upsilon$, and $h^{Y}:=h^T \cap N^Y$ in the second case, with $h^\Upsilon$ being seen as an hyperedge of $Y$. Using our terminology, it should be clear to the reader that, in the former case, $h^{\Upsilon}$ is not connected to $Y$ (with the respect to the coupling $H^T$). In Definition \ref{definition:CouplingFunctions}, the operation that maps $h^T$ to either $h^\Upsilon$ or $h^Y$ is encoded by means of a function $\phi^{T,\Upsilon,Y}:H^T \rightarrow H^\Upsilon \cup H^Y$: $\phi^{T,\Upsilon,Y}[h^T]$ is either equal to $h^\Upsilon$ or $h^Y$ whether it contains nodes of $Y$ or not. 

With this detail clarified, we introduce the core contribution of this work:
\begin{definition}
\label{definition:RandomControlledRP}
Let $\alpha, \beta>0$ and $\gamma:=\inf\{ \scG_{\alpha, \beta}[T]: T \in \scF, \scG_{\alpha, \beta}[T]>1-\alpha \}$. Let $(p,q)$ be a dual pair of integrability functionals and let $\rw$ be an $(\scH^{\gamma, \alpha, \beta}, p, q)$-probabilistic rough path. 

A path $\bX:[0,1]\to \scH^{\gamma-, \alpha, \beta}$ is called a $p$-Random Controlled Rough Path (RCRP) controlled by $\rw$ if $\forall Y \in \scF^{\gamma-, \alpha, \beta}$, $\forall s, t\in[0,1]$ with $s<t$, 
\begin{align}
\label{eq:definition:RandomControlledRP1}
\Big\langle \bX_{s, t},& \rId \Big\rangle(\omega_0) = \sum_{T \in \scF}^{\gamma-, \alpha, \beta} \bE^{H^T} \bigg[ \Big\langle \bX_s, T \Big\rangle(\omega_0, \omega_{H^T}) \cdot \Big\langle \rw_{s, t}, T \Big\rangle(\omega_0, \omega_{H^T}) \bigg] + \Big\langle \bX^{\sharp}_{s, t}, \rId \Big\rangle(\omega_0), 
\\
\nonumber
\Big\langle \bX_{s, t},& Y \Big\rangle(\omega_0, \omega_{H^Y}) 
\\
\nonumber
=& \sum_{T \in \scF}^{\gamma-,\alpha, \beta} \sum_{\Upsilon \in \scF} c'\Big( T, \Upsilon, Y \Big) \cdot \bE^{E^{T, \Upsilon, Y}}\bigg[ \Big\langle \bX_s, T\Big\rangle(\omega_0, \omega_{\phi^{ T, \Upsilon, Y}[H^T]}) \cdot \Big\langle \rw_{s, t}, \Upsilon \Big\rangle(\omega_0, \omega_{\varphi^{T, \Upsilon, Y}[H^\Upsilon]}) \bigg] 
\\
\label{eq:definition:RandomControlledRP2}
&+ \Big\langle \bX^{\sharp}_{s, t}, Y \Big\rangle(\omega_0, \omega_{H^Y}), 
\end{align}
and $\forall Y \in \scF_0^{\gamma-, \alpha, \beta}$, 
\begin{equation}
\label{eq:definition:RandomControlledRP4}
\begin{split}
\bE^0\bigg[ \sup_{t\in[0,1]} & \bE^{H^Y}\bigg[ \Big| \big\langle \bX_t, Y \big\rangle (\omega_{0}, \omega_{H^Y}) \Big|^{p[Y]} \bigg] \bigg]<\infty, 
\\
\bE^0\Bigg[ \sup_{s, t\in[0,1]} &
\frac{\bE^{H^Y}\Big[\big| \langle \bX^{\sharp}_{s, t}, Y \rangle(\omega_0, \omega_{H^Y}) \big|^{p[Y]} \Big]}{|t-s|^{p[Y](\gamma - \scG_{\alpha, \beta}[Y])}} \Bigg]^{\tfrac{1}{p[Y]}} <\infty. 
\end{split}
\end{equation}

We call $\cD_{\rw}^{\gamma, p, q}(\scH^{\gamma, \alpha, \beta})$ the space of all $p$-RCRPs controlled by $\rw$. We equip it with the (random) seminorm $\big\| \cdot \big\|_{\rw, \gamma, p, q, 0}: \cD_{\rw}^{\gamma, p, q}(\scH^{\gamma, \alpha, \beta}) \to L^{p[\rId]}\big( \Omega, \bP; \bR^+\big)$ defined by
\begin{align}
\nonumber
\big\| \bX \big\|_{\rw, \gamma, p, q, 0}(\omega_0)
=& \sum_{T \in \scF_0}^{\gamma-, \alpha, \beta} \bE^{H^T}\bigg[ \Big| \big\langle \bX_0, T \big\rangle(\omega_0, \omega_{H^T}) \Big|^{p[T]} \bigg]^{\tfrac{1}{p[T]}}
\\
\label{eq:definition:RandomControlledRP3}
&+ \sum_{T\in \scF_0}^{\gamma-, \alpha, \beta} \sup_{s, t\in[0,1]} 
\frac{ \bE^{H^T}\bigg[ \Big| \big\langle \bX_{s,t}^{\sharp}, T \big\rangle (\omega_{0}, \omega_{H^T}) \Big|^{p[T]} \bigg]^{\tfrac{1}{p[T]}} }{|t-s|^{\gamma - \scG_{\alpha, \beta}[T]}}. 
\end{align}

When there is no ambiguity in the choice of coupled Hopf algebra $\scH^{\gamma, \alpha, \beta}$, we denote the space of random controlled rough paths (controlled by $\rw$) by $\cD_\rw^{\gamma, p,q}$. 
\end{definition}

\begin{remark}
\label{remark:timeint-length}
Throughout this entire work, the time interval is taken to be $[0,1]$, but we could equivalently take it to be any finite interval of the real line. When another time interval is specified, it is always as a subset $[u, v]\subseteq [0,1]$ which will be used to demonstrate the dependence in those estimates to the length of the time interval. This is vital for key rough path localisation techniques. 
\end{remark}

\begin{remark}
Let us consider the use of the set of untagged hyperedges $E^{T, \Upsilon, Y}$: in Equation \eqref{eq:definition:RandomControlledRP1}, we should think that the only hyperedge being tagged is the 0-hyperedge $h_0$. Hence, $\omega_0$ is the only variable on the left-hand side of the expression and on the right hand side of the expression all other hyperedges (the set $H^T$) are being integrated over. 

By contrast, in Equation \eqref{eq:definition:RandomControlledRP2}, we only integrate over the set of decoupled hyperedges $E^{T, \Upsilon, Y}$. Once again, we recall that this is the collection of hyperedges of $\Upsilon$ that are not connected to $Y$ and this does not contain the $0$-hyperedge. The 0-hyperedge remains coupled as before via $\omega_0$, but we also fix $\omega_{H^Y}$ on the left-hand side. The operation $\phi^{T, \Upsilon, Y}$ maps hyperedges of the Lions tree $T$ onto the hyperedges of the Lions tree $Y$ except when there is no associated hyperedge, in which case it maps to $H^\Upsilon$. In the latter case, the image is regarded as a non-connected hyperedge of $\Upsilon$, i.e., as an element of $E^{T,\Upsilon,Y}$. Thus, any hyperedge $h\in H^T$ such that $\phi^{T, \Upsilon, Y}[h] \in H^{\Upsilon}$ will be contained in $E^{T, \Upsilon, Y}$ and any other hyperedges are contained in $H^Y$ and so are tagged. 

In the same fashion, $\varphi^{T, \Upsilon, Y}$ maps to hyperedges of $\Upsilon$ to $H^Y$ when there is an associated hyperedge tagged by $H^T$, and $H^\Upsilon$ otherwise. Thus $\varphi^{T, \Upsilon, Y}[h] \in H^\Upsilon$ will be contained in $E^{T, \Upsilon, Y}$ and all other hyperedges will be contained in $H^Y$, making them tagged. 
\end{remark}

\begin{example}
\label{example:Tikz}
Let us visualise Definition \ref{definition:RandomControlledRP} in the case where $d=1$ (so that all nodes are labelled identically). Let $\alpha = \frac{11}{50}$ and $\beta = \tfrac{15}{50}$. Thus $\alpha \in \big( \tfrac{1}{5}, \tfrac{1}{4} \big)$ and $\beta \in \big( \tfrac{1}{4}, \tfrac{1}{3}\big)$ and $\gamma = \tfrac{41}{50}=\alpha+2\beta$ since
$$
4\alpha < 3\alpha + \beta < 1 < 2\alpha + 2\beta < 5 \alpha < \alpha + 3 \beta. 
$$
Hence
$$
\Big\{ \alpha i + \beta j: (i, j)\in \bN_0: \alpha i + \beta j<\gamma\big\} = \Big\{ \alpha, \beta, 2\alpha, \alpha + \beta, 2\beta, 3\alpha, 2\alpha + \beta \Big\}
$$
and we are left considering the collection of Lions forests:
\begin{align*}
\scF^{\gamma-, \alpha, \beta}=\Bigg\{
&
\vcenter{\hbox{
\begin{tikzpicture}
\node[vertexS] at (0,0) {}; 
\begin{pgfonlayer}{background}
\draw[zhyedge1S] (0,0) -- (-0,0);
\draw[zhyedge2S] (0,0) -- (-0,0);
\end{pgfonlayer}
\end{tikzpicture}
}}
, 
\vcenter{\hbox{
\begin{tikzpicture}
\node[vertexS] at (0,0) {}; 
\begin{pgfonlayer}{background}
\draw[hyedge1S, color=red] (0,0) -- (-0,0);
\draw[hyedge2S] (0,0) -- (-0,0);
\end{pgfonlayer}
\end{tikzpicture}
}}
,
\vcenter{\hbox{
\begin{tikzpicture}
\node[vertexS] at (0,0) {}; 
\node[vertexS] at (0,0.5) {}; 
\draw[edge] (0,0) -- (0,0.5);
\begin{pgfonlayer}{background}
\draw[zhyedge1S] (0,0) -- (0,0.5);
\draw[zhyedge2S] (0,0) -- (0,0.5);
\end{pgfonlayer}
\end{tikzpicture}
}}
, 
\vcenter{\hbox{
\begin{tikzpicture}
\node[vertexS] at (0,0) {}; 
\node[vertexS] at (0.5,0) {}; 
\begin{pgfonlayer}{background}
\draw[zhyedge1S] (0,0) -- (0.5,0);
\draw[zhyedge2S] (0,0) -- (0.5,0);
\end{pgfonlayer}
\end{tikzpicture}
}}
,
\vcenter{\hbox{
\begin{tikzpicture}
\node[vertexS] at (0,0) {}; 
\node[vertexS] at (0,0.5) {}; 
\draw[edge] (0,0) -- (0,0.5);
\begin{pgfonlayer}{background}
\draw[zhyedge1S] (0,0) -- (0,0);
\draw[zhyedge2S] (0,0) -- (0,0);
\draw[hyedge1S, color=red] (0,0.5) -- (0,0.5);
\draw[hyedge2S] (0,0.5) -- (0,0.5);
\end{pgfonlayer}
\end{tikzpicture}
}}
, 
\vcenter{\hbox{
\begin{tikzpicture}
\node[vertexS] at (0,0) {}; 
\node[vertexS] at (0.5,0) {}; 
\begin{pgfonlayer}{background}
\draw[zhyedge1S] (0,0) -- (0,0);
\draw[zhyedge2S] (0,0) -- (0,0);
\draw[hyedge1S, color=red] (0.5,0) -- (0.5,0);
\draw[hyedge2S] (0.5,0) -- (0.5,0);
\end{pgfonlayer}
\end{tikzpicture}
}}
,
\vcenter{\hbox{
\begin{tikzpicture}
\node[vertexS] at (0,0) {}; 
\node[vertexS] at (0, 0.5) {}; 
\draw[edgeS] (0,0) -- (0,0.5);
\begin{pgfonlayer}{background}
\draw[hyedge1S, color=red] (0,0) -- (0,0.5);
\draw[hyedge2S] (0,0) -- (0,0.5);
\end{pgfonlayer}
\end{tikzpicture}
}}
,
\vcenter{\hbox{
\begin{tikzpicture}
\node[vertexS] at (0,0) {}; 
\node[vertexS] at (0.5,0) {}; 
\begin{pgfonlayer}{background}
\draw[hyedge1S, color=red] (0,0) -- (0.5,0);
\draw[hyedge2S] (0,0) -- (0.5,0);
\end{pgfonlayer}
\end{tikzpicture}
}}
,
\vcenter{\hbox{
\begin{tikzpicture}
\node[vertexS] at (0,0) {}; 
\node[vertexS] at (0, 0.5) {}; 
\draw[edgeS] (0,0) -- (0,0.5);
\begin{pgfonlayer}{background}
\draw[hyedge1S, color=red] (0,0) -- (0,0);
\draw[hyedge2S] (0,0) -- (0,0);
\draw[hyedge1S, color=blue] (0,0.5) -- (0,0.5);
\draw[hyedge2S] (0,0.5) -- (0,0.5);
\end{pgfonlayer}
\end{tikzpicture}
}}
,
\vcenter{\hbox{
\begin{tikzpicture}
\node[vertexS] at (0,0) {}; 
\node[vertexS] at (0.5,0) {}; 
\begin{pgfonlayer}{background}
\draw[hyedge1S, color=red] (0,0) -- (0,0);
\draw[hyedge2S] (0,0) -- (0,0);
\draw[hyedge1S, color=blue] (0.5,0) -- (0.5,0);
\draw[hyedge2S] (0.5,0) -- (0.5,0);
\end{pgfonlayer}
\end{tikzpicture}
}}
,
\\
&
\vcenter{\hbox{
\begin{tikzpicture}
\node[vertexS] at (0,0) {}; 
\node[vertexS] at (0,0.5) {}; 
\node[vertexS] at (0,1) {}; 
\draw[edgeS] (0,0) -- (0,0.5);
\draw[edgeS] (0,0.5) -- (0,1);
\begin{pgfonlayer}{background}
\draw[zhyedge1S] (0,0) -- (0,1);
\draw[zhyedge2S] (0,0) -- (-0,1);
\end{pgfonlayer}
\end{tikzpicture}
}}
,
\vcenter{\hbox{
\begin{tikzpicture}
\node[vertexS] at (0,0) {}; 
\node[vertexS] at (0.25,0.5) {}; 
\node[vertexS] at (-0.25,0.5) {}; 
\draw[edgeS] (0,0) -- (0.25,0.5);
\draw[edgeS] (0,0) -- (-0.25,0.5);
\begin{pgfonlayer}{background}
\draw[zhyedge1S] (-0.25,0.5) -- (0,0) -- (0.25,0.5);
\draw[zhyedge2S] (-0.25,0.5) -- (0,0) -- (0.25,0.5);
\end{pgfonlayer}
\end{tikzpicture}
}}
,
\vcenter{\hbox{
\begin{tikzpicture}
\node[vertexS] at (0,0) {}; 
\node[vertexS] at (0.5,0) {}; 
\node[vertexS] at (0.5,0.5) {}; 
\draw[edgeS] (0.5,0) -- (0.5,0.5);
\begin{pgfonlayer}{background}
\draw[zhyedge1S] (0,0) -- (0.5,0) -- (0.5,0.5);
\draw[zhyedge2S] (0,0) -- (0.5,0) -- (0.5,0.5);
\end{pgfonlayer}
\end{tikzpicture}
}}
,
\vcenter{\hbox{
\begin{tikzpicture}
\node[vertexS] at (0,0) {}; 
\node[vertexS] at (0.25,0) {}; 
\node[vertexS] at (0.5,0) {}; 
\begin{pgfonlayer}{background}
\draw[zhyedge1S] (0,0) -- (0.5,0);
\draw[zhyedge2S] (0,0) -- (0.5,0);
\end{pgfonlayer}
\end{tikzpicture}
}}
,
\vcenter{\hbox{
\begin{tikzpicture}
\node[vertexS] at (0,0) {}; 
\node[vertexS] at (0,0.5) {}; 
\node[vertexS] at (0,1) {}; 
\draw[edgeS] (0,0) -- (0,0.5);
\draw[edgeS] (0,0.5) -- (0,1);
\begin{pgfonlayer}{background}
\draw[zhyedge1S] (0,0) -- (0,0.5);
\draw[zhyedge2S] (0,0) -- (-0,0.5);
\draw[hyedge1S, color=red] (0,1) -- (0,1);
\draw[hyedge2S] (0,1) -- (0,1);
\end{pgfonlayer}
\end{tikzpicture}
}}
,
\vcenter{\hbox{
\begin{tikzpicture}
\node[vertexS] at (0,0) {}; 
\node[vertexS] at (0.25,0.5) {}; 
\node[vertexS] at (-0.25,0.5) {}; 
\draw[edgeS] (0,0) -- (0.25,0.5);
\draw[edgeS] (0,0) -- (-0.25,0.5);
\begin{pgfonlayer}{background}
\draw[zhyedge1S] (-0.25,0.5) -- (0,0);
\draw[zhyedge2S] (-0.25,0.5) -- (0,0);
\draw[hyedge1S, color=red] (0.25,0.5) -- (0.25,0.5);
\draw[hyedge2S] (0.25,0.5) -- (0.25,0.5);
\end{pgfonlayer}
\end{tikzpicture}
}}
,
\vcenter{\hbox{
\begin{tikzpicture}
\node[vertexS] at (0,0) {}; 
\node[vertexS] at (0.5,0) {}; 
\node[vertexS] at (0.5,0.5) {}; 
\draw[edgeS] (0.5,0) -- (0.5,0.5);
\begin{pgfonlayer}{background}
\draw[zhyedge1S] (0,0) -- (0.5,0);
\draw[zhyedge2S] (0,0) -- (0.5,0);
\draw[hyedge1S, color=red] (0.5,0.5) -- (0.5,0.5);
\draw[hyedge2S] (0.5,0.5) -- (0.5,0.5);
\end{pgfonlayer}
\end{tikzpicture}
}}
,
\vcenter{\hbox{
\begin{tikzpicture}
\node[vertexS] at (0,0) {}; 
\node[vertexS] at (0.5,0) {}; 
\node[vertexS] at (0.5,0.5) {}; 
\draw[edgeS] (0.5,0) -- (0.5,0.5);
\begin{pgfonlayer}{background}
\draw[zhyedge1S] (0.5,0.5) -- (0.5,0);
\draw[zhyedge2S] (0.5,0.5) -- (0.5,0);
\draw[hyedge1S, color=red] (0,0) -- (0,0);
\draw[hyedge2S] (0,0) -- (0,0);
\end{pgfonlayer}
\end{tikzpicture}
}}
,
\vcenter{\hbox{
\begin{tikzpicture}
\node[vertexS] at (0,0) {}; 
\node[vertexS] at (0.25,0) {}; 
\node[vertexS] at (0.75,0) {}; 
\begin{pgfonlayer}{background}
\draw[zhyedge1S] (0,0) -- (0.25,0);
\draw[zhyedge2S] (0,0) -- (0.25,0);
\draw[hyedge1S, color=red] (0.75,0) -- (0.75,0);
\draw[hyedge2S] (0.75,0) -- (0.75,0);
\end{pgfonlayer}
\end{tikzpicture}
}}
\Bigg\}
\end{align*}
Then each of the jets described by Equation \eqref{eq:definition:RandomControlledRP2} are equivalent to the following:
\begin{align*}
\Big\langle \bX_{s, t}, &
\vcenter{\hbox{
\begin{tikzpicture}
\node[vertexS] at (0,0) {}; 
\begin{pgfonlayer}{background}
\draw[zhyedge1S] (0,0) -- (-0,0);
\draw[zhyedge2S] (0,0) -- (-0,0);
\end{pgfonlayer}
\end{tikzpicture}
}}
\Big\rangle(\omega_0) 
= 
\bE^1\Bigg[ \Big\langle \bX_s, 
\vcenter{\hbox{
\begin{tikzpicture}
\node[vertexS] at (0,0) {}; 
\node[vertexS] at (0,0.5) {}; 
\draw[edge] (0,0) -- (0,0.5);
\begin{pgfonlayer}{background}
\draw[zhyedge1S] (0,0) -- (0,0);
\draw[zhyedge2S] (0,0) -- (0,0);
\draw[hyedge1S, color=red] (0,0.5) -- (0,0.5);
\draw[hyedge2S] (0,0.5) -- (0,0.5);
\end{pgfonlayer}
\end{tikzpicture}
}}
+
\vcenter{\hbox{
\begin{tikzpicture}
\node[vertexS] at (0,0) {}; 
\node[vertexS] at (0.5,0) {}; 
\begin{pgfonlayer}{background}
\draw[zhyedge1S] (0,0) -- (0,0);
\draw[zhyedge2S] (0,0) -- (0,0);
\draw[hyedge1S, color=red] (0.5,0) -- (0.5,0);
\draw[hyedge2S] (0.5,0) -- (0.5,0);
\end{pgfonlayer}
\end{tikzpicture}
}}
\Big\rangle(\omega_0, {\color{red}\omega_1}) \Big\langle \rw_{s, t}, 
\vcenter{\hbox{
\begin{tikzpicture}
\node[vertexS] at (0,0) {}; 
\begin{pgfonlayer}{background}
\draw[hyedge1S, color=red] (0,0) -- (-0,0);
\draw[hyedge2S] (0,0) -- (-0,0);
\end{pgfonlayer}
\end{tikzpicture}
}}
\Big\rangle(\omega_0, {\color{red}\omega_1}) \Bigg]
\\
&+
\Big\langle \bX_s, 
\vcenter{\hbox{
\begin{tikzpicture}
\node[vertexS] at (0,0) {}; 
\node[vertexS] at (0,0.5) {}; 
\draw[edge] (0,0) -- (0,0.5);
\begin{pgfonlayer}{background}
\draw[zhyedge1S] (0,0) -- (0,0.5);
\draw[zhyedge2S] (0,0) -- (0,0.5);
\end{pgfonlayer}
\end{tikzpicture}
}}
+ 2 
\vcenter{\hbox{
\begin{tikzpicture}
\node[vertexS] at (0,0) {}; 
\node[vertexS] at (0.5,0) {}; 
\begin{pgfonlayer}{background}
\draw[zhyedge1S] (0,0) -- (0.5,0);
\draw[zhyedge2S] (0,0) -- (0.5,0);
\end{pgfonlayer}
\end{tikzpicture}
}}
\Big\rangle(\omega_0) \Big\langle \rw_{s, t}, 
\vcenter{\hbox{
\begin{tikzpicture}
\node[vertexS] at (0,0) {}; 
\begin{pgfonlayer}{background}
\draw[zhyedge1S] (0,0) -- (-0,0);
\draw[zhyedge2S] (0,0) -- (-0,0);
\end{pgfonlayer}
\end{tikzpicture}
}}
\Big\rangle(\omega_0)
+
\Big\langle \bX_s, 
\vcenter{\hbox{
\begin{tikzpicture}
\node[vertexS] at (0,0) {}; 
\node[vertexS] at (0,0.5) {}; 
\node[vertexS] at (0,1) {}; 
\draw[edgeS] (0,0) -- (0,0.5);
\draw[edgeS] (0,0.5) -- (0,1);
\begin{pgfonlayer}{background}
\draw[zhyedge1S] (0,0) -- (0,1);
\draw[zhyedge2S] (0,0) -- (-0,1);
\end{pgfonlayer}
\end{tikzpicture}
}}
+ 
\vcenter{\hbox{
\begin{tikzpicture}
\node[vertexS] at (0,0) {}; 
\node[vertexS] at (0.5,0) {}; 
\node[vertexS] at (0.5,0.5) {}; 
\draw[edgeS] (0.5,0) -- (0.5,0.5);
\begin{pgfonlayer}{background}
\draw[zhyedge1S] (0,0) -- (0.5,0) -- (0.5,0.5);
\draw[zhyedge2S] (0,0) -- (0.5,0) -- (0.5,0.5);
\end{pgfonlayer}
\end{tikzpicture}
}}
\Big\rangle(\omega_0) \Big\langle \rw_{s, t}, 
\vcenter{\hbox{
\begin{tikzpicture}
\node[vertexS] at (0,0) {}; 
\node[vertexS] at (0,0.5) {}; 
\draw[edge] (0,0) -- (0,0.5);
\begin{pgfonlayer}{background}
\draw[zhyedge1S] (0,0) -- (0,0.5);
\draw[zhyedge2S] (0,0) -- (0,0.5);
\end{pgfonlayer}
\end{tikzpicture}
}}
\Big\rangle(\omega_0)
\\
&+
\Big\langle \bX_s, 
\vcenter{\hbox{
\begin{tikzpicture}
\node[vertexS] at (0,0) {}; 
\node[vertexS] at (0.25,0.5) {}; 
\node[vertexS] at (-0.25,0.5) {}; 
\draw[edgeS] (0,0) -- (0.25,0.5);
\draw[edgeS] (0,0) -- (-0.25,0.5);
\begin{pgfonlayer}{background}
\draw[zhyedge1S] (-0.25,0.5) -- (0,0) -- (0.25,0.5);
\draw[zhyedge2S] (-0.25,0.5) -- (0,0) -- (0.25,0.5);
\end{pgfonlayer}
\end{tikzpicture}
}}
+ 
\vcenter{\hbox{
\begin{tikzpicture}
\node[vertexS] at (0,0) {}; 
\node[vertexS] at (0.5,0) {}; 
\node[vertexS] at (0.5,0.5) {}; 
\draw[edgeS] (0.5,0) -- (0.5,0.5);
\begin{pgfonlayer}{background}
\draw[zhyedge1S] (0,0) -- (0.5,0) -- (0.5,0.5);
\draw[zhyedge2S] (0,0) -- (0.5,0) -- (0.5,0.5);
\end{pgfonlayer}
\end{tikzpicture}
}}
+3
\vcenter{\hbox{
\begin{tikzpicture}
\node[vertexS] at (0,0) {}; 
\node[vertexS] at (0.25,0) {}; 
\node[vertexS] at (0.5,0) {}; 
\begin{pgfonlayer}{background}
\draw[zhyedge1S] (0,0) -- (0.5,0);
\draw[zhyedge2S] (0,0) -- (0.5,0);
\end{pgfonlayer}
\end{tikzpicture}
}}
\Big\rangle(\omega_0) \Big\langle \rw_{s, t}, 
\vcenter{\hbox{
\begin{tikzpicture}
\node[vertexS] at (0,0) {}; 
\node[vertexS] at (0.5,0) {}; 
\begin{pgfonlayer}{background}
\draw[zhyedge1S] (0,0) -- (0.5,0);
\draw[zhyedge2S] (0,0) -- (0.5,0);
\end{pgfonlayer}
\end{tikzpicture}
}}
\Big\rangle(\omega_0)
\\
&+
{\color{red} \bE^1}\Bigg[ 
\Big\langle \bX_s, 
\vcenter{\hbox{
\begin{tikzpicture}
\node[vertexS] at (0,0) {}; 
\node[vertexS] at (0,0.5) {}; 
\node[vertexS] at (0,1) {}; 
\draw[edgeS] (0,0) -- (0,0.5);
\draw[edgeS] (0,0.5) -- (0,1);
\begin{pgfonlayer}{background}
\draw[zhyedge1S] (0,0) -- (0,0.5);
\draw[zhyedge2S] (0,0) -- (-0,0.5);
\draw[hyedge1S, color=red] (0,1) -- (0,1);
\draw[hyedge2S] (0,1) -- (0,1);
\end{pgfonlayer}
\end{tikzpicture}
}}
+
\vcenter{\hbox{
\begin{tikzpicture}
\node[vertexS] at (0,0) {}; 
\node[vertexS] at (0.5,0) {}; 
\node[vertexS] at (0.5,0.5) {}; 
\draw[edgeS] (0.5,0) -- (0.5,0.5);
\begin{pgfonlayer}{background}
\draw[zhyedge1S] (0,0) -- (0.5,0);
\draw[zhyedge2S] (0,0) -- (0.5,0);
\draw[hyedge1S, color=red] (0.5,0.5) -- (0.5,0.5);
\draw[hyedge2S] (0.5,0.5) -- (0.5,0.5);
\end{pgfonlayer}
\end{tikzpicture}
}}
\Big\rangle(\omega_0, {\color{red} \omega_1}) \Big\langle \rw_{s, t}, 
\vcenter{\hbox{
\begin{tikzpicture}
\node[vertexS] at (0,0) {}; 
\node[vertexS] at (0,0.5) {}; 
\draw[edge] (0,0) -- (0,0.5);
\begin{pgfonlayer}{background}
\draw[zhyedge1S] (0,0) -- (0,0);
\draw[zhyedge2S] (0,0) -- (0,0);
\draw[hyedge1S, color=red] (0,0.5) -- (0,0.5);
\draw[hyedge2S] (0,0.5) -- (0,0.5);
\end{pgfonlayer}
\end{tikzpicture}
}}
\Big\rangle(\omega_0, {\color{red}\omega_1}) \Bigg]
\\
&+
{\color{red}\bE^1}\Bigg[ 
\Big\langle \bX_s, 
\vcenter{\hbox{
\begin{tikzpicture}
\node[vertexS] at (0,0) {}; 
\node[vertexS] at (0.25,0.5) {}; 
\node[vertexS] at (-0.25,0.5) {}; 
\draw[edgeS] (0,0) -- (0.25,0.5);
\draw[edgeS] (0,0) -- (-0.25,0.5);
\begin{pgfonlayer}{background}
\draw[zhyedge1S] (-0.25,0.5) -- (0,0);
\draw[zhyedge2S] (-0.25,0.5) -- (0,0);
\draw[hyedge1S, color=red] (0.25,0.5) -- (0.25,0.5);
\draw[hyedge2S] (0.25,0.5) -- (0.25,0.5);
\end{pgfonlayer}
\end{tikzpicture}
}}
+
\vcenter{\hbox{
\begin{tikzpicture}
\node[vertexS] at (0,0) {}; 
\node[vertexS] at (0.5,0) {}; 
\node[vertexS] at (0.5,0.5) {}; 
\draw[edgeS] (0.5,0) -- (0.5,0.5);
\begin{pgfonlayer}{background}
\draw[zhyedge1S] (0,0) -- (0.5,0);
\draw[zhyedge2S] (0,0) -- (0.5,0);
\draw[hyedge1S, color=red] (0.5,0.5) -- (0.5,0.5);
\draw[hyedge2S] (0.5,0.5) -- (0.5,0.5);
\end{pgfonlayer}
\end{tikzpicture}
}}
+2
\vcenter{\hbox{
\begin{tikzpicture}
\node[vertexS] at (0,0) {}; 
\node[vertexS] at (0.25,0) {}; 
\node[vertexS] at (0.75,0) {}; 
\begin{pgfonlayer}{background}
\draw[zhyedge1S] (0,0) -- (0.25,0);
\draw[zhyedge2S] (0,0) -- (0.25,0);
\draw[hyedge1S, color=red] (0.75,0) -- (0.75,0);
\draw[hyedge2S] (0.75,0) -- (0.75,0);
\end{pgfonlayer}
\end{tikzpicture}
}}
\Big\rangle(\omega_0, {\color{red}\omega_1}) \Big\langle \rw_{s, t}, 
\vcenter{\hbox{
\begin{tikzpicture}
\node[vertexS] at (0,0) {}; 
\node[vertexS] at (0.5,0) {}; 
\begin{pgfonlayer}{background}
\draw[zhyedge1S] (0,0) -- (0,0);
\draw[zhyedge2S] (0,0) -- (0,0);
\draw[hyedge1S, color=red] (0.5,0) -- (0.5,0);
\draw[hyedge2S] (0.5,0) -- (0.5,0);
\end{pgfonlayer}
\end{tikzpicture}
}}
\Big\rangle(\omega_0, {\color{red}\omega_1}) \Bigg]
+
\Big\langle \bX_{s, t}^{\sharp}, 
\vcenter{\hbox{
\begin{tikzpicture}
\node[vertexS] at (0,0) {}; 
\begin{pgfonlayer}{background}
\draw[zhyedge1S] (0,0) -- (-0,0);
\draw[zhyedge2S] (0,0) -- (-0,0);
\end{pgfonlayer}
\end{tikzpicture}
}}
\Big\rangle(\omega_0), 
\end{align*}
\begin{align*}
\Big\langle \bX_{s, t}, &
\vcenter{\hbox{
\begin{tikzpicture}
\node[vertexS] at (0,0) {}; 
\begin{pgfonlayer}{background}
\draw[hyedge1S, color=red] (0,0) -- (-0,0);
\draw[hyedge2S] (0,0) -- (-0,0);
\end{pgfonlayer}
\end{tikzpicture}
}}
\Big\rangle(\omega_0, {\color{red}\omega_1}) 
=
\Big\langle \bX_s, 
\vcenter{\hbox{
\begin{tikzpicture}
\node[vertexS] at (0,0) {}; 
\node[vertexS] at (0.5,0) {}; 
\begin{pgfonlayer}{background}
\draw[zhyedge1S] (0,0) -- (0,0);
\draw[zhyedge2S] (0,0) -- (0,0);
\draw[hyedge1S, color=red] (0.5,0) -- (0.5,0);
\draw[hyedge2S] (0.5,0) -- (0.5,0);
\end{pgfonlayer}
\end{tikzpicture}
}}
\Big\rangle(\omega_0, {\color{red}\omega_1})
\Big\langle \rw_{s, t}, 
\vcenter{\hbox{
\begin{tikzpicture}
\node[vertexS] at (0,0) {}; 
\begin{pgfonlayer}{background}
\draw[zhyedge1S] (0,0) -- (-0,0);
\draw[zhyedge2S] (0,0) -- (-0,0);
\end{pgfonlayer}
\end{tikzpicture}
}}
\Big\rangle(\omega_0)
\\
&+
\Big\langle \bX_s, 
\vcenter{\hbox{
\begin{tikzpicture}
\node[vertexS] at (0,0) {}; 
\node[vertexS] at (0.5,0) {}; 
\node[vertexS] at (0.5,0.5) {}; 
\draw[edgeS] (0.5,0) -- (0.5,0.5);
\begin{pgfonlayer}{background}
\draw[zhyedge1S] (0.5,0.5) -- (0.5,0);
\draw[zhyedge2S] (0.5,0.5) -- (0.5,0);
\draw[hyedge1S, color=red] (0,0) -- (0,0);
\draw[hyedge2S] (0,0) -- (0,0);
\end{pgfonlayer}
\end{tikzpicture}
}}
\Big\rangle(\omega_0, {\color{red}\omega_1})
\Big\langle \rw_{s, t}, 
\vcenter{\hbox{
\begin{tikzpicture}
\node[vertexS] at (0,0) {}; 
\node[vertexS] at (0,0.5) {}; 
\draw[edge] (0,0) -- (0,0.5);
\begin{pgfonlayer}{background}
\draw[zhyedge1S] (0,0) -- (0,0.5);
\draw[zhyedge2S] (0,0) -- (0,0.5);
\end{pgfonlayer}
\end{tikzpicture}
}}
\Big\rangle(\omega_0)
+
\Big\langle \bX_s, 
\vcenter{\hbox{
\begin{tikzpicture}
\node[vertexS] at (0,0) {}; 
\node[vertexS] at (0.25,0) {}; 
\node[vertexS] at (0.75,0) {}; 
\begin{pgfonlayer}{background}
\draw[zhyedge1S] (0,0) -- (0.25,0);
\draw[zhyedge2S] (0,0) -- (0.25,0);
\draw[hyedge1S, color=red] (0.75,0) -- (0.75,0);
\draw[hyedge2S] (0.75,0) -- (0.75,0);
\end{pgfonlayer}
\end{tikzpicture}
}}
\Big\rangle(\omega_0, {\color{red}\omega_1})
\Big\langle \rw_{s, t},
\vcenter{\hbox{ 
\begin{tikzpicture}
\node[vertexS] at (0,0) {}; 
\node[vertexS] at (0.5,0) {}; 
\begin{pgfonlayer}{background}
\draw[zhyedge1S] (0,0) -- (0.5,0);
\draw[zhyedge2S] (0,0) -- (0.5,0);
\end{pgfonlayer}
\end{tikzpicture}
}}
\Big\rangle(\omega_0)
\\
&+
\Big\langle \bX_s, 
\vcenter{\hbox{
\begin{tikzpicture}
\node[vertexS] at (0,0) {}; 
\node[vertexS] at (0, 0.5) {}; 
\draw[edgeS] (0,0) -- (0,0.5);
\begin{pgfonlayer}{background}
\draw[hyedge1S, color=red] (0,0) -- (0,0.5);
\draw[hyedge2S] (0,0) -- (0,0.5);
\end{pgfonlayer}
\end{tikzpicture}
}}
+2
\vcenter{\hbox{
\begin{tikzpicture}
\node[vertexS] at (0,0) {}; 
\node[vertexS] at (0.5,0) {}; 
\begin{pgfonlayer}{background}
\draw[hyedge1S, color=red] (0,0) -- (0.5,0);
\draw[hyedge2S] (0,0) -- (0.5,0);
\end{pgfonlayer}
\end{tikzpicture}
}}
\Big\rangle(\omega_0, {\color{red}\omega_1})
\Big\langle \rw_{s, t}, 
\vcenter{\hbox{
\begin{tikzpicture}
\node[vertexS] at (0,0) {}; 
\begin{pgfonlayer}{background}
\draw[hyedge1S, color=red] (0,0) -- (-0,0);
\draw[hyedge2S] (0,0) -- (-0,0);
\end{pgfonlayer}
\end{tikzpicture}
}}
\Big\rangle(\omega_0, {\color{red}\omega_1})
\\
&+
{\color{blue}\bE^2}\Bigg[ \Big\langle \bX_s, 
\vcenter{\hbox{
\begin{tikzpicture}
\node[vertexS] at (0,0) {}; 
\node[vertexS] at (0, 0.5) {}; 
\draw[edgeS] (0,0) -- (0,0.5);
\begin{pgfonlayer}{background}
\draw[hyedge1S, color=red] (0,0) -- (0,0);
\draw[hyedge2S] (0,0) -- (0,0);
\draw[hyedge1S, color=blue] (0,0.5) -- (0,0.5);
\draw[hyedge2S] (0,0.5) -- (0,0.5);
\end{pgfonlayer}
\end{tikzpicture}
}}
+
\vcenter{\hbox{
\begin{tikzpicture}
\node[vertexS] at (0,0) {}; 
\node[vertexS] at (0.5,0) {}; 
\begin{pgfonlayer}{background}
\draw[hyedge1S, color=red] (0,0) -- (0,0);
\draw[hyedge2S] (0,0) -- (0,0);
\draw[hyedge1S, color=blue] (0.5,0) -- (0.5,0);
\draw[hyedge2S] (0.5,0) -- (0.5,0);
\end{pgfonlayer}
\end{tikzpicture}
}}
\Big\rangle(\omega_0, {\color{red}\omega_1}, {\color{blue}\omega_2})
\Big\langle \rw_{s, t}, 
\vcenter{\hbox{
\begin{tikzpicture}
\node[vertexS] at (0,0) {}; 
\begin{pgfonlayer}{background}
\draw[hyedge1S, color=blue] (0,0) -- (-0,0);
\draw[hyedge2S] (0,0) -- (-0,0);
\end{pgfonlayer}
\end{tikzpicture}
}}
\Big\rangle(\omega_0, {\color{blue}\omega_2}) \Bigg]
+
\Big\langle \bX_{s, t}^{\sharp}, 
\vcenter{\hbox{
\begin{tikzpicture}
\node[vertexS] at (0,0) {}; 
\begin{pgfonlayer}{background}
\draw[hyedge1S, color=red] (0,0) -- (-0,0);
\draw[hyedge2S] (0,0) -- (-0,0);
\end{pgfonlayer}
\end{tikzpicture}
}}
\Big\rangle(\omega_0, {\color{red}\omega_1}), 
\end{align*}
\begin{align*}
\Big\langle \bX_{s, t}, &
\vcenter{\hbox{
\begin{tikzpicture}
\node[vertexS] at (0,0) {}; 
\node[vertexS] at (0,0.5) {}; 
\draw[edge] (0,0) -- (0,0.5);
\begin{pgfonlayer}{background}
\draw[zhyedge1S] (0,0) -- (0,0.5);
\draw[zhyedge2S] (0,0) -- (0,0.5);
\end{pgfonlayer}
\end{tikzpicture}
}}
\Big\rangle(\omega_0) 
=
\Big\langle \bX_s, 
\vcenter{\hbox{
\begin{tikzpicture}
\node[vertexS] at (0,0) {}; 
\node[vertexS] at (0,0.5) {}; 
\node[vertexS] at (0,1) {}; 
\draw[edgeS] (0,0) -- (0,0.5);
\draw[edgeS] (0,0.5) -- (0,1);
\begin{pgfonlayer}{background}
\draw[zhyedge1S] (0,0) -- (0,1);
\draw[zhyedge2S] (0,0) -- (-0,1);
\end{pgfonlayer}
\end{tikzpicture}
}}
+
\vcenter{\hbox{  
\begin{tikzpicture}
\node[vertexS] at (0,0) {}; 
\node[vertexS] at (0.5,0) {}; 
\node[vertexS] at (0.5,0.5) {}; 
\draw[edgeS] (0.5,0) -- (0.5,0.5);
\begin{pgfonlayer}{background}
\draw[zhyedge1S] (0,0) -- (0.5,0) -- (0.5,0.5);
\draw[zhyedge2S] (0,0) -- (0.5,0) -- (0.5,0.5);
\end{pgfonlayer}
\end{tikzpicture}
}}
+2
\vcenter{\hbox{
\begin{tikzpicture}
\node[vertexS] at (0,0) {}; 
\node[vertexS] at (0.25,0.5) {}; 
\node[vertexS] at (-0.25,0.5) {}; 
\draw[edgeS] (0,0) -- (0.25,0.5);
\draw[edgeS] (0,0) -- (-0.25,0.5);
\begin{pgfonlayer}{background}
\draw[zhyedge1S] (-0.25,0.5) -- (0,0) -- (0.25,0.5);
\draw[zhyedge2S] (-0.25,0.5) -- (0,0) -- (0.25,0.5);
\end{pgfonlayer}
\end{tikzpicture}
}}
\Big\rangle(\omega_0) \Big\langle \rw_{s, t}, 
\vcenter{\hbox{
\begin{tikzpicture}
\node[vertexS] at (0,0) {}; 
\begin{pgfonlayer}{background}
\draw[zhyedge1S] (0,0) -- (-0,0);
\draw[zhyedge2S] (0,0) -- (-0,0);
\end{pgfonlayer}
\end{tikzpicture}
}}
\Big\rangle(\omega_0)
\\
&+
{\color{red}\bE^1}\Bigg[ \Big\langle \bX_s, 
\vcenter{\hbox{
\begin{tikzpicture}
\node[vertexS] at (0,0) {}; 
\node[vertexS] at (0,0.5) {}; 
\node[vertexS] at (0,1) {}; 
\draw[edgeS] (0,0) -- (0,0.5);
\draw[edgeS] (0,0.5) -- (0,1);
\begin{pgfonlayer}{background}
\draw[zhyedge1S] (0,0) -- (0,0.5);
\draw[zhyedge2S] (0,0) -- (-0,0.5);
\draw[hyedge1S, color=red] (0,1) -- (0,1);
\draw[hyedge2S] (0,1) -- (0,1);
\end{pgfonlayer}
\end{tikzpicture}
}}
+
\vcenter{\hbox{
\begin{tikzpicture}
\node[vertexS] at (0,0) {}; 
\node[vertexS] at (0.25,0.5) {}; 
\node[vertexS] at (-0.25,0.5) {}; 
\draw[edgeS] (0,0) -- (0.25,0.5);
\draw[edgeS] (0,0) -- (-0.25,0.5);
\begin{pgfonlayer}{background}
\draw[zhyedge1S] (-0.25,0.5) -- (0,0);
\draw[zhyedge2S] (-0.25,0.5) -- (0,0);
\draw[hyedge1S, color=red] (0.25,0.5) -- (0.25,0.5);
\draw[hyedge2S] (0.25,0.5) -- (0.25,0.5);
\end{pgfonlayer}
\end{tikzpicture}
}}
+
\vcenter{\hbox{
\begin{tikzpicture}
\node[vertexS] at (0,0) {}; 
\node[vertexS] at (0.5,0) {}; 
\node[vertexS] at (0.5,0.5) {}; 
\draw[edgeS] (0.5,0) -- (0.5,0.5);
\begin{pgfonlayer}{background}
\draw[zhyedge1S] (0.5,0.5) -- (0.5,0);
\draw[zhyedge2S] (0.5,0.5) -- (0.5,0);
\draw[hyedge1S, color=red] (0,0) -- (0,0);
\draw[hyedge2S] (0,0) -- (0,0);
\end{pgfonlayer}
\end{tikzpicture}
}}
\Big\rangle(\omega_0, {\color{red}\omega_1}) 
\Big\langle \rw_{s, t}, 
\vcenter{\hbox{
\begin{tikzpicture}
\node[vertexS] at (0,0) {}; 
\begin{pgfonlayer}{background}
\draw[hyedge1S, color=red] (0,0) -- (-0,0);
\draw[hyedge2S] (0,0) -- (-0,0);
\end{pgfonlayer}
\end{tikzpicture}
}}
\Big\rangle(\omega_0, {\color{red}\omega_1})\Bigg]
+
\Big\langle \bX_{s, t}^{\sharp}, 
\vcenter{\hbox{
\begin{tikzpicture}
\node[vertexS] at (0,0) {}; 
\node[vertexS] at (0,0.5) {}; 
\draw[edge] (0,0) -- (0,0.5);
\begin{pgfonlayer}{background}
\draw[zhyedge1S] (0,0) -- (0,0.5);
\draw[zhyedge2S] (0,0) -- (0,0.5);
\end{pgfonlayer}
\end{tikzpicture}
}}
\Big\rangle(\omega_0), 
\end{align*}
\begin{align*}
\Big\langle \bX_{s, t}, &
\vcenter{\hbox{
\begin{tikzpicture}
\node[vertexS] at (0,0) {}; 
\node[vertexS] at (0.5,0) {}; 
\begin{pgfonlayer}{background}
\draw[zhyedge1S] (0,0) -- (0.5,0);
\draw[zhyedge2S] (0,0) -- (0.5,0);
\end{pgfonlayer}
\end{tikzpicture}
}}
\Big\rangle(\omega_0) 
=
\Big\langle \bX_s, 
\vcenter{\hbox{
\begin{tikzpicture}
\node[vertexS] at (0,0) {}; 
\node[vertexS] at (0.5,0) {}; 
\node[vertexS] at (0.5,0.5) {}; 
\draw[edgeS] (0.5,0) -- (0.5,0.5);
\begin{pgfonlayer}{background}
\draw[zhyedge1S] (0,0) -- (0.5,0) -- (0.5,0.5);
\draw[zhyedge2S] (0,0) -- (0.5,0) -- (0.5,0.5);
\end{pgfonlayer}
\end{tikzpicture}
}}
+3
\vcenter{\hbox{
\begin{tikzpicture}
\node[vertexS] at (0,0) {}; 
\node[vertexS] at (0.25,0) {}; 
\node[vertexS] at (0.5,0) {}; 
\begin{pgfonlayer}{background}
\draw[zhyedge1S] (0,0) -- (0.5,0);
\draw[zhyedge2S] (0,0) -- (0.5,0);
\end{pgfonlayer}
\end{tikzpicture}
}}
\Big\rangle(\omega_0) \Big\langle \rw_{s, t}, 
\vcenter{\hbox{
\begin{tikzpicture}
\node[vertexS] at (0,0) {}; 
\begin{pgfonlayer}{background}
\draw[zhyedge1S] (0,0) -- (-0,0);
\draw[zhyedge2S] (0,0) -- (-0,0);
\end{pgfonlayer}
\end{tikzpicture}
}}
\Big\rangle(\omega_0)
\\
&+
{\color{red}\bE^1}\Bigg[ \Big\langle \bX_s, 
\vcenter{\hbox{
\begin{tikzpicture}
\node[vertexS] at (0,0) {}; 
\node[vertexS] at (0.5,0) {}; 
\node[vertexS] at (0.5,0.5) {}; 
\draw[edgeS] (0.5,0) -- (0.5,0.5);
\begin{pgfonlayer}{background}
\draw[zhyedge1S] (0,0) -- (0.5,0);
\draw[zhyedge2S] (0,0) -- (0.5,0);
\draw[hyedge1S, color=red] (0.5,0.5) -- (0.5,0.5);
\draw[hyedge2S] (0.5,0.5) -- (0.5,0.5);
\end{pgfonlayer}
\end{tikzpicture}
}}
+
\vcenter{\hbox{
\begin{tikzpicture}
\node[vertexS] at (0,0) {}; 
\node[vertexS] at (0.25,0) {}; 
\node[vertexS] at (0.75,0) {}; 
\begin{pgfonlayer}{background}
\draw[zhyedge1S] (0,0) -- (0.25,0);
\draw[zhyedge2S] (0,0) -- (0.25,0);
\draw[hyedge1S, color=red] (0.75,0) -- (0.75,0);
\draw[hyedge2S] (0.75,0) -- (0.75,0);
\end{pgfonlayer}
\end{tikzpicture}
}}
\Big\rangle(\omega_0, {\color{red}\omega_1}) 
\Big\langle \rw_{s, t}, 
\vcenter{\hbox{
\begin{tikzpicture}
\node[vertexS] at (0,0) {}; 
\begin{pgfonlayer}{background}
\draw[hyedge1S, color=red] (0,0) -- (-0,0);
\draw[hyedge2S] (0,0) -- (-0,0);
\end{pgfonlayer}
\end{tikzpicture}
}}
\Big\rangle(\omega_0, {\color{red}\omega_1})\Bigg]
+
\Big\langle \bX_{s, t}^{\sharp}, 
\vcenter{\hbox{
\begin{tikzpicture}
\node[vertexS] at (0,0) {}; 
\node[vertexS] at (0.5,0) {}; 
\begin{pgfonlayer}{background}
\draw[zhyedge1S] (0,0) -- (0.5,0);
\draw[zhyedge2S] (0,0) -- (0.5,0);
\end{pgfonlayer}
\end{tikzpicture}
}}
\Big\rangle(\omega_0), 
\end{align*}
\begin{align*}
\Big\langle \bX_{s, t}, &
\vcenter{\hbox{
\begin{tikzpicture}
\node[vertexS] at (0,0) {}; 
\node[vertexS] at (0,0.5) {}; 
\draw[edge] (0,0) -- (0,0.5);
\begin{pgfonlayer}{background}
\draw[zhyedge1S] (0,0) -- (0,0);
\draw[zhyedge2S] (0,0) -- (0,0);
\draw[hyedge1S, color=red] (0,0.5) -- (0,0.5);
\draw[hyedge2S] (0,0.5) -- (0,0.5);
\end{pgfonlayer}
\end{tikzpicture}
}}
\Big\rangle(\omega_0, {\color{red}\omega_1}) 
=
\Big\langle \bX_s, 
\vcenter{\hbox{
\begin{tikzpicture}
\node[vertexS] at (0,0) {}; 
\node[vertexS] at (0.25,0.5) {}; 
\node[vertexS] at (-0.25,0.5) {}; 
\draw[edgeS] (0,0) -- (0.25,0.5);
\draw[edgeS] (0,0) -- (-0.25,0.5);
\begin{pgfonlayer}{background}
\draw[zhyedge1S] (-0.25,0.5) -- (0,0);
\draw[zhyedge2S] (-0.25,0.5) -- (0,0);
\draw[hyedge1S, color=red] (0.25,0.5) -- (0.25,0.5);
\draw[hyedge2S] (0.25,0.5) -- (0.25,0.5);
\end{pgfonlayer}
\end{tikzpicture}
}}
+
\vcenter{\hbox{
\begin{tikzpicture}
\node[vertexS] at (0,0) {}; 
\node[vertexS] at (0.5,0) {}; 
\node[vertexS] at (0.5,0.5) {}; 
\draw[edgeS] (0.5,0) -- (0.5,0.5);
\begin{pgfonlayer}{background}
\draw[zhyedge1S] (0,0) -- (0.5,0);
\draw[zhyedge2S] (0,0) -- (0.5,0);
\draw[hyedge1S, color=red] (0.5,0.5) -- (0.5,0.5);
\draw[hyedge2S] (0.5,0.5) -- (0.5,0.5);
\end{pgfonlayer}
\end{tikzpicture}
}}
\Big\rangle(\omega_0, {\color{red}\omega_1}) 
\Big\langle \rw_{s,t}, 
\vcenter{\hbox{
\begin{tikzpicture}
\node[vertexS] at (0,0) {}; 
\begin{pgfonlayer}{background}
\draw[zhyedge1S] (0,0) -- (-0,0);
\draw[zhyedge2S] (0,0) -- (-0,0);
\end{pgfonlayer}
\end{tikzpicture}
}}
\Big\rangle(\omega_0)
+
\Big\langle \bX_{s, t}^{\sharp},
\vcenter{\hbox{ 
\begin{tikzpicture}
\node[vertexS] at (0,0) {}; 
\node[vertexS] at (0,0.5) {}; 
\draw[edge] (0,0) -- (0,0.5);
\begin{pgfonlayer}{background}
\draw[zhyedge1S] (0,0) -- (0,0);
\draw[zhyedge2S] (0,0) -- (0,0);
\draw[hyedge1S, color=red] (0,0.5) -- (0,0.5);
\draw[hyedge2S] (0,0.5) -- (0,0.5);
\end{pgfonlayer}
\end{tikzpicture}
}}
\Big\rangle(\omega_0, {\color{red}\omega_1}), 
\end{align*}
\begin{align*}
\Big\langle \bX_{s, t}, &
\vcenter{\hbox{
\begin{tikzpicture}
\node[vertexS] at (0,0) {}; 
\node[vertexS] at (0.5,0) {}; 
\begin{pgfonlayer}{background}
\draw[zhyedge1S] (0,0) -- (0,0);
\draw[zhyedge2S] (0,0) -- (0,0);
\draw[hyedge1S, color=red] (0.5,0) -- (0.5,0);
\draw[hyedge2S] (0.5,0) -- (0.5,0);
\end{pgfonlayer}
\end{tikzpicture}
}}
\Big\rangle(\omega_0, {\color{red}\omega_1}) 
\\
&=
\Big\langle \bX_s, 
\vcenter{\hbox{
\begin{tikzpicture}
\node[vertexS] at (0,0) {}; 
\node[vertexS] at (0.5,0) {}; 
\node[vertexS] at (0.5,0.5) {}; 
\draw[edgeS] (0.5,0) -- (0.5,0.5);
\begin{pgfonlayer}{background}
\draw[zhyedge1S] (0.5,0.5) -- (0.5,0);
\draw[zhyedge2S] (0.5,0.5) -- (0.5,0);
\draw[hyedge1S, color=red] (0,0) -- (0,0);
\draw[hyedge2S] (0,0) -- (0,0);
\end{pgfonlayer}
\end{tikzpicture}
}}
+2
\vcenter{\hbox{
\begin{tikzpicture}
\node[vertexS] at (0,0) {}; 
\node[vertexS] at (0.25,0) {}; 
\node[vertexS] at (0.75,0) {}; 
\begin{pgfonlayer}{background}
\draw[zhyedge1S] (0,0) -- (0.25,0);
\draw[zhyedge2S] (0,0) -- (0.25,0);
\draw[hyedge1S, color=red] (0.75,0) -- (0.75,0);
\draw[hyedge2S] (0.75,0) -- (0.75,0);
\end{pgfonlayer}
\end{tikzpicture}
}}
\Big\rangle(\omega_0, {\color{red}\omega_1}) 
\Big\langle \rw_{s, t}, 
\vcenter{\hbox{
\begin{tikzpicture}
\node[vertexS] at (0,0) {}; 
\begin{pgfonlayer}{background}
\draw[zhyedge1S] (0,0) -- (-0,0);
\draw[zhyedge2S] (0,0) -- (-0,0);
\end{pgfonlayer}
\end{tikzpicture}
}}
\Big\rangle(\omega_0)
+
\Big\langle \bX_{s, t}^{\sharp}, 
\vcenter{\hbox{
\begin{tikzpicture}
\node[vertexS] at (0,0) {}; 
\node[vertexS] at (0.5,0) {}; 
\begin{pgfonlayer}{background}
\draw[zhyedge1S] (0,0) -- (0,0);
\draw[zhyedge2S] (0,0) -- (0,0);
\draw[hyedge1S, color=red] (0.5,0) -- (0.5,0);
\draw[hyedge2S] (0.5,0) -- (0.5,0);
\end{pgfonlayer}
\end{tikzpicture}
}}
\Big\rangle(\omega_0, {\color{red}\omega_1}) 
\end{align*}
and for all Lions trees in $\scF^{\gamma-, \alpha, \beta}$ not listed above, 
$$
\Big\langle \bX_{s, t}, T\Big\rangle(\omega_0, \omega_{H^T}) 
= 
\Big\langle \bX_{s, t}^{\sharp}, T\Big\rangle(\omega_0, \omega_{H^T}).  
$$
for the simple reason that in this case we cannot find a forest $T \in 
\scF^{\gamma-, \alpha, \beta}$ together with a non-trivial cut such that $Y$ is the root of $T$ under this cut. This claim is obvious when $Y$ contains three nodes since $\scF^{\gamma-, \alpha, \beta}$ does not contain any forest with (strictly) more than three nodes. When $Y$ is one of the forests with two nodes that is non-listed in the above enumeration, this should be checked case by case taking into account the colouring of the hyperedges of $Y$. For instance, 
$
\vcenter{\hbox{
\begin{tikzpicture}
\node[vertexS] at (0,0) {}; 
\node[vertexS] at (0.5,0) {}; 
\begin{pgfonlayer}{background}
\draw[hyedge1S, color=red] (0,0) -- (0.5,0);
\draw[hyedge2S] (0,0) -- (0.5,0);
\end{pgfonlayer}
\end{tikzpicture}
}}
$
cannot be regarded as the root of an element of $\scF^{\gamma-, \alpha, \beta}$, whilst 
$
\vcenter{\hbox{
\begin{tikzpicture}
\node[vertexS] at (0,0) {}; 
\node[vertexS] at (0.5,0) {}; 
\begin{pgfonlayer}{background}
\draw[hyedge1S] (0,0) -- (0.5,0);
\draw[hyedge2S] (0,0) -- (0.5,0);
\end{pgfonlayer}
\end{tikzpicture}
}}
$ is the root of an element of $\scF^{\gamma-, \alpha, \beta}$.

Further, we have that $\bP$-almost surely:
\begin{align*}
\sup_{s, t\in[0,1]} \frac{\Big\langle \bX_{s, t}^{\sharp}, \rId \Big\rangle(\omega_0)}{|t-s|^{\gamma}} < \infty, 
\quad&\quad
\\
\sup_{s, t\in[0,1]} \frac{\Big\langle \bX_{s, t}^{\sharp}, 
\vcenter{\hbox{
\begin{tikzpicture}
\node[vertexS] at (0,0) {}; 
\begin{pgfonlayer}{background}
\draw[zhyedge1S] (0,0) -- (-0,0);
\draw[zhyedge2S] (0,0) -- (-0,0);
\end{pgfonlayer}
\end{tikzpicture}
}}
\Big\rangle(\omega_0)}
{|t-s|^{\gamma-\alpha}} < \infty, 
\quad&\quad
\sup_{s, t\in[0,1]} \frac{ {\color{red}\bE^1}\bigg[ \Big| \Big\langle \bX_{s, t}^{\sharp}, 
\vcenter{\hbox{
\begin{tikzpicture}
\node[vertexS] at (0,0) {}; 
\begin{pgfonlayer}{background}
\draw[hyedge1S, color=red] (0,0) -- (-0,0);
\draw[hyedge2S] (0,0) -- (-0,0);
\end{pgfonlayer}
\end{tikzpicture}
}}
\Big\rangle(\omega_0, {\color{red}\omega_1})\Big| \bigg]}
{|t-s|^{\gamma - \beta}} < \infty, 
\\
\sup_{s, t\in[0,1]} \frac{
\Big\langle \bX_{s, t}^{\sharp}, 
\vcenter{\hbox{
\begin{tikzpicture}
\node[vertexS] at (0,0) {}; 
\node[vertexS] at (0,0.5) {}; 
\draw[edge] (0,0) -- (0,0.5);
\begin{pgfonlayer}{background}
\draw[zhyedge1S] (0,0) -- (0,0.5);
\draw[zhyedge2S] (0,0) -- (0,0.5);
\end{pgfonlayer}
\end{tikzpicture}
}}
\Big\rangle(\omega_0)}
{|t-s|^{\gamma - 2\alpha}} < \infty, 
\quad&\quad
\sup_{s, t\in[0,1]} \frac{
\Big\langle \bX_{s, t}^{\sharp}, 
\vcenter{\hbox{
\begin{tikzpicture}
\node[vertexS] at (0,0) {}; 
\node[vertexS] at (0.5,0) {}; 
\begin{pgfonlayer}{background}
\draw[zhyedge1S] (0,0) -- (0.5,0);
\draw[zhyedge2S] (0,0) -- (0.5,0);
\end{pgfonlayer}
\end{tikzpicture}
}}
\Big\rangle(\omega_0)}
{|t-s|^{\gamma - 2\alpha}} < \infty, 
\\
\sup_{s, t\in[0,1]} \frac{ {\color{red}\bE^1}\bigg[ \Big|
\Big\langle \bX_{s, t}^{\sharp}, 
\vcenter{\hbox{
\begin{tikzpicture}
\node[vertexS] at (0,0) {}; 
\node[vertexS] at (0,0.5) {}; 
\draw[edge] (0,0) -- (0,0.5);
\begin{pgfonlayer}{background}
\draw[zhyedge1S] (0,0) -- (0,0);
\draw[zhyedge2S] (0,0) -- (0,0);
\draw[hyedge1S, color=red] (0,0.5) -- (0,0.5);
\draw[hyedge2S] (0,0.5) -- (0,0.5);
\end{pgfonlayer}
\end{tikzpicture}
}}
\Big\rangle(\omega_0, {\color{red}\omega_1})
\Big| \bigg]}
{|t-s|^{\gamma - \alpha - \beta}} < \infty, 
\quad&\quad
\sup_{s, t\in[0,1]} \frac{ {\color{red}\bE^1}\bigg[ \Big|
\Big\langle \bX_{s, t}^{\sharp}, 
\vcenter{\hbox{
\begin{tikzpicture}
\node[vertexS] at (0,0) {}; 
\node[vertexS] at (0.5,0) {}; 
\begin{pgfonlayer}{background}
\draw[zhyedge1S] (0,0) -- (0,0);
\draw[zhyedge2S] (0,0) -- (0,0);
\draw[hyedge1S, color=red] (0.5,0) -- (0.5,0);
\draw[hyedge2S] (0.5,0) -- (0.5,0);
\end{pgfonlayer}
\end{tikzpicture}
}}
\Big\rangle(\omega_0, {\color{red}\omega_1}) 
\Big| \bigg]}
{|t-s|^{\gamma - \alpha - \beta}} < \infty, 
\end{align*}
and for all other Lions trees $Y \in \scF^{\gamma-, \alpha, \beta}$ not listed above, 
$$
\sup_{s, t\in[0,1]} \frac{\bE^{H^Y}\bigg[ \Big| \Big\langle \bX^{\sharp}_{s, t}, Y \Big\rangle(\omega_0, \omega_{H^Y}) \Big| \bigg]}{|t-s|^{\gamma - \scG_{\alpha, \beta}[Y]}} <\infty. 
$$
For these forests, we note that $0< \gamma - \scG_{\alpha, \beta}[Y]\leq \alpha$ so that there are no Lions forests $T, \Upsilon\in \scF^{\gamma-, \alpha, \beta}$ such that $c'(T, \Upsilon, Y)>0$. 
\end{example}

For $Y\in\scF_0^{\gamma, \alpha, \beta}$, we denote
\begin{equation}
\label{eq:Jet_Operator}
\Big\langle \fJ\big[ \bX \big]_{s, t}, Y \Big\rangle(\omega_0, \omega_{H^{Y}}) = \Big\langle \bX_{s, t}, Y \Big\rangle(\omega_0, \omega_{H^{Y}}) - \Big\langle \bX_{s, t}^{\sharp}, Y \Big\rangle(\omega_0, \omega_{H^{Y}}). 
\end{equation}
In broad terms, the operation $\fJ$ is a concise way for denoting the jet of a random controlled rough path, which has differing regularity from the remainder term $\langle \bX_{s, t}^{\sharp}, Y\rangle$. 

\subsection{Motivation: Mean-field elementary differentials}
\label{subsec:Int_MF_ElemDiff}

In this Subsection, we wish to briefly consider how one could provide a local expansion of Butcher type for the olution of the McKean-Vlasov equation
\begin{equation}
\label{eq:McKean-Increment}
X_{s, t}(\omega_0) = \int_s^t f\big( X_r(\omega_0), \cL_r^X \big) \otimes dW_r(\omega_0). 
\end{equation}
where (for this Subsection alone) $X$ and $W$ are some smooth path valued random variables and $f$ is a smooth function on $\bR^d \times \cP_2(\bR^d)$. By Taylor expanding the function $f$ by means of Theorem 
\ref{theorem:LionsTaylor2}, we hope to motivate the choice of structure for random controlled rough paths as stated in Definition \ref{definition:RandomControlledRP} above. 

We first fix $\alpha,\beta >0$ as in Definition \ref{definition:RandomControlledRP}. For simplicity, we assume $\alpha=\beta$ and we define $\gamma$ accordingly, i.e., $\gamma = \alpha \big\lfloor \tfrac{1}{\alpha} \big\rfloor$. Indeed, for an arbitrary tree $T$, $\scG_{\alpha,\beta}[T] = |N^T|\alpha$ and $\scG_{\alpha,\beta}[T] > 1 - \alpha \Leftrightarrow \vert N^T \vert > \tfrac{1}{\alpha}-1 \Leftrightarrow |N^T| \geq \tfrac{1}{\alpha}$. Thus $\gamma$ is indeed equal to $\alpha \big\lfloor \tfrac{1}{\alpha} \big\rfloor$ and $T \in \scF^{\gamma-,\alpha,\alpha}$ if and only if $|N^T| \leq n-1$, with $n:=\big\lfloor \tfrac{1}{\alpha} \big\rfloor$. Intuitively, $\gamma$ prescribes the order that one wants to reach in the local expansion of the increments of $X$, with the following rule: the smaller $\alpha$, the larger the order of the expansion. We then suppose that the field $f$ belongs to $C_b^{n, (n)}$. By applying Theorem \ref{theorem:LionsTaylor2}, we get, for fixed $s,t$ such that $0<s<t$,
\begin{align}
\nonumber
\eqref{eq:McKean-Increment} =& f\Big( X_s(\omega_0), \cL_s^X \Big) \otimes W_{s, t} + \int_s^t \Big[ f\big( X_r(\omega_0), \cL_r^X \big) - f\big( X_s(\omega_0), \cL_s^X \big) \Big] \otimes dW_r(\omega_0) 
\\
\nonumber
=& f\Big( X_s(\omega_0), \cL_s^X \Big) \otimes W_{s, t}(\omega_0) 
+ \sum_{i=1}^n \sum_{a\in \A{i}} \int_{s}^{t} \frac{\rD^a f\big( X_s(\omega_0), \cL_s^X \big)\big[ X_{s, r}(\omega_0), \Pi_{s, r}^X \big] }{i!} \otimes dW_r
\\
\label{eq:TaylorMotivation}
& + \int_s^t R_n^{X_{s, r}(\omega_0), \Pi_{s, r}^X} \otimes dW_r,
\end{align}
where $\Pi_{s,r}^X$ is here the (joint) law of $(X_s,X_r)$, which is here `the' natural coupling between $\cL(X_s)$ and $\cL(X_r)$. The exact expression for the integrand in the time-integral on the second line is then given by \eqref{eq:rDa}. 

Next, we substitute the whole right hand side of Equation \eqref{eq:TaylorMotivation} for each increment of $X_{s, r}(\omega_0)$ within this formula. This provides
\begin{align}
\nonumber
\eqref{eq:McKean-Increment} =& f\Big( X_s(\omega_0), \cL_s^X \Big) \otimes W_{s, t}(\omega_0) 
\\
\nonumber
&+ \sum_{i=1}^n \sum_{a\in \A{i}} \frac{1}{i!} \bE^{1, ..., m[a]}\Bigg[ \bigg\langle \partial_a f\Big( X_s(\omega_0), \cL_s^X, X_s(\omega_1), ..., X_s(\omega_{m[a]})\Big), \bigotimes_{j=1}^{i} f\Big(X_s(\omega_{a_j}), \cL_s^X\Big) \bigg\rangle
\\
&
\label{eq:ElemDiff_motivation}
\qquad \cdot \int_{s}^{t}\Big( \bigotimes_{j=1}^{i} W_{s,r}(\omega_{a_j}) \Big) \otimes dW_r(\omega_0) \Bigg] + \Big\{ Remainder\Big\}, 
\end{align}
where \textit{Remainder} is a remainder term which is intuitively not important at this stage of the discussion and whose exact contribution is addressed next. For the time being, we want to explain first how the above expansion fits the framework used in Definition \ref{definition:RandomControlledRP}. In particular, we now spend some time reformulating the second term in the right-hand side by means of Lions' trees. 

In order to so, we start with the following observation.
By Lemma \ref{Lemma:Bijection-Partition}, for any partition sequence $a\in \A{i}$, we can consider the pre-image of the set $\{ 1, ..., i\}$ by $a$ and then associate a partition of the integers $\{1, ..., i\}$ which we write $H^a$ along with a (possibly empty) tagged partition element $h_0:=\big\{ j\in \{1, ..., i\}: a_j=0\big\}$. Let us consider the Lions tree $T^a=(N, E, h_0^a, H^a, L)$ where
\begin{align*}
N:=&\Big\{ 0, 1, ..., i \Big\}, 
\quad
E:=\Big\{ (1, 0), ... (i, 0) \Big\},
\quad
L:N \to \{1, ..., d\}, 
\\
h_0:=&\{0\}\cup \Big\{ j\in \{1, ..., i\}: a_j = 0\Big\}, 
\quad 
H:= \Big\{ \big\{ j\in\{1, ..., i\}: a_j=k\big\}: k=1, ..., m[a]\Big\}. 
\end{align*}
It is worth observing that $H$ and $H^a$ can be identified since the hyperedge $\{ j : a_j =k \}$ is nothing but the collection of nodes that are labelled by a pre-image of $k$ by $a$.

Let $\rw$ be the probabilistic rough path associated the lift of the smooth path-valued random variable $W$ (see \cite{salkeld2021Probabilistic} for the construction). Then we have that
$$
\int_s^t \prod_{j=1}^i \big\langle W_{s, r}(\omega_{a_j}), e_{L[j]} \big\rangle \cdot d \big\langle W_r, e_{L[0]}\big\rangle = \Big\langle \rw_{s, t}, T \Big\rangle(\omega_0, \omega_1, ..., \omega_{m[a]}). 
$$
which gives another expression of the stochastic integrals in \eqref{eq:ElemDiff_motivation}, based on Lions' trees. What is more, we can index the $\omega$'s in the above formula by the hyperedges of the tree itself. Indeed, there is an obvious bijection between $\{1,...,m[a]\}$ and the collection $H^a$ of non-0 hyperedges. By relabelling the probability spaces, the above can be re-expressed of the form
\begin{equation}
\label{eq:tree:stochastic:integral:elementary}
\int_s^t \prod_{j=1}^i \big\langle W_{s, r}(\omega_{h_{a_j}}), e_{L[j]} \big\rangle \cdot d \big\langle W_r, e_{L[0]}\big\rangle = \Big\langle \rw_{s, t}, T^{a} \Big\rangle(\omega_0, \omega_{H^a}),
\end{equation}
where $h_1,...,h_{m[a]}$ is an enumeration of the hyperedges of $T^a$, i.e., of the elements of $H^a$. This formulation is very advantageous as it permits to track how two random variables of the above type are correlated when they are labelled by two different forests. In short, the statistical correlations are then exhaustively described by means of the couplings between the two forests.

Similarly, we want to associate to the Lions tree the random variable
\begin{equation}
\label{eq:ElementaryDiff_intro}
\frac{1}{i!} \bigg\langle \partial_a f\Big( X_s(\omega_0), \cL_s^X, X_s(\omega_1), ..., X_s(\omega_{m[a]})\Big), \bigotimes_{j=1}^{i} f\Big(X_s(\omega_{a_j}), \cL_s^X\Big) \bigg\rangle. 
\end{equation}
which also appears in \eqref{eq:ElemDiff_motivation}. This however requires a preliminary discussion about the term \textit{Remainder} in the expansion \eqref{eq:ElemDiff_motivation}. Basically, the remainder has been obtained by replacing $[X_r-X_s](\omega_0)$ by $f(X_s(\omega_0),\cL_s^X) \otimes W_{s,r}$, which is indeed licit up to a rest of order $(r-s)^2$ (using the smoothness of $W$). Although it prompted us to introduce the rough path $\rw$ and then to derive the identity \eqref{eq:tree:stochastic:integral:elementary}, this approach has the severe drawback to lead to a global remainder (namely \textit{Remainder}) of a low order. Indeed, when taking the index $i$ in the summand in \eqref{eq:ElemDiff_motivation} as being equal to 1 (or 2, 3...), we may get a contribution to \textit{Remainder} that is in fact of  a lower order than some of the terms indexed by higher values of $i$. The strategy to improve the value of the remainder is in fact well known. For ordinary differential equations, it relies on the elementary differentials associated to the Connes-Kreimer-Hopf algebra. 

When adapted to our setting, it may be implemented as follows. The very main idea is to replace the random variable \eqref{eq:ElementaryDiff_intro} by a random variable of the more general form 
$$
\tilde{\Psi}^f \Big[ T^a \Big] \Big(X_s(\omega_{0}), \cL_s^X \Big)\Big( X_s(\omega)_{H^{a}} \Big)
$$
and then to postulate an expansion of the type
\begin{align}
\nonumber
&\eqref{eq:McKean-Increment} 
\\
\label{eq:RCRP_SortOf}
&= 
\sum_{T \in \scT_{0, n}} \bE^{H^T} \bigg[ \tilde{\Psi}^f \Big[ T \Big] \Big(X_s(\omega_{0}), \cL_s^X \Big)\Big( X_s(\omega)_{H^T} \Big) \cdot \Big\langle \rw_{s, t}, T \Big\rangle\bigl(\omega_0, \omega_{H^T}\bigr) \bigg]
+
O \Big( |t-s|^{n+1}\Big),
\end{align}
with a remainder whose contribution is now clearly identified. The core idea of elementary differentials is that they can be defined inductively, by iterating on the grades of the forests. Notice that in the above right-hand side, we used the notation $X_s(\omega)_{H^T}$ to denote $(X_s(\omega_h))_{h \in H^T}$. 

By abstracting the technique used to obtain \eqref{eq:ElemDiff_motivation}, we come to the following definition:
\begin{definition}
\label{definition:MeanField-ElementaryDiffer-Primary}
Let $f: \bR^e \times \cP_2(\bR^e) \to \lin(\bR^d, \bR^e)$ be a vector field such that $f \in C^{ n, (n)}\big( \bR^e \times \cP_2(\bR^e)\big)$. Let
\begin{align*}
\scT_{0, n}:=& \Big\{ T \in \scT: h_0^T \neq \emptyset, \quad |N^T|\leq n \Big\}
\\
A_\ast^{n}:=& \bigcup_{i=1}^n \A{i}
\end{align*}

Let $\tilde{\Psi}^f: \scT_{0, n} \to C\Big( \bR^e \times \cP_2(\bR^e), \lip\Big( (\bR^e)^{\times \bN_0} , \lin\Big( \bigoplus_{j=0}^\infty  (\bR^d)^{\otimes j}, \bR^e\Big)\Big)\Big)$ so that
$$
\tilde{\Psi}^f\Big[ (N, E, h_0, H, L) \Big] (x, \mu) \in \lip\Big( (\bR^e)^{\times |H|}, \lin\Big( (\bR^d)^{\otimes |N|}, \bR^e\Big) \Big)
$$
and be defined inductively $\forall (x, \mu) \in \bR^e \times \cP_2(\bR^e)$ and $\forall T \in \scT_{0,n}$ by 
\begin{align*}
&\tilde{\Psi}^f \Big[ \big\lfloor \rId \big\rfloor \Big] \Big( x_{h_0}, \mu \Big) = f\Big( x_{h_0}, \mu \Big)
\\
&\tilde{\Psi}^f \Big[ T \Big] \Big(x_{h_0^{T}}, \mu\Big)\Big( x_{H^T} \Big) 
\\
&=
\sum_{a \in A_\ast^{n}} \sum_{\substack{T_1, ..., T_{|a|} \\ \in \scT_0}} \frac1{|a|!} \mathbbm{1}_{\Big\{ T= \big\lfloor \cE^a[T_1, ..., T_{|a|}]\big\rfloor \Big\}}
\cdot 
\partial_a f\Big( x_{h_0^{T}}, \mu, x_{\tilde{h}_1^{T}}, ..., x_{\tilde{h}_{m[a]}^{T}} \Big) \cdot \bigotimes_{r=1}^{|a|} \tilde{\Psi}^f\Big[ T_{r} \Big]\Big( x_{\tilde{h}_{a_{r}}^{T}} , \mu\Big) ( x_{H^{T_{r}}})
\end{align*}
where $x_{H^{T}}$ denotes the tuple $(\underbrace{x_h, ...}_{h \in H^{T}})$ and the sets 
$$
\tilde{h}_{r}^{T} = \underset{\substack{p\\ a_p = r}}{\bigcup} h_0^{T_p} \neq \emptyset
$$
are defined as in Definition \ref{definition:Z-set} and $T=\cE^a[T_1, ..., T_{|a|}]$.
\end{definition}
In order to check that this definition is appropriate, it suffices to return to \eqref{eq:McKean-Increment}. It is then pretty easy to get 
\eqref{eq:RCRP_SortOf} when $n=1$. Next, we can iterate inductively on the value of $n$ in order to get \eqref{eq:RCRP_SortOf} at any order. The key point now is to replace $\Pi_{s,r}^X$ in the right-hand side by the expansion \eqref{eq:RCRP_SortOf} but at rank $n-1$ so that there is no loop in the derivation. Returning to \eqref{eq:definition:RandomControlledRP1}, this prompts us to regard $X_{s,t}$ as $\langle \bX_{s, t} \rId \rangle$ for a random controlled rough path $\bX$ whose higher order components (in 
$\scH$) are precisely given by the elementary differentials. 

Thus the next step is to compute the increment of an elementary differential and check that the resulting expansion 
is consistent with the form of the jets postulated in Definition 
\ref{definition:RandomControlledRP}. By taking a Lions-Taylor expansion and substituting Equation \eqref{eq:RCRP_SortOf}, we can first prove that for $Y=\big( \{0\}, \emptyset, \{0\}, \emptyset, L \big)$, we have
\begin{align*}
\tilde{\Psi}^f \Big[& Y \Big] \Big(X(\omega_{0}), \cL^X \Big)_{s, t} = f\Big( X(\omega_0), \cL^X\Big)_{s, t}
\\
&= \sum_{i=1}^{n-1} \sum_{a\in \A{i}} \frac{\rD^a f\big( X_s(\omega_0), \cL_s^X \big)\big[ X_{s, t}(\omega_0), \Pi_{s, r}^X \big] }{i!} + O\Big( |t-s|^{n} \Big)
\\
&= \sum_{i=1}^{n-1} \sum_{a\in \A{i}} \frac1{|a|!} \bE^{1,...,m[a]} \bigg[  \partial_a f\Big(X_s(\omega_0), \cL_s^X, X_s(\omega_1), ..., X_s(\omega_{m[a]} )\Big) \cdot \bigotimes_{j=1}^{|a|} X_{s,t}(\omega_{a_j}) \bigg]
\\
&\quad + O\Big( |t-s|^{n} \Big),
\end{align*}
where the order of the remainder here follows from the smoothness of $X$. The key point now is to replace $\Pi_{s,r}^X$ in the right-hand side by the expansion \eqref{eq:RCRP_SortOf} to get
\begin{align*}
\tilde{\Psi}^f &\Big[ Y \Big] \Big(X(\omega_{0}), \cL^X \Big)_{s, t}  
\\
&= \sum_{i=1}^{n-1} \sum_{a \in \A{i}} \sum_{\substack{T_1, ..., T_{|a|} \\ \in \scT_0}} \frac1{|a|!} \bE^{1,...,m[a]} \bigg[ \partial_a f\Big(X_s(\omega_0), 
\cL_s^X, X_s(\omega_1), ..., X_s(\omega_{m[a]}) \Big) 
\\
&\qquad  \cdot \bigotimes_{j=1}^{|a|} \tilde{\Psi}^f\Big[ T_{j} \Big]\Big( X_s(\omega_{a_j}) , \cL_s^X \Big) \Big( X_s(\omega)_{H^{T_{j}}} \Big) \cdot \bigotimes_{j=1}^{|a|} \Big\langle \rw_{s, t}, T_j \Big\rangle (\omega_{a_j}, \omega_{H^{T_j}} ) \bigg]
+ 
O\Big( |t-s|^{n} \Big),
\end{align*}

For $a$ and $T_1,...,T_{|a|}$ as in the summand above, we use the operator $\cE^a$ in \eqref{eq:definition:E^a-notation} to define $T=\cE^a[ T_1, ..., T_n]$. Then, by Definition \ref{definition:Z-set} (recalling that $T_1,...,T_{|a|}$ have non-empty 0-hyperedges),
$$
\Big\langle \rw_{s, t}, \Upsilon \Big\rangle(\omega_0, \omega_{H^\Upsilon})
=
\bigotimes_{j=1}^{|a|} \Big\langle \rw_{s, t}, T_j \Big\rangle\big(\omega_{\tilde{h}_j^{T}}, \omega_{H^{T_j}}\big),
$$
so that, by \eqref{eq:proposition:EAction} (noticing that $Z^a[T_1,...,T_n]$ is empty here) and by Definition \ref{definition:MeanField-ElementaryDiffer-Primary}, 
\begin{align*}
&\tilde{\Psi}^f \Big[ Y \Big] \Big(X(\omega_{0}), \cL^X \Big)_{s, t}  
\\
&=
\sum_{T \in \scF} \sum_{i=1}^{n-1} \sum_{a \in \A{i}}  \sum_{\substack{T_1, ..., T_{|a|} \\ \in \scT_0}} \frac1{|a|!} {\mathbf 1}_{\bigl\{ T=\cE^a[T_1,...,T_n]\bigr\}} \bE^{1,...,m[a]} \bigg[ \partial_a f\Big( X_s(\omega_0), \cL_s^X, X_s(\omega_1), ..., X_s(\omega_{m[a]}) \Big) 
\\
&\qquad \cdot \bigotimes_{j=1}^{|a|} \tilde{\Psi}^f\Big[ T_{j} \Big]\Big( X_s(\omega_{a_j}), \cL_s^X \Big) \Big(  X_s(\omega)_{H^{T_{j}}} \Big)
\cdot
\Big\langle \rw_{s, t}, \Upsilon \Big\rangle(\omega_0, \omega_{H^T}) \biggr]
\\
&=\sum_{\substack{T\in\scF: \\ \lfloor T\rfloor \in \scT_{0,n}}} \bE^{H^T} \bigg[ \tilde{\Psi}^f \Big[ \lfloor T \rfloor \Big] \Big( X_s(\omega_{0}), \cL_s^X \Big) \Big( X_s(\omega)_{H^T} \Big) \cdot \Big\langle \rw_{s, t}, T \Big\rangle(\omega_0, \omega_{H^T}) \bigg] + O\Big( |t-s|^{n} \Big).
\end{align*}

We remark that for any choice of $(T, \Upsilon)$ such that $c'(T, \Upsilon, Y)>0$ (for the same choice of $Y$ as above), a non-0 hyperedge of $\Upsilon$ cannot be connected to $Y$ as otherwise it would contain the root of $Y$ and thus of $T$, which is impossible since it is a non-0 hyperedge. Therefore,  $E^{T, \Upsilon, Y}=H^\Upsilon$, $\phi^{T, \Upsilon, Y}[H^T] = H^{\Upsilon}$ and $\varphi^{T, \Upsilon, Y}[H^{\Upsilon}] = H^{\Upsilon}$. By making the appropriate substitution, we get 
\begin{align}
\nonumber
&\tilde{\Psi}^f \Big[ Y \Big] \Big(X(\omega_{0}), \cL^X \Big)_{s, t} = f\Big( X(\omega_0), \cL^X\Big)_{s, t}
\\
\nonumber
&= \sum_{\substack{T \in \scT_{0, n} \\ \Upsilon\in \scF}} c'\Big( T, \Upsilon, Y \Big) \cdot \bE^{E^{T, \Upsilon, Y}}\bigg[ \tilde{\Psi}^f \Big[ T \Big] \Big(X_s(\omega_{0}), \cL_s^X \Big)\Big( X_s(\omega)_{\phi^{T, \Upsilon, Y}[H^T]}\Big) 
\\
\label{eq:Motivation_ElemDiff_increm1}
&\qquad \cdot \Big\langle \rw_{s, t}, \Upsilon \Big\rangle(\omega_0, \omega_{\varphi^{T, \Upsilon, Y}[H^{\Upsilon}]}) \bigg] + O \Big( |t-s|^{n+1-|N^Y|} \Big), 
\end{align}
from which the similarity with Equation \eqref{eq:definition:RandomControlledRP2} can be seen. The order of the remainder should be compared with \eqref{eq:definition:RandomControlledRP4} in Definition \ref{definition:RandomControlledRP}. Here $\scG_{\alpha,\beta}=\alpha$ and in turn $\gamma - \scG_{\alpha,\beta} = \alpha (n-1)$. 

Now proceeding via induction, let us suppose that Equation \eqref{eq:Motivation_ElemDiff_increm1} holds for all $\hat{Y}\in \scT_0$ such that $|N^{\hat{Y}}|\leq k$ and let $Y\in \scT_0$ such that $|N^Y|=k+1$. Then $Y$ can be expressed of the form
$$
Y= \Big\lfloor \cE^a\big[ Y_1, ..., Y_n \big] \Big\rfloor
$$
where the inductive hypothesis applies for each $Y_i$. Thanks to Definition \ref{definition:MeanField-ElementaryDiffer-Primary}, we have that
\begin{align*}
&\tilde{\Psi}^f \Big[ Y \Big] \Big(X(\omega_{h_0^{Y}}), \cL^{X} \Big)\Big( X(\omega)_{H^Y} \Big)_{s, t}
=
\sum_{i=1}^{n-1} \sum_{a \in \A{i}} \sum_{\substack{Y_1, ..., Y_{|a|} \\ \in \scT_0}} \frac1{|a|!} \mathbbm{1}_{\Big\{ Y= \big\lfloor \cE^a[Y_1, ..., Y_{|a|}]\big\rfloor \Big\}}
\\
&\Bigg( 
\bigg\langle \partial_a f\Big( X(\omega_{h_0^{Y}}), \cL^X, ..., X(\omega_{\tilde{h}_{m[a]}^{Y}}) \Big)_{s, t}, \bigotimes_{r=1}^{|a|} \tilde{\Psi}^f\Big[ Y_{r} \Big]\Big( X_s(\omega_{\tilde{h}_{a_{r}}^{Y}}), \cL_s^X \Big) \Big( X_s(\omega)_{H^{Y_{r}}} \Big) \bigg\rangle
\\
&\quad+
\bigg\langle \partial_a f\Big( X_s(\omega_{h_0^{Y}}), \cL_s^X, ..., X_s(\omega_{\tilde{h}_{m[a]}^{Y}}) \Big), \bigg( \bigotimes_{r=1}^{|a|} \tilde{\Psi}^f\Big[ Y_{r} \Big]\Big( X(\omega_{\tilde{h}_{a_{r}}^{Y}}), \cL^X \Big) \Big( X(\omega)_{H^{Y_{r}}} \Big)\bigg)_{s, t} \bigg\rangle
\\
&\quad+
\bigg\langle \partial_a f\Big( X(\omega_{h_0^{Y}}), \cL^X, ..., X(\omega_{\tilde{h}_{m[a]}^{Y}}) \Big)_{s, t}, \bigg( \bigotimes_{r=1}^{|a|} \tilde{\Psi}^f\Big[ Y_{r} \Big]\Big( X(\omega_{\tilde{h}_{a_{r}}^{Y}}), \cL^X \Big) \Big( X(\omega)_{H^{Y_{r}}} \Big)\bigg)_{s, t} \bigg\rangle \Bigg). 
\end{align*}
At this stage, we feel better not to carry out the computations 
explicitly as this would be too lengthy. However, the strategy should be clear to the reader. By Taylor expanding the functions $\partial_a f$ and applying the inductive hypothesis, we should indeed eventually obtain that for any choice of $Y\in \scT_{0, n}$, 
\begin{align}
\nonumber
\tilde{\Psi}^f \Big[ Y \Big]& \Big(X(\omega_0), \cL^{X} \Big)\Big( X(\omega)_{H^Y} \Big)_{s, t}
\\
\nonumber
&= \sum_{\substack{T \in \scT_{0, n} \\ \Upsilon \in \scF}} c'\Big( T, \Upsilon, Y \Big) \cdot \bE^{E^{T, \Upsilon, Y}}\bigg[ \tilde{\Psi}^f \Big[ T \Big] \Big(X_s(\omega_0), \cL_s^{X} \Big)\Big( X_s(\omega)_{\phi^{T, \Upsilon, Y}[H^T]} \Big) 
\\
\label{eq:Motivation_ElemDiff_increm2}
&\qquad \cdot \Big\langle \rw_{s, t}, \Upsilon \Big\rangle(\omega_0, \omega_{\varphi^{T, \Upsilon, Y}[H^{\Upsilon}]}) \bigg] 
+ 
O \Big( |t-s|^{n+1-|N^Y|} \Big). 
\end{align}
The technical details of this proof are unimportant at this point in the proceedings, but can easily be reproduced from results that follow. We refer the reader to \cite{salkeld2021Probabilistic3} for further results relating to elementary differentials. 

\begin{remark}
The reader may wonder why, unlike the general Definition \ref{definition:RandomControlledRP}, Definition \ref{definition:MeanField-ElementaryDiffer-Primary} solely involves trees with a non-empty 0-hyperedge. In fact, this is a peculiarity of the structure inherited from the mean-field equation \eqref{eq:McKean-Increment}. Basically, the formula for the increment $X_{s,t}(\omega_0)$ relies on the outcome of the noise $W$ for the same realisation $\omega_0$ and does not involve the realisation of any independent copy of the noise. To emphasise this point, such an equation would take the form
$$
dX_{r}(\omega_0) = f\big( X_r(\omega_0), \cL_r^X \big) \otimes dW_r(\omega_0) + \hat{\bE}\Big[ f\big( X_r(\hat{\omega}), \cL_r^X \big) \otimes dW_r(\hat{\omega}) \Big]. 
$$
We propose such equations as a direction of future study. 
\end{remark}

\subsection{Operations on random controlled rough paths}
\label{subsec:Operat_RCRPs}

In this section, our goal is to explore a collection of key results that demonstrate that random controlled rough paths are the appropriate tool for studying mean-field equations. 

\subsubsection{The Reconstruction Theorem}
\label{subsubsection:ReconstructionTheorem}

The reconstruction theorem is a now celebrated result, proved in its first form in \cite{gubinelli2004controlling}*{Proposition 1} and later in a more general form  in \cite{hairer2014theory}*{Theorem 3.10}. In this next Theorem, we describe how the probabilistic structure is interwoven into this classical result. For the reader who is not aware of the notion of `reconstruction' but who is more familiar with mean-field systems driven by a standard Brownian motion, the following statement should be seen as the definition of the integral of a random controlled rough path with respect to a probabilistic rough path.

\begin{theorem}
\label{theorem:Reconstruction}
Let $\alpha, \beta>0$ and $\gamma:=\inf\{ \scG_{\alpha, \beta}[T]: T \in \scF, \scG_{\alpha, \beta}[T]>1-\alpha \}$. Let $(\ddot{\scH})^{\gamma, \alpha, \beta}$ denote the $L^0 \big(\Omega, \bP; \lin(\bR^d, \bR^e) \big)$-module
\begin{equation}
\label{eq:theorem:Reconstruction_Space}
(\ddot{\scH})^{\gamma, \alpha, \beta} = \bigoplus_{T\in \scF_0}^{\gamma, \alpha, \beta} L^0\Big( \Omega \times \Omega^{\times |H^T|}, \bP \times \bP^{\times |H^T|}; \lin\big( (\bR^d)^{\otimes |N^T|}, \lin(\bR^d, \bR^e) \big) \Big) \cdot T
\end{equation}
so that $\big( (\ddot{\scH})^{\gamma, \alpha, \beta}, \circledast, \rId, \Delta, \epsilon, \scS\big)$ is a coupled Hopf algebra. 

Let $(p_x,q)$ be a dual pair of integrability functionals and suppose that 
\begin{equation}
\label{eq:proposition:Reconstruction2}
\sup_{T\in \scF^{\gamma, \alpha, \beta}} \tfrac{1}{q[T]}
<
\sup_{T\in \scF_0^{\gamma-\alpha, \alpha, \beta}} \Big( \tfrac{1}{p_x[T]} + \tfrac{1}{q\big[ \lfloor T \rfloor \big]} \Big)
\leq 1. 
\end{equation}
We define $p_y:\scF_0^{\gamma, \alpha, \beta} \to [1, \infty)$ by
\begin{equation}
\label{eq:proposition:Reconstruction1}
\frac{1}{p_y[\rId]}:= \sup_{T\in \scF_0^{\gamma-\alpha, \alpha, \beta}} \bigg( \frac{1}{p_x[T]} + \frac{1}{q\big[ \lfloor T \rfloor \big]} \bigg), \quad \frac{1}{p_y[T]} := \frac{1}{p_y[\rId]} - \frac{1}{q[T]}.  
\end{equation}
Then $(p_y, q)$ is also a dual pair of integrability functional. 

Let $\rw$ be an $(\scH^{\gamma, \alpha, \beta}, p_x, q)$-probabilistic rough path and let $\bX\in \cD_{\rw}^{\gamma, p_x, q}\big( (\ddot{\scH})^{\gamma, \alpha, \beta} \big)$. We define
\begin{align}
\nonumber
&\int_0^1 X_r dW_r(\omega_0) 
\\
\label{eq:theorem:Reconstruction}
&= \lim_{|D_n|\to 0} \sum_{[u, v]\in D_n} \sum_{T \in \scF_0}^{\gamma-\alpha, \alpha, \beta} \bE^{H^T} \bigg[ \Big\langle \bX_u, T \Big\rangle(\omega_0, \omega_{H^T}) \cdot \Big\langle \rw_{u, v}, \lfloor T \rfloor_i \Big\rangle (\omega_0, \omega_{H^T}) \bigg]. 
\end{align}
Then the following two results hold:
\begin{enumerate}
\item The limit in Equation \eqref{eq:theorem:Reconstruction} exists and there exists a constant $C>0$ depending only on $\alpha, \beta$ such that for any $u, v\in[0,1]$, 
\begin{align}
\nonumber
\Bigg| \int_u^v& X_r dW_r(\omega_0) - \sum_{T\in \scF_0}^{\gamma-\alpha, \alpha, \beta} \bE^{H^T} \bigg[ \Big\langle \bX_u, T \Big\rangle(\omega_0, \omega_{H^T}) \cdot \Big\langle \rw_{u, v}, \lfloor T \rfloor_i \Big\rangle (\omega_0, \omega_{H^T}) \bigg] \Bigg|
\\
\nonumber
&\leq C \Bigg( \sum_{T \in \scF_0}^{\gamma-\alpha, \alpha, \beta} \Big\| \big\langle \bX^{\sharp}, T \big\rangle (\omega_{0}) \Big\|_{p[T],\gamma - \scG_{\alpha, \beta}[T]} 
\\
\label{eq:theorem:Reconstruction:Inequality}
&\qquad \qquad 
\cdot 
\Big\| \big\langle \rw, \lfloor T\rfloor \big\rangle(\omega_{0}) \Big\|_{q[ \lfloor T \rfloor_i ], \scG_{\alpha, \beta}[T]+\alpha} \Bigg) \cdot |v-u|^{\gamma + \alpha}.  
\end{align}
\item Let $\Phi$ be the map $\Phi: \cD_{\rw}^{\gamma, p_x, q}\big( (\ddot{\scH})^{\gamma, \alpha, \beta} \big) \to \cD_{\rw}^{\gamma, p_y, q} \big( \scH^{\gamma, \alpha, \beta}\big)$ defined by
\begin{equation}
\label{eq:theorem:Reconstruction:Phi}
\begin{split}
\Big\langle \Phi\big[ \bX \big]_{t}, \rId \Big\rangle(\omega_0) = \int_0^t X_r dW_r(\omega_0), 
&\quad 
\Big\langle \Phi\big[ \bX \big]_{t}, \lfloor T \rfloor \Big\rangle(\omega_0, \omega_{H^T}) = \Big\langle \bX_{t}, T\Big\rangle(\omega_0, \omega_{H^T}), 
\\
\Big\langle \Phi\big[ \bX\big]_t, T_1 \circledast T_2 \Big\rangle&(\omega_0, \omega_{H^{T_1}}, \omega_{H^{T_2}}) =0
\end{split}
\end{equation}
Then $\Phi$ is a continuous map between Banach spaces and satisfied that for some constant $C=C(\alpha, \beta)>0$ dependent only on $\alpha$ and $\beta$,
\end{enumerate}
\begin{align}
\nonumber
\sum_{Y \in \scF_0}^{\gamma-, \alpha, \beta}& \sup_{s,t\in[u,v]} \frac{ \Big\| \big\langle \Phi[ \bX]_{s, t}^{\sharp}, Y \big\rangle(\omega_0) \Big\|_{p_y[ Y ]} }{|t-s|^{\gamma - \scG_{\alpha, \beta}[Y] }}  
\\
\nonumber
\leq& C\Bigg( \sum_{T \in \scF_0}^{\gamma-\alpha, \alpha, \beta} 
\Big\| \big\langle \bX^{\sharp}, T \big\rangle(\omega_0) \Big\|_{p_x[T],\gamma - \scG_{\alpha, \beta}[T]} 
\cdot \bigg(
\Big\| \big\langle \rw, \lfloor T \rfloor \big\rangle(\omega_0) \Big\|_{q[T],\scG_{\alpha, \beta}[T] + \alpha} + 1 \bigg) \cdot |v-u|^{\alpha}
\\
\nonumber
&+ \sum_{\substack{ T\in \scF \\ \scG_{\alpha, \beta}[T]\in [\gamma-\alpha, \gamma)}} \Big\| \big\langle \bX^{\sharp}, T\big\rangle(\omega_0) \Big\|_{p_x[T], \gamma-\scG_{\alpha, \beta}[T]}
\cdot 
\bigg( \sum_{\Upsilon\in \scF}^{\gamma-, \alpha, \beta} \Big\| \big\langle \rw, \Upsilon\big\rangle(\omega_0) \Big\|_{q[\Upsilon],\scG_{\alpha, \beta}[\Upsilon]} \bigg) \cdot |v-u|^{\alpha}
\\
\nonumber
&+ \sum_{\substack{ T\in \scF \\ \scG_{\alpha, \beta}[T]\in [\gamma-\alpha, \gamma)}} \Big\| \big\langle \bX_u, T \big\rangle(\omega_0) \Big\|_{p_x[T]} 
\cdot
\bigg( \Big\| \big\langle \rw, \lfloor T \rfloor \big\rangle(\omega_0) \Big\|_{q[\lfloor T \rfloor], \scG_{\alpha, \beta}[\lfloor T \rfloor]} 
\\
\label{eq:theorem:Reconstruction:Phi-2}
&\qquad +
\sum_{\Upsilon\in \scF}^{\gamma, \alpha, \beta} \Big\| \big\langle \rw, \Upsilon \big\rangle(\omega_0) \Big\|_{q[\Upsilon], \scG_{\alpha, \beta}[\Upsilon]} \bigg) \cdot |v-u|^{\scG_{\alpha, \beta}[T] - (\gamma-\alpha)} \Bigg). 
\end{align}
\end{theorem}

The proof of Theorem \ref{theorem:Reconstruction} is delayed until Section \ref{subsec:ProofReconstruction}. 

\begin{remark}
It is worth observing that in  the inequality \eqref{eq:theorem:Reconstruction:Inequality}, $\gamma+\alpha>1$. In particular, \eqref{eq:theorem:Reconstruction:Inequality} provides an approximation of the integral in the left-hand side in terms of a compensated Riemann sum, which is one of the key ingredient of rough path theory. In comparison with the same result but in the standard rough setting, the main novelty here lies precisely in the form of the Riemann sum, which contains extra terms. Typically, those extra terms are due to some mean-field interaction, as it is the case in \eqref{eq:McKean-Increment}. 
\end{remark}

\begin{remark}
We highlight that the difference between $(\ddot{\scH})^{\gamma, \alpha, \beta}$ defined in Equation \eqref{eq:theorem:Reconstruction_Space} and $\scH^{\gamma, \alpha, \beta}$ as described in Equation \eqref{eq:Coupled-HopfAlgebra} is solely that the associated module of measurable functions runs over a different vector space, $\lin(\bR^d, \bR^e)$ and $\bR^e$ respectively. 
\end{remark}

\subsubsection{Continuous Image of Random controlled rough paths}
\label{subsubsect:Contin-Image-RandContRP}

We consider the composition of a RCRP by a smooth functions using a Lions-Taylor expansion. The structure of random controlled rough paths has been designed specifically to naturally combine with Theorem \ref{theorem:LionsTaylor2}, see Theorem \ref{theorem:ContinIm-RCRPs} below. 

For concise notation, we are going to use the following notation:
\begin{definition}
\label{definition:B_set}
Let $\alpha, \beta>0$ and let 
$$
\gamma:=\inf \big\{ \alpha  i + \beta j: (i, j)\in \bN_0^{\times 2}, \alpha i + \beta j > 1-\alpha\big\}, 
\quad
n:=\sup \big\{ m\in \bN_0: m < \tfrac{\gamma}{\alpha\wedge \beta} \big\}.
$$
We define the sets
\begin{align*}
B^{\alpha, \beta}:=& \Big\{ (i, j) \in \bN_0^{\times 2}: 
i+j \leq n, 
\quad
\alpha i + \beta j < \gamma \Big\}, 
\\
B_\ast^{\alpha, \beta}:=& \Big\{ (i, j) \in \bN_0^{\times 2} : 
i+j \leq n, 
\quad
\alpha i + \beta j \geq \gamma \Big\}. 
\end{align*}

Further, for $(i,j)\in B^{\alpha, \beta}$, we define
\begin{align*}
B_{i, j}^{\alpha, \beta} =& \Big\{ (\tilde{\imath}, \tilde{\jmath})\in B^{\alpha, \beta}: \tilde{\imath}\geq i, \tilde{\jmath}\geq j, \tilde{\imath}+\tilde{\jmath}>i+j \Big\} \quad \mbox{and} 
\\
B_{i, j|\ast}^{\alpha, \beta} =& \Big\{ (\tilde{\imath}, \tilde{\jmath})\in B_\ast^{\alpha, \beta}: \tilde{\imath}\geq i, \tilde{\jmath}\geq j, \tilde{\imath}+\tilde{\jmath}>i+j \Big\}. 
\end{align*}
With these in mind, we denote
$$
A^{\alpha, \beta}:= \bigcup_{(i, j)\in B^{\alpha, \beta}} A_{i, j}, 
\quad
A_{\ast}^{\alpha, \beta}:= \bigcup_{(i, j)\in B_{\ast}^{\alpha, \beta}} A_{i, j}, 
$$
and for $a\in A_{i, j}$
$$
A^{\alpha, \beta| a}:= \bigcup_{(\tilde{\imath}, \tilde{\jmath})\in B_{i, j}^{\alpha, \beta}} A_{i, j}, 
\quad
A_{\ast}^{\alpha, \beta| a}:= \bigcup_{(\tilde{\imath}, \tilde{\jmath})\in B_{i, j|\ast}^{\alpha, \beta}} A_{i, j}.  
$$
\end{definition}

The set of partition sequences $A^{\alpha, \beta}$ is a concise way of collecting the necessary derivatives when considering a Taylor expansion in spacial and measure variables when the regularity of the spacial variable is of order $\alpha$ and regularity of the measure variable is of order $\beta$. In particular, this representation captures the appropriate number of derivatives in the spacial and measure variables in our setting. 

For brevity, we use the notation that for $a \in \A{n}$, $i\in\{1, ..., n\}$ and $T \in \scF$, 
\begin{equation}
\label{eq:omega_zero-zero_hyperedge}
\begin{split}
&\Big\langle \big[ \bX, \bY \big]_{s, t}, T\Big\rangle (\omega_{a_i}) = 
\begin{cases}
\Big\langle \bX_{s, t}, T\Big\rangle (\omega_{0}) 
& \quad \mbox{if} \quad a_i=0
\\
\Big\langle \bY_{s, t}, T\Big\rangle (\omega_{a_i})
& \quad \mbox{if} \quad a_i>0
\end{cases}, 
\\
&\Big\langle \big[ \bX, \bY \big]_{s, t}, T\Big\rangle (\omega_{\tilde{h}_{a_i}^{T}}) = 
\begin{cases}
\Big\langle \bX_{s, t}, T\Big\rangle (\omega_{0}) 
& \quad \mbox{if} \quad a_i=0
\\
\Big\langle \bY_{s, t}, T\Big\rangle (\omega_{\tilde{h}_{a_i}^{T}})
& \quad \mbox{if} \quad a_i>0
\end{cases},  
\end{split}
\end{equation}

\begin{theorem}
\label{theorem:ContinIm-RCRPs}
Let $\alpha, \beta>0$ and let 
$$
\gamma:=\inf \big\{ \alpha  i + \beta j: (i, j)\in \bN_0^{\times 2}, \alpha i + \beta j > 1-\alpha\big\}, 
\quad
n:=\sup \big\{ m\in \bN_0: m < \tfrac{\gamma}{\alpha\wedge \beta} \big\}.
$$
Let $(p_x, q)$ and $(p_y, q)$ be pairs of dual integrability functionals and additionally suppose that
\begin{equation}
\label{eq:theorem:ContinIm-RCRPs4}
\sup_{T\in \scF^{\gamma, \alpha, \beta}} \Big( \tfrac{1}{q[T]} \Big) < \tfrac{n+1}{p_y[\rId]} \leq 1,  
\quad
\sup_{T\in \scF^{\gamma, \alpha, \beta}} \Big( \tfrac{1}{q[T]} \Big) < \tfrac{n+1}{p_x[\rId]} \leq 1. 
\end{equation}
We define $p_z:\scF_0^{\gamma, \alpha, \beta}\to [1, \infty)$ such that
\begin{equation}
\label{eq:theorem:ContinIm-RCRPs3}
\frac{1}{p_z[\rId]}:= \frac{n+1}{ p_x[\rId]}, 
\quad
\frac{1}{p_z[T]}:= \frac{1}{p_z[\rId]} - \frac{1}{q[T]}. 
\end{equation}
Then $(p_z, q)$ are a dual pair of integrability functionals.

Let $\rw$ be an $(\scH^{\gamma, \alpha, \beta}, p_z, q)$-probabilistic rough path, and let 
$$
\bX \in \cD_{\rw}^{\gamma, p_x, q}\big( \scH^{\gamma-, \alpha, \beta} \big)
\quad \mbox{and} \quad
\bY\in \cD_{\rw}^{\gamma, p_y, q}\big( \scH^{\gamma-, \alpha, \beta} \big).
$$ 
Let $f: \bR^e \times \cP_2(\bR^e) \to \lin\big( \bR^d, \bR^e\big)$ satisfy that $f \in C_b^{n, (n)}\big( \bR^e \times \cP_2(\bR^e) \big)$. 

Recalling Equation \eqref{eq:theorem:Reconstruction_Space}, we define $\bZ:[0,1]\to (\ddot{\scH})^{\gamma-, \alpha, \beta}$ by 
\begin{equation}
\label{eq:theorem:ContinIm-RCRPs2}
\Big\langle \bZ_t, \rId \Big\rangle(\omega_0) = f\Big(  \big\langle \bX_t, \rId \big\rangle(\omega_0), \cL^{\langle \bY_t, \rId\rangle} \Big). 
\end{equation}
Then there exists a random controlled rough path $\bZ\in \cD_\rw^{\gamma, p_z, q}\big((\ddot{\scH})^{\gamma-, \alpha, \beta} \big)$ that satisfies Equation \eqref{eq:theorem:ContinIm-RCRPs2} given by
\begin{align}
\nonumber
&\Big\langle \bZ_t, Y \Big\rangle (\omega_{0}, \omega_{H^Y})
=
\sum_{a\in A^{\alpha, \beta} }  \frac{1}{|a|!} 
\sum_{\substack{Y_1, ..., Y_{|a|} \in \scF \\ Y = \cE^a[Y_1, ..., Y_{|a|}] }} 
\bE^{Z^a[Y_1, ..., Y_{|a|}]} \bigg[ 
\\
\label{eq:theorem:ContinIm-RCRPs}
&\qquad \partial_a f\Big( \big\langle \bX_{t}, \rId \big\rangle (\omega_{0}), \cL^{\langle \bY_t, \rId\rangle}, ..., \big\langle \bY_{t}, \rId \big\rangle (\omega_{\tilde{h}_{m[a]}^Y}) \Big) \cdot
\bigotimes_{r=1}^{|a|} \Big\langle [\bX, \bY]_t, Y_r \Big\rangle (\omega_{\tilde{h}_{a_r}^Y}, \omega_{H^{Y_r}}) \bigg] . 
\end{align}

Further, there exists a polynomial, $\fP:\big( \bR^+ \big)^{\times 7} \to \bR^+$ increasing in every variable such that
\begin{align}
\nonumber
\sum_{T\in \scF_0}^{\gamma-, \alpha, \beta}& \sup_{s, t\in[u,v]} \frac{ \Big\| \big\langle \bZ_{s, t}^{\sharp}, T \big\rangle(\omega_0) \Big\|_{p_z[T]} }{|t-s|^{\gamma - \scG_{\alpha, \beta}[T]}} 
\\
\nonumber
&\leq \| f \|_{C_b^{n, (n)}} \cdot 
\fP\bigg(
 \sum_{T\in\scF_0}^{\gamma-, \alpha, \beta} \Big\| \big\langle \bX^{\sharp}, T \big\rangle(\omega_0) \Big\|_{p_x[T], \gamma - \scG_{\alpha, \beta}[T]}
 , 
\sum_{T\in \scF}^{\gamma-, \alpha, \beta} \Big\| \big\langle \bX_u, T\big\rangle(\omega_0) \Big\|_{p_x[T]}, 
\\
\nonumber
&\qquad 
\sum_{T\in\scF_0}^{\gamma-, \alpha, \beta} \Big\| \big\langle \bY^{\sharp}, T \big\rangle \Big\|_{p_y[T], \gamma - \scG_{\alpha, \beta}[T]}
, 
\sum_{T\in \scF}^{\gamma-, \alpha, \beta} \Big\| \big\langle \bY_u, T\big\rangle \Big\|_{p_y[T]}, 
\\
\label{eq:theorem:ContinIm-RCRPs-Est}
&\qquad \sum_{T\in \scF}^{\gamma, \alpha, \beta} \Big\| \big\langle \rw, T\big\rangle(\omega_0) \Big\|_{q[T], \scG_{\alpha, \beta}[T]}
, 
|v-u|^\alpha
, 
|v-u|^{\beta} \bigg)
\end{align}
\end{theorem}

The proof of Theorem \ref{theorem:ContinIm-RCRPs} is delayed until Section \ref{subsubsec:SfRCRPs,LTE}. 

\begin{remark}
\label{remark:Polynomial-comments1}
Drawing inspiration from the representations found in \cite{boedihardjo2020lipschitz}, the polynomial 
$$
\fP\Big( x_1, x_2, y_1, y_2, w, t_1, t_2 \Big) = \sum_{i\in I_{\fP}} C_i \cdot x_1^{i_1} \cdot x_2^{i_2} \cdot y_1^{i_3} \cdot y_2^{i_4} \cdot w^{i_5} \cdot t_1^{i_6} \cdot t_2^{i_7}. 
$$
in Equation \eqref{eq:theorem:ContinIm-RCRPs-Est} satisfies the following identities:
\begin{itemize}
\item $i_1, ..., i_5 \in \bN_0$
\item $i_1+i_2+i_3+i_4 \leq n+1$.In particular, this means that $i_1 + i_2 \leq n+1$ and $i_3+i_4 \leq n+1$. 
\item $i_5 \leq i_1 + i_2 + i_3 + i_4$. 
\item While $i_6, i_7\in \bZ$ (that is, may be negative), we always have that $\alpha \cdot i_6 + \beta \cdot i_7\geq 0$. 
\end{itemize}
\end{remark}

\subsection{Coupled coproduct identity}
\label{subsec:CoupledCoproductId}

The following two Lemmas are natural extensions of known results for coproducts extended to the coupled coproduct setting. We need them in the analysis that is carried out below.

\begin{lemma}
\label{lemma:CoprodCountIdent2}
Let $T, \Upsilon$ and $Y \in \scF_0$. Then we have that $\forall i\in\{1, ..., d\}$, 
\begin{equation}
\label{eq:lemma:CoprodCountIdent2}
c\Big( \lfloor T \rfloor_i, \Upsilon, \lfloor Y \rfloor_i\Big) = c\Big( T, \Upsilon, Y\Big)
\end{equation}
\end{lemma}

\begin{proof}
Firstly, $\lfloor T\rfloor_i\in \scT$ so any root from an admissible cut must be expressible as $\lfloor Y\rfloor_i$ where $Y\in \scF_0$. Thanks to Equation \eqref{eq:definition:coproduct}, we have $\Delta\big[\lfloor T \rfloor_i \big] = (I \times^{H^T} \lfloor \cdot \rfloor_i)\circ \Delta\big[ T \big] + \lfloor T \rfloor_i \times^{H^T} \rId$ so that
\begin{align*}
c\Big( \lfloor T \rfloor_i, \Upsilon, \lfloor Y \rfloor_i \Big) =& \Big\langle \Delta\Big[ \lfloor T \rfloor_i \Big], \Upsilon \times^{H^T} \lfloor Y \rfloor_i \Big\rangle
\\
=& \Big\langle \lfloor T \rfloor_i \times^{H^T} \rId + \Big( I \times^{H^T} \lfloor \cdot \rfloor_i \Big) \circ \Delta\big[ T \big], \Upsilon \times^{H^T} \lfloor Y \rfloor_i \Big\rangle
\\
=& \Big\langle \Delta\big[ T \big], \Upsilon \times^{H^T} Y \Big\rangle = c\Big( T, \Upsilon, Y\Big). 
\end{align*}
\end{proof}

\begin{lemma}
\label{lemma:CoprodCountIdent}
Let $a\in \A{n}$, suppose that $T_1, ..., T_n \in \scF_0$ such that $\cE^a[T_1, ..., T_n] \in \scF$. 

Then for $\Upsilon, Y \in \scF$, we have
\begin{equation}
\label{eq:lemma:CoprodCountIdent}
c\Big( \cE^a[T_1, ..., T_n], \Upsilon, Y\Big) = \sum_{\substack{\Upsilon_1, ..., \Upsilon_n\in\scF_0 \\\cE^a[ \Upsilon_1, ..., \Upsilon_n] = \Upsilon}} \sum_{\substack{Y_1, ..., Y_n\in\scF_0 \\\cE^a[ Y_1, ..., Y_n] = Y}} \prod_{i=1}^n c \Big( T_i, \Upsilon_i, Y_i \Big) 
\end{equation}
\end{lemma}

\begin{proof}
Let $a \in \A{n}$ and $T_i \in \scF_0$. To alleviate complicated notation, for $n\in \bN$ such that $n\geq 2$ we denote $\circledast^{(2)(n)}: ( \scF_0 \tilde{\times} \scF_0)^{\times n} \to (\scF_0 \tilde{\times} \scF_0)$ by
$$
\circledast^{(2)(n)}\Big[ \Upsilon_1 \times^{H^{T_1}} Y_1, ..., \Upsilon_n \times^{H^{T_n}} Y_n\Big] = \Big( \circledast_{i=1}^n \Upsilon_i \Big) \times^{\bigcup_{i=1}^n H^{T_i}} \Big( \circledast_{i=1}^n Y_i \Big). 
$$

Using the coproduct identities from Equation \eqref{eq:definition:coproduct}, we get
\begin{align*}
&\Delta\Big[ \cE^a[T_1, ..., T_n] \Big] = \Delta\bigg[ \Big(\underset{\substack{i: \\ a_i = 0}}{\circledast} T_i\Big) \circledast \cE \Big[\underset{\substack{i: \\ a_i = 1}}{\circledast}  T_i\Big] \circledast ... \circledast \cE\Big[\underset{\substack{i: \\ a_i = m[a]}}{\circledast} T_i\Big] \bigg]
\\
=& \circledast^{(2)(m[a]+1)}\bigg[ \circledast^{(2)(l[a]_0)}\Big[ \underbrace{\Delta[T_i], ...}_{i: a_i=0} \Big], \cE \tilde{\times} \cE \circ \circledast^{(2)(l[a]_1)}\Big[ \underbrace{\Delta[T_i], ...}_{i: a_i=1} \Big], ..., \cE \tilde{\times} \cE \circ \circledast^{(2)(l[a]_{m[a]})}\Big[ \underbrace{\Delta[T_i], ...}_{i: a_i=m[a]} \Big] \bigg]
\\
=&\sum_{\Upsilon_1, ..., \Upsilon_n \in \scF_0} \sum_{Y_1, ..., Y_n \in \scF_0} \Big( \prod_{i=1}^n c\big( T_i, \Upsilon_i, Y_i \big)\Big) \cdot \cE^a\Big[ \Upsilon_1, ..., \Upsilon_n \Big] \times^{H^{\cE^a[T_1, ..., T_n]}} \cE^a\Big[ Y_1, ..., Y_n \Big]. 
\end{align*}
Therefore
\begin{align*}
c\Big( \cE^a[T_1, ..., T_n], \Upsilon, Y\Big) =& \Big\langle \Delta\Big[ \cE^a[T_1, ..., T_n]\Big], \Upsilon \times^{H^{\cE^a[T_1, ..., T_n]}} Y \Big\rangle
\\
=&\sum_{\substack{\Upsilon_1, ..., \Upsilon_n \in \scF_0\\ \cE^a[\Upsilon_1, ..., \Upsilon_n] = \Upsilon}} \sum_{\substack{Y_1, ..., Y_n \in \scF_0 \\ \cE^a[Y_1, ..., Y_n] = Y}}  \prod_{i=1}^n c\Big( T_i, \Upsilon_i, Y_i \Big). 
\end{align*}
\end{proof}

\subsection{Duality identities}
\label{subsec:DualityId}

Classically, the increments of controlled rough paths are elements of a Hopf algebra with incremental properties determined by the coproduct counting function. In our setting, we additionally need to encode the couplings from the convolution product into our incremental properties. This is described and addressed in this section. We already gave some intuition about this feature in the introduction of Subsection \ref{subse:def:RCRP}. This is described and addressed in a more systematic and rigorous manner in this subsection.

\subsubsection{Coupling identities}

Motivated by Equation \eqref{eq:definition:psi}, we define the collection of operators for describing a coupled pair of partitions. 

\begin{definition}
\label{definition:CouplingFunctions}
Let $T, \Upsilon, Y\in\scF$ and suppose that $c'(T, \Upsilon, Y)>0$. We define $\psi^{\Upsilon, T}: H^\Upsilon \to H^T$ and $\psi^{Y, T}: H^Y \to H^T$ using the identities
$$
\psi^{\Upsilon, T}[h^\Upsilon] \cap N^{\Upsilon} = h^\Upsilon, 
\quad 
\psi^{Y, T}[h^Y] \cap N^{Y} = h^Y. 
$$
Further, we define $\phi^{T, \Upsilon,Y}: H^T \to H^Y \cup H^\Upsilon$ and $\varphi^{T, \Upsilon, Y}:H^\Upsilon \to H^{Y}\cup H^{\Upsilon}$ as follows:
\begin{align*}
\phi^{T, \Upsilon,Y} [h] =& 
\begin{cases} 
h & \quad \in H^\Upsilon \quad \mbox{if } h\cap N^Y = \emptyset, 
\\
h\cap N^Y \qquad  & \quad \in H^Y  \quad \mbox{if } h\cap N^Y \neq \emptyset. 
\end{cases}
\\
\varphi^{T, \Upsilon, Y}[h] =& 
\begin{cases}
h & \quad \in H^\Upsilon \quad \mbox{if } \psi^{\Upsilon, T}[h]\cap N^Y  = \emptyset, 
\\
\psi^{\Upsilon, T}[h]\cap N^Y & \quad \in H^Y \quad \mbox{if } \psi^{\Upsilon, T}[h]\cap N^Y  \neq \emptyset. 
\end{cases}
\end{align*}
\end{definition}

As with Definition \ref{definition:psi-Mappings}, the operator $\psi$ uses the property of the coproduct counting function that $H^T \in H^\Upsilon \tilde{\cup} H^Y$ and the natural injection for $H^\Upsilon$ and $H^Y$ into $H^T$. The operations $\phi$ and $\varphi$ invert these injections. The mappings $\psi$ are not (always) surjective, so we extend the images of $\phi$ and $\varphi$ in order to define a meaningful mapping. 

The operations $\phi$ and $\varphi$ are key to decoupling random variables that have been obtained via the convolution product (see Definition \ref{definition:DualModule} and Equation \eqref{eq:definition:DualModule_ConvolProd}) and so are necessarily coupled. 

\begin{lemma}
\label{lemma:TechnicalPhiVarphi}
Let $n\in \bN$, let $a\in \A{n}$ and let $T_1, ..., T_n, \Upsilon_1, ..., \Upsilon_n, Y_1, ..., Y_n \in \scF$. 

Suppose that for all $i=1, ..., n$, 
$$
c'\Big( T_i, \Upsilon_i, Y_i \Big)>0. 
$$
Then for all $i=1, ..., n$, 
\begin{align}
\label{eq:lemma:TechnicalPhiVarphi1}
\phi^{T_i, \Upsilon_i, Y_i} = \phi^{\cE^a[T_1, ..., T_n], \cE^a[\Upsilon_1, ..., \Upsilon_n], \cE^a[Y_1, ..., Y_n]}\Big|_{H^{T_i}}
\\
\label{eq:lemma:TechnicalPhiVarphi2}
\varphi^{T_i, \Upsilon_i, Y_i} = \varphi^{\cE^a[T_1, ..., T_n], \cE^a[\Upsilon_1, ..., \Upsilon_n], \cE^a[Y_1, ..., Y_n]}\Big|_{H^{\Upsilon_i}}
\end{align}
\end{lemma}

\begin{proof}
By the hypothesis we know that for all $i=1, ..., n$, $\Big\langle \Delta[T_i], \Upsilon_i \otimes^{H^{T_i}} Y_i \Big\rangle \neq 0$. Using that $\Delta$ is an algebra homomorphism and that $\Delta$ commutes with $\cE$, we have
$$
\Big\langle \Delta\Big[ \cE^a[ T_1, ..., T_n]\Big] , \cE^a[\Upsilon_1, ..., \Upsilon_n] \otimes^{H^{\cE^a[ T_1, ..., T_n]}} \cE^a[Y_1, ..., Y_n] \Big\rangle \neq 0
$$
for any choice of $a\in \A{n}$ so that
$$
c'\Big( \cE^a[T_1, ..., T_n], \cE^a[ \Upsilon_1, ..., \Upsilon_n], \cE^a[ Y_1, ..., Y_n] \Big)>0. 
$$
Suppose for $i=1, ..., n$ that $h^{T_i} \in H^{T_i}$ and $h^{T_i} \cap N^{Y_i} = \emptyset$. Then $h^{T_i} \in H^{\cE^a[T_1, ..., T_n]}$ and since $h^{T_i} \cap N^{Y_j} = \emptyset$ for $j\neq i$ we also have $h^{T_i} \cap N^{\cE^a[Y_1, ..., Y_n]} = \emptyset$. On the other hand if $h^{T_i} \cap N^{Y_i}\neq \emptyset$ then $h^{T_i} \cap N^{Y_i} \subseteq h^{T_i} \cap \Big( \bigcup_{i=1}^n N^{Y_i}\Big) = h^{T_i} \cap N^{\cE^a[Y_1, ..., Y_n]} \neq \emptyset$. Thus
Equation \eqref{eq:lemma:TechnicalPhiVarphi1} holds. 

For $\psi^{\Upsilon, T}[h^{\Upsilon_i}] \in H^{T_i}$, we have $\psi^{\Upsilon, T}[h^{\Upsilon_i}] \cap N^{Y_j}=\emptyset$ for $j\neq i$ so that $\psi^{\Upsilon, T}[h]^{\Upsilon_i} \cap N^{Y_i}=\emptyset$ implies $\psi^{\Upsilon, T}[h]^{\Upsilon_i} \cap N^{\cE^a[Y_1, ..., Y_n]}=\emptyset$. On the other hand if $\psi^{\Upsilon, T}[h^{\Upsilon_i}] \cap N^{Y_i}\neq \emptyset$ then $\psi^{\Upsilon, T}[h^{\Upsilon_i}] \cap N^{\cE^a[Y_1, ...Y_n]} = \psi^{\Upsilon, T}[h^{\Upsilon_i}] \cap N^{Y_i} \neq \emptyset$. Thus
Equation \eqref{eq:lemma:TechnicalPhiVarphi2} holds. 
\end{proof}

\subsubsection{Untagged collections of hyperedges and ghost hyperedges}

Up until this point, when considering a Lions tree $T$, we have thought of the hyperedge $h_0$ as being tagged and distinct from all other hyperedges $h\in H$. Hence, $H$ is the set of untagged hyperedges. However, when expressing the incremental relationships of random controlled rough paths, we should think of all hyperedges as being tagged and additional hyperedges that are uncoupled with these as being untagged. 

This leads us to consider the following set:
\begin{definition}
\label{definition:E-Set}
Let $T, \Upsilon, Y \in \scF$ and suppose that $c'(T, \Upsilon, Y)>0$. Then we define the set
$$
E^{T, \Upsilon, Y} = \{ h\in H^\Upsilon: \varphi^{T, \Upsilon, Y}(h) = h\}. 
$$ 
\end{definition}
As we already alluded to in the introduction of Subsection \ref{subse:def:RCRP}, the set $E^{T, \Upsilon, Y}$ is crucial in defining random controlled rough paths (see Definition \ref{definition:RandomControlledRP} above). Heuristically, this is the set of hyperedges of $\Upsilon$ that are not coupled with any hyperedges of $Y$ with respect to the coupling $H^T$.

Following on from Definition \ref{definition:CouplingFunctions}, we also need a way of describing how ghost hyperedges are coupled. 
\begin{definition}
\label{definition:CouplingFunctions-Ghost}
Let $n\in \bN$, let $a\in \A{n}$ and let $T_1, ..., T_n, \Upsilon_1, ..., \Upsilon_n, Y_1, ...Y_n\in \scF$. For brevity, we denote $T=\cE^a[T_1, ..., T_n]$, $\Upsilon = \cE^a[\Upsilon_1, ..., \Upsilon_n]$ and $Y = \cE^a[Y_1, ..., Y_n]$ and
$$
\tilde{H}^T: = H^T \cup Z^a[T_1, ..., T_n]. 
$$
Suppose that for all $i=1, ..., n$, 
$$
c' \Big( T_i, \Upsilon_i, Y_i \Big) >0. 
$$

Then we define $\tilde{\phi}^{T, \Upsilon, Y}:\{ \tilde{h}_j^T: j=1, ..., m[a]\} \to H^{Y} \cup H^{\Upsilon} \cup Z^a[T_1, ..., T_n]$ by
$$
\tilde{\phi}^{T, \Upsilon, Y}\big[ \tilde{h}_j^{Y} \big] = 
\begin{cases} 
\tilde{h}_j^Y \quad & \quad \mbox{if}\quad \underset{\substack{k:\\ a_k=j}}{\bigcup} h_0^{Y_k}\neq \emptyset 
\\ 
\tilde{h}_j^\Upsilon \quad & \quad \mbox{if}\quad \underset{\substack{k:\\ a_k=j}}{\bigcup} h_0^{Y_k}= \emptyset, \quad \underset{\substack{k:\\ a_k=j}}{\bigcup} h_0^{\Upsilon_k} \neq \emptyset, 
\\
\tilde{h}_j^T \quad & \quad \mbox{if}\quad \underset{\substack{k:\\ a_k=j}}{\bigcup} h_0^{Y_k} = \underset{\substack{k:\\ a_k=j}}{\bigcup} h_0^{\Upsilon_k} = \emptyset 
\end{cases}. 
$$
\end{definition}

\begin{lemma}
\label{lemma:ESets}
Let $n\in \bN$, let $a\in \A{n}$ and let $T_1, ..., T_n, \Upsilon_1, ..., \Upsilon_n, Y_1, ...Y_n\in \scF$. For brevity, we denote $T=\cE^a[T_1, ..., T_n]$, $\Upsilon = \cE^a[\Upsilon_1, ..., \Upsilon_n]$ and $Y = \cE^a[Y_1, ..., Y_n]$. Suppose that for all $i=1, ..., n$, 
$$
c' \Big( T_i, \Upsilon_i, Y_i \Big) >0. 
$$
Then
\begin{equation}
\label{eq:lemma:ESets}
E^{T, \Upsilon, Y} \cup Z^a\big[ T_1, ..., T_n \big]
= 
\bigg( \bigcup_{i=1}^n E^{T_i, \Upsilon_i, Y_i} \bigg) 
\cup \Big\{ \tilde{\phi}^{T, \Upsilon, Y}\big[ \tilde{h}_j^{Y} \big] : \tilde{h}_j^{Y} \in Z^a\big[ Y_1, ..., Y_n \big] \Big\}. 
\end{equation}
where the sets $\tilde{h}_j^{\Upsilon}$ and $Z^a[\Upsilon_1, ..., \Upsilon_n]$ were defined in Definition \ref{definition:Z-set}. 
\end{lemma}

\begin{proof}
By the supposition we have that for all $i=1, ..., n$, $N^{T_i} = N^{\Upsilon_i} \cup N^{Y_i}$, $h_0^{T_i} = h_0^{\Upsilon_i} \cup h_0^{Y_i}$ and $H^{T_i} \in H^{\Upsilon_i} \tilde{\cup} H^{Y_i}$. Thus 
\begin{align*}
&N^{\cE^a[T_1, ..., T_n]} = N^{\cE^a[\Upsilon_1, ..., \Upsilon_n]} \cup N^{\cE^a[Y_1, ..., Y_n]}, 
\quad
h_0^{\cE^a[T_1, ..., T_n]} = h_0^{\cE^a[\Upsilon_1, ..., \Upsilon_n]} \cup h_0^{\cE^a[Y_1, ..., Y_n]}, 
\\
&H^{\cE^a[T_1, ..., T_n]} \in H^{\cE^a[\Upsilon_1, ..., \Upsilon_n]} \tilde{\cup} H^{\cE^a[Y_1, ..., Y_n]} . 
\end{align*}
Further, by construction we have
\begin{align*}
H^{\cE^a[T_1, ..., T_n]}  = \Big( \bigcup_{i=1}^n H^{T_i} \Big) \cup \bigg\{ \bigcup_{\substack{k: \\a_k = i}} h_0^{T_k}: i=1, ..., m[a] \bigg\}\backslash \{\emptyset\}, 
\\
H^{\cE^a[\Upsilon_1, ..., \Upsilon_n]}  = \Big( \bigcup_{i=1}^n H^{\Upsilon_i} \Big) \cup \bigg\{ \bigcup_{\substack{k: \\a_k = i}} h_0^{\Upsilon_k}: i=1, ..., m[a] \bigg\}\backslash \{\emptyset\}. 
\end{align*}
Firstly, let $h \in H^{\Upsilon_i}$ and suppose that $\varphi^{T_i, \Upsilon_i, Y_i}[h] = h$. Then $\exists h' \in H^{T_i}$ such that $\psi^{\Upsilon, T}[h] = h'$ and $h' \cap N^{Y_i} = \emptyset$ (so that $h$ and $h'$ can be identified). Then $h \in H^{\Upsilon}$ and $h' \in H^{T}$.  By Lemma \ref{lemma:TechnicalPhiVarphi}, we have that
$$
\varphi^{T, \Upsilon, Y}[h] = h'
\quad
\mbox{and}
\quad
h' \cap N^{Y} = h' \cap N^{Y_i} = \emptyset
$$
so that $E^{T_i, \Upsilon, Y_i} \subseteq E^{T, \Upsilon, Y}$ for each $i=1, ..., n$. Next consider the sets $\tilde{h}_j^{\Upsilon}$ for $j=1, ..., m[a]$ such that $\tilde{h}_j^{\Upsilon} \notin Z^a[\Upsilon_1, ..., \Upsilon_n]$ and $\tilde{h}_j^{Y} \in Z^a[Y_1, ..., Y_n]$. Then we have that
$$
\tilde{h}_j^\Upsilon = \bigcup_{\substack{k: \\ a_k = j}} h_0^{\Upsilon_k}
\quad
\mbox{and}
\bigcup_{\substack{k: \\ a_k = j}} h_0^{Y_k} = \emptyset 
$$
so that $\tilde{\phi}^{T, \Upsilon, Y}[ \tilde{h}_j^Y] = \tilde{h}_j^{\Upsilon}$. 

Using that $h_0^{T_k} = h_0^{\Upsilon_k} \cup h_0^{Y_k}$, we get that
$$
\bigcup_{\substack{k: \\ a_k = j}} h_0^{T_k} = \bigcup_{\substack{k: \\ a_k = j}} h_0^{\Upsilon_k},  
$$
so that
$$
\varphi^{T, \Upsilon, Y}\bigg[ \bigcup_{\substack{k: \\ a_k = j}} h_0^{\Upsilon_k} \bigg] = \bigcup_{\substack{k: \\ a_k = j}} h_0^{\Upsilon_k}. 
$$
This implies that
\begin{equation}
\label{eq:lemma:ESets-1.1}
\Big\{ \tilde{\phi}^{T, \Upsilon, Y}[ \tilde{h}_j^Y]: 
\tilde{h}_j^Y \in Z^a[Y_1, ..., Y_n], 
\quad
\tilde{h}_j^\Upsilon \notin Z^a[\Upsilon_1, ..., \Upsilon_n] \Big\} \subseteq E^{T, \Upsilon, Y}. 
\end{equation}
Finally, suppose that $\tilde{h}_j^Y \in Z^a[Y_1, ..., Y_n]$ and $\tilde{h}_j^\Upsilon \in Z^a[\Upsilon_1, ..., \Upsilon_n]$. This implies that
$$
\bigcup_{\substack{k: \\ a_k = j}} h_0^{Y_k} = \bigcup_{\substack{k: \\ a_k = j}} h_0^{\Upsilon_k} = \emptyset,  
\quad \iff \quad
\bigcup_{\substack{k: \\ a_k = j}} h_0^{T_k} = \emptyset. 
$$
In this case, $\tilde{\phi}^{T, \Upsilon, Y}[ \tilde{h}_j^Y] = \tilde{h}_j^{T}$ and
\begin{equation}
\label{eq:lemma:ESets-1.2}
\Big\{ \tilde{\phi}^{T, \Upsilon, Y}[ \tilde{h}_j^Y]: 
\tilde{h}_j^Y \in Z^a[Y_1, ..., Y_n], 
\quad
\tilde{h}_j^\Upsilon \in Z^a[\Upsilon_1, ..., \Upsilon_n] \Big\} = Z^a[T_1, ..., T_n]. 
\end{equation}

For the second implication, suppose that $h\in E^{T, \Upsilon, Y}$. By disjointedness of the elements of the partition $H^\Upsilon$, either
$$
h \in \bigg( \bigcup_{i=1}^n H^{\Upsilon_i}\bigg) 
\quad
\mbox{or}
\quad
h \in \bigg\{ \tilde{h}_j^\Upsilon: j=1, ..., m[a], \quad \tilde{h}_j^\Upsilon \notin Z^a[\Upsilon_1, ..., \Upsilon_n]\bigg\}. 
$$
Suppose that $h\in H^{\Upsilon_i}$ for some choice of $i=1, ..., n$. Then an application of Lemma \ref{lemma:TechnicalPhiVarphi} gives that
$$
h =\varphi^{T, \Upsilon, Y}[h] = \varphi^{T_i, \Upsilon_i, Y_i}[h]
\quad \implies \quad
h \in E^{T_i, \Upsilon_i, Y_i}. 
$$
Alternatively, suppose that $h = \tilde{h}_j^{\Upsilon}$ for some choice of $j$ such that $\tilde{h}_j^{\Upsilon} \notin Z^a[\Upsilon_1, ..., \Upsilon_n]$. Then
\begin{align*}
&\varphi^{T,\Upsilon, Y} \Big[ \bigcup_{\substack{k \\a_k = j}} h_0^{\Upsilon_k} \Big] = \bigcup_{\substack{k \\a_k = i}} h_0^{\Upsilon_k} 
\quad \iff \quad 
\bigg( \bigcup_{\substack{k \\a_k = j}} h_0^{T_k} \bigg) \cap N^{Y} = \emptyset, 
\\
&\iff \quad \emptyset = \bigg( \bigcup_{\substack{k \\a_k = j}} h_0^{\Upsilon_k} \cup h_0^{Y_k} \bigg) \cap \bigg( \bigcup_{i=1}^n N^{Y_i} \bigg) = \bigcup_{\substack{k \\a_k = j}} h_0^{Y_k}
\quad \iff \quad
\tilde{h}_j^Y \in Z^a[Y_1, ..., Y_n]. 
\end{align*}
Thus
\begin{equation}
\label{eq:lemma:ESets-1.3}
E^{T, \Upsilon, Y} \subseteq \bigg( \bigcup_{i=1}^n E^{T_i, \Upsilon_i, Y_i} \bigg) \cup \Big\{ \tilde{\phi}^{T, \Upsilon, Y}\big[ \tilde{h}_j^{Y} \big] : \tilde{h}_j^{Y} \in Z^a\big[ Y_1, ..., Y_n \big], 
\quad
\tilde{h}_j^{\Upsilon} \in Z^a\big[ \Upsilon_1, ..., \Upsilon_n \big] \Big\}
\end{equation}
Combining Equations \eqref{eq:lemma:ESets-1.1}, \eqref{eq:lemma:ESets-1.2} and \eqref{eq:lemma:ESets-1.3} yields Equation \eqref{eq:lemma:ESets}. 
\end{proof}

\subsection{The space of random controlled rough paths}
\label{subsec:RCRPs}

The positive functional introduced in Equations  \eqref{eq:definition:RandomControlledRP3} integrate over the free variables associated to the hyperedges $H^T$ before taking the supremum. Hence, when taking limits these free variables should only lead to convergence in mean type results. By contrast, one takes supremums over $s, t\in[0,1]$ before taking expectations on the tagged free variable $\omega_0$ in Equation \eqref{eq:definition:RandomControlledRP3} so that this is very much an expectation over a graded norm on pathspace. Thus Equation \eqref{eq:definition:RandomControlledRP3} describes an almost sure convergence in the tagged variable. 

The theory of rough paths is a pathwise theory, so statements in this work will often be in an almost sure setting. However, we additionally need to establish integrability of the associated distributions so that convergence in mean will also be established. This should usually only be a verification task since
\begin{align*}
\bE^0\Big[ \big\| \bX \big\|_{\rw, \gamma, p, q, 0}(\omega_0)^{p[\rId]} \Big]^{\tfrac{1}{p[\rId]}}
\lesssim& 
\sum_{T\in \scF_0}^{\gamma-, \alpha, \beta} \bE^0\bigg[ \bE^{H^T}\bigg[ \Big| \big\langle \bX_0, T\big\rangle(\omega_0, \omega_{H^T}) \Big|^{p[T]} \bigg] \bigg]^{\tfrac{1}{p[T]}}
\\
&+\sum_{T\in\scF_0}^{\gamma-, \alpha, \beta} \bE^0\Bigg[ \sup_{s, t\in[0,1]} \frac{ \bE^{H^T}\bigg[ \Big| \big\langle \bX_{s,t}^{\sharp}, T \big\rangle (\omega_{0}, \omega_{H^T}) \Big|^{p[T]} \bigg]}{|t-s|^{p[T](\gamma - \scG_{\alpha, \beta}[T])}} \Bigg]^{\tfrac{1}{p[T]}} < \infty
\end{align*}
since $p[\rId]\leq p[T]$ (due to Equation \eqref{eq:IntegrabilityP1}). 

In order to compress notation, we will often write
\begin{align*}
\Big\| \big\langle \bX_t, T\big\rangle(\omega_0) \Big\|_p := \bE^{H^T}\bigg[ \Big| \big\langle \bX_t, T\big\rangle(\omega_0, \omega_{H^T}) \Big|^p \bigg]^{\tfrac{1}{p}},
\quad
&\Big\| \big\langle \bX_t, T\big\rangle \Big\|_p := \bE^0\bigg[ \bE^{H^T}\bigg[ \Big| \big\langle \bX_t, T\big\rangle(\omega_0, \omega_{H^T}) \Big|^p \bigg] \bigg]^{\tfrac{1}{p}}, 
\\
\Big\| \big\langle \bX, T \big\rangle(\omega_0) \Big\|_{p, \alpha}:= \sup_{s, t\in[0,1]} \frac{ \Big\| \big\langle \bX_{s, t}, T \big\rangle(\omega_0) \Big\|_{p} }{|t-s|^{\alpha} }, 
\quad
&\Big\| \big\langle \bX, T \big\rangle \Big\|_{p, \alpha}:= \bE^0\bigg[ \Big\| \big\langle \bX, T \big\rangle(\omega_0) \Big\|_{p, \alpha}^p \bigg]^{\tfrac{1}{p}} . 
\end{align*}

\begin{remark}
\label{remark:(dual)_integrab-functional}
In Definition \ref{definition:RandomControlledRP}, we see that a dual pair of integrability functionals $(p, q)$ are key to describing the integrability of the components of a random controlled rough path. 

Following on from Lemma \ref{lemma:(dual)_integrab-functional} and applying the H\"older inequality gives that
\begin{align*}
\Big\| \big\langle \bX_{s, t},\rId\big\rangle(\omega_0) \Big\|_{p[\rId]}
\leq&
\sum_{T\in\scF}^{\gamma-, \alpha, \beta} \Big\| \big\langle \bX_s, T \big\rangle(\omega_0) \Big\|_{p[T]} \cdot \Big\| \big\langle \rw_{s, t}, T \big\rangle(\omega_0) \Big\|_{q[T]}
\\
&+ \Big\| \big\langle \bX_{s, t}^{\sharp}, \rId \big\rangle(\omega_0) \Big\|_{p[\rId]}, 
\\
\Big\| \big\langle \bX_{s, t}, Y \big\rangle(\omega_0) \Big\|_{p[Y]}
\leq&
\sum_{T\in\scF}^{\gamma-, \alpha, \beta} \sum_{\Upsilon\in \scF} c'\big(T, \Upsilon, Y\big) \cdot \Big\| \big\langle \bX_s, T \big\rangle(\omega_0) \Big\|_{p[T]} \cdot \Big\| \big\langle \rw_{s, t}, \Upsilon \big\rangle(\omega_0) \Big\|_{q[\Upsilon]}
\\
&+ \Big\| \big\langle \bX_{s, t}^{\sharp}, Y \big\rangle(\omega_0) \Big\|_{p[Y]}. 
\end{align*}
These H\"older type estimates are baked into every estimate that we use in this paper. 
\end{remark}

Although it looks quite standard, the next statement is in fact a key ingredient for further investigations on the solvability of mean-field equations driven by rough paths.

\begin{theorem}
\label{theorem:Banachspace}
Let $\alpha, \beta>0$ and $\gamma:=\inf\{ \scG_{\alpha, \beta}[T]: T \in \scF, \scG_{\alpha, \beta}[T]>1-\alpha \}$. Let $(p,q)$ be a dual pair of integrability functionals and let $\rw$ be an $(\scH^{\gamma, \alpha, \beta}, p, q)$-probabilistic rough path. 

Let $\| \cdot \|_{\rw, \gamma, p, q}: \cD_{\rw}^{\gamma, p, q} \to \bR^+$ be defined by
\begin{equation}
\label{eq:theorem:Banachspace}
\big\| \bX \big\|_{\rw, \gamma, p, q} = \bE^0\Big[ \big\| \bX \big\|_{\rw, \gamma, p, q, 0}(\omega_0)^{p[\rId]} \Big]^{\tfrac{1}{p[\rId]}} . 
\end{equation}
Then $\big\| \cdot \big\|_{\rw, \gamma, p, q}$ is a norm and $\Big( \cD_{\rw}^{\gamma, p, q}, \big\| \cdot \big\|_{\rw, \gamma, p, q} \Big)$ is a Banach space over the field $\bR^e$. 
\end{theorem}

\begin{proof}
It easy to verify that for any choice of $(\scH^{\gamma, \alpha, \beta}, p, q)$-probabilistic rough path, the set $\cD_{\rw}^{\gamma, p, q}$ is a vector space over the field $\bR^e$. 

\emph{Part (i)} We verify that Equation \eqref{eq:theorem:Banachspace} is a norm: The first thing to check is that this integral is finite for any choice of $\bX \in \cD_\rw^{\gamma, p, q}$. Thanks to Equation \eqref{eq:definition:RandomControlledRP3} and Equation \eqref{eq:IntegrabilityP1},   
\begin{align*}
\bE^0\Big[ \big\| \bX &\big\|_{\rw, \gamma, p, q, 0}(\omega_0)^{p[\rId]} \Big]^{\tfrac{1}{p[\rId]}}
\\
&= \bE^0\Bigg[ \bigg| \sum_{T\in \scF_0}^{\gamma-, \alpha, \beta} \Big\| \big\langle \bX_0, T\big\rangle(\omega_0) \Big\|_{p[T]} + \Big\| \big\langle \bX^{\sharp}, T\big\rangle(\omega_0) \Big\|_{p[T], \gamma - \scG_{\alpha, \beta}[T]} \bigg|^{p[\rId]} \Bigg]^{\tfrac{1}{p[\rId]}}
\\
&\leq \sum_{T\in\scF_0}^{\gamma-, \alpha, \beta} \bE^0\bigg[ \Big\| \big\langle \bX_0, T \big\rangle(\omega_0) \Big\|_{p[T]}^{p[\rId]} \bigg]^{\tfrac{1}{p[\rId]}} + \bE^0\bigg[ \Big\| \big\langle \bX^{\sharp}, T \big\rangle(\omega_0) \Big\|_{p[T], \gamma - \scG_{\alpha, \beta}[T]}^{p[\rId]} \bigg]^{\tfrac{1}{p[\rId]}} 
\\
&\leq \sum_{T\in\scF_0}^{\gamma-, \alpha, \beta} \Big\| \big\langle \bX_0, T \big\rangle \Big\|_{p[T]} 
+
\Big\| \big\langle \bX^{\sharp}, T \big\rangle \Big\|_{p[T], \gamma - \scG_{\alpha, \beta}[T]} < \infty. 
\end{align*}
so that $\| \bX \|_{\rw, \gamma, p, q}<\infty$ for any choice of $\bX \in \cD_\rw^{\gamma, p, q}$. 

Now let $\bX \in \cD_\rw^{\gamma, p, q}$ and suppose that $\| \bX \|_{\rw, \gamma, p, q}=0$. Then $\forall T\in \scF_0^{\gamma-, \alpha, \beta}$, 
$$
\Big\| \big\langle \bX_0, T\big\rangle \Big\|_{p[T]}, 
\quad 
\Big\| \big\langle \bX^{\sharp}, T \big\rangle \Big\|_{p[T], \gamma-\scG_{\alpha, \beta}[T]} = 0. 
$$

By assumption, we have that $\forall t\in[0,1]$, 
\begin{align*}
\Big\langle \bX_{t}, \rId \Big\rangle(\omega_0) =& \Big\langle \bX_{0}, \rId \Big\rangle(\omega_0) + \Big\langle \bX_{s, t}^{\sharp}, \rId\Big\rangle(\omega_0)
\\
&+ \sum_{T\in \scF}^{\gamma-, \alpha, \beta} \bE^{H^T} \bigg[  \Big\langle \bX_{0}, T\Big\rangle(\omega_0, \omega_{H^T}) \cdot \Big\langle \rw_{0, t}, T\Big\rangle(\omega_0, \omega_{H^T}) \bigg]
\end{align*}
so that
\begin{align*}
\bE^0\bigg[ \sup_{t\in[0,1]}& \Big| \big\langle \bX_{t}, \rId \big\rangle(\omega_0)\Big|^{p[\rId]} \bigg]^{\tfrac{1}{p[\rId]}} 
\\
\leq& 
\Big\| \big\langle \bX_{0}, \rId\big\rangle\Big\|_{p[\rId]}
+
\Big\| \big\langle \bX^{\sharp}, \rId\big\rangle\Big\|_{p[\rId], \gamma}
+
\sum_{T\in \scF}^{\gamma-, \alpha, \beta}  \Big\| \big\langle \bX_{0}, T \big\rangle\Big\|_{p[T]} \cdot \Big\| \big\langle \rw, T \big\rangle\Big\|_{q[T], \scG_{\alpha, \beta}[T]} 
=0. 
\end{align*}
Similarly, for any choice of $Y\in \scF^{\gamma-, \alpha, \beta}$, 
\begin{align*}
\Big\langle \bX_{t}&, Y \Big\rangle(\omega_0, \omega_{H^Y}) 
= 
\Big\langle \bX_{0}, Y \Big\rangle(\omega_0, \omega_{H^Y})
+ 
\Big\langle \bX_{0, t}^{\sharp}, Y \Big\rangle(\omega_0, \omega_{H^Y})
\\
&+ \sum_{T\in \scF}^{\gamma-, \alpha, \beta} \sum_{\Upsilon\in \scF} c'\Big(T, \Upsilon, Y\Big) \bE^{E^{T, \Upsilon, Y}} \bigg[  \Big\langle \bX_{0}, T\Big\rangle(\omega_0, \omega_{\phi^{T, \Upsilon, Y}[H^T]}) \cdot \Big\langle \rw_{0, t}, \Upsilon \Big\rangle(\omega_0, \omega_{\varphi^{T, \Upsilon, Y}[H^\Upsilon]}) \bigg] 
\end{align*}
so that
\begin{align*}
\bE^0\bigg[ \sup_{t\in[0,1]} \bE^{H^Y}\bigg[ \Big| \big\langle \bX_{t}&, Y \big\rangle(\omega_0, \omega_{H^Y}) \Big|^{p[Y]}\bigg] \bigg]^{\tfrac{1}{p[Y]}} 
\leq 
\Big\| \big\langle \bX_0, \rId\big\rangle\Big\|_{p[Y]}
+
\Big\| \big\langle \bX^{\sharp}, \rId\big\rangle\Big\|_{p[Y], \gamma-\scG_{\alpha, \beta}[Y]}
\\
&+ 
\sum_{T\in \scF}^{\gamma-, \alpha, \beta}  \sum_{\Upsilon\in \scF} c'\Big( T, \Upsilon, Y\Big) \cdot \Big\| \big\langle \bX_{0}, T \big\rangle\Big\|_{p[T]} \cdot \Big\| \big\langle \rw, \Upsilon \big\rangle\Big\|_{q[\Upsilon], \scG_{\alpha, \beta}[\Upsilon]} 
=0. 
\end{align*}

In particular, $\forall T \in \scF_0^{\gamma-, \alpha, \beta}$, $\exists \cN^T\subseteq \Omega \times \Omega^{\times |H^T|}$ such that 
$$
(\bP \times \bP^{\times |H^T|} )\Big[ \cN^T \Big] = 1
$$
and $\forall (\omega_0, \omega_{H^T}) \in \cN^T$, $\forall t\in[0,1]$
$$
\Big\langle \bX_t, T \Big\rangle(\omega_0, \omega_{H^T}) = 0. 
$$
Thus $\| \bX \|_{\rw, \gamma, p, q} = 0$ if and only if $\bX = 0$. Finally, for any $r\in \bR^e$ and $\bX_1, \bX_2\in \cD_{\rw}^{\gamma, p, q}$
$$
\big\| r \cdot \bX \big\|_{\rw, \gamma, p, q} = |r| \cdot \big\| \bX \big\|_{\rw, \gamma, p, q}, \qquad \big\| \bX_1 + \bX_2 \big\|_{\rw, \gamma, p, q} \leq \big\| \bX_1 \big\|_{\rw, \gamma, p, q} + \big\| \bX_2 \big\|_{\rw, \gamma, p, q}. 
$$

\emph{Part (ii)} We prove completeness of the normed vector space $\Big( \cD_{\rw}^{\gamma, p, q}, \big\| \cdot \big\|_{\rw, \gamma, p, q} \Big)$. Let $( \bX^{(i)})_{i\in \bN} \in \cD_{\rw}^{\gamma, p, q}$ be an absolutely convergent sequence, so that
\begin{equation}
\label{eq:theorem:Banachspace-1.1}
\sum_{i=1}^\infty \big\| \bX^{(i)} \big\|_{\rw, \gamma, p, q}< \infty. 
\end{equation}
Let $n\in \bN$. For each $T\in \scF_0^{\gamma-, \alpha, \beta}$ and $t\in[0,1]$, we define the random variables
$$
\Big\langle \bY_{t}^{(n)}, T \Big\rangle(\omega_0, \omega_{H^T}) = \sum_{i=1}^n \Big\langle \bX_{t}^{(i)}, T \Big\rangle(\omega_0, \omega_{H^T}). 
$$
Then $\bY^{(n)}$ satisfies the identies
\begin{align*}
\Big\langle \bY_{s, t}^{(n)},& \rId \Big\rangle(\omega_0) = \sum_{T \in \scF}^{\gamma-, \alpha, \beta} \bE^{H^T} \bigg[ \Big\langle \bY_s^{(n)}, T \Big\rangle(\omega_0, \omega_{H^T}) \cdot \Big\langle \rw_{s, t}, T \Big\rangle(\omega_0, \omega_{H^T}) \bigg] + \Big\langle \bY^{(n), {\sharp}}_{s, t}, \rId \Big\rangle(\omega_0), 
\\
\Big\langle \bY_{s, t}^{(n)},& Y \Big\rangle(\omega_0, \omega_{H^Y}) 
\\
=& \sum_{T, \Upsilon \in \scF}^{\gamma-,\alpha, \beta}  c'\Big( T, \Upsilon, Y \Big) \cdot \bE^{E^{T, \Upsilon, Y}}\bigg[ \Big\langle \bY_s^{(n)}, T\Big\rangle(\omega_0, \omega_{\phi^{ T, \Upsilon, Y}[H^T]}) \cdot \Big\langle \rw_{s, t}, \Upsilon \Big\rangle(\omega_0, \omega_{\varphi^{T, \Upsilon, Y}[H^\Upsilon]}) \bigg] 
\\
&+ \Big\langle \bY^{(n), {\sharp}}_{s, t}, Y \Big\rangle(\omega_0, \omega_{H^Y}), 
\end{align*}
where, for $s, t\in[0,1]$ and $T\in \scF_0^{\gamma-, \alpha, \beta}$, we have
$$
\Big\langle \bY_{s, t}^{(n), {\sharp}}, T \Big\rangle(\omega_0, \omega_{H^T}) = \sum_{i=1}^{n} \Big\langle \bX_{s, t}^{(i),{\sharp}}, T \Big\rangle(\omega_0, \omega_{H^T}). 
$$
For each choice of $T\in\scF_0^{\gamma-, \alpha, \beta}$ the sequence of random variables
$$
\Big\langle \bY_0^{(n)}, T\Big\rangle \in L^{p[T]}\Big( \Omega \times \Omega^{\times |H^T|} , \bP \times \bP^{\times |H^T|}; \lin\big( (\bR^d)^{\otimes |N^T|}, \bR^e\big) \Big)
$$
are absolutely convergent thanks to Equation \eqref{eq:theorem:Banachspace-1.1}. Due to the Riesz-Fischer Theorem, they converge and we denote $\lim_{n\to \infty} \big\langle \bY_0^{(n)}, T\big\rangle = \big\langle \bY_0, T\big\rangle$ and
$$
\Big\langle \bY_0, T\Big\rangle = \sum_{i=1}^\infty \Big\langle \bX_0^{(i)}, T\Big\rangle. 
$$
Similarly, for each $s, t\in[0,1]$ and $T\in \scF_0^{\gamma-, \alpha, \beta}$, the sequence of random variables
$$
\Big\langle \bY_{s, t}^{\sharp}, T \Big\rangle \in L^{p[T]}\Big( \Omega \times \Omega^{\times |H^T|}, \bP \times \bP^{\times |H^T|}; \lin\big( (\bR^d)^{\otimes |N^T|}, \bR^e\big) \Big)
$$
are absolutely convergent thanks to Equation \eqref{eq:theorem:Banachspace-1.1}. Due to the Riesz-Fischer Theorem, they converge and
$$
\Big\langle \bY_{s, t}^{\sharp}, T \Big\rangle = \sum_{i=1}^{\infty} \Big\langle \bX_{s, t}^{\sharp}, T \Big\rangle. 
$$
Next, we define for $s, t\in[0,1]$ 
\begin{align*}
&\Big\langle \bY_{0, t}, \rId\Big\rangle(\omega_0) 
\\
&= \lim_{n\to \infty} \Bigg( \sum_{T\in \scF}^{\gamma-, \alpha, \beta} \bE^{H^T} \bigg[ \Big\langle \bY_0^{(n)}, T\Big\rangle(\omega_0, \omega_{H^T}) \cdot \Big\langle \rw_{0, t}, T\Big\rangle(\omega_0, \omega_{H^T}) \bigg] + \Big\langle \bY_{0, t}^{(n),{\sharp}}, \rId\Big\rangle(\omega_0) \Bigg), 
\\
&\Big\langle \bY_{s, t}, \rId\Big\rangle(\omega_0) 
\\
&= \lim_{n\to \infty} \Bigg( \sum_{T\in \scF}^{\gamma-, \alpha, \beta} \bE^{H^T} \bigg[ \bigg( \Big\langle \bY_0^{(n)} + \bY_{0, s}^{(n)}, T\Big\rangle(\omega_0, \omega_{H^T})\bigg) \cdot \Big\langle \rw_{s, t}, T\Big\rangle(\omega_0, \omega_{H^T}) \bigg] + \Big\langle \bY_{s, t}^{(n),{\sharp}}, \rId\Big\rangle(\omega_0) \Bigg), 
\end{align*}
and for each $Y\in \scF^{\gamma-, \alpha, \beta}$
\begin{align*}
&\Big\langle \bY_{0, t}, Y\Big\rangle(\omega_0, \omega_{H^Y}) 
= \lim_{n\to \infty} \Big\langle \bY_{0, t}^{(n)}, Y\Big\rangle(\omega_0, \omega_{H^Y}) 
\\
&= \lim_{n\to \infty} \Bigg( \sum_{T, \Upsilon \in \scF}^{\gamma-, \alpha, \beta} c'\Big(T, \Upsilon, Y) \cdot \bE^{E^{T, \Upsilon, Y}} \bigg[ \Big\langle \bY_0^{(n)}, T\Big\rangle(\omega_0, \omega_{\phi^{T, \Upsilon, Y}[H^T]})
\\
&\qquad \cdot \Big\langle \rw_{0, t}, \Upsilon \Big\rangle(\omega_0, \omega_{\varphi^{T, \Upsilon, Y}[H^\Upsilon]}) \bigg] + \Big\langle \bY_{0, t}^{(n),{\sharp}}, Y \Big\rangle(\omega_0, \omega_{H^Y}) \Bigg), 
\\
&\Big\langle \bY_{s, t}, Y\Big\rangle(\omega_0, \omega_{H^Y}) 
= \lim_{n\to \infty} \Big\langle \bY_{s, t}^{(n)}, Y\Big\rangle(\omega_0, \omega_{H^Y}) 
\\
&=\lim_{n\to \infty} \Bigg( \sum_{T, \Upsilon\in \scF}^{\gamma-, \alpha, \beta} c'\Big(T, \Upsilon, Y) \cdot \bE^{E^{T, \Upsilon, Y}} \bigg[ \bigg( \Big\langle \bY_0^{(n)} + \bY_{0, s}^{(n)}, T\Big\rangle(\omega_0, \omega_{\phi^{T, \Upsilon, Y}[H^T]})\bigg) 
\\
&\qquad \cdot \Big\langle \rw_{s, t}, \Upsilon \Big\rangle(\omega_0, \omega_{\varphi^{T, \Upsilon, Y}[H^\Upsilon]}) \bigg] + \Big\langle \bY_{s, t}^{(n),{\sharp}}, Y \Big\rangle(\omega_0, \omega_{H^Y}) \Bigg). 
\end{align*}
Finally, we define $\bY:[0,1] \to \scH^{\gamma-, \alpha, \beta}$ by
$$
\bY_t = \sum_{T\in\scF_0}^{\gamma-, \alpha, \beta} \bigg( \Big\langle \bY_0, T\Big\rangle + \Big\langle \bY_{0,t}, T\Big\rangle \bigg) \cdot T. 
$$
By construction, $\bY$ satisfies Equations \eqref{eq:definition:RandomControlledRP1} and \eqref{eq:definition:RandomControlledRP2}, and
$$
\bY_t = \sum_{i=1}^\infty \bX_t^{(i)}. 
$$
Next, for $T\in \scF_0^{\gamma-, \alpha, \beta}$ we consider the random variable
$$
\Big\| \big\langle \bY_{s, t}^{(n),{\sharp}}, T \big\rangle(\omega_0) \Big\|_{p[T], \gamma - \scG_{\alpha, \beta}[T]} 
\in 
L^{p[T]} \Big( \Omega, \bP, \lin\big( (\bR^d)^{\otimes |N^T|}, \bR^e \big) \Big). 
$$
These are absolutely convergent thanks to Equation \eqref{eq:theorem:Banachspace-1.1}, and a final application of the Riesz-Fischer Theorem implies that a limit exists. We also have that
$$
\Big\| \big\langle \bY^{\sharp}, T\big\rangle \Big\|_{p[T],\gamma - \scG_{\alpha, \beta}[T]} 
\leq \sum_{i=1}^\infty \Big\| \big\langle \bX^{(i),{\sharp}}, T\big\rangle \Big\|_{p[T],\gamma - \scG_{\alpha, \beta}[T]} < \infty. 
$$
Finally
\begin{align*}
&\bE^0\Bigg[ \bigg| \sup_{s, t\in[0,1]} \frac{\bE^{H^T}\bigg[ \Big| \big\langle \bY_{s, t}^{(n),{\sharp}}, T\big\rangle(\omega_0, \omega_{H^T}) \Big|^{p[T]} \bigg]^{\tfrac{1}{p[T]}}}{|t-s|^{\gamma - \scG_{\alpha, \beta}[T]}} 
\\
&\qquad - \sup_{s, t\in[0,1]} \frac{\bE^{H^T}\bigg[ \Big| \big\langle \bY_{s, t}^{\sharp}, T\big\rangle(\omega_0, \omega_{H^T}) \Big|^{p[T]} \bigg]^{\tfrac{1}{p[T]}}}{|t-s|^{\gamma - \scG_{\alpha, \beta}[T]}} \bigg|^{p[T]} \Bigg]^{\tfrac{1}{p[T]}}
\\
&\leq \Big\| \big\langle \bY^{(n),{\sharp}} - \bY^{\sharp}, T\big\rangle \Big\|_{p[T], \gamma - \scG_{\alpha, \beta}[T]} 
\leq
\sum_{i=n+1}^\infty \Big\| \big\langle \bX^{(i),{\sharp}}, T \big\rangle \Big\|_{p[T],\gamma - \scG_{\alpha, \beta}[T]}  \to 0 
\end{align*}
as $n\to \infty$. Thus $\bY\in \cD_{\rw}^{\gamma, p, q}$ and
$$
\lim_{n\to \infty} \big\| \bY - \bY^{(n)} \big\|_{\rw, \gamma, p, q} = 0
$$
Thus absolutely convergence sequences converge and the normed vector space $\Big( \cD_{\rw}^{\gamma, p, q}, \big\| \cdot \big\|_{\rw, \gamma, p, q} \Big)$ is complete. 
\end{proof}

\subsubsection{Regularity of random controlled rough paths}

The following is motivated by Definition \ref{definition:E-Set} and Definition \ref{definition:CouplingFunctions}:
\begin{definition}
Let $T, \Upsilon, \Upsilon', Y \in \scF_0$. We define $c': \scF \times \scF \times \scF \times \scF \to \bN_0$ by
$$
c'\Big( T, \Upsilon', \Upsilon, Y\Big) = \Big\langle I \tilde{\otimes} \Delta \circ \Delta\big[ T \big], \big( \Upsilon', \Upsilon, Y, H^T \big) \Big\rangle
$$
or equivalently
$$
c'\Big( T, \Upsilon', \Upsilon, Y \Big) = \sum_{T'\in \scF_0} c'\Big( T, \Upsilon', T'\Big) \cdot c'\Big( T', \Upsilon, Y\Big). 
$$
We denote
$$
H^T \cap \Big( N^{\Upsilon'} \cup N^{\Upsilon}\Big):=  \Big\{ h \cap \big( N^{\Upsilon'} \cup N^{\Upsilon} \big): h\in H^T\Big\} \backslash\{ \emptyset\}. 
$$
Then for any $T, \Upsilon', \Upsilon, Y \in \scF$ such that $c'( T, \Upsilon', \Upsilon, Y)>0$, we have that
$$
H^T \cap \Big( N^{\Upsilon'} \cup N^{\Upsilon}\Big) \in H^{\Upsilon'} \tilde{\cup} H^{\Upsilon}. 
$$
We define 
\begin{itemize}
\item $\psi^{(\Upsilon', \Upsilon), T}: H^T \cap ( N^{\Upsilon'} \cup N^{\Upsilon}) \to H^T$ by
$$
\psi^{(\Upsilon', \Upsilon), T}\big[ h \cap ( N^{\Upsilon'} \cup N^{\Upsilon}) \big] = h. 
$$
\item $\psi^{\Upsilon',(T, \Upsilon)}: H^{\Upsilon'}\to H^T \cap ( N^{\Upsilon'} \cup N^{\Upsilon})$ by
$$
\psi^{\Upsilon',(T, \Upsilon)}\big[ h^{\Upsilon'} \big] \cap N^{\Upsilon'} = h^{\Upsilon'}. 
$$
\item $\psi^{\Upsilon,(T, \Upsilon')}: H^{\Upsilon}\to H^T \cap ( N^{\Upsilon'} \cup N^{\Upsilon})$ by
$$
\psi^{\Upsilon,(T, \Upsilon')}\big[ h^{\Upsilon} \big] \cap N^{\Upsilon} = h^{\Upsilon}. 
$$
\end{itemize}
Notice that the above definition makes sense because elements of $H^T \cap (N^{\Upsilon '} \cap N^{\Upsilon})$ cannot be empty (which is part of the above definition): When, $h  \cap (N^{\Upsilon '} \cap N^{\Upsilon})$ is not empty, the knowledge of any element of the intersection is sufficient to identify the entire hyperedge 
$h$; Of course, this would be false if the intersection were empty.

We define 
\begin{align*}
\phi^{T, (\Upsilon', \Upsilon), Y}:& H^T \to \big(H^T \cap (N^{\Upsilon'} \cup N^{\Upsilon}) \big) \cup H^Y, 
\\
\varphi^{T, (\Upsilon', \Upsilon), Y}:& H^T \cap (N^{\Upsilon'} \cup N^{\Upsilon}) \to H^{Y}\cup \big( H^T \cap (N^{\Upsilon'} \cup N^{\Upsilon}) \big)
\end{align*}
as follows:
\begin{align*}
\phi^{T, (\Upsilon', \Upsilon),Y} \big[ h^T \big] =& 
\begin{cases} 
h^T & \quad \in H^T \cap (N^{\Upsilon'} \cup N^{\Upsilon}) \quad \mbox{if } h^T\cap N^Y = \emptyset, 
\\
h^T \cap N^Y \qquad  & \quad \in H^Y  \quad \mbox{if } h^T \cap N^Y \neq \emptyset. 
\end{cases}
\\
\varphi^{T, (\Upsilon',\Upsilon), Y}\big[ h^T \big] =& 
\begin{cases}
h^T & \quad \in H^T\cap(N^{\Upsilon'} \cup N^{\Upsilon}) \quad \mbox{if } \psi^{(\Upsilon', \Upsilon), T}\big[ h^T \big]\cap N^Y  = \emptyset, 
\\
\psi^{(\Upsilon, \Upsilon'), T}\big[ h^T \big] \cap N^Y & \quad \in H^Y \quad \mbox{if } \psi^{(\Upsilon', \Upsilon), T}\big[ h^T \big] \cap N^Y  \neq \emptyset. 
\end{cases}
\end{align*}

Finally, we define the set
$$
E^{T, (\Upsilon', \Upsilon), Y} = \Big\{ h^T \in H^T\cap \big( N^{\Upsilon'} \cup N^{\Upsilon}\big): \varphi^{T, (\Upsilon',\Upsilon), Y}\big[h^T \big] = h^T \Big\}. 
$$ 
\end{definition}

With these definitions at hand, we are able to restate Equation \eqref{eq:definition:RandomControlledRP2} as follows: $\forall Y\in \scF_0^{\gamma-, \alpha, \beta}$ and $\forall s, t\in[u, v]$, 
\begin{align}
\nonumber
\Big\langle& \bX_{s, t}, Y\Big\rangle(\omega_0, \omega_{H^Y}) = \Big\langle \bX_{s, t}^{\sharp}, Y\Big\rangle(\omega_, \omega_{H^Y})
\\
\nonumber
&+ \sum_{T, \Upsilon\in \scF}^{\gamma-, \alpha, \beta} c'\Big( T, \Upsilon, Y\Big) \cdot \bE^{E^{T, \Upsilon, Y}}\bigg[ \Big\langle \bX_u, T\Big\rangle(\omega_0, \omega_{\phi^{T, \Upsilon, Y}[H^T]}) \cdot \Big\langle \rw_{s, t}, \Upsilon\Big\rangle(\omega_0, \omega_{\varphi^{T, \Upsilon, Y}[H^{\Upsilon}]}) \bigg]
\\
\nonumber
&+ \sum_{T, \Upsilon\in \scF}^{\gamma-, \alpha, \beta} c'\Big( T, \Upsilon, Y\Big) \cdot \bE^{E^{T, \Upsilon, Y}}\bigg[ \Big\langle \bX_{u, s}^{\sharp}, T\Big\rangle(\omega_0, \omega_{\phi^{T, \Upsilon, Y}[H^T]}) \cdot \Big\langle \rw_{s, t}, \Upsilon\Big\rangle(\omega_0, \omega_{\varphi^{T, \Upsilon, Y}[H^{\Upsilon}]}) \bigg]
\\
\nonumber
&+\sum_{T, \Upsilon', \Upsilon \in \scF}^{\gamma-, \alpha, \beta} c'\Big( T, \Upsilon', \Upsilon, Y\Big) \cdot \bE^{E^{T, (\Upsilon', \Upsilon), Y}}\bigg[ \Big\langle \bX_u, T \Big\rangle(\omega_0, \omega_{\phi^{T, (\Upsilon', \Upsilon), Y}[H^T]}) 
\\
\label{eq:increment:identity}
&\qquad \cdot \Big\langle \rw_{u, s}, \Upsilon' \Big\rangle(\omega_0, \omega_{\varphi^{T, (\Upsilon', \Upsilon), Y}[\psi^{\Upsilon',(T,\Upsilon)}[H^{\Upsilon'}]]}) \cdot \Big\langle \rw_{s, t}, \Upsilon \Big\rangle(\omega_0, \omega_{\varphi^{T, (\Upsilon', \Upsilon), Y}[\psi^{\Upsilon,(T,\Upsilon')}[H^{\Upsilon}]]}) \bigg]
\end{align}

We now briefly highlight some regularity upper bounds well known in the field of rough paths. Due to dual integrability of the components of the rough path and the controlled rough path, some extra care must be taken but many aspects of the proofs are similar. 

\begin{proposition}
\label{prop:Regu-RCRP}
Let $\alpha, \beta>0$ and $\gamma:=\inf\{ \scG_{\alpha, \beta}[T]: T \in \scF, \scG_{\alpha, \beta}[T]>1-\alpha \}$. Let $(p,q)$ be a dual pair of integrability functionals and let $\rw$ be an $(\scH^{\gamma, \alpha, \beta}, p, q)$-probabilistic rough path. 

Let $\bX \in \cD_{\rw}^{\gamma, p, q}$. Let $u, v\in [0,1]$ and denote $|v-u|=\eta$. Then 
\begin{enumerate}
\item For $Y\in \scF_0^{\gamma-, \alpha, \beta}$, 
\begin{align}
\nonumber
\sup_{t\in[u, v]}& \Big\| \big\langle \bX_t, Y \big\rangle(\omega_0) \Big\|_{p[Y]} 
\\
\label{eq:prop:Regu-RCRP-1}
&\leq \Big\| \big\langle \bX_u, Y \big\rangle(\omega_0) \Big\|_{p[Y]} 
+ \sup_{s, t\in[u,v]} \frac{ \Big\| \big\langle \bX_{s,t}, Y \big\rangle(\omega_0) \Big\|_{p[Y]} }{|t-s|^{\alpha\wedge(\gamma - \scG_{\alpha, \beta}[Y])}} \cdot \eta^{\alpha\wedge(\gamma - \scG_{\alpha, \beta}[Y])}. 
\end{align}
\item We have that
\begin{align}
\nonumber
\sup_{s, t \in [u, v]}& \frac{\Big| \big\langle \bX_{s, t}, \rId \big\rangle(\omega_0)\Big|}{|t-s|^{\alpha}} 
\leq 
\sup_{s, t\in[u, v]} \frac{ \Big| \big\langle \bX_{s, t}^{\sharp}, \rId \big\rangle (\omega_0) \Big|}{|t-s|^{\gamma}} \cdot \eta^{\gamma - \alpha}
\\
\label{eq:prop:Regu-RCRP-2}
& +\sum_{T\in\scF}^{\gamma-, \alpha, \beta} \sup_{t\in[u, v]} \Big\| \big\langle \bX_t, T \big\rangle(\omega_0) \Big\|_{p[T]} 
\cdot 
\sup_{s,t\in[u, v]} \frac{\Big\| \big\langle \rw_{s, t}, T \big\rangle(\omega_0) \Big\|_{q[T]} }{|t-s|^{\scG_{\alpha, \beta}[T]}} \cdot \eta^{\scG_{\alpha, \beta}[T] - \alpha}. 
\end{align}
\item For $Y \in \scF^{\gamma-,\alpha, \beta}$, 
\begin{align}
\nonumber
\sup_{s, t\in[u,v]}& \frac{ \Big\| \big\langle \bX_{s,t}, Y \big\rangle(\omega_0) \Big\|_{p[Y]} }{|t-s|^{\alpha\wedge(\gamma - \scG_{\alpha, \beta}[Y])}}
\\
\nonumber
\leq& \sum_{T, \Upsilon\in \scF}^{\gamma-, \alpha, \beta} c'\Big(T, \Upsilon, Y\Big) \cdot \sup_{t\in[u,v]} \Big\| \big\langle \bX_{t}, T \big\rangle(\omega_0) \Big\|_{p[T]}
\cdot \sup_{s, t\in[u,v]} \frac{ \Big\| \big\langle \rw_{s,t}, \Upsilon \big\rangle(\omega_0) \Big\|_{q[\Upsilon]}  }{|t-s|^{\scG_{\alpha, \beta}[\Upsilon]}} \cdot \eta^{\scG_{\alpha, \beta}[\Upsilon] - \alpha} 
\\
\label{eq:prop:Regu-RCRP-3}
&+ \sup_{s, t\in[u,v]} \frac{ \Big\| \big\langle \bX_{s,t}^{\sharp}, Y \big\rangle(\omega_0) \Big\|_{p[Y]} }{|t-s|^{\gamma - \scG_{\alpha, \beta}[Y]}} \cdot  \eta^{\gamma - \scG_{\alpha, \beta}[Y] - \alpha \wedge (\gamma - \scG_{\alpha, \beta}[Y])}. 
\end{align}
\end{enumerate}
In particular, for all $Y\in \scF^{\gamma-, \alpha, \beta}$, 
\begin{align}
\nonumber
\sup_{s, t\in[u, v]}& \frac{ \Big\| \big\langle \bX_{s, t}, Y\big\rangle(\omega_0) \Big\|_{p[Y]} }{|t-s|^{\alpha \wedge (\gamma - \scG_{\alpha, \beta}[Y])}} 
\leq 
\Big\| \big\langle \bX^{\sharp}, Y\big\rangle(\omega_0) \Big\|_{p[Y], \alpha \wedge (\gamma - \scG_{\alpha, \beta}[Y])}  \cdot \eta^{\gamma - \scG_{\alpha, \beta}[Y] - \alpha\wedge(\gamma - \scG_{\alpha, \beta}[Y])} 
\\
\nonumber
&+ \sum_{T, \Upsilon\in \scF }^{\gamma-, \alpha, \beta} c'\Big( T, \Upsilon, Y\Big) \cdot \Big\| \big\langle \bX_u, T\big\rangle(\omega_0) \Big\|_{p[T]} 
\cdot 
\Big\| \big\langle \rw, \Upsilon \big\rangle(\omega_0) \Big\|_{q[\Upsilon], \scG_{\alpha, \beta}[\Upsilon]} 
\cdot 
\eta^{\scG_{\alpha, \beta}[\Upsilon] - \alpha\wedge(\gamma - \scG_{\alpha, \beta}[Y])}
\\
\nonumber
&+ \sum_{T, \Upsilon\in \scF}^{\gamma-, \alpha, \beta} c'\Big( T, \Upsilon, Y\Big) \cdot \Big\| \big\langle \bX^{\sharp}, T\big\rangle(\omega_0) \Big\|_{p[T], \gamma - \scG_{\alpha, \beta}[T]}  
\\
\nonumber
&\qquad \qquad\cdot \Big\| \big\langle \rw, \Upsilon \big\rangle(\omega_0) \Big\|_{q[\Upsilon], \scG_{\alpha, \beta}[\Upsilon]}
\cdot 
\eta^{\gamma - \scG_{\alpha, \beta}[Y] - \alpha\wedge(\gamma - \scG_{\alpha, \beta}[Y])}
\\
\nonumber
&+ \sum_{T, \Upsilon', \Upsilon \in \scF}^{\gamma-, \alpha, \beta} c'\Big( T, \Upsilon', \Upsilon, Y\Big) 
\cdot
\Big\| \big\langle \bX_u, T\big\rangle(\omega_0) \Big\|_{p[T]} 
\cdot
\Big\| \big\langle \rw, \Upsilon' \big\rangle(\omega_0) \Big\|_{q[\Upsilon'], \scG_{\alpha, \beta}[\Upsilon']} 
\\
\label{eq:prop:Regu-RCRP-4}
&\qquad \qquad \cdot
\Big\| \big\langle \rw, \Upsilon \big\rangle(\omega_0) \Big\|_{q[\Upsilon], \scG_{\alpha, \beta}[\Upsilon]} \cdot \eta^{\scG_{\alpha, \beta}[\Upsilon'] + \scG_{\alpha, \beta}[\Upsilon] - \alpha\wedge(\gamma-\scG_{\alpha, \beta}[Y])}. 
\end{align}
and (recalling Equation \eqref{eq:Jet_Operator})
\begin{align}
\nonumber
\sup_{s, t\in[u, v]}& \frac{ \Big\| \big\langle \fJ[\bX]_{s, t}, Y\big\rangle(\omega_0) \Big\|_{p[Y]} }{|t-s|^{\alpha \wedge (\gamma - \scG_{\alpha, \beta}[Y])}}
\\
\nonumber
&+ \sum_{T, \Upsilon\in \scF }^{\gamma-, \alpha, \beta} c'\Big( T, \Upsilon, Y\Big) \cdot \Big\| \big\langle \bX_u, T\big\rangle(\omega_0) \Big\|_{p[T]} 
\cdot 
\Big\| \big\langle \rw, \Upsilon \big\rangle(\omega_0) \Big\|_{q[\Upsilon], \scG_{\alpha, \beta}[\Upsilon]} 
\cdot 
\eta^{\scG_{\alpha, \beta}[\Upsilon] - \alpha\wedge(\gamma - \scG_{\alpha, \beta}[Y])}
\\
\nonumber
&+ \sum_{T, \Upsilon\in \scF}^{\gamma-, \alpha, \beta} c'\Big( T, \Upsilon, Y\Big) \cdot \Big\| \big\langle \bX^{\sharp}, T\big\rangle(\omega_0) \Big\|_{p[T], \gamma - \scG_{\alpha, \beta}[T]}  
\\
\nonumber
&\qquad \qquad \cdot \Big\| \big\langle \rw, \Upsilon \big\rangle(\omega_0) \Big\|_{q[\Upsilon], \scG_{\alpha, \beta}[\Upsilon]}
\cdot 
\eta^{\gamma - \scG_{\alpha, \beta}[Y] - \alpha\wedge(\gamma - \scG_{\alpha, \beta}[Y])}
\\
\nonumber
&+ \sum_{T, \Upsilon', \Upsilon \in \scF}^{\gamma-, \alpha, \beta} c'\Big( T, \Upsilon', \Upsilon, Y\Big) 
\cdot
\Big\| \big\langle \bX_u, T\big\rangle(\omega_0) \Big\|_{p[T]} 
\cdot
\Big\| \big\langle \rw, \Upsilon' \big\rangle(\omega_0) \Big\|_{q[\Upsilon'], \scG_{\alpha, \beta}[\Upsilon']} 
\\
\label{eq:prop:Regu-RCRP-5}
&\qquad \qquad \cdot
\Big\| \big\langle \rw, \Upsilon \big\rangle(\omega_0) \Big\|_{q[\Upsilon], \scG_{\alpha, \beta}[\Upsilon]} \cdot \eta^{\scG_{\alpha, \beta}[\Upsilon'] + \scG_{\alpha, \beta}[\Upsilon] - \alpha\wedge(\gamma-\scG_{\alpha, \beta}[Y])}. 
\end{align}

Finally, we have that for all $Y\in \scF_0^{\gamma-, \alpha, \beta}$ that
\begin{equation}
\label{eq:prop:Regu-RCRP-6}
\bE^0 \Bigg[ \sup_{s, t\in[u, v]} \frac{ \Big\| \big\langle \bX_{s, t}, Y\big\rangle(\omega_0) \Big\|_{p[Y]}^{p[Y]} }{|t-s|^{p[Y](\alpha \wedge (\gamma - \scG_{\alpha, \beta}[Y]))}} \Bigg], 
\quad
\bE^0 \Bigg[ \sup_{s, t\in[u, v]} \frac{ \Big\| \big\langle \fJ[\bX]_{s, t}, Y\big\rangle(\omega_0) \Big\|_{p[Y]}^{p[Y]} }{|t-s|^{p[Y](\alpha \wedge (\gamma - \scG_{\alpha, \beta}[Y]))}} \Bigg] < \infty. 
\end{equation}
\end{proposition}

\begin{proof}
The proof Equation \eqref{eq:prop:Regu-RCRP-1} is standard. Equation \eqref{eq:prop:Regu-RCRP-2} follows from Equation \eqref{eq:definition:RandomControlledRP1} and the H\"older inequality. Similarly, Equation \eqref{eq:prop:Regu-RCRP-3} follows from Equation \eqref{eq:definition:RandomControlledRP2} and the H\"older inequality. 

Equation \eqref{eq:prop:Regu-RCRP-4} follows from Equation \eqref{eq:increment:identity} and the H\"older inequality. Equation \eqref{eq:prop:Regu-RCRP-5} comes similarly thanks to Equation \eqref{eq:Jet_Operator}. 

Equation \eqref{eq:prop:Regu-RCRP-6} follows by integrating over the tagged probability space and using Equation \eqref{eq:definition:RandomControlledRP4}. 
\end{proof}

\subsection{Proof of the results in Section \ref{subsec:Operat_RCRPs}}
\label{subsection:Proofs-Sec4}

The proofs contained in this subsection are, for the most part, novel adaptions of methods that have been well established in the rough path literature to the probabilistic setting. As such, a reader not well read on classical results relating to rough paths may find that some important details of proofs have been skipped for conciseness. On the other hand, a reader familiar with rough path techniques but not with the probabilistic framework developed in \cite{salkeld2021Probabilistic} and this work may not recognise how the probabilistic framework convolutes these proofs with only a fleeting glance at them. 

This technical Lemma will by used a number of times. 
\begin{lemma}
\label{lemma:TechLem_Products}
Let $V$ be a vector space. Let $n\in \bN$ and for $i=1, ..., n$, let $x_i, y_i \in V$. Then
$$
\bigotimes_{i=1}^n (x_i+y_i) - \bigotimes_{i=1}^n x_i = \sum_{k=1}^n \bigg( \bigotimes_{i=1}^{k-1} x_i \bigg) \otimes y_k \otimes \bigg( \bigotimes_{i=k+1}^n (x_i+y_i) \bigg). 
$$
\end{lemma}

\begin{proof}
By a telescoping summation, 
\begin{align*}
\bigotimes_{i=1}^n (x_i + y_i) - \bigotimes_{i=1}^n x_i =& \sum_{k=1}^{n} \Bigg[ \bigg( \bigotimes_{i=1}^{k-1} x_i \otimes \bigotimes_{i=k}^{n} (x_i+y_i) \bigg) - \bigg( \bigotimes_{i=1}^{k} x_i \otimes \bigotimes_{i=k+1}^{n} (x_i + y_i) \bigg) \Bigg]
\\
=& \sum_{k=1}^{n} \bigg( \bigotimes_{i=1}^{k-1} x_i \bigg) \otimes (x_k+y_k) \otimes \bigg( \bigotimes_{i=k+1}^n (x_i + y_i) \bigg)
\end{align*}
\end{proof}

\subsubsection{Proof of Theorem \ref{theorem:Reconstruction}}
\label{subsec:ProofReconstruction}

This first Proposition will allow us to prove the existence of the limit described in Equation \eqref{eq:theorem:Reconstruction} using an easy application of the Sewing lemma (see for instance \cite{frizhairer2014}*{Lemma 4.2}). 

\begin{proposition}
\label{proposition:Reconstruction}
Let $\alpha, \beta>0$ and let $\gamma:=\inf\{ \scG_{\alpha, \beta}[T]: T \in \scF, \scG_{\alpha, \beta}[T]>1-\alpha \}$. Let $(p, q)$ be a pair of dual integrability functionals and suppose $\exists r>1$ such that, for any $i\in\{1, ..., d\}$,
\begin{equation}
\label{eq:proposition:Reconstruction-pq}
\frac{1}{r} := \sup_{T\in \scF_0^{\gamma-\alpha, \alpha, \beta}} \Big( \frac{1}{p_x[T]} + \frac{1}{q\big[ \lfloor T \rfloor_i \big]} \Big)< 1. 
\end{equation}

Let $i \in\{1, ..., d\}$, let $\rw$ be a $(\scH^{\gamma, \alpha, \beta}, p, q)$- probabilistic rough path and let $\bX \in \cD_{\rw}^{\gamma, p, q}$. Define $\Xi:[0,1]^2 \to L^r \big( \Omega, \bP; \bR^e \big)$ by
\begin{align*}
\Xi_{s, t}(\omega_0):=& \sum_{T \in \scF_0}^{\gamma-\alpha,\alpha, \beta} \bE^{H^T} \bigg[ \Big\langle \bX_s, T \Big\rangle(\omega_0, \omega_{H^T}) \cdot \Big\langle \rw_{s, t}, \lfloor T \rfloor_i \Big\rangle (\omega_0, \omega_{H^T}) \bigg], 
\\
\Xi_{s, t, u}(\omega_0) :=& \Xi_{s, u}(\omega_0) - \Xi_{s,t}(\omega_0) - \Xi_{t, u}(\omega_0). 
\end{align*}
Then 
\begin{equation}
\label{eq:proposition:Reconstruction}
\bP\bigg[ \sup_{s, t, u} \frac{ \Xi_{s, t, u}(\omega_0) }{|u-s|^{\gamma+\alpha}} < \infty \bigg] = 1, 
\quad \mbox{and} \quad 
\bE^0\bigg[ \sup_{s, t, u} \frac{ \big| \Xi_{s, t, u}(\omega_0) \big|^r }{|u-s|^{r(\gamma+\alpha)}}  \bigg] < \infty. 
\end{equation}
\end{proposition}

\begin{proof}
Using the coproduct identity and relabelling appropriately,
\begin{align*}
\Xi_{s, t, u}(\omega_0) =& \sum_{T \in \scF_0}^{\gamma-\alpha,\alpha, \beta} \bE^{H^T}\bigg[ \Big\langle \bX_s, T \Big\rangle (\omega_0, \omega_{H^T}) \cdot \Big\langle \rw_{s,t} \otimes^{H^T} \rw_{t, u}, \Delta'\big[ \lfloor T \rfloor_i \big] \Big\rangle (\omega_0, \omega_{H^T}) \bigg]
\\
& - \sum_{T \in \scF_0}^{\gamma-\alpha,\alpha, \beta} \bE^{H^T} \bigg[ \Big\langle \bX_{s, t}, T \Big\rangle (\omega_0, \omega_{H^T}) \cdot \Big\langle \rw_{t, u}, \lfloor T \rfloor_i \Big\rangle (\omega_0, \omega_{H^T}) \bigg]
\\
=& \sum_{T \in \scF_0}^{\gamma-\alpha,\alpha, \beta} \bE^{H^T}\bigg[ \Big\langle \bX_s, T \Big\rangle (\omega_0, \omega_{H^T}) \cdot \sum_{\Upsilon \in \scF} \sum_{Y \in \scF} c'\Big( \lfloor T\rfloor_i, \Upsilon, \lfloor Y \rfloor_i \Big) 
\\
& \qquad \cdot \Big\langle \rw_{s,t}, \Upsilon \Big\rangle (\omega_0, \omega_{\psi^{\Upsilon, T}[H^\Upsilon]}) \otimes \Big\langle \rw_{t, u}, \lfloor Y \rfloor_i \Big\rangle (\omega_0, \omega_{\psi^{Y, T}[H^Y]}) \bigg]
\\
& - \sum_{Y \in \scF_0}^{\gamma-\alpha,\alpha, \beta} \bE^{H^Y} \bigg[ \Big\langle \bX_{s, t}, Y \Big\rangle (\omega_0, \omega_{H^Y}) \cdot \Big\langle \rw_{t, u}, \lfloor Y \rfloor_i \Big\rangle (\omega_0, \omega_{H^Y}) \bigg]
\\
=&\sum_{Y\in \scF_0}^{\gamma-\alpha,\alpha, \beta} \bE^{H^Y}\Bigg[ \bigg( \bigg( \sum_{\substack{T \in \scF \\ \Upsilon \in \scF}}^{\gamma-, \alpha, \beta} c'\Big( \lfloor T \rfloor_i, \Upsilon, \lfloor Y\rfloor_i \Big) \cdot \bE^{E^{T, \Upsilon, Y}} \bigg[ \Big\langle \bX_s, T \Big\rangle (\omega_0, \omega_{\phi^{T, \Upsilon, Y}[H^T]}) 
\\
&\quad \cdot \Big\langle \rw_{s, t}, \Upsilon \Big\rangle (\omega_0, \omega_{\varphi^{T, \Upsilon, Y}[H^\Upsilon]}) \bigg] \bigg) - \Big\langle \bX_{s, t}, Y \Big\rangle (\omega_0, \omega_{H^Y}) \bigg) \cdot \Big\langle \rw_{t, u} \lfloor Y \rfloor_i \Big\rangle (\omega_0, \omega_{H^Y}) \Bigg]. 
\end{align*}
Notice that to get the last expression, we used Fubini's theorem by distinguishing in $H^T$ between the hyperedges that are connected to $Y$ and those that are not.

Now using the Equation \eqref{eq:lemma:CoprodCountIdent2} and applying Definition \ref{definition:RandomControlledRP} gives
$$
\Xi_{s, t, u}(\omega_0) = - \sum_{Y \in \scF_0}^{\gamma-\alpha, \alpha, \beta} \bE^{H^Y}\bigg[ \Big\langle \bX_{s, t}^{\sharp}, Y \Big\rangle (\omega_0, \omega_{H^Y}) \cdot \Big\langle \rw_{t, u}, \lfloor Y\rfloor_i \Big\rangle (\omega_0, \omega_{H^Y}) \bigg]. 
$$
An application of the H\"older inequality via Equation \eqref{eq:proposition:Reconstruction1} gives
\begin{align}
\nonumber
\bP\bigg[& \sup_{s, t, u\in[0,1]} \frac{\big| \Xi_{s, t, u}(\omega_0) \big|}{|u-s|^{\gamma+\alpha}}< \infty \bigg]
\\ 
\label{eq:proposition:Reconstruction.1}
&\geq \bP\Bigg[ \sum_{Y \in \scF_0}^{\gamma-\alpha, \alpha, \beta} \Big\| \big\langle \rw, \lfloor Y\rfloor_i \big\rangle(\omega_{0}) \Big\|_{q[\lfloor Y \rfloor_i], \scG_{\alpha, \beta}[Y]+\alpha}
\cdot 
\Big\| \big\langle \bX^{\sharp}, Y \big\rangle (\omega_{0}) \Big\|_{p[Y], \gamma - \scG_{\alpha, \beta}[Y]} < \infty \Bigg]=1. 
\end{align}

Further, thanks to Equation \eqref{eq:proposition:Reconstruction-pq} and the H\"older inequality again
\begin{align*}
\bE^0\bigg[& \sup_{s, t, u\in[0,1]} \frac{\big| \Xi_{s, t, u}(\omega_0) \big|^r}{|u-s|^{r(\gamma+\alpha)}} \bigg]^{\tfrac{1}{r}} 
\\
\lesssim &
\sum_{Y\in \scF_0}^{\gamma-\alpha, \alpha, \beta} 
\bE^0\bigg[ \Big\| \big\langle \rw, \lfloor Y\rfloor_i \big\rangle(\omega_0) \Big\|_{q[\lfloor Y \rfloor_i], \scG_{\alpha, \beta}[Y]+\alpha}^r
\cdot
\Big\| \big\langle \bX^{\sharp}, Y \big\rangle(\omega_0) \Big\|_{p[Y], \gamma - \scG_{\alpha, \beta}[Y]}^r \bigg]^{\tfrac{1}{r}}
\\
\lesssim &
\sum_{Y\in \scF_0}^{\gamma-\alpha, \alpha, \beta} 
\Big\| \big\langle \rw, \lfloor Y\rfloor_i \big\rangle \Big\|_{q[\lfloor Y \rfloor_i], \scG_{\alpha, \beta}[Y]+\alpha}
\cdot
\Big\| \big\langle \bX^{\sharp}, Y \big\rangle \Big\|_{p[Y], \gamma - \scG_{\alpha, \beta}[Y]}
< \infty. 
\end{align*}
\end{proof}

\begin{proof}[Proof of Theorem \ref{theorem:Reconstruction}]
First of all, we verify that $p_y:\scF_0^{\gamma, \alpha, \beta}\to (1, \infty)$ as defined in Equation \eqref{eq:proposition:Reconstruction1} paired with $q$ form a pair of dual integrability functionals. By construction, $p_y$ satisfies Equation \eqref{eq:IntegrabilityP1} so that $(p_y, q)$ is a pair of dual integrability functionals. 

Let $D_n$ be a sequence of partitions of $[0,1]$ such that the mesh size $|D_n|\to 0$ as $n\to \infty$. By the Sewing Lemma (see for instance \cite{frizhairer2014}*{Lemma 4.2}) and Equation \eqref{eq:proposition:Reconstruction}, the limit in Equation \eqref{eq:theorem:Reconstruction} exists and we define
$$
\int_0^1 X_r dW_r(\omega_0) := \lim_{|D_n| \to 0} \sum_{[s, t]\in D_n} \Xi_{s, t}(\omega_0), 
$$ 
where the limit is $\bP$-almost surely. Similarly, for any $u, v\in[0,1]$ and $D_n^{[u, v]}$ a sequence of partitions of $[u, v]$ such that the mesh size $|D_n^{[u, v]}| \to 0$ as $n\to \infty$, we can define
$$
\int_u^v X_r dW_r(\omega_0):= \lim_{|D_n^{[u, v]}|\to \infty} \sum_{[s, t]\in D_n^{[u, v]}} \Xi_{s, t}(\omega_0)
$$
and obtain that $\forall u ,v, w\in[0,1]$
$$
\int_u^v X_r dW_r(\omega_0) + \int_v^w X_r dW_r(\omega_0) = \int_u^w X_r dW_r(\omega_0). 
$$
Further, a dominated convergence type result ensures this also converges in $L^{p_y[\rId]}\big( \Omega, \bP; \bR^e \big)$. 

Finally, we have that there exists a constant $C>0$ dependent on $\alpha, \gamma+\alpha$ and the random variable
$$
\big\| \Xi^\flat(\omega_0) \big\|_{\gamma+\alpha}: = \sup_{s, t, u\in[0,1]} \frac{ \big\| \Xi_{s, t, u}(\omega_0) \big\| }{|u-s|^{\gamma+\alpha}}
$$
such that
$$
\Big| \int_s^t X_r dW_r(\omega_0) - \Xi_{s, t}(\omega_0) \Big| \leq C \big\| \Xi^\flat(\omega_0) \big\|_{\gamma+\alpha} \cdot  \big| t-s \big|^{\gamma+\alpha} . 
$$
However, by applying Equation \eqref{eq:proposition:Reconstruction.1} to this, we have proved Equation \eqref{eq:theorem:Reconstruction:Inequality}. 

Using the construction from Equation \eqref{eq:theorem:Reconstruction:Phi} and applying Equation \eqref{eq:theorem:Reconstruction:Inequality}, we have that for $s, t\in [u, v]$, 
\begin{align}
\nonumber
&\bigg| \Big\langle \Phi[ \bX]_{s, t}^{\sharp}, \rId \Big\rangle(\omega_0) \bigg|
\\
\nonumber
&= \bigg| \Big\langle \Phi[ \bX]_{s, t}, \rId \Big\rangle(\omega_0) - \sum_{T\in \scF}^{\gamma-, \alpha, \beta} \bE^{H^T} \bigg[ \Big\langle \Phi[\bX]_s, T\Big\rangle(\omega_0, \omega_{H^T}) \cdot \Big\langle \rw_{s, t}, T \Big\rangle(\omega_0, \omega_{H^T}) \bigg] \bigg|
\\
\nonumber
&\leq C \Bigg( \sum_{T \in \scF_0}^{\gamma-\alpha, \alpha, \beta} \sup_{r, s, t\in[u,v] } \frac{\bE^{H^T}\bigg[ \Big| \big\langle \bX_{r, s}^{\sharp}, T \big\rangle (\omega_{0}, \omega_{H^T}) \Big| \cdot \Big| \big\langle \rw_{s, t}, \lfloor T\rfloor \big\rangle(\omega_{0}, \omega_{H^T}) \Big| \bigg] }{|s-r|^{\gamma - \scG_{\alpha, \beta}[T]} \cdot |t-s|^{\scG_{\alpha, \beta}[T]+\alpha}} \Bigg) \cdot |v-u|^{\gamma + \alpha} 
\\
\label{eq:norm-Phi-1.1}
&+ \Bigg( \sum_{\substack{T\in \scF \\ \scG_{\alpha, \beta}[ \lfloor T \rfloor ] =\gamma}} \sup_{s, t\in[u, v]}
\frac{\bE^{H^T} \bigg[ \Big| \Big\langle \bX_s, T \Big\rangle(\omega_0, \omega_{H^T}) \Big| \cdot \Big| \Big\langle \rw_{s, t}, \big\lfloor T\big\rfloor \Big\rangle(\omega_0, \omega_{H^T}) \Big| \bigg]}{|t-s|^{\scG_{\alpha, \beta}[\lfloor T \rfloor]}} \Bigg) \cdot |v-u|^{\gamma}. 
\end{align}
Thanks to Equation \eqref{eq:proposition:Reconstruction1}, we have that for any $[u, v]\subseteq [0,1]$, 
\begin{align*}
\sup_{s, t\in[u,v]}& \frac{ \Big| \big\langle \Phi[ \bX]_{s, t}^{\sharp}, \rId \big\rangle(\omega_0) \Big|}{|t-s|^{\gamma}} 
\\
\leq& C \sum_{T\in \scF_0}^{\gamma - \alpha, \alpha, \beta} \Big\| \big\langle \bX^{\sharp}, T\big\rangle(\omega_0) \Big\|_{p_x[T],\gamma - \scG_{\alpha, \beta}[T]}
\cdot 
\Big\| \big\langle \rw, \lfloor T \rfloor \big\rangle(\omega_0) \Big\|_{q[\lfloor T\rfloor], \scG_{\alpha, \beta}[\lfloor T \rfloor]} \cdot |v-u|^{\alpha}
\\
&+ \sum_{\substack{T\in \scF \\ \scG_{\alpha, \beta}[ \lfloor T \rfloor ] = \gamma}} \sup_{s\in[u, v]} \Big\| \big\langle \bX_s, T\big\rangle(\omega_0) \Big\|_{p_x[T]}  \cdot \Big\| \big\langle \rw, \lfloor T \rfloor \big\rangle(\omega_0) \Big\|_{q[\lfloor T \rfloor ],\scG_{\alpha, \beta}[\lfloor T \rfloor]}
\end{align*}
and for $T\in \scF$ such that $\scG_{\alpha, \beta}[T]=\gamma-\alpha$, an application of Proposition \ref{prop:Regu-RCRP} yields
\begin{equation}
\label{eq:norm-Phi-1.6}
\sup_{s\in[u, v]} \Big\| \big\langle \bX_s, T\big\rangle(\omega_0) \Big\|_{p_x[T]} \leq \Big\| \big\langle \bX_u, T\big\rangle(\omega_0) \Big\|_{p_x[T]} + \Big\| \big\langle \bX^{\sharp}, T\big\rangle(\omega_0) \Big\|_{p_x[T], \alpha} \cdot |v-u|^{\alpha}. 
\end{equation}
Thus we have that $(\Omega_0, \bP)$-almost surely
\begin{align}
\nonumber
\sup_{s, t\in[u,v]}& \frac{ \Big| \big\langle \Phi[ \bX]_{s, t}^{\sharp}, \rId \big\rangle(\omega_0) \Big|}{|t-s|^{\gamma}} 
\\
\nonumber
\leq& C \sum_{T\in \scF_0}^{\gamma - \alpha, \alpha, \beta} \Big\| \big\langle \bX^{\sharp}, T\big\rangle(\omega_0) \Big\|_{p_x[T],\gamma - \scG_{\alpha, \beta}[T]}
\cdot 
\Big\| \big\langle \rw, \lfloor T \rfloor \big\rangle(\omega_0) \Big\|_{q[\lfloor T\rfloor], \scG_{\alpha, \beta}[\lfloor T \rfloor]} \cdot |v-u|^{\alpha}
\\
\label{eq:norm-Phi-1.2}
&+ \sum_{\substack{T\in \scF \\ \scG_{\alpha, \beta}[ \lfloor T \rfloor ] = \gamma}} \Big\| \big\langle \bX_u, T\big\rangle(\omega_0) \Big\|_{p_x[T]}  \cdot \Big\| \big\langle \rw, \lfloor T \rfloor \big\rangle(\omega_0) \Big\|_{q[\lfloor T \rfloor ],\scG_{\alpha, \beta}[\lfloor T \rfloor]}. 
\end{align}
Similarly, thanks to Equation \eqref{eq:proposition:Reconstruction1} and the H\"older inequality we also have that
\begin{align}
\nonumber
\bE^0\Bigg[& \sup_{s, t\in [u,v]} \frac{\Big| \big\langle \Phi[\bX]_{s, t}^{\sharp}, \rId\big\rangle(\omega_0)\Big|^{p_y[\rId]}}{|t-s|^{p_y[\rId] \gamma}} \Bigg]^{\tfrac{1}{p_y[\rId]}}
\\
\nonumber
\lesssim& \sum_{T\in\scF_0}^{\gamma-\alpha, \alpha, \beta}\bE^0\bigg[ \Big\| \big\langle \bX^{\sharp}, T\big\rangle(\omega_0) \Big\|_{p_x[T],\gamma - \scG_{\alpha, \beta}[T]}^{p_y[\rId]}
\cdot 
\Big\| \big\langle \rw, \lfloor T \rfloor \big\rangle(\omega_0) \Big\|_{q[\lfloor T\rfloor], \scG_{\alpha, \beta}[\lfloor T \rfloor]}^{p_y[\rId]}\bigg]^{\tfrac{1}{p_y[\rId]}} \cdot |v-u|^{\alpha}
\\
\nonumber
&+\sum_{\substack{T\in \scF \\ \scG_{\alpha, \beta}[ \lfloor T \rfloor ] = \gamma}} \bE^0\bigg[ \Big\| \big\langle \bX_u, T\big\rangle(\omega_0) \Big\|_{p_x[T]}^{p_y[\rId]}
\cdot \Big\| \big\langle \rw, \lfloor T \rfloor \big\rangle(\omega_0) \Big\|_{q[\lfloor T \rfloor ],\scG_{\alpha, \beta}[\lfloor T \rfloor]}^{p_y[\rId]} \bigg]^{\tfrac{1}{p_y[\rId]}}
\\
\nonumber
\lesssim& \sum_{T\in\scF_0}^{\gamma-\alpha, \alpha, \beta} \Big\| \big\langle \bX^{\sharp}, T\big\rangle \Big\|_{p_x[T],\gamma - \scG_{\alpha, \beta}[T]}
\cdot 
\Big\| \big\langle \rw, \lfloor T \rfloor \big\rangle \Big\|_{q[\lfloor T\rfloor], \scG_{\alpha, \beta}[\lfloor T \rfloor]} \cdot |v-u|^{\alpha}
\\
\label{eq:norm-Phi-1.2-integrated}
&+ \sum_{\substack{T\in \scF \\ \scG_{\alpha, \beta}[ \lfloor T \rfloor ] = \gamma}} \Big\| \big\langle \bX_u, T\big\rangle \Big\|_{p_x[T]} 
\cdot 
\Big\| \big\langle \rw, \lfloor T \rfloor \big\rangle \Big\|_{q[\lfloor T \rfloor ],\scG_{\alpha, \beta}[\lfloor T \rfloor]}. 
\end{align}

Arguing in the same fashion, we also get 
\begin{align}
\nonumber
\Big\langle \Phi&[\bX]_{s, t}^{\sharp}, \lfloor Y \rfloor \Big\rangle (\omega_0, \omega_{H^Y} ) = \Big\langle \bX_{s, t}^{\sharp}, Y \Big\rangle (\omega_0, \omega_{H^Y} )
\\
\nonumber
&+ \sum_{\substack{T\in\scF, \Upsilon\in\scF \\ \scG_{\alpha, \beta}[T]\in [\gamma-\alpha, \gamma)}} c'\Big( T, \Upsilon, Y\Big) \cdot \bE^{E^{T, \Upsilon, Y}}\bigg[ \Big\langle \bX_s, T\Big\rangle(\omega_0, \omega_{\phi^{T, \Upsilon, Y}[H^T]})
\\
\label{eq:norm-Phi-1.3}
&\qquad \cdot 
\Big\langle \rw_{s,t}, \Upsilon \Big\rangle (\omega_0, \omega_{\varphi^{T, \Upsilon, Y}[H^\Upsilon]}) \bigg] 
\end{align}
so that an application of Equation \eqref{eq:proposition:Reconstruction1} along with Lemma \ref{lemma:(dual)_integrab-functional}, and Equation \eqref{eq:norm-Phi-1.6} yields that $(\Omega_0, \bP)$-almost surely
\begin{align}
\nonumber
\sup_{s,t\in[u,v]}& \frac{ \Big\| \big\langle \Phi[ \bX]_{s, t}^{\sharp}, \lfloor Y \rfloor \big\rangle(\omega_0) \Big\|_{p_y[ \lfloor Y \rfloor ]}  }{|t-s|^{\gamma - \scG_{\alpha, \beta}[\lfloor Y \rfloor] }} 
\leq
\Big\| \big\langle \bX^{\sharp}, Y \big\rangle(\omega_0) \Big\|_{p_x[Y],\gamma - \scG_{\alpha, \beta}[Y]} \cdot |v-u|^{\alpha}
\\
\nonumber
+& \sum_{\substack{T\in\scF, \Upsilon\in\scF \\ \scG_{\alpha, \beta}[T]\in [\gamma-\alpha, \gamma)}} c'\Big( T, \Upsilon, Y\Big) 
\cdot 
\bigg( \Big\| \big\langle \bX^{\sharp}, T\big\rangle(\omega_0) \Big\|_{p_x[T], \gamma-\scG_{\alpha, \beta}[T]} 
\cdot 
\Big\| \big\langle \rw, \Upsilon \big\rangle(\omega_0) \Big\|_{q[\Upsilon],\scG_{\alpha, \beta}[\Upsilon]}
\cdot 
|v-u|^{\alpha} 
\\
\label{eq:norm-Phi-1.4}
&\quad +
\Big\| \big\langle \bX_u, T\big\rangle(\omega_0) \Big\|_{p_x[T]}
\cdot 
\Big\| \big\langle \rw, \Upsilon \big\rangle(\omega_0) \Big\|_{q[\Upsilon],\scG_{\alpha, \beta}[\Upsilon]} \cdot |v-u|^{\scG_{\alpha, \beta}[T] - (\gamma - \alpha)} \bigg). 
\end{align}
In the same fashion as before, integrating through gives
\begin{align}
\nonumber
\bE^0\Bigg[& \sup_{s,t\in[u,v]} \frac{ \Big\| \big\langle \Phi[ \bX]_{s, t}^{\sharp}, \lfloor Y \rfloor \big\rangle(\omega_0) \Big\|_{p_y[ \lfloor Y \rfloor ]}^{p_y[\lfloor Y\rfloor]} }{|t-s|^{p_y[\lfloor Y\rfloor](\gamma - \scG_{\alpha, \beta}[\lfloor Y \rfloor]) }} \Bigg]
\leq
\Big\| \big\langle \bX^{\sharp}, Y \big\rangle \Big\|_{p_x[Y],\gamma - \scG_{\alpha, \beta}[Y]} \cdot |v-u|^{\alpha}
\\
\nonumber
+& \sum_{\substack{T\in\scF, \Upsilon\in\scF \\ \scG_{\alpha, \beta}[T]\in [\gamma-\alpha, \gamma)}} c'\Big( T, \Upsilon, Y\Big) 
\cdot 
\bigg( \Big\| \big\langle \bX^{\sharp}, T\big\rangle \Big\|_{p_x[T], \gamma-\scG_{\alpha, \beta}[T]} 
\cdot 
\Big\| \big\langle \rw, \Upsilon \big\rangle \Big\|_{q[\Upsilon],\scG_{\alpha, \beta}[\Upsilon]}
\cdot 
|v-u|^{\alpha} 
\\
\label{eq:norm-Phi-1.4-integrated}
&\quad +
\Big\| \big\langle \bX_u, T\big\rangle \Big\|_{p_x[T]}
\cdot 
\Big\| \big\langle \rw, \Upsilon \big\rangle \Big\|_{q[\Upsilon],\scG_{\alpha, \beta}[\Upsilon]} \cdot |v-u|^{\scG_{\alpha, \beta}[T] - (\gamma - \alpha)} \bigg)
\end{align}
Combining Equations \eqref{eq:norm-Phi-1.2} and \eqref{eq:norm-Phi-1.4} and applying Proposition \ref{prop:Regu-RCRP}, we get Equation \eqref{eq:theorem:Reconstruction:Phi-2}. 

Similarly, by summing over the terms in Equations \eqref{eq:norm-Phi-1.2-integrated} and \eqref{eq:norm-Phi-1.4-integrated} yields
\begin{align*}
\sum_{Y \in \scF_0}^{\gamma-, \alpha, \beta}& \sup_{s,t\in[u,v]} \frac{ \Big\| \big\langle \Phi[ \bX]_{s, t}^{\sharp}, Y \big\rangle \Big\|_{p_y[ Y ]} }{|t-s|^{\gamma - \scG_{\alpha, \beta}[Y] }}  
\\
\lesssim& \Bigg( \sum_{T \in \scF_0}^{\gamma-\alpha, \alpha, \beta} 
\Big\| \big\langle \bX^{\sharp}, T \big\rangle \Big\|_{p_x[T],\gamma - \scG_{\alpha, \beta}[T]} 
\cdot \bigg(
\Big\| \big\langle \rw, \lfloor T \rfloor \big\rangle \Big\|_{q[T],\scG_{\alpha, \beta}[T] + \alpha} + 1 \bigg) \cdot |v-u|^{\alpha}
\\
&+ \sum_{\substack{ T\in \scF \\ \scG_{\alpha, \beta}[T]\in [\gamma-\alpha, \gamma)}} \Big\| \big\langle \bX^{\sharp}, T\big\rangle \Big\|_{p_x[T], \gamma-\scG_{\alpha, \beta}[T]}
\cdot 
\bigg( \sum_{\Upsilon\in \scF}^{\gamma-, \alpha, \beta} \Big\| \big\langle \rw, \Upsilon\big\rangle \Big\|_{q[\Upsilon],\scG_{\alpha, \beta}[\Upsilon]} \bigg) \cdot |v-u|^{\alpha}
\\
&+ \sum_{\substack{ T\in \scF \\ \scG_{\alpha, \beta}[T]\in [\gamma-\alpha, \gamma)}} \Big\| \big\langle \bX_u, T \big\rangle \Big\|_{p_x[T]} 
\cdot
\bigg( \Big\| \big\langle \rw, \lfloor T \rfloor \big\rangle \Big\|_{q[\lfloor T \rfloor], \scG_{\alpha, \beta}[\lfloor T \rfloor]} 
\\
&\qquad +
\sum_{\Upsilon\in \scF}^{\gamma, \alpha, \beta} \Big\| \big\langle \rw, \Upsilon \big\rangle \Big\|_{q[\Upsilon], \scG_{\alpha, \beta}[\Upsilon]} \bigg) \cdot |v-u|^{\scG_{\alpha, \beta}[T] - (\gamma-\alpha)} \Bigg). 
\end{align*}
\end{proof}

\subsubsection{Products and Expectations of RCRPs}

In this section, we study products (see Proposition \ref{proposition:RCRP-2product}) and expectations (see Proposition \ref{proposition:ExpectationRCRP}) of random controlled rough paths. Then we combine these techniques to achieve Proposition \ref{proposition:productRCRP} which plays a key role in the proof of Theorem \ref{theorem:ContinIm-RCRPs}. 

\begin{proposition}
\label{proposition:RCRP-2product}
Let $\alpha, \beta>0$ and $\gamma:=\inf\{ \scG_{\alpha, \beta}[T]: T \in \scF, \scG_{\alpha, \beta}[T]>1-\alpha \}$. Let $(\Omega, \cF, \bP)$ be a probability space and let $\hat{\scH}^{\gamma, \alpha, \beta}(\Omega)$ denote the $L^0\big( \Omega, \bP; \bR^e \otimes \bR^e \big)$-module
$$
\hat{\scH}^{\gamma, \alpha, \beta}(\Omega) = \bigoplus_{T\in \scF_0}^{\gamma, \alpha, \beta} L^0\Big( \Omega \times \Omega^{\times |H^T|}, \bP \times \bP^{\times |H^T|}; \lin\big( (\bR^d)^{\otimes |N^T|}, \bR^e \otimes \bR^e\big) \Big) \cdot T
$$
so that $\big( \hat{\scH}^{\gamma, \alpha, \beta}(\Omega), \circledast, \rId, \Delta, \epsilon, \scS\big)$ is a coupled Hopf algebra (which is very similar to the first claim in the statement of Theorem \ref{theorem:Reconstruction} except that the arrival space right above is not the same). 

Let $(p_1, q)$ and $(p_2, q)$ be two pairs of dual integrability functionals such that 
\begin{equation}
\label{eq:proposition:RCRP-2product-pq-cond}
\sup_{T\in \scF^{\gamma, \alpha, \beta}} \frac{1}{q[T]} \leq \frac{1}{p_1[\rId]} + \frac{1}{p_2[\rId]}\leq 1. 
\end{equation}
Let $p:\scF_0^{\gamma, \alpha, \beta}\to [1, \infty)$ defined by
\begin{equation}
\label{eq:proposition:RCRP-2product-pq}
\frac{1}{p[\rId]} = \frac{1}{p_1[\rId]} + \frac{1}{p_2[\rId]}, 
\quad
\frac{1}{p[T]} =  \frac{1}{p[\rId]} - \frac{1}{q[T]}. 
\end{equation}
Let $\rw$ be a $\big( \scH^{\gamma, \alpha, \beta}(\Omega), p, q\big)$-probabilistic rough path, let $\bX^1 \in \cD_{\rw}^{\gamma, p_1, q}\big( \scH^{\gamma, \alpha, \beta}(\Omega) \big)$ and let $\bX^2 \in \cD_{\rw}^{\gamma, p_2, q}\big( \scH^{\gamma, \alpha, \beta}(\Omega) \big)$. We define
\begin{equation}
\label{eq:proposition:RCRP-2product}
\Big\langle \bY_t, \rId \Big\rangle(\omega_0) = \Big\langle \bX^1_t, \rId \Big\rangle(\omega_0) \otimes \Big\langle \bX^2_t, \rId\Big\rangle(\omega_0). 
\end{equation}
Then there is a random controlled rough path $\bY\in \cD_{\rw}^{\gamma, p, q}\big( \hat{\scH}^{\gamma, \alpha, \beta}(\Omega) \big)$ that satisfies Equation \eqref{eq:proposition:RCRP-2product} and $\forall T \in \scF_0^{\gamma-, \alpha, \beta}$, 
\begin{equation}
\label{eq:proposition:RCRP-2product:GubDeriv}
\Big\langle \bY_s, T \Big\rangle(\omega_0, \omega_{H^T}) = \sum_{\substack{T_1, T_2\in \scF_0\\ T_1\circledast T_2 = T}} \Big\langle \bX_s^1, T_1 \Big\rangle(\omega_0, \omega_{H^{T_1}}) \otimes \Big\langle \bX_s^2, T_2 \Big\rangle(\omega_0, \omega_{H^{T_2}}). 
\end{equation}
\end{proposition}

\begin{proof}
Firstly, we verify that $(p, q)$ as described in Equation \eqref{eq:proposition:RCRP-2product-pq} satisfy Definition \ref{definition:(dual)_integrab-functional}. 

Suppose that $(p_1, q)$ and $(p_2, q)$ are a dual pair of integrability functionals that satisfy Equation \eqref{eq:proposition:RCRP-2product-pq-cond}. Then by Equation \eqref{eq:proposition:RCRP-2product-pq}, $p[\rId]\in (1, \infty)$. Further, for any choice of $T\in \scF^{\gamma, \alpha, \beta}$, we have that $\tfrac{1}{p[\rId]} - \tfrac{1}{q[T]} \in (0,1)$ so that $p[T]\in (1, \infty)$. By construction, $p$ satisfies Equation \eqref{eq:IntegrabilityP1} so that $(p, q)$ is a pair of dual integrability functionals. 

Suppose that $\bY$ satisfies Equation \eqref{eq:proposition:RCRP-2product}. We want to verify that it can be restated in the form described in Equation \eqref{eq:definition:RandomControlledRP1} and \eqref{eq:definition:RandomControlledRP2}. Then we verify that $\bY$ satisfies Definition \ref{definition:RandomControlledRP} by confirming the integrability and regularity of the remainder terms. Firstly, 
\begin{align*}
\Big\langle \bY_{s, t}&, \rId\Big\rangle (\omega_0)
=
\bigg( \Big\langle \bX^1, \rId\Big\rangle \otimes \Big\langle \bX^2, \rId\Big\rangle \bigg)_{s, t}(\omega_0)
\\
=& \sum_{T \in \scF}^{\gamma-, \alpha, \beta} \bE^{H^T}\bigg[ \bigg( \sum_{\substack{T_1, T_2\in\scF_0 \\ T_1\circledast T_2 = T}} \Big\langle \bX_s^1, T_1 \Big\rangle(\omega_0, \omega_{H^{T_1}}) \otimes \Big\langle \bX_s^2, T_2\Big\rangle(\omega_0, \omega_{H^{T_2}}) \bigg) \cdot \Big\langle \rw_{s,t}, T \Big\rangle(\omega_0, \omega_{H^T}) \bigg]
\\
&+ \Big\langle \bY_{s,t}^{\sharp}, \rId \Big\rangle(\omega_0), 
\end{align*}
where 
\begin{align}
\nonumber
\Big\langle& \bY_{s,t}^{\sharp}, \rId \Big\rangle(\omega_0) 
\\
\nonumber
=& \sum_{\substack{T_1, T_2\in \scF_0 \\ \scG_{\alpha, \beta}[T_1\circledast T_2] \geq \gamma }}^{\gamma-, \alpha, \beta} \bE^{H^{T_1\circledast T_2}} \Big[ \big\langle \bX_{s}^1, T_1 \big\rangle(\omega_0, \omega_{H^{T_1}}) \otimes \big\langle \bX_{s}^2, T_2 \big\rangle(\omega_0, \omega_{H^{T_2}}) \cdot \big\langle \rw_{s, t}, T_1 \circledast T_2 \big\rangle(\omega_0, \omega_{H^{T_1\circledast T_2} }) \Big]
\\
\label{eq:proposition:RCRP-2product:sharp}
&+ \bigotimes_{i=1}^2 \Big( \big\langle \bX_s^{i} + \bX_{s, t}^i, \rId\big\rangle(\omega_0) \Big) - \bigotimes_{i=1}^2 \Big( \big\langle \bX_s^{i} + \fJ\big[ \bX^{i} \big]_{s, t}, \rId\big\rangle(\omega_0) \Big). 
\end{align}
Substituting in Equation \eqref{eq:proposition:RCRP-2product:GubDeriv} yields Equation \eqref{eq:definition:RandomControlledRP1}. 

By measuring the resulting regularity (and denoting $\eta=|v-u|$), thanks to Equation \eqref{eq:proposition:RCRP-2product-pq} we obtain that
\begin{align*}
\sup_{s, t\in [u,v]}& \frac{ \Big| \big\langle \bY^{\sharp}_{s, t}, \rId \big\rangle(\omega_0) \Big|}{|t-s|^{\gamma}}
\\
\leq& \sum_{\substack{T_1, T_2\in \scF_0 \\ \scG_{\alpha, \beta}[T_1\circledast T_2] \geq \gamma }}^{\gamma-, \alpha, \beta} \bigotimes_{i=1}^2 \Big( \sup_{t\in [u,v]} \Big\| \big\langle \bX_t^i, T_i\big\rangle(\omega_0) \Big\|_{p_i[T_i]} \cdot \Big\| \big\langle \rw, T_i\big\rangle(\omega_0) \Big\|_{q[T_i], \scG_{\alpha, \beta}[T_i]}\Big)
\cdot
\eta^{\scG_{\alpha, \beta}[T_1 \circledast T_2] - \gamma}
\\
&+ \sup_{s, t\in[u, v]} \Big| \big\langle \bX_s^1, \rId\big\rangle(\omega_0) + \big\langle \fJ[\bX^1]_{s,t}, \rId\big\rangle(\omega_0) \Big| \cdot \Big\| \big\langle \bX^{2,{\sharp}}, \rId\big\rangle(\omega_0) \Big\|_{\gamma}
\\
&+ \Big\| \big\langle \bX^{1,{\sharp}}, \rId\big\rangle(\omega_0) \Big\|_{\gamma} \cdot \sup_{s, t\in[u, v]} \Big| \big\langle \bX_s^2, \rId\big\rangle(\omega_0) + \big\langle \fJ[\bX^2]_{s,t}, \rId\big\rangle(\omega_0) \Big| 
\\
&+ \Big\| \big\langle \bX^{1,{\sharp}}, \rId\big\rangle(\omega_0) \Big\|_{\gamma} \cdot \Big\| \big\langle \bX^{2,{\sharp}}, \rId\big\rangle(\omega_0) \Big\|_{\gamma} \cdot \eta^{\gamma}
<
\infty \quad (\Omega_0, \bP)\mbox{-almost surely}. 
\end{align*}
Further, integrating over the tagged probability space and applying the H\"older inequality with Equation \eqref{eq:proposition:RCRP-2product-pq} yields
\begin{align*}
&\bE^0\Bigg[ \sup_{s, t\in [u,v]} \frac{\Big| \big\langle \bY^{\sharp}_{s, t}, \rId \big\rangle(\omega_0) \Big|^{p[\rId]}}{|t-s|^{p[\rId]\gamma}} \Bigg]^{\tfrac{1}{p[\rId]}}
\\
&\leq \hspace{-10pt}
\sum_{\substack{T_1, T_2\in \scF_0 \\ \scG_{\alpha, \beta}[T_1\circledast T_2] \geq \gamma }}^{\gamma-, \alpha, \beta} \bigotimes_{i=1}^2 \Bigg( \bE^0\bigg[ \sup_{t\in [u,v]} \Big\| \big\langle \bX_t^i, T_i\big\rangle(\omega_0) \Big\|_{p_i[T_i]}^{p_i[T_i]} \bigg]^{\tfrac{1}{p_i[T_i]}} \cdot \Big\| \big\langle \rw, T_i\big\rangle \Big\|_{q[T_i], \scG_{\alpha, \beta}[T_i]}\Bigg)
\cdot
\eta^{\scG_{\alpha, \beta}[T_1 \circledast T_2] - \gamma}
\\
&+ \bE^0\bigg[ \sup_{s, t\in[u, v]} \Big| \big\langle \bX_s^1, \rId\big\rangle(\omega_0) + \big\langle \fJ[\bX^1]_{s,t}, \rId\big\rangle(\omega_0) \Big|^{p_1[\rId]}\bigg]^{\tfrac{1}{p_1[\rId]}} \cdot \Big\| \big\langle \bX^{2,{\sharp}}, \rId\big\rangle \Big\|_{\gamma, p_2[\rId]}
\\
&+ \Big\| \big\langle \bX^{1,{\sharp}}, \rId\big\rangle \Big\|_{\gamma, p_1[\rId]} \cdot \bE^0\bigg[ \sup_{s, t\in[u, v]} \Big| \big\langle \bX_s^2, \rId\big\rangle(\omega_0) + \big\langle \fJ[\bX^2]_{s,t}, \rId\big\rangle(\omega_0) \Big|^{p_2[\rId]} \bigg]^{\tfrac{1}{p_2[\rId]}} 
\\
&+ \Big\| \big\langle \bX^{1,{\sharp}}, \rId\big\rangle \Big\|_{\gamma, p_1[\rId]} \cdot \Big\| \big\langle \bX^{2,{\sharp}}, \rId\big\rangle \Big\|_{\gamma, p_2[\rId]} \cdot \eta^{\gamma}
<
\infty 
\end{align*}

Next, for $Y \in \scF^{\gamma-, \alpha, \beta}$ we verify that Equation \eqref{eq:proposition:RCRP-2product:GubDeriv} satisfies Equation \eqref{eq:definition:RandomControlledRP2}. We have
\begin{align*}
\Big\langle \bY_{s,t}, Y \Big\rangle(\omega_0, \omega_{H^Y}) =& \sum_{\substack{Y_1, Y_2\in \scF_0\\ Y_1\circledast Y_2 = Y}} \Big\langle \bX_{s,t}^1, Y_1 \Big\rangle(\omega_0, \omega_{H^{Y_1}}) \otimes \Big\langle \bX_s^2, Y_2 \Big\rangle(\omega_0, \omega_{H^{Y_2}})
\\
&+ \sum_{\substack{Y_1, Y_2\in \scF_0\\ Y_1\circledast Y_2 = Y}} \Big\langle \bX_{s}^1, Y_1 \Big\rangle(\omega_0, \omega_{H^{Y_1}}) \otimes \Big\langle \bX_{s,t}^2, Y_2 \Big\rangle(\omega_0, \omega_{H^{Y_2}})
\\
&+ \sum_{\substack{Y_1, Y_2\in \scF_0\\ Y_1\circledast Y_2 = Y}} \Big\langle \bX_{s,t}^1, Y_1 \Big\rangle(\omega_0, \omega_{H^{Y_1}}) \otimes \Big\langle \bX_{s,t}^2, Y_2 \Big\rangle(\omega_0, \omega_{H^{Y_2}}). 
\end{align*}
We substitute using the fact that both $\bX^1$ and $\bX^2$ satisfy Equation \eqref{eq:definition:RandomControlledRP2} to get
\begin{align}
\nonumber
\Big\langle \bY_{s,t}&, Y \Big\rangle(\omega_0, \omega_{H^Y})
= \sum_{\substack{Y_1, Y_2\in \scF_0\\ Y_1\circledast Y_2 = Y}} 
\sum_{T\in \scF}^{\gamma-, \alpha, \beta} \sum_{\substack{T_1, T_2 \in \scF_0\\ T = T_1\circledast T_2}} \sum_{\Upsilon_1, \Upsilon_2 \in \scF_0}
\bigg( \prod_{i=1}^2 c\Big(T_i, \Upsilon_i, Y_i \Big) - \prod_{i=1}^2 \delta_{T_i = Y_i, \Upsilon_i = \rId}\bigg)
\\
\nonumber
& \cdot \bigotimes_{i=1}^2 \bE^{E^{T_i, \Upsilon_i, Y_i}}\bigg[ \Big\langle \bX_s^1, T_i \Big\rangle (\omega_0, \omega_{\phi^{ T_i, \Upsilon_i, Y_i}[H^{T_i}]}) \cdot \Big\langle \rw_{s, t}, \Upsilon_i \Big\rangle (\omega_0, \omega_{\varphi^{T_i, \Upsilon_i, Y_i}[H^{\Upsilon_i}]}) \bigg] 
\\
\label{proposition:RCRP-2product:(simple)1.1}
&+\Big\langle \bY_{s, t}^{\sharp}, Y\Big\rangle (\omega_0, \omega_{H^Y}), 
\end{align}
where 
\begin{align}
\nonumber
\Big\langle \bY_{s, t}^{\sharp}&, Y\Big\rangle(\omega_0, \omega_{H^Y}) 
\\
\nonumber
=& \sum_{\substack{Y_1, Y_2\in \scF_0\\ Y_1\circledast Y_2 = Y}} \bigg( \bigotimes_{i=1}^2 \Big( \big\langle \bX_s^i + \bX_{s, t}^i, Y_i\big\rangle \Big)(\omega_0, \omega_{H^Y}) - \bigotimes_{i=1}^2 \Big(  \big\langle \bX_s^i + \fJ\big[ \bX^i \big]_{s, t}, Y_i\big\rangle \Big)(\omega_0, \omega_{H^Y})
\\
\nonumber
&+ \sum_{\substack{Y_1, Y_2\in \scF_0\\ Y_1\circledast Y_2 = Y}} 
\sum_{\substack{T_1, T_2 \in \scF_0 \\ \scG_{\alpha,\beta}[T_1 \circledast T_2]\geq \gamma}}^{\gamma-, \alpha, \beta} \sum_{ \Upsilon_1, \Upsilon_2 \in \scF_0} \bigg( \prod_{i=1}^2 c\Big(T_i, \Upsilon_i, Y_i \Big) - \prod_{i=1}^2 \delta_{\{ T_i = Y_i, \Upsilon_i = \rId\} }\bigg)
\\
\label{proposition:RCRP-2product:sharp2}
&\qquad \cdot \bigotimes_{i=1}^2 \bE^{E^{T_i, \Upsilon_i, Y_i}}\bigg[ \Big\langle \bX_s^i, T_i \Big\rangle(\omega_0, \omega_{\phi^{ T_i, \Upsilon_i, Y_i}[H^{T_i}]}) \cdot \Big\langle \rw_{s, t}, \Upsilon_i \Big\rangle(\omega_0, \omega_{\varphi^{T_i, \Upsilon_i, Y_i}[H^{\Upsilon_i}]}) \bigg]. 
\end{align}
Applying Lemma \ref{lemma:CoprodCountIdent} and Lemma \ref{lemma:TechnicalPhiVarphi} to Equation \eqref{proposition:RCRP-2product:(simple)1.1} and substituting in Equation \eqref{eq:proposition:RCRP-2product:GubDeriv} gives
\begin{align*}
\Big\langle \bY_{s,t}, Y \Big\rangle(\omega_0, \omega_{H^Y})
&= 
\sum_{T\in \scF}^{\gamma-,\alpha, \beta} \sum_{\Upsilon\in \scF} c'\Big( T, \Upsilon, Y\Big) 
\cdot
\bE^{E^{T, \Upsilon,Y}}\bigg[ \sum_{\substack{T_1, T_2\in \scF_0\\ T_1\circledast T_2 = T}} \bigotimes_{i=1}^2 \Big\langle \bX_s^i, T_i \Big\rangle(\omega_0, \omega_{\phi^{T, \Upsilon, Y}[H^{T_i}]}) 
\\
&\quad \cdot \Big\langle \rw_{s, t}, \Upsilon \Big\rangle(\omega_0, \omega_{\varphi^{T, \Upsilon, Y}[H^\Upsilon]}) \bigg] 
+ 
\Big\langle \bY^{\sharp}_{s, t}, Y \Big\rangle( \omega_0, \omega_{H^Y}), 
\end{align*}
so that $\bY$ satisfies Equation \eqref{eq:definition:RandomControlledRP2}. 

To conclude, we integrate over the probability space $\Omega^{\times |H^Y|}$ in \eqref{proposition:RCRP-2product:sharp2} and apply Lemma \ref{lemma:(dual)_integrab-functional} so that for any $Y\in \scF^{\gamma-, \alpha, \beta}$,
\begin{align*}
\sup_{s, t\in[u,v]} &
\frac{\bE^{H^T}\bigg[ \Big| \big\langle \bY^{\sharp}_{s, t}, Y \big\rangle (\omega_0, \omega_{H^Y}) \Big|^{p[Y]} \bigg]^{\tfrac{1}{p[Y]}}}{|t-s|^{\gamma - \scG_{\alpha, \beta}[Y]}} 
\\
&\leq \sum_{\substack{Y_1, Y_2\in \scF_0\\ Y_1\circledast Y_2 = Y}} 
\sum_{\substack{T_1, T_2 \in \scF_0 \\ \scG_{\alpha,\beta}[T_1 \circledast T_2]\geq \gamma}}^{\gamma-, \alpha, \beta} \sum_{ \Upsilon_1, \Upsilon_2 \in \scF_0} \bigg( \prod_{i=1}^2 c\Big( T_i, \Upsilon_i, Y_i \Big) - \prod_{i=1}^2 \delta_{\{T_i = Y_i, \Upsilon_i = \rId\}} \bigg)
\\
&\quad \cdot \prod_{i=1}^2 \Big( \sup_{t\in [u,v]} \Big\| \big\langle \bX_t^i, Y_i\big\rangle(\omega_0) \Big\|_{p_i[Y_i]} \cdot \Big\| \big\langle \rw, Y_i\big\rangle(\omega_0) \Big\|_{q[T_i], \scG_{\alpha, \beta}[T_i]}\Big) \cdot \eta^{\scG_{\alpha, \beta}[T_1 \circledast T_2] - \gamma}
\\
&+ \sum_{\substack{Y_1, Y_2\in \scF_0\\ Y_1\circledast Y_2 = Y}} \bigg( \sup_{s, t\in[u, v]} \Big\| \big\langle \bX_s^1 + \fJ[\bX^1]_{s,t}, Y_1 \big\rangle(\omega_0) \Big\|_{p_1[Y_1]} \cdot \Big\| \big\langle \bX^{2,{\sharp}}, Y_2\big\rangle(\omega_0) \Big\|_{p_2[Y_2],\gamma} \cdot \eta^{\scG_{\alpha, \beta}[Y_1]}
\\
&\quad + \Big\| \big\langle \bX^{1,{\sharp}}, Y_1\big\rangle(\omega_0) \Big\|_{p_1[Y_1],\gamma} \cdot \sup_{s, t\in[0, 1]} \Big\| \big\langle \bX_s^2 + \fJ[\bX^2]_{s,t}, Y_2 \big\rangle(\omega_0) \Big\|_{p_2[Y_2]} \cdot \eta^{\scG_{\alpha, \beta}[Y_2]}
\\
&\quad + \Big\| \big\langle \bX^{1,{\sharp}}, Y_1 \big\rangle(\omega_0) \Big\|_{p_1[Y_1], \gamma-\scG_{\alpha, \beta}[Y_1]} \cdot \Big\| \big\langle \bX^{2,{\sharp}}, Y_2 \big\rangle(\omega_0) \Big\|_{p_2[Y_2], \gamma-\scG_{\alpha, \beta}[Y_2]} \cdot \eta^{\gamma} \bigg)
< \infty
\end{align*}
$(\Omega_0, \bP)$-almost surely and
$$
\bE^0\Bigg[ \sup_{s, t\in[0,1]} 
\frac{\bE^{H^T}\bigg[ \Big| \big\langle \bY^{\sharp}_{s, t}, Y \big\rangle (\omega_0, \omega_{H^Y}) \Big|^{p[Y]} \bigg]}{|t-s|^{p[Y](\gamma - \scG_{\alpha, \beta}[Y])} } \Bigg]<\infty. 
$$
\end{proof}

Next, we consider the expectation of a random controlled rough path. 
\begin{proposition}
\label{proposition:ExpectationRCRP}
Let $\alpha, \beta>0$ and $\gamma:=\inf\{ \scG_{\alpha, \beta}[T]: T \in \scF, \scG_{\alpha, \beta}[T]>1-\alpha \}$. Let $(p, q)$ be a pair of dual integrability functionals. 

Let $\rw$ be a $(\scH^{\gamma, \alpha, \beta}, p, q)$-probabilistic rough path and let $\bX \in \cD_{\rw}^{\gamma, p, q}\big( \scH^{\gamma, \alpha, \beta}\big)$. We define
\begin{equation}
\label{eq:proposition:ExpectationRCRP_Id}
\Big\langle \bY_t, \rId \Big\rangle(\omega_0) = \tilde{\bE}\Big[ \big\langle \bX_t, \rId \big\rangle(\tilde{\omega}) \Big]. 
\end{equation}
Then there is a random controlled rough path $\bY\in \cD_{\rw}^{\gamma, p, q}\big( \scH^{\gamma, \alpha, \beta}\big)$ that satisfies Equation \eqref{eq:proposition:ExpectationRCRP_Id} and for $T\in \scF^{\gamma-, \alpha, \beta}$ such that $h_0^T \neq \emptyset$, 
\begin{equation}
\label{eq:proposition:ExpectationRCRP}
\begin{split}
\Big\langle \bY_t, T\Big\rangle&(\omega_0, \omega_{H^T}) = 0, 
\\
\Big\langle \bY_t, \cE[T] \Big\rangle&(\omega_0, \omega_{h_0^T},  \omega_{H^T}) = \Big\langle \bX_t, T\Big\rangle(\omega_{h_0^T}, \omega_{H^T}) + \tilde{\bE}\Big[ \Big\langle \bX_t, \cE[T] \Big\rangle(\tilde{\omega}, \omega_{h_0^T}, \omega_{H^T}) \Big]. 
\end{split}
\end{equation}
\end{proposition}

We emphasise that each of the random variables $\big\langle \bY_t, T \big\rangle(\omega_0, \omega_{H^T})$ are constant in $\omega_0$ (deterministic). 

\begin{proof}
We have
\begin{align*}
\Big\langle \bY_{s, t}, \rId \Big\rangle(\omega_0) =& \tilde{\bE}\bigg[ \Big\langle \bX_{s, t}, \rId \Big\rangle(\tilde{\omega}) \bigg]
\\
=& \sum_{T\in\scF}^{\gamma-, \alpha, \beta} \tilde{\bE}\bigg[ \bE^{H^T}\bigg[ \Big\langle \bX_s, T \Big\rangle(\tilde{\omega}, \omega_{H^T}) \cdot \Big\langle \rw_{s, t}, T\Big\rangle(\tilde{\omega}, \omega_{H^T}) \bigg]\bigg] + \tilde{\bE}\bigg[ \Big\langle \bX_{s, t}^{\sharp}, \rId\Big\rangle(\tilde{\omega}) \bigg]. 
\end{align*}
Next, an application of Proposition \ref{proposition:EAction} yields that for $T\in\scF^{\gamma-, \alpha, \beta}$ such that $h_0^T\neq \emptyset$, 
\begin{align*}
\tilde{\bE}\bigg[ \bE^{H^T}\bigg[& \Big\langle \bX_s, T \Big\rangle(\tilde{\omega}, \omega_{H^T}) \cdot \Big\langle \rw_{s, t}, T\Big\rangle(\tilde{\omega}, \omega_{H^T}) \bigg]\bigg] 
\\
&= \bE^{(H^T)'} \bigg[ \Big\langle \bX_s, T \Big\rangle(\omega_{h_0^T}, \omega_{H^T}) \cdot \Big\langle \rw_{s, t}, \cE[T] \Big\rangle(\omega_0, \omega_{h_0^T}, \omega_{H^T}) \bigg]
\quad 
\mbox{and}
\\
\tilde{\bE}\bigg[ \bE^{H^{\cE[T]}}\bigg[& \Big\langle \bX_s, \cE[T] \Big\rangle(\tilde{\omega}, \omega_{H^{\cE[T]}}) \cdot \Big\langle \rw_{s, t}, \cE[T] \Big\rangle(\tilde{\omega}, \omega_{H^{\cE[T]}}) \bigg]\bigg] 
\\
&= \bE^{(H^T)'} \bigg[ \tilde{\bE}\bigg[ \Big\langle \bX_s, T \Big\rangle(\tilde{\omega}, \omega_{h_0^T}, \omega_{H^T})\bigg] \cdot \Big\langle \rw_{s, t}, \cE[T] \Big\rangle(\omega_0, \omega_{h_0^T}, \omega_{H^T}) \bigg]. 
\end{align*}
Thus
\begin{align*}
\Big\langle \bY_{s, t}, \rId\Big\rangle(\omega_0) 
=& 
\sum_{\substack{T\in\scF: h_0^T\neq \emptyset \\ \scG_{\alpha, \beta}[\cE[T]]< \gamma}} \bE^{(H^T)'}\Big[ \Big( \big\langle \bX_s, T\big\rangle(\omega_{h_0^T}, \omega_{H^T}) + \tilde{\bE}\Big[ \big\langle \bX_s, \cE[T]\big\rangle(\tilde{\omega}, \omega_{h_0^T}, \omega_{H^T}) \Big] \Big) 
\\
&\quad \cdot \big\langle \rw_{s, t}, \cE[T] \big\rangle(\omega_0, \omega_{h_0^T}, \omega_{H^T}) \Big]
+ \Big\langle \bY_{s, t}^{\sharp}, \rId \Big\rangle(\omega_0), 
\end{align*}
where
\begin{align}
\nonumber
&\Big\langle \bY_{s, t}^{\sharp}, \rId\Big\rangle(\omega_0) 
\\
&= \tilde{\bE}\Big[ \big\langle \bX_{s, t}^{\sharp}, \rId\big\rangle(\tilde{\omega}) \Big] + \sum_{\substack{ T\in \scF^{\gamma-, \alpha, \beta} \\ \scG_{\alpha, \beta}[\cE[T]]\geq \gamma}} \bE^{(H^T)'} \Big[ \big\langle \bX_s, T \big\rangle(\omega_{h_0^T}, \omega_{H^T}) \cdot \big\langle \rw_{s, t}, T \big\rangle (\omega_{h_0^T}, \omega_{H^T}) \Big], 
\end{align}
and choosing $\langle \bY, T\rangle$ according to Equation \eqref{eq:proposition:ExpectationRCRP} yields Equation \eqref{eq:definition:RandomControlledRP1}. 

By measuring the resulting regularity (and denoting $\eta=|v-u|$), we obtain that
\begin{align*}
\sup_{s,t\in[u, v]}& \frac{\Big| \big\langle \bY_{s, t}^{\sharp}, \rId\big\rangle(\omega_0)\Big| }{|t-s|^{\gamma}} 
\\
\leq& 
\sum_{\substack{ T\in \scF^{\gamma-, \alpha, \beta} \\ \scG_{\alpha, \beta}[\cE[T]]\geq \gamma}} \sup_{t\in[u,v]} \Big\| \big\langle \bX_s, T \big\rangle\Big\|_{p[T]} \cdot \Big\| \big\langle \rw, \cE[T] \big\rangle \Big\|_{q[T], \scG_{\alpha, \beta}[\cE[T]]} \cdot \eta^{\scG_{\alpha, \beta}[\cE[T]] - \gamma}
\\
&+ 
\sup_{s, t\in[u,v]} \frac{\tilde{\bE}\Big[ \big\langle \bX_{s, t}^{\sharp}, \rId\big\rangle(\tilde{\omega})\Big]}{|t-s|^{\gamma}} < \infty.  
\end{align*}
Further, this is constant in the tagged probability space so that
$$
\bE^0\bigg[ \sup_{s,t\in[0, 1]} \frac{\big| \langle \bY_{s, t}^{\sharp}, \rId\rangle(\omega_0)\big| }{|t-s|^{\gamma}} \bigg]<\infty. 
$$

Secondly, we verify Equation \ref{eq:definition:RandomControlledRP2} is satisfied. By construction, for $Y \in \scF^{\gamma-, \alpha, \beta}$ such that $h_0^Y\neq \emptyset$, 
\begin{align}
\nonumber
&\Big\langle \bY_{s, t}, \cE[Y] \Big\rangle(\omega_0, \omega_{h_0^Y}, \omega_{H^Y}) 
\\
\nonumber
&= \Big\langle \bX_{s, t}, Y \Big\rangle(\omega_{h_0^Y}, \omega_{H^Y}) + \tilde{\bE}\bigg[ \Big\langle \bX_{s, t}, \cE[Y]\Big\rangle(\tilde{\omega}, \omega_{h_0^Y}, \omega_{H^Y}) \bigg], 
\\
\label{eq:proposition:ExpectationRCRP:t(1.1)}
&=\sum_{\substack{T\in \scF \\ \Upsilon\in \scF}}^{\gamma-, \alpha, \beta} c'\Big( T, \Upsilon, Y\Big) \cdot \bE^{E^{T, \Upsilon, Y}}\bigg[ \Big\langle \bX_s, T \Big\rangle(\omega_{h_0^Y}, \omega_{\phi^{T, \Upsilon, Y}[H^T]}) \cdot \Big\langle \rw_{s,t}, \Upsilon\Big\rangle(\omega_{h_0^Y}, \omega_{\varphi^{T, \Upsilon, Y}[H^{\Upsilon}]}) \bigg]
\\
\label{eq:proposition:ExpectationRCRP:t(1.2)}
&+\sum_{\substack{T\in \scF \\ \Upsilon\in \scF }}^{\gamma-, \alpha, \beta} c'\Big( T, \Upsilon, \cE[Y] \Big) \cdot \bE^{E^{T, \Upsilon, \cE[Y]}}\bigg[ \tilde{\bE}\bigg[ \Big\langle \bX_s, T \Big\rangle(\tilde{\omega}, \omega_{\phi^{T, \Upsilon, \cE[Y]}[H^T]}) \cdot \Big\langle \rw_{s,t}, \Upsilon\Big\rangle(\tilde{\omega}, \omega_{\varphi^{T, \Upsilon, \cE[Y]}[H^{\Upsilon}]}) \bigg]\bigg]
\\
\nonumber
&+ \Big\langle \bX_{s, t}^{\sharp}, Y \Big\rangle(\omega_{h_0^Y}, \omega_{H^Y}) + \tilde{\bE}\bigg[ \Big\langle \bX_{s, t}^{\sharp}, \cE[Y]\Big\rangle(\tilde{\omega}, \omega_{h_0^Y}, \omega_{H^Y}) \bigg]. 
\end{align}

By assumption, $h_0^Y\neq \emptyset$, so that $h_0^{T}\neq \emptyset$ whenever $c'(T, \Upsilon, Y)>0$. Recalling Definition \ref{definition:CouplingFunctions} and Definition \ref{definition:E-Set}, consider the triple $\big( \cE[T], \cE[\Upsilon], \cE[Y] \big)$. We remark that 
\begin{align*}
c'\Big( \cE[T], \cE[\Upsilon], \cE[Y]\Big) =& c'\Big( T, \Upsilon, Y\Big), \quad&\quad 
E^{\cE[T], \cE[\Upsilon], \cE[Y]} =& E^{T, \Upsilon, Y} 
\\
\phi^{\cE[T], \cE[\Upsilon], \cE[Y]}\big[ h_0^T \big] =& h_0^Y,
\quad&\quad 
\varphi^{\cE[T], \cE[\Upsilon], \cE[Y]}\big[ h_0^\Upsilon \big] =& h_0^Y
\end{align*}
so that
\begin{align*}
\eqref{eq:proposition:ExpectationRCRP:t(1.1)} =& \sum_{\substack{T\in \scF: h_0^T\neq \emptyset \\ \Upsilon\in \scF}}^{\gamma-,\alpha, \beta} c'\Big( \cE[T], \cE[\Upsilon] , \cE[Y] \Big) \cdot \bE^{E^{\cE[T], \cE[\Upsilon], \cE[Y] }}\bigg[ \Big\langle \bX_s, T \Big\rangle( \omega_{\phi^{\cE[T], \cE[\Upsilon], \cE[Y]}[H^{\cE[T]}]}) 
\\
&\quad \cdot \Big\langle \rw_{s,t}, \Upsilon\Big\rangle ( \omega_{\varphi^{\cE[T], \cE[\Upsilon], \cE[Y]}[H^{\cE[\Upsilon]}]}) \bigg]. 
\end{align*}
In the same fashion, the Lions tree $\cE[Y]$ satisfies that $h_0^{\cE[Y]} = \emptyset$ so for any triple $(T, \Upsilon, \cE[Y])$ such that $c'( T, \Upsilon, \cE[Y])>0$, we have that $h_0^T \cap N^{\Upsilon} = h_0^\Upsilon$ and 
$$
c'\Big( \cE[T], \cE[\Upsilon], \cE[Y]\Big) = c'\Big( T, \Upsilon, \cE[Y] \Big). 
$$
Therefore, 
\begin{align*}
\eqref{eq:proposition:ExpectationRCRP:t(1.2)} =& \sum_{\substack{T\in \scF:h_0^T \neq \emptyset \\ \Upsilon\in \scF}}^{\gamma-,\alpha, \beta} c'\Big( \cE[T], \cE[\Upsilon], \cE[Y] \Big) \cdot \bE^{E^{\cE[T], \cE[\Upsilon], \cE[Y]}}\bigg[ \tilde{\bE}\bigg[ \Big\langle \bX_s, \cE[T] \Big\rangle(\tilde{\omega}, \omega_{\phi^{\cE[T], \cE[\Upsilon], \cE[Y]}[H^{\cE[T]}]}) 
\\
&\quad \cdot \Big\langle \rw_{s,t}, \Upsilon\Big\rangle(\tilde{\omega}, \omega_{\varphi^{\cE[T], \cE[\Upsilon], \cE[Y]}[H^{\Upsilon}]}) \bigg]\bigg]. 
\end{align*}

Making these substitutions and noticing that by assumption, $c'(T, \Upsilon, Y)>0$ implies that $h_0^\Upsilon = \emptyset$ and we obtain
\begin{align*}
\Big\langle \bY_{s, t}&, \cE[Y] \Big\rangle(\omega_0, \omega_{h_0^Y}, \omega_{H^Y}) = \sum_{T \in \scF: h_0^T \neq \emptyset }^{\gamma-,\alpha, \beta} \sum_{\Upsilon\in\scF} c'\Big( \cE[T], \Upsilon, \cE[Y] \Big) \cdot \bE^{E^{\cE[T], \Upsilon, \cE[Y]}} \Bigg[ 
\\
&\bigg( \Big\langle \bX_s, T\Big\rangle (\omega_{\phi^{\cE[T], \Upsilon, \cE[Y]}[H^{\cE[T]}]}) 
+ \tilde{\bE}\Big[ \Big\langle \bX_s, \cE[T]\Big\rangle (\tilde{\omega}, \omega_{\phi^{\cE[T], \Upsilon, \cE[Y]}[H^{\cE[T]}]})\Big] \bigg) 
\\
&\cdot \Big\langle \rw_{s, t}, \Upsilon\Big\rangle( \omega_0, \omega_{\varphi^{\cE[T], \Upsilon, \cE[Y]}[H^{\Upsilon}]} ) \Bigg]
+ \Big\langle \bY_{s, t}^{\sharp}, \cE[Y] \Big\rangle(\omega_0, \omega_{h_0^Y}, \omega_{H^Y}), 
\end{align*}
where
\begin{align*}
\Big\langle \bY_{s, t}^{\sharp}&, \cE[Y] \Big\rangle(\omega_0, \omega_{h_0^Y}, \omega_{H^Y}) = \Big\langle \bX_{s, t}^{\sharp}, Y \Big\rangle(\omega_{h_0^Y}, \omega_{H^Y}) + \tilde{\bE}\bigg[ \Big\langle \bX_{s, t}^{\sharp}, \cE[Y]\Big\rangle(\tilde{\omega}, \omega_{h_0^Y}, \omega_{H^Y}) \bigg]
\\
&+ \sum_{\substack{T\in \scF\\ \scG_{\alpha, \beta}[\cE[T]]\geq \gamma}}^{\gamma-, \alpha, \beta}  \sum_{\Upsilon\in \scF} c'\Big( T, \Upsilon, Y\Big) \cdot \tilde{\bE} \bigg[ \bE^{E^{T, \Upsilon, Y}}\bigg[ \Big\langle \bX_s, T\Big\rangle(\tilde{\omega}, \omega_{\phi^{T, \Upsilon, Y}[H^T]}) 
\\
&\qquad \cdot \Big\langle \rw_{s, t}, \Upsilon\Big\rangle (\omega_0, \omega_{\varphi^{T, \Upsilon, Y}[H^{\Upsilon}]}) \bigg] \bigg], 
\end{align*}
and substituting in Equation \eqref{eq:proposition:ExpectationRCRP} yields Equation \ref{eq:definition:RandomControlledRP2}. 

Finally, for $Y\in\scF^{\gamma, \alpha, \beta}$ such that $h_0^Y \neq \emptyset$
\begin{align*}
\sup_{s, t\in[u, v]}& \frac{ \Big\| \big\langle \bY_{s, t}^{\sharp}, \cE[Y] \big\rangle(\omega_0) \Big\|_{p[\cE[Y]]} }{|t-s|^{\gamma - \scG_{\alpha, \beta}[\cE[Y]]}}
\\
\leq& \sup_{s, t\in[u, v]} \frac{ \Big\| \big\langle \bX_{s, t}^{\sharp}, Y \big\rangle \Big\|_{p[Y]} }{|t-s|^{\gamma - \scG_{\alpha, \beta}[Y]}}
+
\sup_{s, t\in[u, v]} \frac{ \Big\| \big\langle \bX_{s, t}^{\sharp}, \cE[Y] \big\rangle \Big\|_{p[\cE[Y]]} }{|t-s|^{\gamma - \scG_{\alpha, \beta}[\cE[Y]]}}
\\
&+ \sum_{\substack{T\in \scF\\ \scG_{\alpha, \beta}[\cE[T]]\geq \gamma}}^{\gamma-, \alpha, \beta}  \sum_{\Upsilon\in \scF} c'\Big( T, \Upsilon, Y\Big) \cdot \sup_{t\in[u,v]} \Big\| \big\langle \bX_t, T\big\rangle \Big\|_{p[T]}
\cdot 
\Big\| \big\langle \rw, \Upsilon \big\rangle \Big\|_{q[\Upsilon], \scG_{\alpha, \beta}[\Upsilon]} \cdot \eta^{\scG_{\alpha, \beta}[\cE[T]] - \gamma}. 
\end{align*}
\end{proof}

The combination of Propositions \ref{proposition:RCRP-2product} and \ref{proposition:ExpectationRCRP} yields the following wider result:

\begin{proposition}
\label{proposition:productRCRP}
Let $\alpha, \beta>0$ and $\gamma:=\inf\{ \scG_{\alpha, \beta}[T]: T \in \scF, \scG_{\alpha, \beta}[T]>1-\alpha \}$. Let $V$ be a vector space and let $\tilde{\scH}^{\gamma, \alpha, \beta}$ denote the $L^0\big( \Omega, \bP; V\big)$-module
$$
\tilde{\scH}^{\gamma, \alpha, \beta} = \bigoplus_{T\in \scF_0}^{\gamma, \alpha, \beta} L^0\Big( \Omega \times \Omega^{\times |H^T|}, \bP \times \bP^{\times |H^T|}; \lin\big( (\bR^d)^{\otimes |N^T|}, V \big) \Big) \cdot T
$$
so that $\big( \tilde{\scH}^{\gamma, \alpha, \beta}, \circledast, \rId, \Delta, \epsilon, \scS\big)$ is a coupled Hopf algebra. 

Let $r>1$ and let $a\in \A{n}$ . For each $i=0, ...n$, let $(p_i, q)$ be a collection of pairs of dual integrability functionals and suppose that
\begin{align}
\label{eq:proposition:productRCRP-pq-cond}
\sup_{T\in \scF^{\gamma, \alpha, \beta}} \frac{1}{q[T]} \leq \inf_{j=0,...,m[a]} \bigg( \frac{1}{r} + \sum_{\substack{i: \\a_i=j}} \frac{1}{p_i[\rId]}\bigg), 
\qquad
\sup_{j=0, ..., m[a]} \bigg( \frac{1}{r} + \sum_{\substack{i: \\a_i=j}} \frac{1}{p_i[\rId]}\bigg) \leq 1. 
\end{align}
We define $p_0:\scF^{\gamma, \alpha, \beta} \to [1, \infty)$ such that 
\begin{equation}
\label{eq:proposition:productRCRP-pq}
\frac{1}{p_0[\rId]} = \inf_{j=0, ..., m[a]} \bigg( \frac{1}{r} + \sum_{\substack{i: \\ a_i=j}} \frac{1}{p_i[\rId]} \bigg), 
\quad
\frac{1}{p_0[T]} = \frac{1}{p_0[\rId]} - \frac{1}{q[T]}. 
\end{equation}

Let $\rw$ be a $(\scH^{\gamma, \alpha, \beta}, p_0, q)$-probabilistic rough path. Let 
$$
f \in L^r\Big( \Omega\times \Omega^{\times m[a]}, \bP \times \bP^{\times m[a]}; \lin\big( (\bR^e)^{\otimes n}, V \big) \Big).
$$
For $i=1, ...,n$, let $\bX^i \in \cD_{\rw}^{\gamma, p_i, q}(\scH^{\gamma, \alpha, \beta})$ and let $\bY:[0,1]\to \scH^{\gamma-, \alpha, \beta}$ defined by
\begin{equation}
\label{eq:proposition:productRCRP(simple*)}
\Big\langle \bY_t, \rId\Big\rangle(\omega_0) = \bE^{1, ..., m[a]} \bigg[ f(\omega_0, ..., \omega_{m[a]}) \cdot \bigotimes_{i=1}^n \Big\langle \bX_t^i, \rId\Big\rangle(\omega_{a_i}) \bigg]. 
\end{equation}

Then there is a random controlled rough path $\bY\in \cD_{\rw}^{\gamma, p_0, q}\big( \tilde{\scH}^{\gamma, \alpha, \beta}\big)$ that satisfies Equation \eqref{eq:proposition:productRCRP(simple*)} and $\forall T \in \scF_0^{\gamma-, \alpha, \beta}$
\begin{align}
\nonumber
\Big\langle &\bY_t, T\Big\rangle(\omega_{0}, \omega_{H^T}) 
\\
\label{eq:proposition:productRCRP(simple)}
&= \sum_{\substack{T_1, ..., T_n\in \scF_0\\ \cE^a[ T_1, ..., T_n] = T}} \bE^{Z^a[T_1, ..., T_n]}\Big[ f\big(\omega_{0}, \omega_{\tilde{h}_1^T}, ..., \omega_{\tilde{h}_{m[a]}^T}\big) 
\cdot 
\bigotimes_{i=1}^n \big\langle \bX_t^i, T_i \big\rangle (\omega_{\tilde{h}_{a_i}^T}, \omega_{H^{T_i}}) \Big]
\end{align}
where for $T_1, ..., T_n\in \scF$ and $T=\cE^a[T_1, ..., T_n]$, the sets $\tilde{h}_j^{T}$ and $Z^a[T_1, ..., T_n]$ are defined as in Definition \ref{definition:Z-set}, see Equation \eqref{eq:omega_zero-zero_hyperedge}). 

Proposition \ref{proposition:ExpectationRCRP}
is an extension of Proposition \ref{proposition:RCRP-2product}, which is restricted to the case $n=2$ and $a=(0,0)$, and the proof of 
the former one precisely consists in iterating the latter one.
\end{proposition}

\begin{remark}
When we state Equation \eqref{eq:definition:RandomControlledRP2} for the random controlled rough path described in Proposition \ref{proposition:productRCRP}, we find that we additionally need to couple the ghost hyperedges according to the function described in Definition \eqref{definition:CouplingFunctions-Ghost}. We see this in Equation \eqref{eq:proposition:productRCRP:5} below.
\end{remark}

For Lions tree $T$, we use the notation
\begin{equation}
\label{eq:a_integral}
\Big\| \big\langle \bX, T \big\rangle(\omega_{a_i}) \Big\|_{[a], p} = \begin{cases}
\bE^{H^T}\bigg[ \Big| \big\langle \bX,  T \big\rangle(\omega_{0}, \omega_{H^T}) \Big|^p \bigg]^{\tfrac{1}{p}} \quad & \quad \mbox{if} \quad a_i=0,  
\\
\bE^{a_i}\bigg[ \bE^{H^T}\bigg[ \Big| \big\langle \bX,  T \big\rangle(\omega_{a_i}, \omega_{H^T}) \Big|^{p} \bigg] \bigg]^{\tfrac{1}{p}} \quad & \quad \mbox{if} \quad a_i>0. 
\end{cases}
\end{equation}

\begin{proof}
Firstly, we verify that $(p_0, q)$ as described in Equation \eqref{eq:proposition:productRCRP-pq} satisfy Definition \ref{definition:(dual)_integrab-functional}. 

Suppose that for each $i=1, ..., n$, $(p_i, q)$ satisfy Equation \eqref{eq:proposition:productRCRP-pq-cond}. Then by Equation \eqref{eq:proposition:productRCRP-pq}, $p_0[\rId]\in (1, \infty)$. Further, for any choice of $T\in \scF^{\gamma, \alpha, \beta}$, we have that $\tfrac{1}{p_0[\rId]} - \tfrac{1}{q[T]} \in (0,1)$ so that $p_0[T]\in (1, \infty)$. By construction, $p_0$ satisfies Equation \eqref{eq:IntegrabilityP1} so that $(p_0, q)$ is a pair of dual integrability functionals. 

Iterative applications of Proposition \ref{proposition:RCRP-2product} and Proposition \ref{proposition:ExpectationRCRP} mean that Equation \eqref{eq:proposition:productRCRP(simple*)} is a $p_0$-RCRP. This proof is to verify that Equation \eqref{eq:proposition:productRCRP(simple)} is the correct representation. 

Let $a \in \A{n}$. Define
\begin{equation}
\begin{split}
&\Big\langle \tilde{\bY}_t, \rId\Big\rangle (\omega_{0}, \omega_{1}, ..., \omega_{m[a]}) := \bigotimes_{i=1}^n \Big\langle \bX^i_t, \rId\Big\rangle(\omega_{a_i}), 
\\
&\Big\langle \bY_t, \rId \Big\rangle(\omega_0) := \bE^{1, ..., m[a]}\Big[ f\big(\omega_0, ..., \omega_{m[a]} \big) \cdot \big\langle \tilde{\bY}_t, \rId \big\rangle(\omega_0, ..., \omega_{m[a]}) \Big]. 
\end{split}
\end{equation}
Expanding as before, we have
\begin{align*}
\Big\langle \tilde{\bY}_{s,t}&, \rId\Big\rangle (\omega_{0}, \omega_{1}, ..., \omega_{m[a]}) = \bigg( \bigotimes_{i=1}^n \Big\langle \bX^i, \rId\Big\rangle(\omega_{a_i})  \bigg)_{s,t} 
\\
=&
\sum_{T\in \scF}^{\gamma-, \alpha, \beta} \sum_{\substack{T_1, ..., T_n\in \scF_0 \\ T=\cE^a[T_1, ..., T_n]}} 
\bigotimes_{i=1}^n \bE^{H^{T_i}}\bigg[ \Big\langle \bX_s^i, T_i \Big\rangle (\omega_{a_i}, \omega_{H^{T_i}})  
\cdot 
\Big\langle \rw_{s,t}, T_i \Big\rangle(\omega_{a_i}, \omega_{H^{T_k}})  \bigg]
\\
&+ \Big\langle \tilde{\bY}^{\sharp}_{s,t}, \rId\Big\rangle (\omega_{0}, ..., \omega_{m[a]})
\end{align*}
where
\begin{align*}
\Big\langle \tilde{\bY}^{\sharp}_{s,t}&, \rId\Big\rangle (\omega_{0}, ..., \omega_{m[a]})
\\
=& \sum_{k=1}^n \bigg( \bigotimes_{i=1}^{k-1} \Big\langle \bX_s^i + \fJ\big[ \bX^i \big]_{s,t}, \rId \Big\rangle(\omega_{a_i})\bigg) \otimes \Big\langle \bX_{s,t}^{i, {\sharp}}, \rId\Big\rangle(\omega_{a_k}) \otimes \bigg( \bigotimes_{i=k+1}^n \Big\langle \bX_{t}^i, \rId \Big\rangle(\omega_{a_i}) \bigg)
\\
&+\sum_{\substack{ T_1, ..., T_n\in \scF_0 \\ \scG_{\alpha, \beta}[\cE^a[T_1, ..., T_n]]\geq \gamma}} \bigotimes_{i=1}^n \bE^{H^{T_i}}\bigg[ \Big\langle \bX_s^i, T_i\Big\rangle(\omega_{a_i}, \omega_{H^{T_i}}) \cdot \Big\langle \rw_{s, t}, T_i \Big\rangle(\omega_{a_i}, \omega_{H^{T_i}}) \bigg]. 
\end{align*}

Next, we fix $T_1, ..., T_n$ and denote $\cE^a[T_1, ..., T_n] = T$. We relabel $\omega_{i} = \omega_{\tilde{h}_i^{\cE^a[T_1, ..., T_n]}}$ for $i=1, ..., m[a]$ where $\tilde{h}_i^{\cE^a[T_1, ..., T_n]}$ are the ghost hyperedges introduced in Definition \ref{definition:Z-set}. For brevity, we will denote these sets by $\tilde{h}_i^{T}$. 

We apply Proposition \ref{proposition:EAction} to observe that the random variable $\langle \rw_{s,t}, T \rangle$ is independent of $\omega_{\tilde{h}_j^T}$ for all $\tilde{h}_j^T \in Z^a[T_1, ..., T_n]$. Any $\tilde{h}_j^T \notin Z^a[T_1, ..., T_n]$ must be a hyperedge of the forest $T = \cE^a[T_1, ..., T_n]$. By interchangeably writing $\omega_0 = \omega_{\tilde{h}_0^T}$, we take expectations to get
\begin{align*}
\Big\langle \bY_{s,t}, \rId\Big\rangle (\omega_{0})
=& \sum_{T\in \scF}^{\gamma-, \alpha, \beta} \bE^{H^T} \bigg[ \sum_{\substack{T_1, ..., T_n \in \scF_0\\ \cE^a[T_1, ..., T_n] = T}} \hspace{-10pt} \bE^{Z^a[T_1, ..., T_n]}\Big[ f(\omega_{0}, \omega_{\tilde{h}_1^T}, ..., \omega_{\tilde{h}_{m[a]}^T}) 
\cdot 
\bigotimes_{i=1}^n \big\langle \bX_s^i, T_i\big\rangle( \omega_{\tilde{h}_{a_i}^T}, \omega_{H^{T_i}}) \Big]
\\
&\qquad 
 \cdot
\Big\langle \rw_{s, t}, T \Big\rangle (\omega_{0}, \omega_{H^T}) \bigg]
\\
&+ \bE^{1, ..., m[a]}\Big[ f(\omega_{0}, \omega_{1}, ..., \omega_{m[a]}) \cdot \big\langle \tilde{\bY}^{\sharp}_{s,t}, \rId\big\rangle (\omega_{0}, \omega_1, ..., \omega_{m[a]}) \Big], 
\end{align*}
which verifies Equation \eqref{eq:proposition:productRCRP(simple)}. 

Using Equation \eqref{eq:proposition:productRCRP-pq}, we can additionally verify that 
\begin{align*}
&\sup_{s, t\in[u,v]} \frac{\bE^{1, ..., m[a]}\Big[ f(\omega_0, \omega_1, ..., \omega_{m[a]}) \cdot \big\langle \tilde{\bY}_{s, t}^{\sharp}, \rId\big\rangle(\omega_0, \omega_1, ..., \omega_{m[a]})\Big] }{|t-s|^{\gamma}}
\\
&\leq \big\| f(\omega_0) \big\|_r \cdot \sum_{k=1}^n \bigg( \bigotimes_{i=1}^{k-1} \bigg\| \Big\| \big\langle \bX^i, \rId \big\rangle(\omega_{a_i}) \Big\|_{\infty} \bigg\|_{[a],p_i[\rId]} 
+ 
\bigg\| \Big\| \big\langle \fJ\big[ \bX^i \big], \rId \big\rangle(\omega_{a_i})\Big\|_{\alpha} \bigg\|_{[a],p_i[\rId]} \cdot |v-u|^{\alpha} \bigg)
\\
&\quad \otimes \bigg\| \Big\| \big\langle \bX^{i, {\sharp}}, \rId \big\rangle(\omega_{a_k}) \Big\|_{\gamma}\bigg\|_{[a],p_i[\rId]} 
\otimes 
\bigg( \bigotimes_{i=k+1}^n \bigg\| \Big\| \big\langle \bX^i, \rId \big\rangle(\omega_{a_i})\Big\|_{\infty} \bigg\|_{[a],p_i[\rId]} \bigg)
\\
&+ \big\| f(\omega_0) \big\|_r \cdot \sum_{\substack{T_1, ..., T_n \in \scF_0 \\ \scG_{\alpha, \beta}[\cE^a[T_1, ..., T_n]]\geq \gamma}}^{\gamma-, \alpha, \beta} \prod_{i=1}^n \bigg\| \Big\| \big\langle \bX^i, T_i\big\rangle(\omega_{a_i})\Big\|_{p_i[T_i], \infty} \bigg\|_{[a],p_i[T_i]} 
\\
&\quad\qquad \cdot \bigg\| \Big\| \big\langle \rw, \cE^a[T_1, ..., T_n] \big\rangle \Big\|_{q\big[ \cE^a[T_1, ..., T_n] \big], \scG_{\alpha, \beta}[\cE^a[T_1, ..., T_n]]} \bigg\|_{[a],q\big[ \cE^a[T_1, ..., T_n] \big]} \cdot |v-u|^{\scG_{\alpha, \beta}[\cE^a[T_1, ..., T_n]]-\gamma}
\\
&\quad < \infty \quad (\Omega_0, \bP)\mbox{-almost surely,}
\end{align*}
and similarly
\begin{align*}
\bE^0\Bigg[ \bigg| \sup_{s, t\in[0,1]}& \frac{\bE^{1, ..., m[a]}\Big[ f(\omega_0, \omega_1, ..., \omega_{m[a]}) \cdot \big\langle \tilde{\bY}_{s, t}^{\sharp}, \rId\big\rangle(\omega_0, \omega_1, ..., \omega_{m[a]})\Big] }{|t-s|^{\gamma}} \bigg|^{p_0[\rId]} \Bigg]^{\tfrac{1}{p_0[\rId]}}<\infty. 
\end{align*}

Next, we verify that an increment of Equation \eqref{eq:proposition:productRCRP(simple)} satisfies Definition \ref{definition:RandomControlledRP}. For any $Y \in \scF^{\gamma-, \alpha, \beta}$, we fix $(\omega_0, \omega_{H^Y}) \in \Omega \times \Omega^{\times |H^Y|}$ and get
\begin{align}
\nonumber
\bigg(& \sum_{\substack{Y_1, ..., Y_n \in \scF_0\\ \cE^a[Y_1, ..., Y_n] = Y}} \bE^{Z^a[Y_1, ..., Y_n]}\Big[ f(\omega_{0}, \omega_{\tilde{h}_1^Y}, ..., \omega_{\tilde{h}_{m[a]}^Y}) \cdot \bigotimes_{i=1}^n \big\langle \bX^{i}, Y_i \big\rangle(\omega_{\tilde{h}_{a_i}^Y}, \omega_{H^{Y_i}})\Big] \bigg)_{s, t}
\\
\nonumber
=&\sum_{\substack{Y_1, ..., Y_n \in \scF_0\\ \cE^a[Y_1, ..., Y_n] = Y}} 
\sum_{T\in \scF}^{\gamma-, \alpha, \beta} \sum_{\substack{T_1, ..., T_n \in \scF_0\\ T=\cE^a[T_1, ..., T_n]}}
\sum_{\Upsilon_1, ..., \Upsilon_n \in \scF_0}
\bigg( \prod_{i=1}^n  c\Big( T_i, \Upsilon_i, Y_i \Big) - \prod_{i=1}^n \delta_{\{ \Upsilon_i = \rId , T_i = Y_i\}} \bigg)
\\
\nonumber
&\quad 
\cdot \bE^{Z^a[Y_1, ..., Y_n]} \bigg[ \bE^{\bigcup_{i=1}^n E^{T_i, \Upsilon_i, Y_i}} \Big[ \bigotimes_{i=1}^n \big\langle \bX_s^i, T_i\big\rangle(\omega_{\tilde{h}_{a_i}^Y}, \omega_{\phi^{T_i, \Upsilon_i, Y_i}[H^{T_i}]}) 
\\
\label{eq:proposition:productRCRP:3}
&\quad \cdot \bigotimes_{i=1}^n \big\langle \rw_{s, t}, \Upsilon_i \big\rangle(\omega_{\tilde{h}_{a_i}^Y}, \omega_{\varphi^{T_i, \Upsilon_i, Y_i}[H^{\Upsilon_i}]}) \Big] \bigg]
+ 
\Big\langle \bY_{s, t}^{\sharp}, Y \Big\rangle( \omega_0, \omega_{H^{Y}}) 
\end{align}
where 
\begin{align}
\nonumber
&\Big\langle \bY_{s, t}^{\sharp}, Y \Big\rangle( \omega_0, \omega_{H^{Y}}) 
\\
\nonumber
&=
\sum_{\substack{Y_1, ..., Y_n \in \scF_0 \\ \cE^a[Y_1, ..., Y_n]=Y}} \bE^{Z^a[Y_1, ..., Y_n]} \bigg[ f(\omega_0, ..., \omega_{\tilde{h}_{m[a]}^Y}) \cdot \sum_{k=1}^n \bigg( \bigotimes_{i=1}^{k-1} \Big\langle \bX_s^i + \fJ[\bX]_{s, t}^i, Y_i \big\rangle(\omega_{\tilde{h}_{a_i}^Y}, \omega_{H^{Y_i}}) \bigg)
\\
\nonumber
&\qquad \otimes \Big\langle \bX_{s, t}^{k, {\sharp}}, Y_k \Big\rangle(\omega_{\tilde{h}_{a_k}^Y}, \omega_{H^{Y_k}}) 
\otimes 
\bigotimes_{i=k+1}^n \Big\langle \bX_t^i, Y_i \Big\rangle (\omega_{\tilde{h}_{a_i}^Y}, \omega_{H^{Y_i}}) \Big) \bigg]
\\
\nonumber
&+\sum_{\substack{Y_1, ..., Y_n \in \scF_0\\ \cE^a[Y_1, ..., Y_n] = Y}} 
\sum_{\substack{T_1, ..., T_n \in \scF_0\\ \scG_{\alpha, \beta}[\cE^a[T_1, ..., T_n]]\geq \gamma}}^{\gamma-, \alpha, \beta}
\sum_{\Upsilon_1, ..., \Upsilon_n \in \scF_0}
\bigg( \prod_{i=1}^n  c\Big( T_i, \Upsilon_i, Y_i \Big) - \prod_{i=1}^n \delta_{\{  T_i = Y_i, \Upsilon_i = \rId\}} \bigg)
\\
\nonumber
&\quad 
\cdot \bE^{Z^a[Y_1, ..., Y_n]} \bigg[ \bE^{\bigcup_{i=1}^n E^{T_i, \Upsilon_i, Y_i}} \Big[ \bigotimes_{i=1}^n \big\langle \bX_s, T_i\big\rangle(\omega_{\tilde{h}_{a_i}^Y}, \omega_{\phi^{T_i, \Upsilon_i, Y_i}[H^{T_i}]}) 
\\
\label{eq:proposition:productRCRP:4}
&\quad \cdot \bigotimes_{i=1}^n \big\langle \rw_{s, t}, \Upsilon_i \big\rangle(\omega_{\tilde{h}_{a_i}^Y}, \omega_{\varphi^{T_i, \Upsilon_i, Y_i}[H^{\Upsilon_i}]}) \Big] \bigg]
\end{align}

Thanks to Lemma \ref{lemma:TechnicalPhiVarphi}, we can replace $\phi^{T_i, \Upsilon_i, Y_i}$ by $\phi^{T, \Upsilon, Y}$. Next, using the same reasoning as Lemma \ref{lemma:ESets}, we note that for any $\tilde{h}_j^{Y} \notin Z^a[Y_1, ..., Y_n]$, we have that $\tilde{h}_j^{Y}\in H^{Y}$ so that $\omega_{\tilde{h}_j^{Y}}$ is fixed whereas for $\tilde{h}_j^{Y} \in Z^a[Y_1, ..., Y_n]$, we integrate over $\omega_{\tilde{h}_j^{Y}}$. We apply Proposition \ref{proposition:EAction} with the set $Z^a[Y_1, ..., Y_n]$ (instead of the set $\{1, ..., m[a]\}$) and relabel the ghost hyperedges according to the coupling function defined in Definition \ref{definition:CouplingFunctions-Ghost}. $\tilde{\phi}^{T, \Upsilon, Y}$ is the identity for all fixed $\omega_{\tilde{h}_j^{Y}}$ and for each probability space that is integrated over, $\tilde{\phi}^{T, \Upsilon, Y}$ divides these into the probability spaces that $\langle \rw_{s,t}, \cE^a[\Upsilon_1, ..., \Upsilon_n]\rangle$ is dependent on and independent of. 

This substitution motivates us to apply Equation \eqref{eq:lemma:ESets} to Equation \eqref{eq:proposition:productRCRP:3}, and additionally using  Lemma \ref{lemma:CoprodCountIdent} yields that 
\begin{align}
\nonumber
\Big\langle \bY_{s, t}&, Y\Big\rangle(\omega_{0}, \omega_{H^{Y}}) 
\\
\nonumber
&=\sum_{\substack{T\in \scF \\ \Upsilon \in \scF}}^{\gamma-,\alpha, \beta}
c'\Big( T, \Upsilon, Y\Big)
\cdot 
\bE^{E^{T, \Upsilon, Y}} \bigg[
\sum_{\substack{T_1, ..., T_n \in \scF_0\\ \cE^a[T_1, ..., T_n] = T }}
\bE^{Z^a[T_1, ..., T_n]}\Big[ f(\omega_{0}, \omega_{\tilde{\phi}^{T, \Upsilon, Y}[\tilde{h}_{1}^Y]}, ..., \omega_{\tilde{\phi}^{T, \Upsilon, Y}[\tilde{h}_{m[a]}^Y]}) 
\\
\nonumber
&\quad \cdot 
\bigotimes_{i=1}^n \Big\langle \bX_s, T_i\Big\rangle(\omega_{\tilde{\phi}^{T, \Upsilon, Y}[\tilde{h}_{a_i}^{Y}]}, \omega_{\phi^{T, \Upsilon, Y}[H^{T_i}]}) \Big] 
\cdot 
\Big\langle \rw_{s, t}, \Upsilon \Big\rangle (\omega_{0}, \omega_{\varphi^{T, \Upsilon, Y}[H^{\Upsilon}]}) \bigg]
\\
\label{eq:proposition:productRCRP:5}
&+ \Big\langle \bY_{s, t}^{\sharp}, Y \Big\rangle(\omega_0, \omega_{H^Y}). 
\end{align}
Thus we see that $\langle \bY_{s, t}, Y\rangle$ satisfies Equation \eqref{eq:definition:RandomControlledRP2} with the addition of the coupling to account for the ghost hyperedges.  

To conclude, thanks to Equation \eqref{eq:proposition:productRCRP:4} we have that
\begin{align*}
&\sup_{s, t\in[u,v]} \frac{\bE^{H^Y}\bigg[ \Big| \big\langle \bY_{s,t}^{\sharp}, Y \big\rangle(\omega_0, \omega_{H^Y}) \Big|^{p_0[Y]} \bigg]^{\tfrac{1}{p_0[Y]}}}{|t-s|^{\gamma - \scG_{\alpha, \beta}[Y]}} 
\leq 
\big\| f(\omega_0) \big\|_r \cdot \sum_{\substack{Y_1, ..., Y_n \in \scF_0 \\ \cE^a[Y_1, ..., Y_n]=Y}} \sum_{k=1}^n
\\
&\quad \bigg( \bigotimes_{i=1}^{k-1} \bigg\| \Big\| \big\langle\bX^i, Y_i \big\rangle(\omega_{a_i})\Big\|_{\infty} \bigg\|_{[a],p_i[Y_i]} 
 + \bigg\| \Big\| \big\langle  \fJ[\bX^i], Y_i \big\rangle(\omega_{a_i}) \Big\|_{\alpha} \bigg\|_{[a],p_i[Y_i]} \cdot |v-u|^{\alpha} \bigg) 
\\
&\qquad \otimes \bigg\| \Big\| \big\langle \bX^{k,{\sharp}}, Y_k\big\rangle(\omega_{a_k})\Big\|_{\gamma-\scG_{\alpha, \beta}[Y_k]} \bigg\|_{[a],p_k[Y_k]} \cdot |v-u|^{\scG_{\alpha, \beta}[Y] - \scG_{\alpha, \beta}[Y_k]} 
\\
&\qquad \otimes \bigg( \bigotimes_{i=k+1}^n \bigg\| \Big\| \big\langle \bX^i, Y_i\big\rangle(\omega_{a_i})\Big\|_{\infty} \bigg\|_{[a],p_i[Y_i]} \bigg)
\\
&+ \big\| f(\omega_0) \big\|_r \cdot \hspace{-20pt} \sum_{\substack{Y_1, ..., Y_n \in \scF_0\\ \cE^a[Y_1, ..., Y_n] = Y}} 
\sum_{\substack{T_1, ..., T_n \in \scF_0\\ \scG_{\alpha, \beta}[\cE^a[T_1, ..., T_n]]\geq \gamma}}^{\gamma-, \alpha, \beta}
\sum_{\Upsilon_1, ..., \Upsilon_n \in \scF_0}
\Big( \prod_{i=1}^n  c\big( T_i, \Upsilon_i, Y_i \big) - \prod_{i=1}^n \delta_{\{ \Upsilon_i = \rId , T_i = Y_i\}} \Big)
\\
&\quad \cdot \prod_{i=1}^n \bigg\| \sup_{s\in[u, v]} \Big\| \big\langle \bX_s, T_i\big\rangle(\omega_{a_i})\Big\|_{p_i[T_i]} \bigg\|_{[a],p_i[T_i]} 
\\
&\qquad \cdot \sup_{s, t\in [u,v]} \frac{ \prod_{i=1}^n  \bigg\| \Big\| \big\langle \rw, \Upsilon_i\big\rangle(\omega_{a_i}) \Big\|_{q[\Upsilon_i] }\bigg\|_{[a],q[\Upsilon_i]}}{|t-s|^{\scG_{\alpha, \beta}[\cE^a[\Upsilon_1, ..., \Upsilon_n]]}} \cdot |v-u|^{\scG_{\alpha, \beta}[\cE^a[T_1, ..., T_n]]-\gamma}
\\
&\quad < \infty \quad (\Omega_0, \bP)\mbox{-almost surely}
\end{align*}
and integrating through yields
\begin{align*}
&\bE^0\Bigg[ \sup_{s, t\in[0,1]} \frac{\bE^{H^Y}\bigg[ \Big| \big\langle \bY_{s,t}^{\sharp}, Y \big\rangle(\omega_0, \omega_{H^Y}) \Big|^{p_0[Y]} \bigg]}{|t-s|^{p_0[Y](\gamma - \scG_{\alpha, \beta}[Y])}} \Bigg] < \infty. 
\end{align*}
\end{proof}

\begin{remark}
This result illustrates that it is not really the choice of integrability functional $p$ that matters when defining a probabilistic rough path but rather the dual integrability functional $q$. Indeed, there can be multiple choices of $p$ for the same choice of $q$. This can be seen in Definition \ref{definition:(dual)_integrab-functional}, but this is the first time in this work where we use this property. 
\end{remark}

\subsubsection{Smooth functions of RCRPs and Lions-Taylor expansions}
\label{subsubsec:SfRCRPs,LTE}

Next, we use Proposition \ref{proposition:productRCRP} to find an expression for a Taylor expansion of some smooth function of a RCRP. These results demonstrate the interconnectivity between the Lions-Taylor expansion (from Section \ref{section:TaylorExpansions}), and Lions trees and probabilistic rough paths (from \cite{salkeld2021Probabilistic}). 

To streamline notation, we denote
$$
\scG_{\alpha, \beta}[a]:=\alpha\cdot l[a]_0 + \beta\cdot(|a| - l[a]_0). 
$$

In order to prove Theorem \ref{theorem:ContinIm-RCRPs}, we will first prove Equation \eqref{eq:theorem:ContinIm-RCRPs} followed by Equation \eqref{eq:theorem:ContinIm-RCRPs-Est}. To do this, first we reframe Theorem \ref{theorem:LionsTaylor2} and Corollary \ref{corollary:LionsTaylor2} as follows:
\begin{corollary}
\label{corollary:LionsTaylor3}
Let $\alpha, \beta>0$ and let 
$$
\gamma:=\inf \big\{ \alpha  i + \beta j: (i, j)\in \bN_0^{\times 2}, \alpha i + \beta j > 1-\alpha\big\}, 
\quad
n:=\sup \big\{ m\in \bN_0: m < \tfrac{\gamma}{\alpha\wedge \beta} \big\}.
$$
Let $f\in C_b^{n, (n)}\big( \bR^d \times \cP_2(\bR^d) \big)$. 

Let $(p_x, q)$ and $(p_y, q)$ be two dual pairs of integrability functionals that satisfy \eqref{eq:theorem:ContinIm-RCRPs4}. Let $\rw$ be a $(\scH^{\gamma, \alpha, \beta}, p, q)$-probabilistic rough paths and let $\bX\in \cD_{\rw}^{\gamma, p_x, q}$ and $\bY\in \cD_{\rw}^{\gamma, p_y, q}$. Then with the same notations as in Definition \ref{definition:B_set} and \eqref{eq:omega_zero-zero_hyperedge}
\begin{align}
\nonumber
f&\Big( \big\langle \bX, \rId\big\rangle(\omega_0), \cL^{\langle \bY, \rId\rangle} \Big)_{s,t}
\\
\nonumber
&= \sum_{a\in A^{\alpha, \beta}} \frac{1}{|a|!} \bE^{1, ..., m[a]}\bigg[ \partial_a f\Big( \big\langle \bX_s, \rId\big\rangle(\omega_0), \cL^{\langle \bY_s, \rId\rangle}, \big\langle \bY_s, \rId\big\rangle(\omega_{1}), ..., \big\langle \bY_s, \rId\big\rangle(\omega_{m[a]})\Big) 
\\
\label{eq:corollary:LionsTaylor2.1}
&\qquad \cdot 
\bigotimes_{r=1}^{|a|} \Big\langle \big[ \bX, \bY \big]_{s, t}, \rId \Big\rangle(\omega_{a_r}) \bigg] + \fR_{s, t}(\omega_0), 
\end{align}
where $\exists \cN \subset \Omega$ such that $\bP\big[ \cN \big] = 0$ and $\forall \omega_0 \in \Omega \backslash \cN$, 
$$
\fR_{s, t}(\omega_0) = O\Big( |t-s|^\gamma \Big), 
$$
with the constant underpinning the Landau notation being implicitly random, and 
\begin{equation}
\label{eq:corollary:LionsTaylor2.1-int}
\bE^0\Bigg[ \bigg( \sup_{s, t\in[0,1]} \frac{ \big| \fR_{s, t}(\omega_0) \big|}{|t-s|^{\gamma}} \bigg)^{\tfrac{p_x[\rId]}{n+1}} \Bigg]<\infty. 
\end{equation}

Next, for $i\in \{1, ..., n\}$, suppose that $a\in \A{i}$. Similar to before, using the notations from Definition \ref{definition:B_set}
\begin{align}
\nonumber
&\partial_a f\Big( \big\langle \bX, \rId\big\rangle(\omega_0), \cL^{\langle \bY, \rId\rangle}, \big\langle \bY, \rId\big\rangle(\omega_{1}), ... \big\langle \bY, \rId\big\rangle(\omega_{m[a]}) \Big)_{s, t}
\\
\nonumber
&= \sum_{\overline{a}\in A^{\alpha, \beta|a}} \frac{1}{(|\overline{a}| - |a|)!} \bE^{m[a]+1, ..., m[\overline{a}]} \bigg[ \partial_{\overline{a}} f\Big( \big\langle \bX_s, \rId\big\rangle(\omega_0), \cL^{\langle \bY_s, \rId\rangle}, \big\langle \bY_s, \rId\big\rangle(\omega_1), ..., \big\langle \bY_s, \rId\big\rangle(\omega_{m[\overline{a}]}) \Big)
\\
\label{eq:corollary:LionsTaylor2.2}
&\qquad  \cdot \bigotimes_{r=|a|+1}^{|\overline{a}|} \Big\langle \big[ \bX, \bY\big]_{s, t}, \rId \Big\rangle(\omega_{a_{r}}) \bigg] + \fR_{s, t}^a(\omega_0, ..., \omega_{m[a]})
\end{align}
where $\exists \cN^a\subset \Omega \times \Omega^{\times m[a]}$ such that $\bP \times \bP^{\times m[a]}\big[ \cN^a \big]=0$ and $\forall (\omega_0, ..., \omega_{m[a]}) \in \big( \Omega \times \Omega^{\times m[a]} \big) \backslash \cN^a$, 
$$
\fR_{s, t}^a(\omega_0, ..., \omega_{m[a]}) = O \Big( | t-s |^{\gamma-\scG_{\alpha, \beta}[a]} \Big), 
$$
and 
\begin{equation}
\label{eq:corollary:LionsTaylor2.2-int}
\bE^{0, 1, ..., m[a]}\Bigg[ \bigg( \sup_{s, t\in[0,1]} \frac{\big| \fR_{s, t}^a(\omega_0, \omega_1, ..., \omega_{m[a]}) \big|}{|t-s|^{\gamma - \scG_{\alpha, \beta}[a]}}\bigg)^{\tfrac{p_x[\rId]}{n+1-|a|}} \Bigg] < \infty. 
\end{equation}
\end{corollary}

\begin{proof}
Firstly, $p_y[\rId]> n+1$ so that 
$$
\bE\bigg[ \Big| \big\langle \bY_t, \rId\big\rangle(\omega) \Big|^{n+1} \bigg]< \infty. 
$$
Therefore, by Theorem \ref{theorem:LionsTaylor2} we have
\begin{align}
\nonumber
f\Big( \big\langle \bX, \rId \big\rangle&(\omega_0), \cL^{\langle \bY, \rId\rangle} \Big)_{s, t}
\\
\nonumber
=& \sum_{a\in A^{\alpha, \beta} } \frac{1}{|a|!} \bE^{1, ..., m[a]}\bigg[ \partial_a f\Big( \big\langle \bX_s, \rId \big\rangle (\omega_0), \cL^{\langle \bY_s, \rId\rangle}, \big\langle \bY_s, \rId \big\rangle (\omega_1), ..., \big\langle \bY_s, \rId \big\rangle (\omega_{m[a]}) \Big) 
\\
\label{eq:theorem:ContinIm-RCRPs1.1}
&\qquad \cdot \bigotimes_{r=1}^{|a|} \Big\langle \big[ \bX, \bY \big]_{s,t}, \rId\Big\rangle (\omega_{a_r}) \bigg]
\\
\nonumber
&+ \sum_{a\in A_\ast^{\alpha, \beta}} \frac{1}{|a|!} \bE^{1, ..., m[a]}\bigg[ \partial_a f\Big( \big\langle \bX_s, \rId \big\rangle (\omega_0), \cL^{\langle \bY_s, \rId\rangle}, \big\langle \bY_s, \rId \big\rangle (\omega_1), ..., \big\langle \bY_s, \rId \big\rangle (\omega_{m[a]}) \Big) 
\\
\label{eq:theorem:ContinIm-RCRPs1.2}
&\qquad \cdot \bigotimes_{r=1}^{|a|} \Big\langle \big[ \bX, \bY \big]_{s,t}, \rId\Big\rangle (\omega_{a_r}) \bigg]
\\
\nonumber
&+ \frac{1}{ (n -1)!} \sum_{a \in \A{n} } \bE^{1, ..., m[a]}\bigg[ f^a \Big[ \big\langle \bX_s, \rId \big\rangle (\omega_0), \big\langle \bX_t, \rId \big\rangle (\omega_0), \Pi^{\langle \bY_s, \rId\rangle, \langle \bY_t, \rId\rangle} \Big](\omega_1, ..., \omega_{m[a]}) 
\\
\label{eq:theorem:ContinIm-RCRPs1.3}
&\qquad \cdot \bigotimes_{r=1}^{n} \Big\langle \big[ \bX, \bY \big]_{s, t}(\omega_{a_r}), \rId\Big\rangle \bigg], 
\end{align}
where
\begin{align}
\nonumber
f^a \Big[ \big\langle \bX_s, \rId \big\rangle& (\omega_0), \big\langle \bX_t, \rId \big\rangle (\omega_0), \Pi^{\langle \bY_s, \rId\rangle, \langle \bY_t, \rId\rangle} \Big](\omega_1, ..., \omega_{m[a]})
\\
\nonumber
=&\int_0^1 \bigg( \partial_a f\Big( \big\langle \bX_s + \xi \cdot \bX_{s,t}, \rId\big\rangle (\omega_0) , \Pi_\xi^{\langle \bY_s, \rId\rangle, \langle \bY_t, \rId\rangle}, ..., \big\langle \bY_s + \xi \cdot \bY_{s,t}, \rId \big\rangle (\omega_{m[a]})  \Big) 
\\
\label{eq:theorem:ContinIm-RCRPs1.2_rem1}
&- \partial_a f\Big( \big\langle \bX_s, \rId\big\rangle (\omega_0), \cL^{\langle \bY_s, \rId\rangle}, ..., \big\langle \bY_s, \rId \big\rangle (\omega_{m[a]}) \Big)\bigg) \cdot (1-\xi)^{n-1} d\xi. 
\end{align}
We emphasise that the summations in Equation \eqref{eq:theorem:ContinIm-RCRPs1.1} and \eqref{eq:theorem:ContinIm-RCRPs1.2} are distinct. 

The terms from Equation \eqref{eq:theorem:ContinIm-RCRPs1.2_rem1} correspond to those of the remainder terms \eqref{eq:theorem:LionsTaylor2_remainder1_1}. Using Equation  \eqref{eq:theorem:LionsTaylor2_Rem1}, we conclude that all the terms in \eqref{eq:theorem:ContinIm-RCRPs1.2} and \eqref{eq:theorem:ContinIm-RCRPs1.3} will be $O\Big( |t-s|^\gamma \Big)$ and
\begin{align}
\nonumber
\big| \fR_{s, t}&(\omega_0) \big| \leq \sum_{a \in A_\ast^{\alpha, \beta}} \frac{1}{|a|!} \Big\| \partial_a f \Big\|_\infty \cdot \Big| \big\langle \bX_{s, t}, \rId\big\rangle(\omega_0) \Big|^{l[a]_0} \cdot \prod_{r=1}^{m[a]} \tilde{\bE}\bigg[ \Big| \big\langle \bY_{s, t}, \rId \big\rangle(\tilde{\omega}) \Big|^{l[a]_r} \bigg]
\\
\nonumber
&+\frac{1}{n!} \sum_{a\in \A{n}}  \Big\| \partial_a f \Big\|_{\lip, 0} \cdot \Big| \big\langle \bX_{s, t}, \rId\big\rangle(\omega_0) \Big|^{l[a]_0+1} \cdot \prod_{r=1}^{m[a]} \tilde{\bE}\bigg[ \Big| \big\langle \bY_{s, t}, \rId\big\rangle(\tilde{\omega}) \Big|^{l[a]_r} \bigg]
\\
\nonumber
&+\frac{1}{n!} \sum_{a\in \A{n}} \Big\| \partial_a f \Big\|_{\lip, \mu} \cdot \Big| \big\langle \bX_{s, t}, \rId\big\rangle(\omega_0) \Big|^{l[a]_0} \cdot \tilde{\bE}\bigg[ \Big| \big\langle \bY_{s, t}, \rId\big\rangle(\tilde{\omega}) \Big|\bigg] \cdot \prod_{r=1}^{m[a]} \tilde{\bE}\bigg[ \Big| \big\langle \bY_{s, t}, \rId\big\rangle(\tilde{\omega}) \Big|^{l[a]_r} \bigg]
\\
\label{eq:corollary:LionsTaylor3:remUpper}
&+\frac{1}{n!} \sum_{a\in \A{n}} \sum_{j=1}^{m[a]} \Big\| \partial_a f \Big\|_{\lip, j} \cdot \Big| \big\langle \bX_{s, t}, \rId\big\rangle(\omega_0) \Big|^{l[a]_0} \cdot \prod_{r=1}^{m[a]} \tilde{\bE}\bigg[ \Big| \big\langle \bY_{s, t}, \rId\big\rangle(\tilde{\omega}) \Big|^{l[a]_r + \delta_{j=r}} \bigg]
\end{align}
where $\| \partial_a f\|_\infty$ is defined as in Equation \eqref{eq:definition:FunctionNorm1} and $\| \partial_a f\|_{\lip, 0}$, $\| \partial_a f \|_{\lip, \mu}$ and $\| \partial_a f\|_{\lip, j}$ are defined as in Equation \eqref{eq:definition:FunctionNorm2}. 

By choosing $a=(0, ..., 0) \in \A{n}$, we get the least integrable terms but even this satisfies Equation \eqref{eq:corollary:LionsTaylor2.1-int}. 

Similarly, by Corollary \ref{corollary:LionsTaylor2} we have that 
\begin{align}
\nonumber
\partial_a f\Big( \big\langle \bX, \rId\big\rangle& (\omega_0), \cL^{\langle \bY,\rId\rangle}, \big\langle \bY, \rId \big\rangle (\omega_{\tilde{h}_1^Y}) ,..., \big\langle \bY, \rId \big\rangle (\omega_{\tilde{h}_{m[a]}^Y }) \Big)_{s, t}
\\
\nonumber
=& \sum_{\overline{a}\in A^{\alpha, \beta|a}} \frac{1}{(|\overline{a}| - |a|)!} \bE^{m[a]+1, ..., m[\overline{a}]} \bigg[ \partial_{\overline{a}} f\Big( \big\langle \bX_s, \rId \big\rangle (\omega_0), \cL^{\langle \bY_s, \rId\rangle}, ..., \big\langle \bY_s, \rId \big\rangle (\omega_{\tilde{h}_{m[a]}^Y} ) 
\\
\label{eq:theorem:ContinIm-RCRPs2.1}
&\quad \cdot \big\langle \bY_s, \rId \big\rangle (\omega_{m[a]+1}), ..., \big\langle \bY_s, \rId \big\rangle (\omega_{m[\overline{a}]}) \Big) \cdot \bigotimes_{r=|a|+1}^{|\overline{a}|} \Big\langle \big[ \bX, \bY \big]_{s, t}, \rId\Big\rangle (\omega_{\overline{a}_r}) \bigg]
\\
\nonumber
&+ \sum_{\overline{a} \in A_{\ast}^{\alpha, \beta|a}} \frac{1}{(|\overline{a}| - |a|)!} \bE^{m[a]+1, ..., m[\overline{a}]} \bigg[ \partial_{\overline{a}} f\Big( \big\langle \bX_s, \rId \big\rangle (\omega_0), \cL^{\langle \bY_s, \rId\rangle}, ..., \big\langle \bY_s, \rId \big\rangle (\omega_{\tilde{h}_{m[a]}^Y} ) 
\\
\label{eq:theorem:ContinIm-RCRPs2.2}
&\quad \cdot \big\langle \bY_s, \rId \big\rangle (\omega_{m[a]+1}), ..., \big\langle \bY_s, \rId \big\rangle (\omega_{m[\overline{a}]}) \Big) \cdot \bigotimes_{r=|a|+1}^{|\overline{a}|} \Big\langle [\bX, \bY]_{s, t}, \rId\Big\rangle (\omega_{\overline{a}_r}) \bigg]
\\
\nonumber
&+\sum_{\overline{a}\in \A{n|a}} \bE^{m[a]+1, ..., m[\overline{a}]} \bigg[ \Big( \partial_a f\Big)^{\overline{a}} (\omega_0, \omega_{\tilde{h}_1^Y}, ..., \omega_{\tilde{h}_{m[a]}^Y}, \omega_{m[a]+1}, ..., \omega_{m[\overline{a}]}) 
\\
\label{eq:theorem:ContinIm-RCRPs2.3}
&\quad 
\cdot \bigotimes_{r=|a|+1}^{n} \Big\langle \big[ \bX, \bY \big]_{s, t}, \rId\Big\rangle (\omega_{\overline{a}_r}) \bigg]
\end{align}
where
\begin{align}
\nonumber
&\Big( \partial_a f \Big)^{\overline{a}} (\omega_0, \omega_{\tilde{h}_1^Y}, ..., \omega_{\tilde{h}_{m[a]}^Y}, \omega_{m[a]+1}, ..., \omega_{m[\overline{a}]})
\\
\nonumber
&=\int_0^1 \bigg( \partial_{\overline{a}} f\Big( \big\langle \bX_s, \rId\big\rangle (\omega_0) + \xi \big\langle \bX_{s,t}, \rId \big\rangle (\omega_0), \Pi_\xi^{\langle \bY_s, \rId\rangle, \langle \bY_t, \rId\rangle}, ..., \big\langle \bY_s, \rId \big\rangle (\omega_{m[\overline{a}]}) + \xi \big\langle \bY_{s,t}, \rId \big\rangle (\omega_{m[\overline{a}]}) \Big) 
\\
\label{eq:theorem:ContinIm-RCRPs1.2_rem2}
&\quad
- \partial_{\overline{a}} f\Big( \big\langle \bX_s, \rId\big\rangle (\omega_0), \cL^{\langle \bY_s, \rId\rangle}, ..., \big\langle \bY_s, \rId\big\rangle (\omega_{m[\overline{a}]}) \Big)\bigg) \cdot (1-\xi)^{n-1-|a|} d\xi. 
\end{align}
The terms from Equation \eqref{eq:theorem:ContinIm-RCRPs1.2_rem2} correspond to those of the remainder terms \eqref{eq:corollary:LionsTaylor2_remainder1_1}. 

For $\overline{a} \in A^{\alpha, \beta|a}$ and for $k=0, ..., m[\overline{a}]$, we denote
$$
l[\overline{a}|a]_k = \sum_{i=|a|+1}^{|\overline{a}|} {\mathbf 1}_{k}(a_{i}). 
$$

All terms in \eqref{eq:theorem:ContinIm-RCRPs2.2} and \eqref{eq:theorem:ContinIm-RCRPs2.3} will be $O\Big( |t-s|^{\gamma - \scG_{\alpha, \beta}[a]} \Big)$ and a similar argument as before implies 
\begin{align}
\nonumber
\big| \fR_{s, t}^{a}&(\omega_0, \omega_1, ..., \omega_{m[a]}) \big|
\\
\nonumber
\leq& 
\sum_{\overline{a} \in A_{\ast}^{\alpha, \beta|a}} \frac{1}{(|\overline{a}| - |a|)!} \cdot \Big\| \partial_{\overline{a}} f \Big\| \cdot \Big| \big\langle \bX_{s, t}, \rId\big\rangle(\omega_0) \Big|^{l[\overline{a}|a]_0} \cdot \prod_{r=1}^{m[a]} \Big| \big\langle \bY_{s, t}, \rId \big\rangle(\omega_{r}) \Big|^{l[\overline{a}|a]_r} 
\\
\nonumber
&\qquad \cdot \prod_{r=m[a]+1}^{m[\overline{a}]} \tilde{\bE}\bigg[ \Big| \big\langle \bY_{s, t}, \rId \big\rangle(\tilde{\omega}) \Big|^{l[\overline{a}|a]_r} \bigg]
\\
\nonumber
&+\frac{1}{(n-|a|)!} \sum_{\overline{a}\in \A{n|a}} \Big\| \partial_{\overline{a}} f \Big\|_{\lip, 0} \cdot \Big| \big\langle \bX_{s, t}, \rId\big\rangle(\omega_0) \Big|^{l[\overline{a}|a]_0 + 1} \cdot \prod_{r=1}^{m[a]} \Big| \big\langle \bY_{s, t}, \rId \big\rangle(\omega_r) \Big|^{l[\overline{a}|a]_r} 
\\
\nonumber
&\qquad \cdot \prod_{r=1}^{m[\overline{a}]} \tilde{\bE}\bigg[ \Big| \big\langle \bY_{s, t}, \rId \big\rangle(\tilde{\omega}) \Big|^{l[\overline{a}|a]_r} \bigg]
\\
\nonumber
&+\frac{1}{(n-|a|)!} \sum_{\overline{a}\in \A{n|a}} \Big\| \partial_{\overline{a}} f \Big\|_{\lip, \mu} \cdot \Big| \big\langle \bX_{s, t}, \rId\big\rangle(\omega_0) \Big|^{l[\overline{a}|a]_0} \cdot \tilde{\bE}\bigg[ \Big| \big\langle \bY_{s, t}, \rId \big\rangle(\tilde{\omega}) \Big| \bigg]
\\
\nonumber
&\qquad \cdot \prod_{r=1}^{m[a]} \Big| \big\langle \bY_{s, t}, \rId \big\rangle(\omega_r) \Big|^{l[\overline{a}|a]_r} \cdot \prod_{r=m[a]+1}^{m[\overline{a}]} \tilde{\bE}\bigg[ \Big| \big\langle \bY_{s, t}, \rId \big\rangle(\tilde{\omega}) \Big|^{l[\overline{a}|a]_r} \bigg]
\\
\nonumber
&+\frac{1}{(n-|a|)!} \sum_{\overline{a}\in \A{n|a}} \sum_{j=1}^{m[\overline{a}]} \Big\| \partial_{\overline{a}} f \Big\|_{\lip, j} \cdot \Big| \big\langle \bX_{s, t}, \rId\big\rangle(\omega_0) \Big|^{l[\overline{a}|a]_0} \cdot \prod_{r=1}^{m[a]} \Big| \big\langle \bY_{s, t}, \rId \big\rangle(\omega_r) \Big|^{l[\overline{a}|a]_r} 
\\
\label{eq:corollary:LionsTaylor3:remUpper_Y}
&\qquad \cdot \prod_{r=1}^{m[\overline{a}]} \tilde{\bE}\bigg[ \Big| \big\langle \bY_{s, t}, \rId \big\rangle(\tilde{\omega}) \Big|^{l[\overline{a}|a]_r + \delta_{j=r} } \bigg]. 
\end{align}
where $\| \partial_a f\|_\infty$ is defined as in Equation \eqref{eq:definition:FunctionNorm1} and $\| \partial_a f\|_{\lip, 0}$, $\| \partial_a f \|_{\lip, \mu}$ and $\| \partial_a f\|_{\lip, j}$ are defined as in Equation \eqref{eq:definition:FunctionNorm2}. 

Equation \eqref{eq:corollary:LionsTaylor2.2-int} follows by integrating over each probability space and we conclude. 
\end{proof}

Our first step in the proof of Theorem \ref{theorem:ContinIm-RCRPs}
is to build on Corollary \ref{corollary:LionsTaylor3} in order to prove \eqref{eq:theorem:ContinIm-RCRPs}:

\begin{proof}[Proof of Equation \eqref{eq:theorem:ContinIm-RCRPs}]
Let $T\in \scF^{\gamma, \alpha, \beta}$ and define $p_z[T]$ according to Equation \eqref{eq:theorem:ContinIm-RCRPs3}. By assumption, $p_z[T] \in (1, \infty)$. Further, $p_z[\rId] \in (1, \infty)$ and $\tfrac{1}{p_z[\rId]} = \tfrac{1}{p_z[T]} + \tfrac{1}{q[T]}$ so that Equation \eqref{eq:IntegrabilityP1} is satisfied. 

In order to draw comparison with classical techniques (see for instance \cite{gubinelli2010ramification}*{Lemma 8.4}), we denote for $i=1, ..., n$
$$
A_{i}^{\alpha, \beta} = \big\{ a\in A^{\alpha, \beta}: |a|=i\big\}, 
\quad
A_{i}^{\alpha, \beta|a} = \big\{ \overline{a}\in A^{\alpha, \beta|a}: |\overline{a}|=i \big\}, 
$$

We apply Corollary \ref{corollary:LionsTaylor3} (and recalling $\fJ$ as introduced in Equation \eqref{eq:Jet_Operator}) to obtain
\begin{align}
\nonumber
f\Big( \big\langle \bX, \rId\big\rangle&(\omega_0), \cL^{\langle \bY, \rId\rangle} \Big)_{s, t}
\\
\nonumber
=& \sum_{a\in A^{\alpha, \beta}} \frac{1}{|a|!} \bE^{1, ..., m[a]}\bigg[
\partial_a f\Big( \big\langle \bX_{s}, \rId \big\rangle (\omega_0), \cL^{\langle \bY_s, \rId\rangle}, ..., \big\langle \bY_{s}, \rId \big\rangle (\omega_{\tilde{h}_{m[a]}^T}) \Big) 
\\
\label{eq:theorem:ContinIm-RCRPs1.1.5}
&\quad \cdot 
\bigotimes_{r=1}^{|a|} \Big\langle \fJ\big[ \bX, \bY\big]_{s, t}, \rId\Big\rangle(\omega_{a_r}) \bigg]
\\
\nonumber
&+ \sum_{a\in A^{\alpha, \beta}} \frac{1}{|a|!} \bE^{1, ..., m[a]}\bigg[
\partial_a f\Big( \big\langle \bX_{s}, \rId \big\rangle (\omega_0), \cL^{\langle \bY_s, \rId\rangle}, ..., \big\langle \bY_{s}, \rId \big\rangle (\omega_{\tilde{h}_{m[a]}^T}) \Big) 
\\
\label{eq:theorem:ContinIm-RCRPs1.1.3}
&\quad \cdot \bigg( \bigotimes_{r=1}^{|a|} \Big\langle \big[ \bX, \bY\big]_{s, t}, \rId\Big\rangle(\omega_{a_r}) - \bigotimes_{r=1}^{|a|} \Big\langle \fJ\big[ \bX, \bY\big]_{s, t}, \rId\Big\rangle(\omega_{a_r}) \bigg)\bigg]
\\
\label{eq:theorem:ContinIm-RCRPs1.1.4}
&+ 
\fR_{s, t}(\omega_0). 
\end{align}
Thanks to Corollary \ref{corollary:LionsTaylor3}, the terms of Equation \eqref{eq:theorem:ContinIm-RCRPs1.1.4} will be of of order $O\big( |t-s|^{\gamma} \big)$. 

An application of Lemma \ref{lemma:TechLem_Products} allows us to write
\begin{align*}
\bigg( \bigotimes_{r=1}^{|a|}& \Big\langle \big[ \bX, \bY\big]_{s, t}, \rId\Big\rangle(\omega_{a_r}) - \bigotimes_{r=1}^{|a|} \Big\langle \fJ\big[ \bX, \bY\big]_{s, t}, \rId\Big\rangle(\omega_{a_r}) \bigg)
\\
=& \sum_{k=1}^{|a|} \bigg( \bigotimes_{r=1}^{k-1} \Big\langle \fJ\big[ \bX, \bY\big]_{s, t}, \rId\Big\rangle(\omega_{a_r}) \bigg) \otimes \Big\langle \big[\bX, \bY\big]_{s, t}^{\sharp}, \rId\Big\rangle(\omega_{a_k}) \otimes \bigg( \bigotimes_{r=k+1}^{|a|} \Big\langle \big[ \bX, \bY\big]_{s, t}, \rId\Big\rangle(\omega_{a_r}) \bigg)
\end{align*}
so that Equation \eqref{eq:theorem:ContinIm-RCRPs1.1.3} will be of order $O\big( |t-s|^{\gamma} \big)$ too.

By expanding out Equation \eqref{eq:theorem:ContinIm-RCRPs1.1.5}, we get
\begin{align}
\nonumber
f\Big( \big\langle \bX&, \rId\big\rangle(\omega_0), \cL^{\langle \bY, \rId\rangle} \Big)_{s, t}
\\
\nonumber
=&\sum_{T\in \scF}^{\gamma-, \alpha, \beta} \sum_{a\in A^{\alpha, \beta}} \frac{1}{|a|!} \bE^{H^T}\Bigg[ \sum_{\substack{T_1, ..., T_{|a|}\in \scF \\ T=\cE^a[T_1, ..., T_{|a|}]}} \bE^{Z^a[T_1, ..., T_{|a|}]} \bigg[ 
\partial_a f\Big( \big\langle \bX_{s}, \rId \big\rangle (\omega_0), \cL^{\langle \bY_s, \rId\rangle}, \big\langle \bY_{s}, \rId \big\rangle (\omega_{\tilde{h}_{1}^T}),
\\
\label{eq:theorem:ContinIm-RCRPs1.1.1}
&\quad  ..., \big\langle \bY_{s}, \rId \big\rangle (\omega_{\tilde{h}_{m[a]}^T}) \Big)
\cdot 
\bigotimes_{r=1}^{|a|} \Big\langle \big[ \bX, \bY \big]_s, T_r \Big\rangle (\omega_{\tilde{h}_{a_r}^T}, \omega_{H^{T_r}}) \bigg]
\cdot \Big\langle \rw_{s, t}, T\Big\rangle(\omega_0, \omega_{H^T}) \Bigg]
\\
\nonumber
&+ \sum_{a\in A^{\alpha, \beta}} \frac{1}{|a|!} \sum_{\substack{T_1, ..., T_{|a|}\in \scF \\ \scG_{\alpha, \beta}\big[ \cE^a[T_1, ..., T_{|a|}]\big]\geq \gamma}}^{\gamma-, \alpha, \beta} \bE^{H^{\cE^a[T_1, ..., T_{|a|}]}}\Bigg[  \bE^{Z^a[T_1, ..., T_{|a|}]} \bigg[ 
\\
\nonumber
&\quad \partial_a f\Big( \big\langle \bX_{s}, \rId \big\rangle (\omega_0), \cL^{\langle \bY_s, \rId\rangle}, \big\langle \bY_{s}, \rId \big\rangle (\omega_{\tilde{h}_1^T}) ,..., \big\langle \bY_{s}, \rId \big\rangle (\omega_{\tilde{h}_{m[a]}^T}) \Big) \cdot
\bigotimes_{r=1}^{|a|} \Big\langle \big[ \bX, \bY \big]_s, T_r \Big\rangle (\omega_{\tilde{h}_{a_r}^T}, \omega_{H^{T_r}}) \bigg]
\\
\label{eq:theorem:ContinIm-RCRPs1.1.2}
&\quad \cdot \Big\langle \rw_{s, t}, \cE^a[T_1, ..., T_{|a|}] \Big\rangle(\omega_0, \omega_{H^{\cE^a[T_1, ..., T_{|a|}]}}) \Bigg] + O\Big( |t-s|^{\gamma} \Big). 
\end{align}

The terms contained in the summation \eqref{eq:theorem:ContinIm-RCRPs1.1.2} are also of order $O(|t-s|^\gamma)$ and the summation in \eqref{eq:theorem:ContinIm-RCRPs1.1.1} motivates the terms of the random controlled rough path described in Equation \eqref{eq:theorem:ContinIm-RCRPs}. 

This leads us to
\begin{align}
\nonumber
\Big\langle& \bZ_{s, t}^{\sharp}, \rId \Big\rangle(\omega_0) 
=
\sum_{a\in A^{\alpha, \beta} } \frac{1}{|a|!} \sum_{\substack{T_1, ..., T_{|a|}\in \scF \\ \scG_{\alpha, \beta}\big[ \cE^a[T_1, ..., T_{|a|}]\big]\geq \gamma}}^{\gamma-, \alpha, \beta} \bE^{H^{\cE^a[T_1, ..., T_{|a|}]}}\Bigg[  \bE^{Z^a[T_1, ..., T_{|a|}]} \bigg[ 
\\
\nonumber
&\quad \partial_a f\Big( \big\langle \bX_{s}, \rId \big\rangle (\omega_0), \cL^{\langle \bY_s, \rId\rangle}, \big\langle \bY_{s}, \rId \big\rangle (\omega_{\tilde{h}_1^T}) ,..., \big\langle \bY_{s}, \rId \big\rangle (\omega_{\tilde{h}_{m[a]}^T}) \Big) \cdot
\bigotimes_{r=1}^{|a|} \Big\langle \big[ \bX, \bY \big]_s, T_r \Big\rangle (\omega_{\tilde{h}_{a_r}^T}, \omega_{H^{T_r}}) \bigg]
\\
\nonumber
&\quad \cdot \Big\langle \rw_{s, t}, \cE^a[T_1, ..., T_{|a|}] \Big\rangle(\omega_0, \omega_{H^{\cE^a[T_1, ..., T_{|a|}]}}) \Bigg]
\\
\nonumber
&+\sum_{a\in A^{\alpha, \beta}} \frac{1}{|a|!} \bE^{1, ..., m[a]}\Bigg[
\partial_a f\Big( \big\langle \bX_{s}, \rId \big\rangle (\omega_0), \cL^{\langle \bY_s, \rId\rangle}, ..., \big\langle \bY_{s}, \rId \big\rangle (\omega_{\tilde{h}_{m[a]}^T}) \Big) 
\\
\nonumber
&\quad \cdot \sum_{k=1}^{|a|} \bigg( \bigotimes_{r=1}^{k-1} \Big\langle \fJ\big[ \bX, \bY\big]_{s, t}, \rId\Big\rangle(\omega_{a_r}) \bigg) \otimes \Big\langle \big[\bX, \bY\big]_{s, t}^{\sharp}, \rId\Big\rangle(\omega_{a_k}) \otimes \bigg( \bigotimes_{r=k+1}^{|a|} \Big\langle \big[ \bX, \bY\big]_{s, t}, \rId\Big\rangle(\omega_{a_r}) \bigg) \Bigg]
\\
\label{eq:theorem:ContinIm-RCRPs_1stRem}
& + \fR_{s, t}(\omega_0). 
\end{align}

Next, we verify that the increments of Equation \eqref{eq:theorem:ContinIm-RCRPs} satisfy Definition \ref{definition:RandomControlledRP}. Fix $Y\in \scF^{\gamma-, \alpha, \beta}$. By taking increments, we get
\begin{align*}
\Big\langle \bZ_{s,t}&, Y \Big\rangle (\omega_0, \omega_{H^Y})
=
\sum_{i=1}^n \sum_{a\in A_i^{\alpha, \beta}}  \frac{1}{i!} 
\sum_{\substack{Y_1, ..., Y_{i} \in \scF \\ Y = \cE^a[Y_1, ..., Y_{i}] }} 
\bE^{Z^a[Y_1, ..., Y_{i}]} \Bigg[ 
\\
& \partial_a f\Big( \big\langle \bX, \rId \big\rangle (\omega_0), \cL^{\langle \bY,\rId\rangle}, ..., \big\langle \bY, \rId \big\rangle (\omega_{\tilde{h}_{m[a]}^Y}) \Big)_{s, t} \cdot
\bigotimes_{r=1}^{i} \Big\langle \big[ \bX, \bY \big]_s, Y_r \Big\rangle (\omega_{\tilde{h}_{a_r}^Y}, \omega_{H^{Y_r}}) 
\\
+&\partial_a f\Big( \big\langle \bX_{s}, \rId \big\rangle (\omega_0), \cL^{\langle \bY_s,\rId\rangle},..., \big\langle \bY_{s}, \rId \big\rangle (\omega_{\tilde{h}_{m[a]}^Y}) \Big) \cdot
\bigg( \bigotimes_{r=1}^{i} \Big\langle \big[ \bX, \bY \big], Y_r \Big\rangle (\omega_{\tilde{h}_{a_r}^Y}, \omega_{H^{Y_r}}) \bigg)_{s, t} 
\\
+& \partial_a f \Big( \big\langle \bX, \rId \big\rangle (\omega_0), \cL^{\langle \bY,\rId\rangle}, ..., \big\langle \bY, \rId \big\rangle (\omega_{\tilde{h}_{m[a]}^Y}) \Big)_{s, t} \cdot 
\bigg( \bigotimes_{r=1}^{i} \Big\langle \big[ \bX, \bY \big], Y_r \Big\rangle (\omega_{\tilde{h}_{a_r}^Y}, \omega_{H^{Y_r}}) \bigg)_{s, t} \Bigg]. 
\end{align*}

By applying Corollary \ref{corollary:LionsTaylor3} to get
\begin{align}
\nonumber
\Big\langle \bZ_{s,t}&, Y \Big\rangle (\omega_0, \omega_{H^Y})
=
\sum_{i=1}^n \sum_{a\in A_i^{\alpha, \beta}}  \frac{1}{i!} 
\sum_{\substack{Y_1, ..., Y_{i} \in \scF \\ Y = \cE^a[Y_1, ..., Y_{i}] }} 
\bE^{Z^a[Y_1, ..., Y_{i}]} \Bigg[ \sum_{j=i+1}^n \sum_{\overline{a} \in A_{j}^{\alpha, \beta|a}} \frac{1}{(j-i)!}
\\
\nonumber
&\bE^{m[a]+1, ..., m[\overline{a}]}\Bigg[ \partial_{\overline{a}} f \bigg( \big\langle \bX_s, \rId \big\rangle (\omega_0), \cL_s^{\langle \bY,\rId\rangle}, ..., \big\langle \bY_s, \rId \big\rangle (\omega_{\tilde{h}_{m[a]}^Y}), ..., \big\langle \bY_s, \rId \big\rangle (\omega_{m[\overline{a}]}) \bigg) 
\\
\label{eq:theorem:ContinIm-RCRPs1.2.1}
&\cdot \bigotimes_{r=i+1}^j \Big\langle \big[\bX, \bY\big]_{s, t}, \rId\Big\rangle(\omega_{\overline{a}_r})\Bigg] \otimes \bigg( \bigotimes_{r=1}^i \Big\langle \big[ \bX, \bY\big]_s + \big[ \bX, \bY\big]_{s,t}, Y_r\Big\rangle(\omega_{\tilde{h}_{a_r}^Y}, \omega_{H^{Y_r}}) \bigg) \Bigg]
\\
\label{eq:theorem:ContinIm-RCRPs1.2.2}
+& \sum_{i=1}^n \sum_{a\in A_i^{\alpha, \beta}}  \frac{1}{i!} 
\sum_{\substack{Y_1, ..., Y_{i} \in \scF \\ Y = \cE^a[Y_1, ..., Y_{i}] }} 
\bE^{Z^a[Y_1, ..., Y_{i}]} \Bigg[ \fR_{s, t}^a(\omega_0, ..., \omega_{\tilde{h}_{m[a]}^Y}) \otimes \bigotimes_{r=1}^{i} \Big\langle \big[ \bX, \bY\big]_t, Y_r\Big\rangle(\omega_{\tilde{h}_{a_r}^Y}, \omega_{H^{Y_r}}) \Bigg]
\\
\nonumber
+& \sum_{i=1}^n \sum_{a\in A_i^{\alpha, \beta}}  \frac{1}{i!} 
\sum_{\substack{Y_1, ..., Y_{i} \in \scF \\ Y = \cE^a[Y_1, ..., Y_{i}] }} 
\bE^{Z^a[Y_1, ..., Y_{i}]} \Bigg[ \partial_a f\Big( \big\langle \bX_{s}, \rId \big\rangle (\omega_0), \cL^{\langle \bY_s,\rId\rangle},..., \big\langle \bY_{s}, \rId \big\rangle (\omega_{\tilde{h}_{m[a]}^Y}) \Big)
\\
\label{eq:theorem:ContinIm-RCRPs1.2.3}
& \cdot \bigg( \bigotimes_{r=1}^{i} \Big\langle \big[ \bX, \bY \big]_s + \big[ \bX, \bY \big]_{s, t}, Y_r \Big\rangle (\omega_{\tilde{h}_{a_r}^Y}, \omega_{H^{Y_r}}) - \bigotimes_{r=1}^{i} \Big\langle \big[ \bX, \bY \big]_s, Y_r \Big\rangle (\omega_{\tilde{h}_{a_r}^Y}, \omega_{H^{Y_r}}) \bigg)  \Bigg]. 
\end{align}

If there exists $a\in A_{i, j} \cap A^{\alpha, \beta}$ such that the Lions tree $Y$ can be expressed of the form $Y=\cE^a\big[ Y_1, ..., Y_{|a|} \big]$, we know that $\scG_{\alpha, \beta}[Y]\geq \alpha i + \beta j$. Hence the terms in Equation \eqref{eq:theorem:ContinIm-RCRPs1.2.2} will be of order at least $O\Big( |t-s|^{\gamma - \scG_{\alpha, \beta}[Y]}\Big)$. 

Addressing Equation \eqref{eq:theorem:ContinIm-RCRPs1.2.3} first, thanks to Lemma \ref{lemma:TechLem_Products} we get that
\begin{align}
\nonumber
\eqref{eq:theorem:ContinIm-RCRPs1.2.3}
=&
\sum_{i=1}^n \sum_{a\in A_i^{\alpha, \beta}}  \frac{1}{i!} 
\sum_{\substack{Y_1, ..., Y_{i} \in \scF \\ Y = \cE^a[Y_1, ..., Y_{i}] }} 
\bE^{Z^a[Y_1, ..., Y_{i}]} \Bigg[ \partial_a f\Big( \big\langle \bX_{s}, \rId \big\rangle (\omega_0), \cL^{\langle \bY_s,\rId\rangle},..., \big\langle \bY_{s}, \rId \big\rangle (\omega_{\tilde{h}_{m[a]}^Y}) \Big)
\\
\nonumber
&\cdot \sum_{k=1}^i \bigg( \bigotimes_{r=1}^{k-1} \Big\langle \big[ \bX, \bY \big]_s + \fJ\big[ \bX, \bY \big]_{s, t}, Y_r \Big\rangle (\omega_{\tilde{h}_{a_r}^Y}, \omega_{H^{Y_r}}) \bigg) \otimes \Big\langle \big[ \bX, \bY \big]_{s, t}^{\sharp}, Y_k \Big\rangle (\omega_{\tilde{h}_{a_k}^Y}, \omega_{H^{Y_k}}) 
\\
\label{eq:theorem:ContinIm-RCRPs1.2.2.1}
&\qquad \otimes \bigg( \bigotimes_{r=k+1}^{i} \Big\langle \big[ \bX, \bY \big]_{t}, Y_r \Big\rangle (\omega_{\tilde{h}_{a_r}^Y}, \omega_{H^{Y_r}}) \bigg)
\\
\label{eq:theorem:ContinIm-RCRPs1.2.2.2}
&+ \bigotimes_{r=1}^{i} \Big\langle \big[ \bX, \bY \big]_{s} + \fJ\big[ \bX, \bY \big]_{s, t}, Y_r \Big\rangle (\omega_{\tilde{h}_{a_r}^Y}, \omega_{H^{Y_r}}) \Bigg]. 
\end{align}

Whenever we have that $Y = \cE^a\big[ Y_1, ..., Y_i\big]$, we know that $\scG_{\alpha, \beta}[Y_k] \leq \scG_{\alpha, \beta}[Y]$ for any choice of $k$ so that Equation \eqref{eq:theorem:ContinIm-RCRPs1.2.2.1} is of order $O\Big( |t-s|^{\gamma - \scG_{\alpha, \beta}[Y]}\Big)$. 

Using the same techniques and making the substitution $Y_{i+1}, ..., Y_j = \rId$, we see that
\begin{align}
\nonumber
&\eqref{eq:theorem:ContinIm-RCRPs1.2.1}
\\
\nonumber
=& 
\sum_{i=1}^n \sum_{j=i+1}^n \sum_{\overline{a} \in A_{j}^{\alpha, \beta}} \frac{1}{i!(j-i)!} 
\sum_{\substack{Y_1, ..., Y_{i} \in \scF \\ Y_{i+1}, ..., Y_j = \rId \\ Y = \cE^{\overline{a}}[Y_1, ..., Y_{j}] }} 
\bE^{Z^{\overline{a}}[Y_1, ..., Y_{j}]} \Bigg[ 
\partial_{\overline{a}} f \Big( \big\langle \bX_s, \rId \big\rangle (\omega_0), \cL_s^{\langle \bY,\rId\rangle}, ..., \big\langle \bY_s, \rId \big\rangle (\omega_{\tilde{h}_{m[\overline{a}]}^Y}) \Big) 
\\
\nonumber
&\cdot \Bigg( \sum_{k=i+1}^j \bigotimes_{r=i+1}^{k-1} \Big\langle \fJ\big[\bX, \bY\big]_{s, t}, Y_r\Big\rangle(\omega_{\tilde{h}_{\overline{a}_r}^Y}) \otimes \Big\langle \big[\bX, \bY\big]_{s, t}^{\sharp}, Y_r \Big\rangle(\omega_{\tilde{h}_{\overline{a}_k}^Y}) \otimes \bigotimes_{r=k+1}^{j} \Big\langle \big[\bX, \bY\big]_{s, t}, Y_r \Big\rangle (\omega_{\tilde{h}_{\overline{a}_r}^Y}) \bigg)
\\
\label{eq:theorem:ContinIm-RCRPs1.2.1.1}
&\quad \bigotimes_{r=1}^i \Big\langle \big[ \bX, \bY\big]_t, Y_r\Big\rangle(\omega_{\tilde{h}_{\overline{a}_r}^Y}, \omega_{H^{Y_r}}) \Bigg]
\\
\nonumber
+& 
\sum_{i=1}^n \sum_{j=i+1}^n \sum_{\overline{a} \in A_{j}^{\alpha, \beta}} \frac{1}{i!(j-i)!}
\sum_{\substack{Y_1, ..., Y_{i} \in \scF \\ Y_{i+1}, ..., Y_j = \rId \\ Y = \cE^{\overline{a}}[Y_1, ..., Y_{j}] }} 
\bE^{Z^{\overline{a}}[Y_1, ..., Y_{j}]} \Bigg[ 
\partial_{\overline{a}} f \Big( \big\langle \bX_s, \rId \big\rangle (\omega_0), \cL_s^{\langle \bY,\rId\rangle}, ..., \big\langle \bY_s, \rId \big\rangle (\omega_{\tilde{h}_{m[\overline{a}]}^Y}) \Big) 
\\
\nonumber
&\cdot \bigg( \bigotimes_{r=i+1}^j \Big\langle \fJ\big[\bX, \bY\big]_{s, t}, Y_r\Big\rangle(\omega_{\tilde{h}_{\overline{a}_r}^Y}) \bigg)
\otimes 
\Bigg( \sum_{k=1}^i \bigg( \bigotimes_{r=1}^{k-1} \Big\langle \big[ \bX, \bY\big]_s + \fJ\big[ \bX, \bY\big]_{s,t}, Y_r\Big\rangle(\omega_{\tilde{h}_{\overline{a}_r}^Y}, \omega_{H^{Y_r}}) \bigg) 
\\
\label{eq:theorem:ContinIm-RCRPs1.2.1.2}
&\otimes \Big\langle \big[ \bX, \bY\big]_{s, t}^{\sharp}, Y_k \Big\rangle(\omega_{\tilde{h}_{\overline{a}_k}^Y}) 
\otimes 
\bigg( \bigotimes_{r=k+1}^{i} \Big\langle \big[ \bX, \bY\big]_t, Y_r\Big\rangle(\omega_{\tilde{h}_{\overline{a}_r}^Y}, \omega_{H^{Y_r}}) \bigg) \Bigg) \Bigg]
\\
\nonumber
+& 
\sum_{i=1}^n \sum_{j=i+1}^n \sum_{\overline{a} \in A_{j}^{\alpha, \beta}} \frac{1}{i!(j-i)!}
\sum_{\substack{Y_1, ..., Y_{i} \in \scF \\ Y_{i+1}, ..., Y_j = \rId \\ Y = \cE^{\overline{a}}[Y_1, ..., Y_{j}] }} 
\bE^{Z^{\overline{a}}[Y_1, ..., Y_{j}]} \Bigg[ 
\partial_{\overline{a}} f \Big( \big\langle \bX_s, \rId \big\rangle (\omega_0), \cL_s^{\langle \bY,\rId\rangle}, ..., \big\langle \bY_s, \rId \big\rangle (\omega_{\tilde{h}_{m[\overline{a}]}^Y}) \Big) 
\\
\label{eq:theorem:ContinIm-RCRPs1.2.1.3}
&\cdot \bigg( \bigotimes_{r=i+1}^j \Big\langle \fJ\big[\bX, \bY\big]_{s, t}, Y_r\Big\rangle(\omega_{\tilde{h}_{\overline{a}_r}^Y}) \bigg)
\otimes 
\bigg( \bigotimes_{r=1}^{i} \Big\langle \big[ \bX, \bY\big]_s + \fJ\big[ \bX, \bY\big]_{s,t}, Y_r\Big\rangle(\omega_{\tilde{h}_{\overline{a}_r}^Y}, \omega_{H^{Y_r}}) \bigg) \Bigg]. 
\end{align}

By the presence of a remainder term, Equation \eqref{eq:theorem:ContinIm-RCRPs1.2.1.1} and Equation \eqref{eq:theorem:ContinIm-RCRPs1.2.1.2} are both of order $O\Big(|t-s|^{\gamma - \scG_{\alpha, \beta}[Y]}\Big)$. 

An application of Lemma \ref{lemma:TechLem_Products} gives us that
\begin{align*}
\bigotimes_{r=1}^{i}& \Big\langle \big[ \bX, \bY\big]_s + \fJ\big[ \bX, \bY\big]_{s,t}, Y_r\Big\rangle(\omega_{\tilde{h}_{\overline{a}_r}^Y}, \omega_{H^{Y_r}})
\\
=&  \sum_{k=1}^{i}\bigg( \bigotimes_{r=1}^{k-1} \Big\langle \big[ \bX, \bY\big]_{s}, Y_r\Big\rangle(\omega_{\tilde{h}_{\overline{a}_r}^Y}, \omega_{H^{Y_r}}) \bigg) 
\otimes 
\Big\langle \fJ\big[ \bX, \bY\big]_{s, t}, Y_k\Big\rangle(\omega_{\tilde{h}_{\overline{a}_k}^Y}, \omega_{H^{Y_k}}) 
\\
&\quad \otimes 
\bigg( \bigotimes_{r=k+1}^{i} \Big\langle \big[ \bX, \bY\big]_{s} + \fJ\big[ \bX, \bY\big]_{s, t}, Y_r\Big\rangle(\omega_{\tilde{h}_{\overline{a}_r}^Y}, \omega_{H^{Y_r}}) \bigg)
\\
&+ \bigg( \bigotimes_{r=1}^{i} \Big\langle \big[ \bX, \bY\big]_{s}, Y_r\Big\rangle(\omega_{\tilde{h}_{\overline{a}_r}^Y}, \omega_{H^{Y_r}}) \bigg). 
\end{align*}

We add the remaining terms together (and addressing the case $j=1$ first) we obtain
\begin{align*}
\Big\langle& \bZ_{s,t}, Y \Big\rangle (\omega_0, \omega_{H^Y}) = \delta_{h_0^Y \neq \emptyset} \cdot \nabla_{x_0}f\Big( \big\langle\bX_s, \rId\big\rangle(\omega_0), \cL_s^{\langle \bY, \rId\rangle}\Big) \cdot \Big\langle \fJ\big[ \bX]_{s, t}, Y\Big\rangle(\omega_0, \omega_{H^Y}) 
\\
&+\delta_{h_0^Y = \emptyset}\Bigg( \sum_{\substack{Y'\in \scF, h_0^{Y'}\neq \emptyset \\ \cE[Y'] = Y}} \partial_{\mu}f\Big( \big\langle \bX_s, \rId\big\rangle(\omega_0), \cL_s^{\langle \bY, \rId\rangle}, \big\langle \bY_s, \rId\big\rangle(\omega_{h_0^{Y'}}) \Big) \cdot\Big\langle \fJ\big[ \bY \big]_{s, t}, Y'\Big\rangle(\omega_{h_0^{Y'}}, \omega_{H^{Y'}})
\\
&\qquad +\bE^1\bigg[ \partial_\mu f\Big( \big\langle \bX_s, \rId\big\rangle(\omega_0), \cL_s^{\langle \bY, \rId\rangle}, \big\langle \bY_s, \rId\big\rangle(\omega_1)\Big) \cdot \Big\langle \fJ\big[ \bY\big]_{s, t}, Y\Big\rangle( \omega_1, \omega_{H^Y}) \bigg] \Bigg)
\\
&+\sum_{j=2}^n \frac{1}{j!} \sum_{i=1}^j \frac{j!}{i!(j-i)!} \sum_{a\in A_j^{\alpha, \beta}} \sum_{\substack{ Y_1, ..., Y_i \in \scF \\ Y_{i+1}, ..., Y_j = \rId \\ \cE^a[ Y_1, ..., Y_j] = Y}} \bE^{Z^a[Y_1, ..., Y_{j}]} \Bigg[ 
\partial_{a} f \Big( \big\langle \bX_s, \rId \big\rangle (\omega_0), \cL_s^{\langle \bY,\rId\rangle}, ..., \big\langle \bY_s, \rId \big\rangle (\omega_{\tilde{h}_{m[a]}^Y}) \Big) 
\\
&\cdot \Bigg( \bigotimes_{r=i+1}^j \Big\langle \fJ\big[\bX, \bY\big]_{s, t}, Y_r\Big\rangle(\omega_{\tilde{h}_{a_r}^Y}) \Bigg)
\otimes  
\Bigg( \sum_{k=1}^{i} \bigg( \bigotimes_{r=1}^{k-1} \Big\langle \big[ \bX, \bY\big]_{s}, Y_r\Big\rangle(\omega_{\tilde{h}_{a_r}^Y}, \omega_{H^{Y_r}}) \bigg) 
\\
&\otimes 
\Big\langle \fJ\big[ \bX, \bY\big]_{s, t}, Y_k\Big\rangle(\omega_{\tilde{h}_{a_k}^Y}, \omega_{H^{Y_k}}) 
\otimes 
\bigg( \bigotimes_{r=k+1}^{i} \Big\langle \big[ \bX, \bY\big]_{s} + \fJ\big[ \bX, \bY\big]_{s, t}, Y_r\Big\rangle(\omega_{\tilde{h}_{a_r}^Y}, \omega_{H^{Y_r}}) \bigg)
\\
&\quad + 
\bigg( \bigotimes_{r=1}^{i} \Big\langle \big[ \bX, \bY\big]_{s}, Y_r\Big\rangle(\omega_{\tilde{h}_{a_r}^Y}, \omega_{H^{Y_r}}) \bigg) \Bigg) + O\Big( |t-s|^{\gamma - \scG_{\alpha, \beta}[Y]}\Big). 
\end{align*}

To conclude, we make the substitution
\begin{align*}
&\Big\langle \fJ\big[ \bX, \bY\big]_{s, t}, Y\Big\rangle(\omega_0, \omega_{H^Y}) 
\\
&= \sum_{T\in \scF}^{\gamma-, \alpha, \beta} \sum_{\Upsilon\in \scF} c'\Big( T, \Upsilon, Y \Big) \cdot \bE^{E^{T, \Upsilon, Y}}\bigg[ \Big\langle \big[\bX, \bY\big]_s, T \Big\rangle(\omega_0, \omega_{\phi^{T, \Upsilon, Y}[H^{T}]}) \cdot \Big\langle \rw_{s, t}, \Upsilon \Big\rangle(\omega_0, \omega_{\varphi^{T, \Upsilon, Y}[H^{\Upsilon}]}) \bigg]
\end{align*}
and use the techniques described in the proof of Proposition \ref{proposition:productRCRP} along with classical combinatorial techniques to obtain that
\begin{align*}
\Big\langle \bZ_{s,t}&, Y \Big\rangle (\omega_0, \omega_{H^Y})
\\
&= \sum_{T\in \scF}^{\gamma-,\alpha, \beta} \sum_{ \Upsilon \in \scF} c'\Big( T, \Upsilon, Y \Big) \cdot \bE^{E^{T, \Upsilon, Y}}\Bigg[ \sum_{a\in A^{\alpha, \beta}} \frac{1}{|a|!} \sum_{\substack{T_1, ..., T_{|a|}\in\scF \\ \cE^a[ T_1, ..., T_{|a|}]=T}} \bE^{Z^a[ T_1, ..., T_{|a|}]}\bigg[ 
\\
&\qquad \partial_a f\Big( \big\langle \bX_s, \rId \big\rangle (\omega_{0}), \cL^{\langle \bY_s, \rId\rangle}, \big\langle \bY_s, \rId \big\rangle (\omega_{\tilde{\phi}^{T, \Upsilon, Y}[\tilde{h}_1^{T}]}), ..., \big\langle \bY_s, \rId \big\rangle (\omega_{\tilde{\phi}^{T, \Upsilon, Y}[\tilde{h}_{m[a]}^{T}]}) \Big) 
\\
&\qquad \cdot \bigotimes_{r=1}^{|a|} \Big\langle [\bX,\bY]_s, T_r \Big\rangle (\omega_{\tilde{\phi}^{T, \Upsilon, Y}[\tilde{h}_{a_r}^{Y}]}, \omega_{\phi^{T, \Upsilon, Y}[H^{T_r}]}) \bigg]
\cdot 
\Big\langle \rw_{s, t}, \Upsilon \Big\rangle (\omega_{0}, \omega_{\varphi^{T, \Upsilon, Y}[H^\Upsilon]}) \Bigg] 
\\
&+ O\Big( |t-s|^{\gamma - \scG_{\alpha, \beta}[Y]}\Big), 
\end{align*}
which verifies that there is a random controlled rough path that satisfies Equation \eqref{eq:theorem:ContinIm-RCRPs}. 
\end{proof}

Adapting the notation from Equation \eqref{eq:a_integral} earlier, we now denote
\begin{equation}
\label{eq:a_integral2}
\Big\| \big\langle [\bX,\bY], T \big\rangle(\omega_{a_i}) \Big\|_{[a], p_{x},p_{y}} = \begin{cases}
\bE^{H^T}\bigg[ \Big| \big\langle \bX,  T \big\rangle(\omega_{0}, \omega_{H^T}) \Big|^{p_x[T]} \bigg]^{\tfrac{1}{p_x[T]}} \quad & \quad \mbox{if} \quad a_i=0,  
\\
\bE^{a_i}\bigg[ \bE^{H^T}\bigg[ \Big| \big\langle \bY,  T \big\rangle(\omega_{a_i}, \omega_{H^T}) \Big|^{p_y[T]} \bigg] \bigg]^{\tfrac{1}{p_y[T]}} \quad & \quad \mbox{if} \quad a_i>0. 
\end{cases}
\end{equation}

\begin{proof}[Proof of Equation \eqref{eq:theorem:ContinIm-RCRPs-Est}]

We start this proof by returning to Equation \eqref{eq:theorem:ContinIm-RCRPs_1stRem}. 

Then
\begin{align}
\nonumber
\sup_{s, t\in[u, v]}& \frac{\Big\langle \bZ_{s, t}^{\sharp}, \rId \Big\rangle(\omega_0)}{|t-s|^{\gamma}} 
\leq
\sum_{a\in A^{\alpha, \beta} } \frac{1}{|a|!} \sum_{\substack{T_1, ..., T_{|a|}\in \scF \\ \scG_{\alpha, \beta}\big[ \cE^a[T_1, ..., T_{|a|}]\big]\geq \gamma}}^{\gamma-, \alpha, \beta} 
\Big\| \partial_a f\Big\|_\infty 
\\
\nonumber
&\cdot
\prod_{r:a_r=0} \Big\| \big\langle \bX, T_r \big\rangle(\omega_0) \Big\|_{p_x[T_r], \infty} 
\cdot 
\prod_{r:a_r>0} \Big\| \big\langle \bY, T_r \big\rangle \Big\|_{p_y[T_r], \infty} 
\\
\nonumber
&\quad \cdot 
\Big\| \big\langle \rw, \cE^a[T_1, ..., T_{|a|}] \big\rangle(\omega_0) \Big\|_{q\big[ \cE^a[T_1, ..., T_{|a|}] \big], \scG_{\alpha, \beta}\big[ \cE^a[T_1, ..., T_{|a|}] \big]} \cdot |v-u|^{\scG_{\alpha, \beta}\big[ \cE^a[T_1, ..., T_{|a|}] \big] - \gamma}
\\
\nonumber
&+\sum_{a\in A^{\alpha, \beta}} \frac{1}{|a|!} \Big\| \partial_a f \Big\|_{\infty}
\cdot 
\sum_{k=1}^{|a|} \Bigg( \prod_{r=1}^{k-1} \bigg\| \Big\| \big\langle \fJ\big[ \bX, \bY\big], \rId\big\rangle(\omega_{a_k}) \Big\|_{\alpha} \bigg\|_{[a],p_x, p_y} \cdot |v-u|^{\alpha} \Bigg) 
\\
\nonumber
&\cdot \bigg\| \Big\| \big\langle \big[\bX, \bY\big]^{\sharp}, \rId\big\rangle(\omega_{a_k}) \Big\|_{\gamma} \bigg\|_{[a], p_x,p_y} \cdot \Bigg( \prod_{r=k+1}^{|a|} \Big\| \big\langle [\bX,\bY], \rId \big\rangle(\omega_{a_r}) \Big\|_{\alpha} \bigg\|_{[a], p_x,p_y} \cdot |v-u|^{\alpha} \Bigg) \Bigg]
\\
\label{eq:remark:Polynomial-comments1.1}
& + \sup_{s, t\in[u, v]} \frac{ \big| \fR_{s, t}(\omega_0) \big| }{|t-s|^{\gamma}}. 
\end{align}

Recall from Definition \ref{definition:B_set} that $(n+1)\alpha\geq \gamma$ and that $\forall a\in A_{\ast}^{\alpha, \beta}$, we have that 
$\scG_{\alpha, \beta}[a] \geq \gamma$. Thanks to Equation \eqref{eq:corollary:LionsTaylor3:remUpper}, we have that
\begin{align}
\nonumber
\sup_{s, t\in[u, v]}& \frac{ \big| \fR_{s, t}(\omega_0) \big| }{|t-s|^{\gamma}} 
\\
\nonumber
\leq& \sum_{a \in A_\ast^{\alpha, \beta}} \frac{1}{|a|!} \Big\| \partial_a f \Big\|_\infty \cdot \Big\| \big\langle \bX, \rId\big\rangle(\omega_0) \Big\|_{\alpha}^{l[a]_0} \cdot \prod_{r=1}^{m[a]} \bigg\| \tilde{\bE}\Big[ \big| \langle \bY, \rId \rangle(\tilde{\omega}) \Big|^{l[a]_r} \Big] \bigg\|_{\beta \cdot l[a]_r} \cdot |v-u|^{\scG_{\alpha,\beta}[a] - \gamma}
\\
\nonumber
&+\frac{1}{n!} \sum_{a\in \A{n}}  \Bigg( \Big\| \partial_a f \Big\|_{\lip, 0} \cdot \Big\| \big\langle \bX, \rId\big\rangle(\omega_0) \Big\|_{\alpha}^{l[a]_0+1} \cdot \prod_{r=1}^{m[a]} \tilde{\bE}\bigg[ \Big\| \big\langle \bY, \rId \big\rangle(\tilde{\omega}) \Big\|_{\alpha}^{l[a]_r} \bigg] 
\\
\nonumber
&+ \Big\| \partial_a f \big\|_{\lip, \mu} \cdot \Big\| \big\langle \bX, \rId \big\rangle(\omega_0) \Big\|_{\alpha}^{l[a]_0} \cdot \tilde{\bE}\bigg[ \Big\| \big\langle \bY, \rId\big\rangle(\tilde{\omega}) \Big\|_{\alpha} \bigg] \cdot \prod_{r=1}^{m[a]} \tilde{\bE}\bigg[ \Big\| \big\langle \bY, \rId \big\rangle (\tilde{\omega}) \Big\|_{\alpha}^{l[a]_r} \bigg]
\\
\label{eq:remark:Polynomial-comments1.2}
&+\sum_{j=1}^{m[a]} \Big\| \partial_a f \Big\|_{\lip, j} \cdot \Big\| \big\langle \bX, \rId\big\rangle(\omega_0) \Big\|_{\alpha}^{l[a]_0} \cdot \prod_{r=1}^{m[a]} \tilde{\bE}\bigg[ \Big\| \big\langle \bY, \rId\big\rangle(\tilde{\omega}) \Big\|_{\alpha}^{l[a]_r + \delta_{j=r}} \bigg] \Bigg) \cdot |v-u|^{(n+1)\alpha - \gamma}. 
\end{align}

Hence, with an application of Equation \eqref{eq:prop:Regu-RCRP-4} and Equation \eqref{eq:prop:Regu-RCRP-5}, we can find a polynomial $\fP_\rId: \big(\bR^+\big)^{\times 7} \to \bR^+$ increasing in every variable such that
\begin{align*}
\sup_{s, t\in[u,v]}& \frac{\big\langle \bZ_{s, t}^{\sharp}, \rId \big\rangle(\omega_0) }{|t-s|^{\gamma}} 
\\
&\leq \| f \|_{C_b^{n, (n)}} \cdot 
\fP_\rId \bigg(
 \sum_{T\in\scF_0}^{\gamma-, \alpha, \beta} \Big\| \big\langle \bX^{\sharp}, T \big\rangle(\omega_0) \Big\|_{p_x[T], \gamma - \scG_{\alpha, \beta}[T]}
 , 
\sum_{T\in \scF}^{\gamma-, \alpha, \beta} \Big\| \big\langle \bX_u, T\big\rangle(\omega_0) \Big\|_{p_x[T]}, 
\\
&\qquad 
\sum_{T\in\scF_0}^{\gamma-, \alpha, \beta} \Big\| \big\langle \bY^{\sharp}, T \big\rangle \Big\|_{p_y[T], \gamma - \scG_{\alpha, \beta}[T]}
, 
\sum_{T\in \scF}^{\gamma-, \alpha, \beta} \Big\| \big\langle \bY_u, T\big\rangle \Big\|_{p_y[T]}, 
\\
&\qquad \sum_{T\in \scF}^{\gamma, \alpha, \beta} \Big\| \big\langle \rw, T\big\rangle(\omega_0) \Big\|_{q[T], \scG_{\alpha, \beta}[T]}
, 
|v-u|^\alpha
, 
|v-u|^{\beta} \bigg). 
\end{align*}

By detailed evaluation of Equations \eqref{eq:remark:Polynomial-comments1.1} and \eqref{eq:remark:Polynomial-comments1.2}, we observe that every time we have a term associated to the function $f$, this is included in Equation \eqref{eq:definition:FunctionNorm} so we can use $\| f\|_{C_b^{n, (n)}}$ as an upper bound. Secondly, for any term of the form $\| \partial_a f\|_\infty$, there will be a product of norms of $\bX$ and $\bY$ where the number of $\bX$ terms will be $l[a]_0$ and the number of $\bY$ terms will be $|a| - l[a]_0$, both of which will be less than or equal to n. Similarly, for any term of the form $\| \partial_a f\|_{\lip}$, there will be a product of $n+1$ terms. Thanks to Proposition \ref{prop:Regu-RCRP}, any of these norms of $\bX$ and $\bY$ can be upper-bounded by the product of any term from the set
\begin{align*}
\bigg\{& \Big\| \big\langle \bX_u, T\rangle(\omega_0)\Big\|_{p_x[T]}, \quad 
\Big\| \big\langle \bX^{\sharp}, T\rangle(\omega_0)\Big\|_{p_x[T], \gamma - \scG_{\alpha, \beta}[T]}, 
\\
&\Big\| \big\langle \bY_u, T\rangle\Big\|_{p_y[T]}, \quad 
\Big\| \big\langle \bY^{\sharp}, T\rangle\Big\|_{p_y[T], \gamma - \scG_{\alpha, \beta}[T]} 
: \quad T\in \scF_0^{\gamma-, \alpha, \beta}
\bigg\}
\end{align*}
and any term from the set
$$
\bigg\{ 1, \quad \Big\| \big\langle\rw, T\big\rangle(\omega_0)\Big\|_{q[T], \scG_{\alpha, \beta}[T]}: T\in \scF^{\gamma, \alpha, \beta} \bigg\}. 
$$
The presence of the unit in the second set means that, however many $\bX$ and $\bY$ terms there are, there will always be less $\rw$ terms. 

Finally, whenever a term of the form $|v-u|$ occurs, it is always to a positive power expressible of the form $\alpha \cdot i_6 + \beta \cdot i_7$. However, by remarking that one such example is $(n+1)\alpha - \gamma$, it should be clear that there are constructions where the integers $i_6$ and $i_7$ may not be positive. 

Thus, the polynomial $\fP_\rId: \big(\bR^+\big)^{\times 7} \to \bR^+$ expressed as
$$
\fP_{\rId} \Big( x_1, x_2, y_1, y_2, w, t_1, t_2 \Big) = \sum_{i\in I_{\fP_{\rId}}} C_i \cdot x_1^{i_1} \cdot x_2^{i_2} \cdot y_1^{i_3} \cdot y_2^{i_4} \cdot w^{i_5} \cdot t_1^{i_6} \cdot t_2^{i_7}. 
$$
satisfies that $i_1, ..., i_5 \in \bN_0$, $i_1+i_2+i_3+i_4 \leq n+1$, $i_5 \leq i_1 + i_2 + i_3 + i_4$ and $\alpha \cdot i_6 + \beta \cdot i_7\geq 0$. This verifies the additional claims made in Remark \ref{remark:Polynomial-comments1}. In particular, the first and second variables of the polynomial $\fP_\rId$ will only ever be of order $n+1$, so that
$$
\bE^0 \bigg[ \sup_{s, t\in [u,v]} \frac{ \Big| \big\langle \bZ_{s, t}^{\sharp}, \rId \big\rangle(\omega_0) \Big|^{p_z[\rId]}}{|t-s|^{p_z[\rId] \alpha}}  \bigg] < \infty. 
$$

For the second part of this proof, we fix $Y\in \scF^{\gamma-, \alpha, \beta}$: We represent the remainder term $\Big\langle \bZ_{s, t}^{\sharp}, Y\Big\rangle(\omega_0, \omega_{H^Y})$ of the form
\begin{equation}
\label{eq:Remainder_Y}
\Big\langle \bZ_{s, t}^{\sharp}, Y \Big\rangle(\omega_0, \omega_{H^Y}) = \hyperlink{Eq:one}{(1)} + \hyperlink{Eq:two}{(2)} + \hyperlink{Eq:three}{(3)}
\end{equation}
where, by combining the terms from \eqref{eq:theorem:ContinIm-RCRPs1.2.2} and \eqref{eq:theorem:ContinIm-RCRPs1.2.2.1}, we obtain
\begin{align*}
\hypertarget{Eq:one}{(1)}&= \sum_{i=1}^n \sum_{a\in A_i^{\alpha, \beta}}  \frac{1}{i!} 
\sum_{\substack{Y_1, ..., Y_{i} \in \scF \\ Y = \cE^a[Y_1, ..., Y_{i}] }} 
\bE^{Z^a[Y_1, ..., Y_{i}]} \Bigg[ \fR_{s, t}^a(\omega_0, ..., \omega_{\tilde{h}_{m[a]}^Y}) \cdot \bigotimes_{r=1}^{i} \Big\langle \big[ \bX, \bY\big]_t, Y_r\Big\rangle(\omega_{\tilde{h}_{a_r}^Y}, \omega_{H^{Y_r}}) 
\\
&\quad + 
\partial_a f\Big( \big\langle \bX_{s}, \rId \big\rangle (\omega_0), \cL^{\langle \bY_s,\rId\rangle},..., \big\langle \bY_{s}, \rId \big\rangle (\omega_{\tilde{h}_{m[a]}^Y}) \Big) \cdot \sum_{k=1}^i \bigg( \bigotimes_{r=1}^{k-1} \Big\langle \big[ \bX, \bY \big]_{t}, Y_r \Big\rangle (\omega_{\tilde{h}_{a_r}^Y}, \omega_{H^{Y_r}}) \bigg)
\\
&\qquad \otimes \Big\langle \big[ \bX, \bY \big]_{s, t}^{\sharp}, Y_k \Big\rangle (\omega_{\tilde{h}_{a_k}^Y}, \omega_{H^{Y_k}}) \otimes \bigg( \bigotimes_{r=k+1}^{i} \Big\langle \big[ \bX, \bY \big]_s + \fJ\big[ \bX, \bY \big]_{s, t}, Y_r \Big\rangle (\omega_{\tilde{h}_{a_r}^Y}, \omega_{H^{Y_r}}) \bigg) \Bigg], 
\end{align*}
by combining the terms from \eqref{eq:theorem:ContinIm-RCRPs1.2.1.1} and \eqref{eq:theorem:ContinIm-RCRPs1.2.1.2} we obtain
\begin{align*}
&\hypertarget{Eq:two}{(2)}
\\
&= 
\sum_{i=1}^n \sum_{j=i+1}^n \sum_{a \in A_{j}^{\alpha, \beta}}  \frac{1}{i!(j-i)!} 
\sum_{\substack{Y_1, ..., Y_{i} \in \scF \\ Y_{i+1}, ..., Y_j = \rId \\ Y = \cE^{a}[Y_1, ..., Y_{j}] }} 
\bE^{Z^{a}[Y_1, ..., Y_{j}]} \Bigg[ 
\partial_{a} f \bigg( \big\langle \bX_s, \rId \big\rangle (\omega_0), \cL_s^{\langle \bY,\rId\rangle}, ..., \big\langle \bY_s, \rId \big\rangle (\omega_{\tilde{h}_{m[a]}^Y}) \bigg) 
\\
&\cdot \sum_{k=i+1}^j \bigg( \bigotimes_{r=i+1}^{k-1} \Big\langle \fJ\big[\bX, \bY\big]_{s, t}, Y_r\Big\rangle(\omega_{\tilde{h}_{a_r}^Y}) \bigg) \otimes \Big\langle \big[\bX, \bY\big]_{s, t}^{\sharp}, Y_r \Big\rangle (\omega_{\tilde{h}_{a_k}^Y}) \otimes \bigg( \bigotimes_{r=k+1}^{j} \Big\langle \big[\bX, \bY\big]_{s, t}, Y_r \Big\rangle (\omega_{\tilde{h}_{a_r}^Y}) \bigg)
\\
&\quad \otimes \bigg( \bigotimes_{r=1}^i \Big\langle \big[ \bX, \bY\big]_t, Y_r\Big\rangle(\omega_{\tilde{h}_{a_r}^Y}, \omega_{H^{Y_r}}) \bigg)
\\
&+ \bigg( \bigotimes_{r=i+1}^j \Big\langle \fJ\big[\bX, \bY\big]_{s, t}, Y_r\Big\rangle(\omega_{\tilde{h}_{a_r}^Y}) \bigg)
\otimes 
\Bigg( \sum_{k=1}^i \bigg( \bigotimes_{r=1}^{k-1} \Big\langle \big[ \bX, \bY\big]_s + \fJ\big[ \bX, \bY\big]_{s,t}, Y_r\Big\rangle(\omega_{\tilde{h}_{a_r}^Y}, \omega_{H^{Y_r}}) \bigg) 
\\
&\otimes \Big\langle \big[ \bX, \bY\big]_{s, t}^{\sharp}, Y_k \Big\rangle(\omega_{\tilde{h}_{a_k}^Y}) 
\otimes 
\bigg( \bigotimes_{r=k+1}^{i} \Big\langle \big[ \bX, \bY\big]_t, Y_r\Big\rangle(\omega_{\tilde{h}_{a_r}^Y}, \omega_{H^{Y_r}}) \bigg) \Bigg) \Bigg],
\end{align*}
and finally
\begin{align*}
\hypertarget{Eq:three}{(3)} 
&= \sum_{j=1}^n \frac{1}{j!} \sum_{a\in A_j^{\alpha, \beta}} \sum_{i=1}^j \frac{j!}{i!(j-i)!} \sum_{\substack{ Y_1, ..., Y_i \in \scF \\ Y_{i+1}, ..., Y_j=\rId \\ \cE^a[ Y_1, ..., Y_j] = Y}} \sum_{\Upsilon_1, ..., \Upsilon_j \in \scF_0} \sum_{\substack{T_1, ..., T_j\in \scF \\ \scG_{\alpha, \beta}\big[ \cE^a[ T_1, ..., T_j] \big] \geq \gamma}}^{\gamma-, \alpha, \beta} \bE^{\bigcup_{k=1}^n E^{T_k, \Upsilon_k, Y_k}}\Bigg[
\\
&\bE^{Z^a[Y_1, ..., Y_{j}]} \Bigg[  \partial_{a} f \Big( \big\langle \bX_s, \rId \big\rangle (\omega_0), \cL_s^{\langle \bY,\rId\rangle}, ..., \big\langle \bY_s, \rId \big\rangle (\omega_{\tilde{h}_{m[a]}^Y}) \Big) 
\\
&\cdot \bigg( \bigotimes_{r=i+1}^j \Big\langle \big[\bX, \bY\big]_{s}, T_r\Big\rangle(\omega_{\tilde{h}_{a_r}^Y}, \omega_{\phi^{T_r, \Upsilon_r, Y_r}[H^{T_r}]}) \cdot \delta_{\{Y_r=\rId, T_r = \Upsilon_r\}} \bigg)
\\
&\otimes \Bigg( \bigg( \bigotimes_{r=1}^{i} \Big\langle \big[ \bX, \bY\big]_{s}, T_r\Big\rangle(\omega_{\tilde{h}_{a_r}^Y}, \omega_{\phi^{T_r, \Upsilon_r, Y_r}[H^{T_r}]}) \cdot \delta_{\{Y_r=T_r, \Upsilon_r=\rId \}} \bigg)
\\
&\quad+ \sum_{k=1}^{i} \bigg( \bigotimes_{r=1}^{k-1} \Big\langle \big[ \bX, \bY\big]_{s}, T_r\Big\rangle(\omega_{\tilde{h}_{a_r}^Y}, \omega_{\phi^{T_r, \Upsilon_r, Y_r}[H^{T_r}]}) \cdot \delta_{\{Y_r=T_r, \Upsilon_r=\rId\}} \bigg) 
\\
&\quad \otimes \Big\langle \big[ \bX, \bY\big]_{s}, T_k\Big\rangle(\omega_{\tilde{h}_{a_k}^Y}, \omega_{\phi^{T_k, \Upsilon_k, Y_k}[H^{T_k}]}) \cdot c'\big( T_k, \Upsilon_k, Y_k\big)
\\
&\quad \otimes \bigg( \bigotimes_{r=k+1}^{i} \Big\langle \big[ \bX, \bY\big]_{s}, T_r\Big\rangle(\omega_{\tilde{h}_{a_r}^Y}, \omega_{\phi^{T_r, \Upsilon_r, Y_r}[H^{T_r}]}) \cdot \Big( c'\big( T_r, \Upsilon_r, Y_r\big) + \delta_{\{ T_r = Y_r, \Upsilon_r=\rId \}} \Big) \bigg) \Bigg) \Bigg]
\\
&\cdot \bigotimes_{r=1}^j \Big\langle \rw_{s, t}, \Upsilon_r \Big\rangle(\omega_{\tilde{h}_{a_r}^Y}, \omega_{\varphi^{T_r, \Upsilon_r, Y_r}[H^{T_r}]}) \Bigg]
\end{align*}

We can upper bound the $\fR_{s, t}^a$ term in \hyperlink{Eq:one}{(1)} using Equation \eqref{eq:corollary:LionsTaylor3:remUpper_Y}. Integrating over the the detagged probability spaces and applying Proposition \ref{prop:Regu-RCRP}, we can find a polynomial $\fP_Y: \big(\bR^+\big)^{\times 7} \to \bR^+$ increasing in every variable such that
\begin{align*}
\sup_{s, t\in[u,v]}& \frac{\bE^{H^Y}\bigg[ \Big| \big\langle \bZ_{s, t}^{\sharp}, Y \big\rangle(\omega_0, \omega_{H^Y})\Big|^{p[Y]}\bigg]^{\tfrac{1}{p[y]}} }{|t-s|^{\gamma - \scG_{\alpha, \beta}[Y]}} 
\\
&\leq \| f \|_{C_b^{n, (n)}} \cdot 
\fP_Y \bigg(
 \sum_{T\in\scF_0}^{\gamma-, \alpha, \beta} \Big\| \big\langle \bX^{\sharp}, T \big\rangle(\omega_0) \Big\|_{p_x[T], \gamma - \scG_{\alpha, \beta}[T]}
 , 
\sum_{T\in \scF}^{\gamma-, \alpha, \beta} \Big\| \big\langle \bX_u, T\big\rangle(\omega_0) \Big\|_{p_x[T]}, 
\\
&\qquad 
\sum_{T\in\scF_0}^{\gamma-, \alpha, \beta} \Big\| \big\langle \bY^{\sharp}, T \big\rangle \Big\|_{p_y[T], \gamma - \scG_{\alpha, \beta}[T]}
, 
\sum_{T\in \scF}^{\gamma-, \alpha, \beta} \Big\| \big\langle \bY_u, T\big\rangle \Big\|_{p_y[T]}, 
\\
&\qquad \sum_{T\in \scF}^{\gamma, \alpha, \beta} \Big\| \big\langle \rw, T\big\rangle(\omega_0) \Big\|_{q[T], \scG_{\alpha, \beta}[T]}
, 
|v-u|^\alpha
, 
|v-u|^{\beta} \bigg)
\end{align*}
By arguing in the same fashion as earlier, we can also verify that this polynomial satisfies the description in Remark \ref{remark:Polynomial-comments1}. 

Summing over each of the polynomials yields Equation \eqref{eq:theorem:ContinIm-RCRPs-Est}. 
\end{proof}

\newpage
\section{Stability of random controlled rough paths}
\label{section:Stability-RCRP}

Let $(\Omega, \cF, \bP)$ be a probability space. The collection of probabilistic rough paths $\scC\big( \scH^{\gamma, \alpha, \beta}(\Omega), p, q \big)$ introduced in Definition \ref{definition:ProbabilisticRoughPaths} is a metric space (see \cite{salkeld2021Probabilistic}), while for any choice of 
$$
\rw\in \scC\big( \scH^{\gamma, \alpha, \beta}(\Omega), p, q \big)
$$ 
the set of random controlled rough paths controlled by $\rw$, $\cD_\rw^{\gamma, \alpha, \beta}$ is a Banach space (see Theorem \ref{theorem:Banachspace}). The set of pairs $(\rw, \bX)$ gives rise to the "\emph{fibre bundle}"
$$
\bigsqcup_{\rw \in \scC(\scH^{\gamma, \alpha, \beta}, p, q)} \cD_\rw^{\gamma, p, q}
$$
with base space $\scC\big( \scH^{\gamma, \alpha, \beta}(\Omega), p, q \big)$ and fibres $\cD_\rw^{\gamma, p, q}$. 

We need to introduce a new concept for comparing two different random controlled rough paths, each controlled by different probabilistic rough paths. In particular, each probabilistic rough path may be defined on different probability spaces so that (motivated by the Wasserstein distance, see Equation \eqref{eq:WassersteinDistance}) we need to consider all possible couplings between the two probability spaces. This should be seen as concrete evidence that we are not really interested in random controlled rough paths as random variables but as paths (when we evaluate the tagged probability space) with associated distributions (coming from all untagged probability spaces). 

\begin{definition}
\label{definition:Wasserstein}
Let $\alpha, \beta>0$ and let $\gamma>\alpha\wedge \beta$. Let $(\Omega, \cF, \bP)$ be probability spaces and denote $(\hat{\Omega}, \hat{\cF}, \hat{\bP})$ an identical probability space. Let $\Pi\in \cP(\Omega \times \hat{\Omega})$ with left marginal $\bP$ and right marginal $\hat{\bP}$. Let $\bX: [0,1] \to \scH^{\gamma-, \alpha, \beta}(\Omega)$ and $\hat{\bX}: [0,1] \to \scH^{\gamma-, \alpha, \beta}(\hat{\Omega})$ be continuous paths. 

We define
\begin{align}
\nonumber
\bW&_{\Pi}^{p}\big[ T \big] \Big( \bX_{t}, \hat{\bX}_{t} \Big)(\omega_0)
\\
\label{eq:WassersteinNorm1}
&:= \bigg( \int_{(\Omega\times \hat{\Omega})^{\times |H^T|}} \Big| \big\langle \bX_{t}, T \big\rangle(\omega_0, \omega_{H^T}) - \big\langle \hat{\bX}_{t}, T \big\rangle(\omega_0, \hat{\omega}_{H^T}) \Big|^{p} d \Pi^{\times |H^T|}\Big((\omega_h, \hat{\omega}_h)_{h\in H^T} \Big) \bigg)^{\tfrac{1}{p}}, 
\\
\nonumber
\bW&_{\Pi}^{p}\big[ T \big] \Big( \bX_{t}, \hat{\bX}_{t} \Big) 
\\
\label{eq:WassersteinNorm3}
&:= \bigg( \int_{(\Omega\times \hat{\Omega})^{\times |(H^T)'|}} \Big| \big\langle \bX_{t}, T \big\rangle(\omega_{(H^T)'}) - \big\langle \hat{\bX}_{t}, T \big\rangle(\hat{\omega}_{(H^T)'}) \Big|^{p} d\Pi^{\times |(H^T)'|} \Big( (\omega, \hat{\omega})_{h\in (H^T)'} \Big) \bigg)^{\tfrac{1}{p}}, 
\end{align}
\begin{align}
\label{eq:WassersteinNorm5}
\bW_{\infty, \Pi}^{p} \big[ T \big] \Big( \bX, \hat{\bX} \Big)(\omega_0)
:=& 
\sup_{t\in[0,1]} \bW_{\Pi}^{p}\big[ T \big] \Big( \bX_{t}, \hat{\bX}_{t} \Big)(\omega_0), 
\\
\label{eq:WassersteinNorm6}
\bW_{\infty, \Pi}^{p}\big[T\big] \Big( \bX, \hat{\bX} \Big) 
:=& 
\sup_{t\in[0,1]} \bW_{\Pi}^{p}\big[ T \big] \Big( \bX_{t}, \hat{\bX}_{t} \Big), 
\end{align}
where we recall that $(H^T)' = \big( H^T \cup \{ h_0\} \big) \backslash\{\emptyset\}$. Moreover, for $\alpha \in (0, 1)$, 
\begin{align}
\label{eq:WassersteinNorm2}
\bW_{\alpha, \Pi}^{p} \big[ T \big] \Big( \bX, \hat{\bX} \Big)(\omega_0)
:=& 
\sup_{s, t\in[0,1]} \frac{\bW_{\Pi}^{p}\big[ T \big] \Big( \bX_{s,t}, \hat{\bX}_{s,t} \Big)(\omega_0) }{|t-s|^{\alpha}} , 
\\
\label{eq:WassersteinNorm4}
\bW_{\alpha, \Pi}^{p}\big[T\big] \Big( \bX, \hat{\bX} \Big) 
:=& 
\sup_{s, t\in[0,1]} \frac{ \bW_{\Pi}^{p}\big[ T \big] \Big( \bX_{s,t}, \hat{\bX}_{s,t} \Big) }{|t-s|^{\alpha}}. 
\end{align}
\end{definition}

\begin{remark}
Let $(\Omega, \cF, \bP)$ and $(\hat{\Omega}, \hat{\cF}, \hat{\bP})$ be identical probability spaces. Recalling Equation \eqref{eq:definition:ProbabilisticRoughPaths3}, we note that for $\rw\in \scC\big( \scH^{\gamma, \alpha, \beta}(\Omega), ,p, q\big)$ and $\hat{\rw} \in \scC\big( \scH^{\gamma, \alpha, \beta}(\hat{\Omega}), p, q\big)$, we can restate
$$
\rho_{\alpha, \beta, p, q, 0}\Big( \rw, \hat{\rw} \Big)(\omega_0) = \inf_{\Pi} \sum_{T\in \scF}^{\gamma, \alpha,\beta} \bW_{\scG_{\alpha, \beta}[T], \Pi}^p\big[T\big] \Big( \rw, \hat{\rw}\Big)(\omega_0)
$$
\end{remark}

For $\Pi$ fixed, Equations \eqref{eq:WassersteinNorm1} and \eqref{eq:WassersteinNorm2} act like norms for a fixed choice of $\omega_0 \in \Omega_0$, but once we take an infimum over the choice of $\Pi$, we obtain an object similar to (but technically distinct from) the Wasserstein distance (see Equation \eqref{eq:WassersteinDistance}). 

By contrast, Equation \eqref{eq:WassersteinNorm3} involves integrating over every probability space, so this is actually a metric induced by a norm (specific to the choice of $\Pi$) and similarly for Equation \eqref{eq:WassersteinNorm4}. Taking an infimum over all choices of $\Pi$ transforms Equation \eqref{eq:WassersteinNorm3} and \eqref{eq:WassersteinNorm4} into Wasserstein distances induced by metrics on Euclidean and pathspace. 

Observe that for a Lions forest $T\in \scF$ such that $H^T = \emptyset$, we have that
$$
\bW_{\Pi}^{p}\big[ T \big] \Big( \bX_{t}, \hat{\bX}_{t} \Big)(\omega_0):= \Big| \big\langle \bX_{t}, T \big\rangle(\omega_0) - \big\langle \hat{\bX}_{t}, T \big\rangle(\omega_0) \Big| 
$$
so that the operation $\bW_{\Pi}^p\big[T\big]$ does not always act as an integral operator. 

To this end, we proceed as follows:
\begin{definition}
\label{definition:InhomogeneousMetric-RCRPs}
Let $\alpha, \beta>0$ and let $\gamma:=\inf\{ \scG_{\alpha, \beta}[T]: T \in \scF, \scG_{\alpha, \beta}[T]>1-\alpha \}$. 

Let $(p_1, q)$ and $(p_2, q)$ be two pairs of dual integrability functionals and define $p:\scF_0^{\gamma, \alpha, \beta} \to [1, \infty)$ by letting, for $T\in \scF_0^{\gamma, \alpha, \beta},$
$$
p[T] = p_1[T] \wedge p_2[T]. 
$$

Let $(\Omega, \cF, \bP)$ and $(\hat{\Omega}, \hat{\cF}, \hat{\bP})$ be identical probability spaces. Let 
$$
\rw\in \scC\big( \scH^{\gamma, \alpha, \beta}(\Omega), p_1, q \big), \quad \hat{\rw} \in \scC\big( \scH^{\gamma, \alpha, \beta}(\hat{\Omega}), p_2, q \big).
$$
For $\bX \in \cD_\rw^{\gamma, p, q}$ and $\hat{\bX} \in  \cD_{\hat{\rw}}^{\gamma, p, q}$, we define
$$
d_{\rw, \hat{\rw}, \gamma, 0}: \cD_\rw^{\gamma, p, q} \times \cD_{\hat{\rw}}^{\gamma, p, q} \to L^{p[\rId]}\big( \Omega, \bP; \bR^+\big)
$$
by
\begin{align}
\nonumber
d_{\rw, \hat{\rw}, \gamma, 0}&\Big( \bX, \hat{\bX}\Big)(\omega_0)
\\
\label{eq:definition:InhomogeneousMetric-RCRPs-0}
:=&
\inf_{\Pi} \sum_{T\in \scF_0}^{\gamma-, \alpha, \beta} \bigg( \bW_{\Pi}^{p[T]}\big[T\big] \Big( \bX_0, \hat{\bX}_0 \Big)(\omega_0)
+ 
\bW_{\gamma - \scG_{\alpha, \beta}[T], \Pi}^{p[T]} \big[T\big] \Big( \bX^{\sharp}, \hat{\bX}^{\sharp} \Big)(\omega_0) \bigg) 
\end{align}
where $\inf_{\Pi}$ runs over all probability measures on $\big( \Omega\times \hat{\Omega}, \cF \otimes \hat{\cF} \big)$ with left and right marginals $\bP$ and $\hat{\bP}$ respectively. 

Separately but of equal interest, for $\bY \in \cD_\rw^{\gamma, p, q}$ and $\hat{\bY} \in  \cD_{\hat{\rw}}^{\gamma, p, q}$, we define
$$
d_{\rw, \hat{\rw}, \gamma}: \cD_\rw^{\gamma, p, q} \times \cD_{\hat{\rw}}^{\gamma, p, q} \to \bR^+
$$
by
\begin{align}
\label{eq:definition:InhomogeneousMetric-RCRPs-E}
d_{\rw, \hat{\rw}, \gamma}&\Big( \bY, \hat{\bY}\Big)
:=
\inf_{\Pi} \sum_{T\in \scF_0}^{\gamma-, \alpha, \beta} \bigg( \bW_{\Pi}^{p[T]}\big[T\big]\Big( \bY_0, \hat{\bY}_0 \Big) 
+ 
\bW_{\gamma - \scG_{\alpha, \beta}[T], \Pi}^{p[T]} \big[T\big] \Big( \bY^{\sharp}, \hat{\bY}^{\sharp} \Big) \bigg) 
\end{align}
where $\inf_{\Pi}$ runs as before. 
\end{definition}

We emphasise that the probability spaces $(\Omega, \cF, \bP)$ and $(\hat{\Omega}, \hat{\cF}, \hat{\bP})$ are identical, although we denote them as distinct to emphasise that there is a coupling between these two spaces. Thus, in Equation \eqref{eq:definition:InhomogeneousMetric-RCRPs-0} we abuse notation so that $\omega_0$ is simultaneously an element of $\Omega$ and $\hat{\Omega}$. 

\begin{example}
Let $\alpha, \beta>0$ and let $\gamma:=\inf\{ \scG_{\alpha, \beta}[T]: T \in \scF, \scG_{\alpha, \beta}[T]>1-\alpha \}$. Let $(\Omega, \cF, \bP)$ be a probability space and let $(p, q)$ be a dual pair of integrability functionals. Let $\rw$ and $\hat{\rw}$ be $(\scH^{\gamma, \alpha, \beta}, p, q)$-probabilistic rough paths such that $\rho_{(\alpha, \beta, p,q), 0}\big(\rw, \hat{\rw}\big)\neq 0$. 

For each $T\in \scF_0^{\gamma-, \alpha, \beta}$ and $s, t\in[0,1]$, let
\begin{align*}
\Big\langle \bX_0, T \Big\rangle(\omega_0, \omega_{H^T}) \in& L^{p[T]}\Big( \Omega \times \Omega^{\times |H^T|}, \bP \times (\bP)^{\times |H^T|}; \lin\big( (\bR^d)^{\otimes |N^T|}, \bR^e\big) \Big), 
\\
\Big\langle \bX_{s, t}^{\sharp}, T \Big\rangle(\omega_0, \omega_{H^T}) \in& L^{p[T]}\Big( \Omega \times \Omega^{\times |H^T|}, \bP \times (\bP)^{\times |H^T|}; \lin\big( (\bR^d)^{\otimes |N^T|}, \bR^e\big) \Big), 
\end{align*}
and additionally suppose that
$$
\bP\Bigg[ \frac{\bE^{H^T}\bigg[ \Big| \big\langle \bX_{s, t}^{\sharp}, T \big\rangle(\omega_0, \omega_{H^T})\Big|^{p[T]} \bigg]}{|t-s|^{p[T](\gamma - \scG_{\alpha, \beta}[T])}} <\infty \Bigg] = 1. 
$$

We define two paths $\bY, \hat{\bY}:[0,1] \to \scH^{\gamma-, \alpha, \beta}(\Omega)$ such that $\forall t\in[0,1]$, 
\begin{align*}
\Big\langle \bY_t, \rId\Big\rangle(\omega_0) =& \Big\langle \bX_0, \rId\Big\rangle(\omega_0) + \Big\langle \bX_{0,t}^{\sharp}, \rId\Big\rangle(\omega_0)
\\
&+ \sum_{T\in \scF}^{\gamma-, \alpha, \beta} \bE^{H^T}\bigg[ \Big\langle \bX_0, T\Big\rangle(\omega_0, \omega_{H^T}) \cdot \Big\langle \rw_{0, t}, T\Big\rangle(\omega_0, \omega_{H^T}) \bigg],
\\
\Big\langle \hat{\bY}_t, \rId\Big\rangle(\omega_0) =& \Big\langle \bX_0, \rId\Big\rangle(\omega_0) + \Big\langle \bX_{0,t}^{\sharp}, \rId\Big\rangle(\omega_0)
\\
&+ \sum_{T\in \scF}^{\gamma-, \alpha, \beta} \bE^{H^T}\bigg[ \Big\langle \bX_0, T\Big\rangle(\omega_0, \omega_{H^T}) \cdot \Big\langle \hat{\rw}_{0, t}, T\Big\rangle(\omega_0, \omega_{H^T}) \bigg],
\end{align*}
and
\begin{align*}
\Big\langle \bY_t&, Y\Big\rangle(\omega_0, \omega_{H^Y}) = \Big\langle \bX_0, Y \Big\rangle(\omega_0, \omega_{H^Y}) + \Big\langle \bX_{0,t}^{\sharp}, Y \Big\rangle(\omega_0, \omega_{H^Y})
\\
&+ \sum_{T, \Upsilon \in \scF}^{\gamma-, \alpha, \beta} c'\Big( T, \Upsilon, Y\Big) \cdot \bE^{E^{T, \Upsilon, Y}}\bigg[ \Big\langle \bX_0, T\Big\rangle(\omega_0, \omega_{\phi^{T, \Upsilon, Y}[H^T]}) \cdot \Big\langle \rw_{0, t}, \Upsilon \Big\rangle(\omega_0, \omega_{\varphi^{T, \Upsilon, Y}[H^{\Upsilon}]}) \bigg],
\\
\Big\langle \hat{\bY}_t&, Y\Big\rangle(\omega_0, \omega_{H^Y}) = \Big\langle \bX_0, Y \Big\rangle(\omega_0, \omega_{H^Y}) + \Big\langle \bX_{0,t}^{\sharp}, Y \Big\rangle(\omega_0, \omega_{H^Y})
\\
&+ \sum_{T, \Upsilon \in \scF}^{\gamma-, \alpha, \beta} c'\Big( T, \Upsilon, Y\Big) \cdot \bE^{E^{T, \Upsilon, Y}}\bigg[ \Big\langle \bX_0, T\Big\rangle(\omega_0, \omega_{\phi^{T, \Upsilon, Y}[H^T]}) \cdot \Big\langle \hat{\rw}_{0, t}, \Upsilon \Big\rangle(\omega_0, \omega_{\varphi^{T, \Upsilon, Y}[H^{\Upsilon}]}) \bigg]. 
\end{align*}
Then it is a relatively simple exercise to verify that $\bY$ is a random controlled rough path controlled by $\rw$ and $\hat{\bY}$ is a random controlled rough path controlled by $\hat{\rw}$. Further, we do not have that $\bY = \hat{\bY}$ since they are controlled by different probabilistic rough paths. None-the-less, by construction we have that
$$
\bP\Bigg[ d_{\rw, \hat{\rw}, \gamma, 0}\Big( \bY, \hat{\bY}\Big)(\omega_0) = 0 \Bigg] = 0. 
$$

We provide this example to demonstrate that these concepts are \emph{not} metrics, but the notation we choose should aid the reader in identifying the links to optimal transport where appropriate. 
\end{example}

\subsection{Stability of operations on random controlled rough paths}
\label{subsec:Stability_RCRPs}

For a choice of norm $\| \cdot \|$, we use the notation
\begin{equation}
\label{eq:Max_2norms}
\Big\| \big\langle \{ \bX \vee \hat{\bX}\}, T \big\rangle \Big\| = \Big\| \big\langle \bX, T \big\rangle \Big\| \vee \Big\| \big\langle \hat{\bX}, T \big\rangle \Big\|
\end{equation}

Having established a ``Wasserstein-like'' metric on the collection of random controlled rough paths, the next question is to establish the stability properties of the operations on random controlled rough paths described in Section \ref{subsec:Operat_RCRPs}. 

\subsubsection{Stability of rough integration}

We saw in Section \ref{subsubsection:ReconstructionTheorem} with the Reconstruction Theorem (Theorem \ref{theorem:Reconstruction}) that random controlled rough paths are an ideal object for defining mean-field stochastic rough integrals. This next result allows us to compare two mean-field rough integrals using techniques that are analogous to the Wasserstein distance. 

\begin{theorem}
\label{theorem:Stability-RoughInt}
Let $\alpha, \beta>0$ and $\gamma:=\inf\{ \scG_{\alpha, \beta}[T]: T \in \scF, \scG_{\alpha, \beta}[T]>1-\alpha \}$. 

Let $(p_1, q_1)$ and $(p_2, q_2)$ be two dual pairs of integrability functionals that both satisfy \ref{eq:proposition:Reconstruction2}. We define $p_3: \scF_0^{\gamma, \alpha, \beta} \to [1, \infty)$ and $p_4: \scF_0^{\gamma, \alpha, \beta} \to [1, \infty)$ by
\begin{align*}
\frac{1}{p_3[\rId]}:=& \sup_{T\in \scF_0^{\gamma-\alpha, \alpha, \beta}} \bigg( \frac{1}{p_1[T]} + \frac{1}{q_1\big[ \lfloor T \rfloor_i \big]} \bigg), \quad \frac{1}{p_3[T]} := \frac{1}{p_3[\rId]} - \frac{1}{q_1[T]}, 
\\
\frac{1}{p_4[\rId]}:=& \sup_{T\in \scF_0^{\gamma-\alpha, \alpha, \beta}} \bigg( \frac{1}{p_2[T]} + \frac{1}{q_2\big[ \lfloor T \rfloor_i \big]} \bigg), \quad \frac{1}{p_4[T]} := \frac{1}{p_4[\rId]} - \frac{1}{q_2[T]}, 
\end{align*}
Additionally, we define
$$
p_5[T] := p_1[T] \wedge p_2[T], 
\quad
p_6[T] := p_3[T] \wedge p_4[T], 
\quad
q[T] := q_1[T] \wedge q_2[T]. 
$$
Then $(p_5, q)$ and $(p_6, q)$ are dual integrability functionals. 

Let $(\Omega, \cF, \bP)$ and $(\hat{\Omega}, \hat{\cF}, \hat{\bP})$ be identical probability spaces and let 
$$
\rw \in \scC\big(\scH^{\gamma, \alpha, \beta}(\Omega), p_1, q_1)\quad  \mbox{and} \quad \hat{\rw} \in \scC\big(\scH^{\gamma, \alpha, \beta}(\hat{\Omega}), p_2, q_2).
$$ 

Recalling Equation \eqref{eq:theorem:Reconstruction_Space} from earlier, let 
$$
\bX\in \cD_{\rw}^{\gamma, p_1, q_1}\big( (\ddot{\scH})^{\gamma, \alpha, \beta}(\Omega) \big)\quad \mbox{and} \quad
\hat{\bX} \in \cD_{\hat{\rw}}^{\gamma, p_2, q_2}\big( (\ddot{\scH})^{\gamma, \alpha, \beta}(\hat{\Omega}) \big).
$$ 
Let $\Phi: \cD_{\rw}^{\gamma, p_5, q}( (\ddot{\scH})^{\gamma, \alpha, \beta} \big) \to \cD_{\rw}^{\gamma, p_6, q} \big( \scH^{\gamma, \alpha, \beta}\big)$ be the operator defined in Equation \eqref{eq:theorem:Reconstruction:Phi}. Then

\begin{align}
\nonumber
\sum_{Y\in \scF_0}^{\gamma-, \alpha, \beta}& \sup_{s, t\in[u,v]} \frac{\bW_{\Pi}^{p_6[Y]}\big[Y\big] \Big( \Phi\big[ \bX \big]_{s, t}^{\sharp}, \Phi\big[ \hat{\bX} \big]_{s, t}^{\sharp} \Big)(\omega_0) }{|t-s|^{\gamma - \scG_{\alpha, \beta}[Y]}}
\\
\nonumber
\lesssim& \sum_{T\in\scF_0}^{\gamma-\alpha, \alpha, \beta} \Bigg( 
\bW_{\gamma-\scG_{\alpha, \beta}[T], \Pi}^{p_5[T]}\big[ T \big]\Big( \bX^{\sharp}, \hat{\bX}^{\sharp} \Big)(\omega_0) \cdot \bigg( \Big\| \big\langle \{ \rw \vee \hat{\rw}\}, \lfloor T\rfloor \big\rangle(\omega_0) \Big\|_{q[\lfloor T\rfloor ], \scG_{\alpha, \beta}[\lfloor T \rfloor]} + 1\bigg) 
\\
\nonumber
&\qquad + \Big\| \big\langle \{ \bX^{\sharp} \vee \hat{\bX}^{\sharp} \}, T\big\rangle(\omega_0) \Big\|_{p_5[T], \gamma-\scG_{\alpha, \beta}[T]} 
\cdot 
\bW_{\scG_{\alpha, \beta}[\lfloor T \rfloor], \Pi}^{q[\lfloor T \rfloor]}\big[\lfloor T \rfloor\big] \Big( \rw, \hat{\rw} \Big)(\omega_0)
\Bigg) \cdot \eta^{\alpha}
\\
\nonumber
&+ \hspace{-10pt}\sum_{\substack{T\in\scF_0 \\ \scG_{\alpha, \beta}[T]\in [\gamma-\alpha, \gamma)}} \hspace{-10pt}\Bigg( 
\bW_{\gamma-\scG_{\alpha, \beta}[T], \Pi}^{p_5[T]}\big[ T \big]\Big( \bX^{\sharp}, \hat{\bX}^{\sharp} \Big)(\omega_0) \cdot \bigg( \sum_{\Upsilon\in\scF}^{\gamma-\alpha-, \alpha, \beta} \Big\| \big\langle \{ \rw \vee \hat{\rw}\} , \Upsilon \big\rangle(\omega_0) \Big\|_{q[\Upsilon], \scG_{\alpha, \beta}[\Upsilon]} \bigg)
\\
\nonumber
&\qquad + \Big\| \big\langle \{ \bX^{\sharp} \vee \hat{\bX}^{\sharp}\} , T\big\rangle(\omega_0)\Big\|_{p_5[T], \gamma-\scG_{\alpha, \beta}[T]}  \cdot \bigg( \sum_{\Upsilon\in\scF}^{\gamma-\alpha-, \alpha, \beta} \bW_{\scG_{\alpha, \beta}[\Upsilon], \Pi}^{q[\Upsilon]}\big[\Upsilon\big] \Big( \rw, \hat{\rw} \Big)(\omega_0) \bigg) \Bigg)\cdot \eta^\alpha 
\end{align}
\begin{align}
\nonumber
&+ \hspace{-10pt}\sum_{\substack{T\in\scF_0 \\ \scG_{\alpha, \beta}[T]\in [\gamma-\alpha, \gamma)}} \hspace{-10pt}\Bigg( 
\bW_{\Pi}^{p_5[T]}\big[ T \big]\Big( \bX_u, \hat{\bX}_u \Big)(\omega_0) \cdot \bigg( \Big\| \big\langle \{ \rw \vee \hat{\rw}\} , \lfloor T\rfloor \big\rangle(\omega_0) \Big\|_{q[\lfloor T\rfloor], \scG_{\alpha, \beta}[\lfloor T\rfloor]} 
\\
\nonumber
&\qquad \quad + \sum_{\Upsilon\in\scF}^{\gamma-\alpha-, \alpha, \beta} \Big\| \big\langle \{ \rw \vee \hat{\rw}\} , \Upsilon \big\rangle(\omega_0) \Big\|_{q[\Upsilon], \scG_{\alpha, \beta}[\Upsilon]} \bigg)
\\
\nonumber
&\qquad + \Big\| \big\langle \{ \bX_u \vee \hat{\bX}_u \} , T\big\rangle(\omega_0)\Big\|_{p_5[T], \gamma-\scG_{\alpha, \beta}[T]}  \cdot \bigg( \bW_{\scG_{\alpha, \beta}[\lfloor T \rfloor], \Pi}^{q[\lfloor T \rfloor]}\big[\lfloor T \rfloor\big] \Big( \rw, \hat{\rw} \Big)(\omega_0)
\\
\label{eq:theorem:Stability-RoughInt-inccom}
&\qquad \quad + \sum_{\Upsilon\in\scF}^{\gamma-\alpha-, \alpha, \beta} \bW_{\scG_{\alpha, \beta}[\Upsilon], \Pi}^{q[\Upsilon]}\big[\Upsilon\big] \Big( \rw, \hat{\rw} \Big)(\omega_0) \bigg) \Bigg). 
\end{align}

In particular, suppose that $\exists \cO \subseteq \Omega$ such that $\forall \omega_0\in \cO$, $\exists M>0$ such that
\begin{equation}
\label{eq:theorem:Stability-RoughInt}
\left.\begin{aligned}
\max\Big( \rho_{(\alpha, \beta, p_y, q), 0}\big( \rw, \rId \big)(\omega_0), 
\quad& 
\rho_{(\alpha, \beta, p_y, q), 0}\big( \hat{\rw}, \rId \big)(\omega_0) \Big) <M, 
\\
\max\Big( \big\| \bX \big\|_{\rw, \gamma, p_y, q, 0} (\omega_0),
\quad&
\big\| \hat{\bX} \big\|_{\hat{\rw}, \gamma, p_y, q, 0} (\omega_0)\Big) < M,  
\end{aligned}\quad \right\rbrace
\end{equation}
with $\rho_{(\alpha, \beta, p_y, q), 0}$ defined in Definition \ref{definition:ProbabilisticRoughPaths}. 

Then for $\omega_0 \in \cO$, we have that there exists a constant $C_M$ dependent only on $M$ such that
\begin{align}
\nonumber
d&_{\rw, \hat{\rw}, \gamma, 0}\Big( \Phi\big[\bX\big], \Phi\big[ \hat{\bX} \big] \Big)(\omega_0) 
\\
\label{eq:theorem:Stability-RoughInt-inc}
&\leq C_M \bigg( \rho_{(\alpha, \beta, p_5, q)}\Big(\rw, \hat{\rw} \Big)(\omega_0) + d_{\rw, \hat{\rw}, \gamma}\Big( \bX, \hat{\bX} \Big)(\omega_0) \bigg). 
\end{align}

In particular, this means that the operator $\Phi$ is locally Lipschitz with respect to  the pseudo-metric $d_{\rw, \hat{\rw}, \gamma, 0}$. 
\end{theorem}

The proof of Theorem \ref{theorem:Stability-RoughInt} is delayed until Section \ref{subsubsec:Proof-Reconstruction-Stability}. 

\subsubsection{Stability of continuous images of Random controlled rough paths}

We saw in Section \ref{subsubsect:Contin-Image-RandContRP} that the smooth image of a random controlled rough path is also a random controlled rough path. This next result allows us to compare the continuous image of two random controlled rough paths using the properties of the Taylor expansion established in Section \ref{subsect:Multivariate_Lions-Taylor} and techniques that are analogous to the Wasserstein distance. 

\begin{theorem}
\label{theorem:Stability-ContinImage}
Let $\alpha, \beta>0$ and let 
$$
\gamma:=\inf \big\{ \alpha  i + \beta j: (i, j)\in \bN_0^{\times 2}, \alpha i + \beta j > 1-\alpha\big\}, 
\quad
n:=\sup \big\{ m\in \bN_0: m < \tfrac{\gamma}{\alpha\wedge \beta} \big\}.
$$

Let $(p_1, q_1)$, $(p_2, q_1)$, $(p_3, q_2)$ $(p_4, q_2)$ be four pairs of dual integrability functionals and additionally suppose that
\begin{equation}
\label{eq:theorem:Stability-ContinImage4}
\begin{split}
\sup_{T\in \scF^{\gamma, \alpha, \beta}} \Big( \tfrac{1}{q_1[T]} \Big) < \tfrac{n+1}{p_1[\rId]} \leq 1,  
\quad
\sup_{T\in \scF^{\gamma, \alpha, \beta}} \Big( \tfrac{1}{q_1[T]} \Big) < \tfrac{n+1}{p_2[\rId]} \leq 1, 
\\
\sup_{T\in \scF^{\gamma, \alpha, \beta}} \Big( \tfrac{1}{q_2[T]} \Big) < \tfrac{n+1}{p_3[\rId]} \leq 1,  
\quad
\sup_{T\in \scF^{\gamma, \alpha, \beta}} \Big( \tfrac{1}{q_2[T]} \Big) < \tfrac{n+1}{p_4[\rId]} \leq 1. 
\end{split}
\end{equation}
We define $p:\scF_0^{\gamma, \alpha, \beta}\to [1, \infty)$ such that 
\begin{equation}
\label{eq:theorem:Stability-ContinImage3}
\begin{split}
\frac{1}{p_5[\rId]}:= \frac{n+2}{ p_1[\rId]}, 
\quad
\frac{1}{p_5[T]}:= \frac{1}{p_5[\rId]} - \frac{1}{q_1[T]}, 
\\
\frac{1}{p_6[\rId]}:= \frac{n+2}{ p_3[\rId]}, 
\quad
\frac{1}{p_6[T]}:= \frac{1}{p_6[\rId]} - \frac{1}{q_2[T]}. 
\end{split}
\end{equation}
Then $(p_5, q_1)$ and $(p_6, q_2)$ are dual pairs of integrability functionals.

Additionally, we define for $T\in \scF^{\gamma, \alpha, \beta}$
\begin{align*}
p_x[T] = p_1[T] \wedge p_3[T], \quad& p_y[T] = p_2[T] \wedge p_4[T],
\\
p[T] = p_5[T] \wedge p_6[T], \quad& q[T] = q_1[T] \wedge q_2[T]. 
\end{align*}
Then $(p, q)$ is a dual pair of integrability functionals. 

Let $(\Omega, \cF, \bP)$ and $(\hat{\Omega}, \hat{\cF}, \hat{\bP})$ be identical probability spaces and let 
$$
\rw \in \scC\big(\scH^{\gamma, \alpha, \beta}(\Omega), p, q)\quad  \mbox{and} \quad \hat{\rw} \in \scC\big(\scH^{\gamma, \alpha, \beta}(\hat{\Omega}), p, q).
$$ 
Let 
\begin{align*}
&\bX \in \cD_{\rw}^{\gamma, p_1, q_1}\big( \scH^{\gamma-, \alpha, \beta}(\Omega)\big), 
\quad
\bY\in \cD_{\rw}^{\gamma, p_2, q_1}\big( \scH^{\gamma-, \alpha, \beta}(\Omega)\big), 
\\
\quad \mbox{and} \quad 
&\hat{\bX} \in \cD_{\hat{\rw}}^{\gamma, p_3, q_2}\big( \scH^{\gamma-, \alpha, \beta}(\hat{\Omega})\big), 
\quad
\hat{\bY} \in \cD_{\hat{\rw}}^{\gamma, p_4, q_2}\big( \scH^{\gamma-, \alpha, \beta}(\hat{\Omega})\big).
\end{align*}
Let $f: \bR^e \times \cP_2(\bR^e) \to \lin(\bR^d, \bR^e)$ satisfy that $f \in C_{b}^{n+1, (n+1)}\big( \bR^d \times \cP_2(\bR^d)\big)$. We define 
$$
\bZ:[0,1]\to (\ddot{\scH})^{\gamma-, \alpha, \beta}
\quad \mbox{and} \quad 
\hat{\bZ}:[0,1]\to (\ddot{\scH})^{\gamma-, \alpha, \beta}
$$
to be random controlled rough paths that satisfy
\begin{align*}
\bZ\in \cD_{\rw}^{\gamma, p_5, q_1}\big( (\ddot{\scH})^{\gamma-, \alpha, \beta}\big),
\quad& 
\Big\langle \bZ_t, \rId \Big\rangle(\omega_0) = f\Big( \big\langle \bX_t, \rId \big\rangle(\omega_0), \cL^{\langle \bY_t, \rId\rangle} \Big)
\\
\hat{\bZ} \in \cD_{\hat{\rw}}^{\gamma, p_6, q_2} \big( (\ddot{\scH})^{\gamma-, \alpha, \beta}\big)
\quad&
\Big\langle \hat{\bZ}_t, \rId \Big\rangle(\omega_0) = f\Big( \big\langle \hat{\bX}_t, \rId \big\rangle(\omega_0), \cL^{\langle \hat{\bY}_t, \rId\rangle} \Big)
\end{align*}
Then there exists polynomials $\fP_1, \fP_2:(\bR^+)^{\times 7} \to \bR^+$ increasing in all variables such that
\begin{align}
\nonumber
\inf_{\Pi} &\sum_{T\in\scF_0}^{\gamma-, \alpha, \beta} \sup_{s, t\in[u,v]} \frac{\bW_{\Pi}^{p[T]}\big[ T\big] \Big( \bZ_{s, t}^{\sharp}, \hat{\bZ}_{s, t}^{\sharp}\Big)(\omega_0) }{|t-s|^{\gamma - \scG_{\alpha, \beta}[T]}} \leq \| f\|_{C_b^{n+1, (n+1)}}
\\
\label{eq:theorem:Stability-ContinImage}
&\cdot \Bigg( \rho_{(\alpha, \beta, p, q)}\Big( \rw, \hat{\rw}\Big)(\omega_0) \cdot \fP_1 + \bigg( d_{\rw, \hat{\rw}, \gamma, 0} \Big( \bX, \hat{\bX} \Big)(\omega_0)
+ d_{\rw, \hat{\rw}, \gamma} \Big( \bY, \hat{\bY} \Big) \bigg) \cdot \fP_2 \Bigg) 
\end{align}
where 
\begin{align*}
\fP_i:&=\fP_i\bigg( \sum_{T\in\scF_0}^{\gamma-, \alpha, \beta} \Big\| \big\langle \{\bX^{\sharp}\vee \hat{\bX}^{\sharp}, T \big\rangle(\omega_0) \Big\|_{p_x[T], \gamma - \scG_{\alpha, \beta}[T]}
, 
\sum_{T\in \scF}^{\gamma-, \alpha, \beta} \Big\| \big\langle \{\bX_u\vee\hat{\bX}_u\}, T \big\rangle(\omega_0) \Big\|_{p_x[T]}, 
\\
&\qquad 
\sum_{T\in\scF_0}^{\gamma-, \alpha, \beta} \Big\| \big\langle \{ \bY^{\sharp} \vee \hat{\bY}^{\sharp}, T \big\rangle \Big\|_{p_y[T], \gamma - \scG_{\alpha, \beta}[T]}
, 
\sum_{T\in \scF}^{\gamma-, \alpha, \beta} \Big\| \big\langle \{ \bY_u \vee \hat{\bY}_u\}, T\big\rangle \Big\|_{p_y[T]}, 
\\
&\qquad \sum_{T\in \scF}^{\gamma, \alpha, \beta} \Big\| \big\langle \{\rw\vee \hat{\rw} \}, T\big\rangle(\omega_0) \Big\|_{q[T], \scG_{\alpha, \beta}[T]}
, 
|v-u|^\alpha
, 
|v-u|^{\beta} \bigg). 
\end{align*}

In particular, suppose that $\exists \cO \subseteq \Omega$ such that $\forall \omega_0\in \cO$, $\exists M>0$ such that
\begin{equation}
\label{eq:theorem:Stability-RoughInt2}
\left.\begin{aligned}
\max\Big( \rho_{(\alpha, \beta, p, q), 0}\Big( \rw, \rId \Big)(\omega_0)
, & \quad
\rho_{(\alpha, \beta, p, q), 0}\Big( \hat{\rw}, \rId \Big)(\omega_0) \Big)<M, 
\\
\max\Big( \big\| \bX \big\|_{\rw, \gamma, p_y, q, 0} (\omega_0)
, & \quad
\big\| \hat{\bX} \big\|_{\hat{\rw}, \gamma, p_y, q, 0} (\omega_0) \Big)< M, 
\\
\max\Big( \big\| \bY \big\|_{\rw, \gamma, p_y, q}
, & \quad
\big\| \hat{\bY} \big\|_{\hat{\rw}, \gamma, p_y, q} \Big)< M. 
\end{aligned}\quad \right\rbrace
\end{equation}

Then for $\omega_0 \in \cO$, we have that there exists a constant $C_M$ dependent only on $M$ such that 
\begin{align}
\nonumber
d&_{\rw, \hat{\rw}, \gamma, 0}\Big( \bZ, \hat{\bZ} \Big)(\omega_0) 
\\
\label{eq:theorem:Stability-ContinImage-inc}
&\leq \| f\|_{C_b^{n+1, (n+1)}} \cdot C_M \cdot \bigg( \rho_{(\alpha, \beta, p, q)}\Big( \rw, \hat{\rw}\Big)(\omega_0) 
+ d_{\rw, \hat{\rw}, \gamma, 0} \Big( \bX, \hat{\bX} \Big)(\omega_0)
+ d_{\rw, \hat{\rw}, \gamma} \Big( \bY, \hat{\bY} \Big)\bigg). 
\end{align}
\end{theorem}

The proof of Theorem \ref{theorem:Stability-ContinImage} is delayed until Section \ref{subsubsection:Stability-ContinImage}. 

\begin{remark}
\label{remark:Polynomial-comments2}
Following on from Remark \ref{remark:Polynomial-comments1}, we have that the polynomials $\fP_1$ and $\fP_2$
\begin{align*}
\fP_1\Big( x_1, x_2, y_1, y_2, w, t_1, t_2 \Big) =& \sum_{i\in I_{1}} C_i \cdot x_1^{i_1} \cdot x_2^{i_2} \cdot y_1^{i_3} \cdot y_2^{i_4} \cdot w^{i_5} \cdot t_1^{i_6} \cdot t_2^{i_7}, 
\\
\fP_2\Big( x_1, x_2, y_1, y_2, w, t_1, t_2 \Big) =& \sum_{j \in J_{2}} C_j \cdot x_1^{j_1} \cdot x_2^{j_2} \cdot y_1^{j_3} \cdot y_2^{j_4} \cdot w^{j_5} \cdot t_1^{j_6} \cdot t_2^{j_7}. 
\end{align*}
in Equation \eqref{eq:theorem:Stability-ContinImage} satisifes the following identities:
\begin{itemize}
\item $i_1, ..., i_5, j_1, ..., j_5 \in \bN_0$
\item $i_1+i_2+i_3+i_4 \leq n+2$. In particular, this means that $i_1 + i_2 \leq n+2$ and $i_3+i_4 \leq n+2$. 
\item $j_1+j_2+j_3+j_4 \leq n+1$. In particular, this means that $j_1 + j_2 \leq n+1$ and $j_3+j_4 \leq n+1$. 
\item $i_5+1 \leq i_1 + i_2 + i_3 + i_4$ and $j_5 \leq j_1 + j_2 + j_3 + j_4$. 
\item While $i_6, i_7, j_6, j_7\in \bZ$ (that is, may be negative), we always have that 
$$
\alpha \cdot i_6 + \beta \cdot i_7\geq 0, \quad \alpha \cdot j_6 + \beta \cdot j_7\geq 0.
$$ 
\end{itemize}
\end{remark}

\subsection{Optimal transport and random controlled rough paths}

Inspired by Proposition \ref{prop:Regu-RCRP}, have the following:
\begin{proposition}
\label{proposition:CompareRCRP}
Let $\alpha, \beta>0$ and $\gamma:=\inf\{ \scG_{\alpha, \beta}[T]: T \in \scF, \scG_{\alpha, \beta}[T]>1-\alpha \}$. Let $(p,q)$ be a dual pair of integrability functionals and let $\rw, \hat{\rw}$ be $(\scH^{\gamma, \alpha, \beta}, p, q)$-probabilistic rough paths. 

Let $\bX \in \cD_{\rw}^{\gamma, p, q}$ and let $\hat{\bX} \in \cD_{\hat{\rw}}^{\gamma, p, q}$. Let $u, v\in [0,1]$ and let $\eta=|v-u|$. Let $\Pi\in \cP(\Omega\times \hat{\Omega})$ with marginals $\bP$. Then 

\begin{enumerate}
\item For $Y\in \scF_0^{\gamma-, \alpha, \beta}$, 
\begin{align}
\nonumber
\sup_{t\in[u, v]}& \bW_{\Pi}^{p[Y]}\big[Y\big]\Big( \bX_t, \hat{\bX}_t \Big)(\omega_0)
\\
\label{eq:proposition:CompareRCRP-1}
\leq&
\bW_{\Pi}^{p[Y]} \big[Y\big] \Big( \bX_u, \hat{\bX}_u\Big)(\omega_0) 
+ 
\bW_{\alpha\wedge(\gamma - \scG_{\alpha, \beta}[Y]), \Pi}^{p[Y]} \big[Y\big] \Big( \bX, \hat{\bX} \Big)(\omega_0)
\cdot
\eta^{\alpha\wedge(\gamma - \scG_{\alpha, \beta}[Y])}.  
\end{align}
\item We have that
\begin{align}
\nonumber
\sup_{s, t \in [u, v]}& \frac{\Big| \big\langle \bX_{s, t}-\hat{\bX}_{s, t}, \rId \big\rangle(\omega_0)\Big|}{|t-s|^{\alpha}} 
\\
\nonumber
\leq& \sum_{T\in\scF}^{\gamma-, \alpha, \beta} \Bigg( \sup_{t\in[u, v]} \Big\| \big\langle \{ \bX_t \vee \hat{\bX}_t\} , T \big\rangle(\omega_0) \Big\|_{p[T]} 
\cdot 
\bW_{\scG_{\alpha, \beta}[T], \Pi}^{q[T]} \big[T\big] \Big( \rw, \hat{\rw} \Big)(\omega_0)
\\
\nonumber
&\quad + \Big\| \big\langle \{ \rw \vee \hat{\rw}\}, T \big\rangle(\omega_0) \Big\|_{q[T], \scG_{\alpha, \beta}[T]} \cdot \sup_{t\in [u,v]} \bW_{\Pi}^{p[T]} \big[T \big] \Big( \bX_t, \hat{\bX}_t \Big)(\omega_0) \Bigg) 
\cdot 
\eta^{\scG_{\alpha, \beta}[T] - \alpha}
\\
\label{eq:proposition:CompareRCRP-2}
&+ \sup_{s, t\in[u, v]} \frac{ \Big| \big\langle \bX_{s, t}^{\sharp} - \hat{\bX}_{s, t}^{\sharp}, \rId \big\rangle (\omega_0) \Big|}{|t-s|^{\gamma}} \cdot \eta^{\gamma - \alpha}. 
\end{align}
\item For $Y \in \scF^{\gamma-,\alpha, \beta}$, 
\begin{align}
\nonumber
\sup_{s, t\in[u,v]}& \frac{ \bW_{\Pi}^{p[Y]}\big[Y\big] \Big( \bX_{s,t}, \hat{\bX}_{s,t} \Big)(\omega_0) }{|t-s|^{\alpha\wedge(\gamma - \scG_{\alpha, \beta}[Y])}}
\\
\nonumber
\leq& \sum_{T, \Upsilon\in \scF}^{\gamma-, \alpha, \beta} c'\Big(T, \Upsilon, Y\Big) \cdot \bigg( \Big\| \big\langle \{ \bX \vee \hat{\bX} \}, T \big\rangle(\omega_0) \Big\|_{p[T], \infty}
\cdot 
\bW_{\scG_{\alpha, \beta}[\Upsilon], \Pi}^{q[\Upsilon]}\big[\Upsilon\big] \Big( \rw, \hat{\rw}\Big)(\omega_0) 
\\
\nonumber
&+ \bW_{\infty, \Pi}^{p[T]}\big[T\big] \Big( \bX, \hat{\bX} \Big)(\omega_0)
\cdot 
\Big\| \big\langle \{ \rw \vee \hat{\rw}\}, \Upsilon \big\rangle(\omega_0)\Big\|_{q[\Upsilon], \scG_{\alpha, \beta}[\Upsilon]} \bigg) \cdot \eta^{\scG_{\alpha, \beta}[\Upsilon] - \alpha}
\\
\label{eq:proposition:CompareRCRP-3}
&+ \sup_{s, t\in[u,v]} \frac{ \bW_{\Pi}^{p[Y]}\big[Y\big]\Big( \bX_{s,t}^{\sharp}, \bX_{s,t}^{\sharp} \Big)(\omega_0)}{|t-s|^{\gamma - \scG_{\alpha, \beta}[Y]}} \cdot \eta^{\gamma - \scG_{\alpha, \beta}[Y] - \alpha\wedge(\gamma - \scG_{\alpha, \beta}[Y])} . 
\end{align}
\end{enumerate}

In particular, 
\begin{align}
\nonumber
\sup_{s, t\in[u, v]}& \frac{ \bW_{\Pi}^{p[Y]}\big[Y\big] \Big( \bX_{s,t}, \hat{\bX}_{s,t} \Big)(\omega_0) }{|t-s|^{\alpha\wedge(\gamma - \scG_{\alpha, \beta}[Y])}} 
\\
\nonumber
\leq& 
\sum_{T, \Upsilon\in \scF}^{\gamma-, \alpha, \beta} c'\Big( T, \Upsilon, Y\Big) \cdot \bigg( \bW_{\Pi}^{p[T]}\big[T\big]\Big( \bX_u, \hat{\bX}_u \Big)(\omega_0) \cdot \Big\| \big\langle \{ \rw\vee \hat{\rw}\}, \Upsilon \big\rangle(\omega_0) \Big\|_{q[\Upsilon], \scG_{\alpha, \beta}[\Upsilon]}
\\
\nonumber
&\qquad + \Big\| \big\langle \{ \bX_u \vee \hat{\bX}_u \}, T \big\rangle(\omega_0) \Big\|_{p[T]} \cdot \bW_{\scG_{\alpha, \beta}[\Upsilon], \Pi}^{q[\Upsilon]}\big[\Upsilon\big] \Big( \rw, \hat{\rw} \Big)(\omega_0) \bigg)\cdot \eta^{\scG_{\alpha, \beta}[\Upsilon] - \alpha} 
\\
\nonumber
&+ \sum_{T, \Upsilon\in \scF}^{\gamma-, \alpha, \beta} c'\Big( T, \Upsilon, Y\Big) \cdot \bigg( \bW_{\gamma-\scG_{\alpha, \beta}[T], \Pi}^{p[T]}\big[T\big]\Big( \bX^{\sharp}, \hat{\bX}^{\sharp} \Big)(\omega_0) \cdot \Big\| \big\langle \{ \rw\vee \hat{\rw}\}, \Upsilon \big\rangle(\omega_0) \Big\|_{q[\Upsilon], \scG_{\alpha, \beta}[\Upsilon]}
\\
\nonumber
&\qquad + \Big\| \big\langle \{ \bX^{\sharp} \vee \hat{\bX}^{\sharp} \}, T \big\rangle(\omega_0) \Big\|_{p[T]} \cdot \bW_{\scG_{\alpha, \beta}[\Upsilon], \Pi}^{q[\Upsilon]}\big[\Upsilon\big] \Big( \rw, \hat{\rw} \Big)(\omega_0) \bigg) \cdot 
\eta^{\gamma - \scG_{\alpha, \beta}[Y] - \alpha\wedge(\gamma - \scG_{\alpha, \beta}[Y])}
\\
\nonumber
&+\sum_{T, \Upsilon', \Upsilon \in \scF}^{\gamma-, \alpha, \beta} c'\Big( T, \Upsilon', \Upsilon, Y\Big) \cdot \eta^{\scG_{\alpha, \beta}[\Upsilon'] + \scG_{\alpha, \beta}[\Upsilon] - \alpha\wedge(\gamma-\scG_{\alpha, \beta}[Y])} \cdot \bigg( \bW_{\Pi}^{p[T]}\big[T\big] \Big( \bX_u, \hat{\bX}_u \Big)(\omega_0)
\\
\nonumber
&\qquad \cdot \Big\| \big\langle \{ \rw\vee \hat{\rw}\}, \Upsilon' \big\rangle(\omega_0) \Big\|_{q[\Upsilon'], \scG_{\alpha, \beta}[\Upsilon']} \cdot \Big\| \big\langle \{ \rw\vee \hat{\rw}\}, \Upsilon \big\rangle(\omega_0) \Big\|_{q[\Upsilon], \scG_{\alpha, \beta}[\Upsilon]}
\\
\nonumber
&+
\Big\| \big\langle \{ \bX_u \vee \hat{\bX}_u \}, T \big\rangle(\omega_0) \Big\|_{p[T]} \cdot \bW_{\scG_{\alpha, \beta}[\Upsilon'], \Pi}^{q[\Upsilon']}\big[\Upsilon'\big] \Big( \rw, \hat{\rw} \Big)(\omega_0) \cdot  \Big\| \big\langle \{ \rw\vee \hat{\rw}\}, \Upsilon \big\rangle(\omega_0) \Big\|_{q[\Upsilon], \scG_{\alpha, \beta}[\Upsilon]}
\\
\nonumber
&+
\Big\| \big\langle \{ \bX_u \vee \hat{\bX}_u \}, T \big\rangle(\omega_0) \Big\|_{p[T]} \cdot  \Big\| \big\langle \{ \rw\vee \hat{\rw}\}, \Upsilon' \big\rangle(\omega_0) \Big\|_{q[\Upsilon'], \scG_{\alpha, \beta}[\Upsilon']} 
\\
\nonumber
&\qquad \cdot \bW_{\scG_{\alpha, \beta}[\Upsilon], \Pi}^{q[\Upsilon]}\big[\Upsilon\big] \Big( \rw, \hat{\rw} \Big)(\omega_0) \bigg)
\\
\label{eq:proposition:CompareRCRP-4}
&+\bW_{\alpha \wedge (\gamma - \scG_{\alpha, \beta}[Y]), \Pi}^{p[Y]}\big[Y\big]\Big( \bX^{\sharp}, \hat{\bX}^{\sharp} \Big)(\omega_0) \cdot \eta^{\gamma - \scG_{\alpha, \beta}[Y] - \big[ \alpha\wedge(\gamma - \scG_{\alpha, \beta}[Y])\big]} 
\end{align}

Further, 
$$
\bE^0 \Bigg[ \bigg| \bW_{\alpha \wedge (\gamma - \scG_{\alpha, \beta}[Y]), \Pi}^{p[Y]}\big[Y\big] \Big( \bX, \hat{\bX} \Big)(\omega_0) \bigg|^{p[Y]} \Bigg] < \infty. 
$$
\end{proposition}

\begin{proof}
Using the norm properties of $\bW_\Pi^p$ from Equation \eqref{eq:WassersteinNorm1}, the proof of Equation \eqref{eq:proposition:CompareRCRP-1} is standard. Equation \eqref{eq:proposition:CompareRCRP-2} follows from Equation \eqref{eq:definition:RandomControlledRP1} and Equation \eqref{eq:proposition:CompareRCRP-3} follows from Equation \eqref{eq:definition:RandomControlledRP2}. The proof of Equation \eqref{eq:proposition:CompareRCRP-4} follows from a similar inductive argument to the proof of Equation \eqref{eq:prop:Regu-RCRP-4}. 
\end{proof}

\subsection{Proof of the results of Section \ref{subsec:Stability_RCRPs}}

Throughout this section, we will regularly use the estimate
\begin{equation}
\label{eq:TechLem_Products2}
\Big| \bigotimes_{i=1}^n x_i - \bigotimes_{i=1}^n \hat{x}_i \Big| 
\leq 
\sum_{k=1}^n \bigg( \prod_{i=1}^{k-1} |x_i| \bigg) \cdot \big| x_k - \hat{x}_k \big| \cdot \bigg( \prod_{i=k+1}^n \big| \hat{x}_i \big| \bigg) 
\end{equation}
which is just an adaption of Lemma \ref{lemma:TechLem_Products}. 

\subsubsection{Proof of Theorem \ref{theorem:Stability-RoughInt}}
\label{subsubsec:Proof-Reconstruction-Stability}

This proof is a reformulation of the proof of Theorem  \ref{theorem:Reconstruction} that additionally incorporates ideas of optimal transport. 

\begin{proof}[Proof of Theorem \ref{theorem:Stability-RoughInt}]
It is a quick exercise to verify that $(p_5, q)$ and $(p_6, q)$ satisfy Definition \ref{definition:(dual)_integrab-functional} and so are dual integrability functionals. 

Inspired by Proposition \ref{proposition:Reconstruction}, we define
\begin{align*}
\Xi_{s, t}(\omega_0):=& \sum_{T\in \scF_0}^{\gamma - \alpha, \alpha, \beta} \bigg( \bE^{H^T} \Big[ \big\langle \bX_s, T\big\rangle(\omega_0, \omega_{H^T}) \cdot \big\langle \rw_{s, t}, \lfloor T\rfloor \big\rangle(\omega_0, \omega_{H^T}) \Big] 
\\
&\qquad - \hat{\bE}^{H^T}\Big[ \big\langle \hat{\bX}_s, T\big\rangle(\omega_0, \hat{\omega}_{H^T}) \cdot \big\langle \hat{\rw}_{s, t}, \lfloor T \rfloor \big\rangle(\omega_0, \hat{\omega}_{H^T}) \Big] \bigg). 
\end{align*}
Arguing as before, we can show that
\begin{align*}
\Xi_{s, t, u}(\omega_0) =& - \sum_{Y\in \scF_0}^{\gamma-\alpha, \alpha, \beta} \bigg( \bE^{H^Y}\Big[ \big\langle \bX_{s, t}^{\sharp}, Y\big\rangle(\omega_0, \omega_{H^Y}) \cdot \big\langle \rw_{t, u}, \lfloor Y \rfloor \big\rangle(\omega_0, \omega_{H^Y}) \Big]
\\
&\qquad - \hat{\bE}^{H^Y}\Big[ \big\langle \hat{\bX}_{s, t}^{\sharp}, Y\big\rangle(\omega_0, \hat{\omega}_{H^Y}) \cdot \big\langle \hat{\rw}_{t, u}, \lfloor Y \rfloor \big\rangle(\omega_0, \hat{\omega}_{H^Y}) \Big] \bigg). 
\end{align*}
Then (recalling Equation \eqref{eq:Max_2norms})
\begin{align*}
\sup_{s, t, u\in[0,1]}& \frac{\big| \Xi_{s, t, u}(\omega_0) \big|}{|u - s|^{\gamma + \alpha}} 
\\
\leq &\sum_{Y\in \scF_0}^{\gamma - \alpha, \alpha, \beta} 
\bigg( \bW_{\gamma - \scG_{\alpha, \beta}[Y], \Pi}^{p_5[Y]}\big[ Y \big]\Big( \bX^{\sharp}, \hat{\bX}^{\sharp}\Big)(\omega_{0}) 
\cdot
\Big\| \big\langle \{ \rw \vee \hat{\rw}\}, \lfloor Y\rfloor_i \big\rangle(\omega_{0}) \Big\|_{q[\lfloor Y \rfloor_i], \scG_{\alpha, \beta}[Y]+\alpha}
\\
&+ \Big\| \big\langle \{ \bX^{\sharp} \vee \hat{\bX}^{\sharp} \}, Y \big\rangle (\omega_{0}) \Big\|_{p_5[Y], \gamma - \scG_{\alpha, \beta}[Y]} 
\cdot
\bW_{\scG_{\alpha, \beta}[Y]+\alpha, \Pi}^{q[\lfloor Y \rfloor_i]}\big[ \lfloor Y\rfloor_i \big] \Big( \rw, \hat{\rw}\Big)(\omega_0) \bigg)<\infty
\end{align*}
$(\Omega_0, \bP)$-almost surely. For $r>1$ such that
$$
\frac{1}{r}:= \sup_{T\in \scF^{\gamma-\alpha, \alpha, \beta}}\bigg( \frac{1}{p_5[T]} + \frac{1}{q[\lfloor T \rfloor ]}\bigg)
$$
we have that
$$
\bE^0\bigg[ \sup_{s, t, u\in[0,1]} \frac{\big| \Xi_{s, t, u}(\omega_0) \big|^r}{|u - s|^{r(\gamma + \alpha)}} \bigg]<\infty. 
$$
We apply the Sewing Lemma to conclude that
\begin{align*}
\Bigg| &\bigg( \int_s^t X_r dW_r(\omega_0) - \int_s^t \hat{X}_r d\hat{W}_r(\omega_0) \bigg) 
\\
&- \sum_{T\in\scF_0}^{\gamma-\alpha, \alpha, \beta} \bigg( \bE^{H^T}\Big[ \big\langle \bX_s, T\big\rangle \cdot \big\langle \rw_{s, t}, \lfloor T \rfloor \big\rangle \Big](\omega_0) - \hat{\bE}^{H^T}\Big[ \big\langle \hat{\bX}_s, T \big\rangle \cdot \big\langle \hat{\rw}_{s, t}, \lfloor T \rfloor \big\rangle \Big](\omega_0) \bigg) \Bigg|
\\
\leq& C \sum_{T \in \scF_0}^{\gamma - \alpha, \alpha, \beta} \sup_{u, \upsilon, v\in [0,1]} \Bigg| \frac{\bE^{H^T}\Big[ \big\langle \bX_{u, \upsilon}^{\sharp}, T\big\rangle \cdot \big\langle \rw_{\upsilon, v}, \lfloor T\rfloor\big\rangle \Big](\omega_0)}{|v-u|^{\gamma+\alpha} } 
\\
&\qquad - \frac{\hat{\bE}^{H^T}\Big[ \big\langle \hat{\bX}_{u, \upsilon}^{\sharp}, T\big\rangle \cdot \big\langle \hat{\rw}_{\upsilon, v}, \lfloor T\rfloor\big\rangle \Big](\omega_0)}{|v-u|^{\gamma+\alpha} } \Bigg| \cdot |t-s|^{\gamma+\alpha}. 
\end{align*}

Inspired by Equation \eqref{eq:norm-Phi-1.1}, we have that for any choice of $\Pi \in \cP(\Omega \times \hat{\Omega})$ with marginals $(\Omega, \bP)$ and $(\hat{\Omega}, \hat{\bP})$ and using the notation $\eta=|v-u|$
\begin{align}
\nonumber
\bW_{\Pi}^{p[\rId]}\big[\rId\big]\Big( \Phi&[\bX]_{u, v}^{\sharp}, \Phi[\hat{\bX}]_{u, v}^{\sharp}\Big) (\omega_0) 
\\
\nonumber
\leq& C \sum_{T\in \scF_0}^{\gamma - \alpha, \alpha, \beta} \Bigg( \bW_{\gamma - \scG_{\alpha, \beta}[T], \Pi}^{p_5[T]} \big[T\big] \Big( \bX^{\sharp}, \hat{\bX}^{\sharp} \Big)(\omega_0)
\cdot 
\Big\| \big\langle \{ \rw \vee \hat{\rw} \}, \lfloor T \rfloor \big\rangle(\omega_0) \Big\|_{q[\lfloor T \rfloor ], \scG_{\alpha, \beta}[\lfloor T \rfloor]}
\\
\nonumber
&\quad + 
\Big\| \big\langle \{ \bX^{\sharp} \vee \hat{\bX}^{\sharp}\}, T\big\rangle(\omega_0) \Big\|_{p_5[T], \gamma - \scG_{\alpha, \beta}[T]} 
\cdot 
\bW_{\scG_{\alpha, \beta}[\lfloor T \rfloor], \Pi}^{q[\lfloor T \rfloor]} \big[ \lfloor T \rfloor \big]\Big( \rw, \hat{\rw} \Big)(\omega_0) \Bigg) \cdot \eta^{\gamma+\alpha}
\\
\nonumber
&+\sum_{\substack{T\in \scF \\ \scG_{\alpha, \beta}[\lfloor T \rfloor] = \gamma}} \Bigg( \bW_{\Pi}^{p_5[T]} \big[T\big] \Big( \bX_u, \hat{\bX}_u \Big)(\omega_0)
\cdot
\Big\| \big\langle \{ \rw \vee \hat{\rw} \}, \lfloor T \rfloor \big\rangle(\omega_0) \Big\|_{q[\lfloor T \rfloor], \scG_{\alpha, \beta}[\lfloor T\rfloor]} 
\\
\label{eq:theorem:Stability-RoughInt-1.1}
&\quad + \Big\| \big\langle \{ \bX_u \vee \hat{\bX}_u \}, T \big\rangle(\omega_0) \Big\|_{p_5[T]} \cdot \bW_{\scG_{\alpha, \beta}[\lfloor T \rfloor], \Pi}^{q[\lfloor T \rfloor]} \big[\lfloor T \rfloor\big] \Big( \rw, \hat{\rw} \Big)(\omega_0) \Bigg) \cdot \eta^{\gamma}. 
\end{align}
Similarly, inspired by Equation \eqref{eq:norm-Phi-1.3}, we have that
\begin{align*}
&\Big\langle \Phi[\bX]_{u, v}^{\sharp}, \lfloor Y \rfloor \Big\rangle (\omega_0, \omega_{H^Y} ) - \Big\langle \Phi[\hat{\bX}]_{u, v}^{\sharp}, \lfloor Y \rfloor \Big\rangle (\omega_0, \hat{\omega}_{H^Y} )
\\
&= \Big\langle \bX_{u, v}^{\sharp}, Y \Big\rangle (\omega_0, \omega_{H^Y} ) - \Big\langle \hat{\bX}_{u, v}^{\sharp}, Y \Big\rangle (\omega_0, \hat{\omega}_{H^Y} )
\\
&+ \sum_{\substack{T, \Upsilon\in \scF \\ \scG_{\alpha, \beta}[T]\in[\gamma-\alpha, \gamma)}} c'\Big( T, \Upsilon, Y\Big) \cdot \bigg( \bE^{E^{T, \Upsilon, Y}}\Big[ \big\langle \bX_u, T\big\rangle(\omega_0, \omega_{\phi^{T, \Upsilon, Y}[H^T]})
\cdot \big\langle \rw_{u,v}, \Upsilon \big\rangle(\omega_0, \omega_{\varphi^{T, \Upsilon, Y}[H^\Upsilon]}) \Big] 
\\
&\qquad - \hat{\bE}^{E^{T, \Upsilon, Y}}\Big[ \big\langle \hat{\bX}_u, T\big\rangle(\omega_0, \hat{\omega}_{\phi^{T, \Upsilon, Y}[H^T]})
\cdot \big\langle \hat{\rw}_{u,v}, \Upsilon \big\rangle(\omega_0, \hat{\omega}_{\varphi^{T, \Upsilon, Y}[H^\Upsilon]}) \Big] \bigg)
\end{align*}
so that
\begin{align}
\nonumber
\bW&_{\Pi}^{p_6[\lfloor Y \rfloor]} \big[\lfloor Y \rfloor\big] \Big( \Phi[\bX]_{u,v}^{\sharp}, \Phi[\hat{\bX}]_{u,v}^{\sharp} \Big)(\omega_0)
\\
\nonumber
\leq& \bW_{\gamma - \scG_{\alpha, \beta}[Y], \Pi}^{p_5[Y]}\big[Y\big] \Big( \bX^{\sharp}, \hat{\bX}^{\sharp} \Big)(\omega_0) \cdot \eta^{\gamma - \scG_{\alpha, \beta}[Y]}
\\
\nonumber
&+ \hspace{-10pt}\sum_{\substack{T, \Upsilon \in \scF \\ \scG_{\alpha, \beta}[T]\in[\gamma-\alpha, \gamma)}}\hspace{-10pt} c'\Big( T, \Upsilon, Y\Big) \cdot \Bigg( \bW_{\Pi}^{p_5[T]} \big[T\big]\Big( \bX_u, \hat{\bX}_u \Big)(\omega_0) \cdot \Big\| \big\langle \{ \rw \vee \hat{\rw} \}, \Upsilon \big\rangle(\omega_0)\Big\|_{q[\Upsilon], \scG_{\alpha, \beta}[\Upsilon]}
\\
\nonumber
&\quad + \Big\| \big\langle \{ \bX_u \vee \hat{\bX}_u \}, T \big\rangle(\omega_0) \Big\|_{p_5[T]} \cdot \bW_{\scG_{\alpha, \beta}[\Upsilon], \Pi}^{q[\Upsilon]} \big[ \Upsilon \big] \Big( \rw, \hat{\rw} \Big)(\omega_0) \Bigg)\cdot \eta^{\scG_{\alpha, \beta}[\Upsilon]}
\\
\nonumber
&+ \hspace{-10pt}\sum_{\substack{T, \Upsilon \in \scF \\ \scG_{\alpha, \beta}[T]\in[\gamma-\alpha, \gamma)}}\hspace{-10pt} c'\Big( T, \Upsilon, Y\Big) \cdot \Bigg( \bW_{\gamma-\scG_{\alpha, \beta}[T], \Pi}^{p_5[T]} \big[T\big]\Big( \bX^{\sharp}, \hat{\bX}^{\sharp} \Big)(\omega_0) \cdot \Big\| \big\langle \{ \rw \vee \hat{\rw} \}, \Upsilon \big\rangle(\omega_0)\Big\|_{q[\Upsilon], \scG_{\alpha, \beta}[\Upsilon]}
\\
\label{eq:theorem:Stability-RoughInt-1.2}
&\quad + \Big\| \big\langle \{ \bX^{\sharp} \vee \hat{\bX}^{\sharp} \}, T \big\rangle(\omega_0) \Big\|_{p_5[T], \gamma - \scG_{\alpha, \beta}[T]} \cdot \bW_{\scG_{\alpha, \beta}[\Upsilon], \Pi}^{q[\Upsilon]} \big[ \Upsilon \big] \Big( \rw, \hat{\rw} \Big)(\omega_0) \Bigg)\cdot \eta^{\gamma - \scG_{\alpha, \beta}[Y]}. 
\end{align}

Therefore, combining Equation \eqref{eq:theorem:Stability-RoughInt-1.1} and Equation \eqref{eq:theorem:Stability-RoughInt-1.2} yields Equation \eqref{eq:theorem:Stability-RoughInt-inccom}. By restricting ourselves to when $\omega_0\in \cO$, thanks to the assumption of Equation \eqref{eq:theorem:Stability-RoughInt}, we can restate \eqref{eq:theorem:Stability-RoughInt-inccom} as a local Lipschitz type estimate and taking an infimum over the choice of $\Pi$ yields Equation \eqref{eq:theorem:Stability-RoughInt-inc}.
\end{proof}

\subsubsection{Proof of Theorem \ref{theorem:Stability-ContinImage}}
\label{subsubsection:Stability-ContinImage}

Similarly, this next proof is a reformulation of the proof of Theorem \ref{theorem:ContinIm-RCRPs} with additional ideas from optimal transport. 

Inspired by Equation \eqref{eq:a_integral2}, for any $T\in \scF_0$ we define
$$
\bW_{\Pi}^{[a],p_x,p_y}\big[T\big]\Big( [\bX, \bY], [\hat{\bX}, \hat{\bY}] \Big) (\omega_{a_i}) = 
\begin{cases}
\bW_{\Pi}^{p_x}\big[T\big]\Big( \bX, \hat{\bX}\Big) (\omega_{0})
\quad & \quad \mbox{if} \quad a_i=0, 
\\
\bW_{\Pi}^{p_y}\big[T\big]\Big( \bY, \hat{\bY}\Big) 
\quad & \quad \mbox{if} \quad a_i>0. 
\end{cases}
$$

\begin{proof}[Proof of Theorem \ref{theorem:Stability-ContinImage}]
Firstly, we consider the empty Lions tree $\rId$: Thanks to Equation \eqref{eq:theorem:ContinIm-RCRPs_1stRem}, we can write 
$$
\bW_{\Pi}^{p[\rId]}\big[\rId\big]\Big( \bZ_{s, t}^{\sharp}, \hat{\bZ}_{s, t}^{\sharp}\Big) (\omega_0) = \hyperlink{Eq:Stab_1_one}{(1)} + \hyperlink{Eq:Stab_1_two}{(2)} + \hyperlink{Eq:Stab_1_three}{(3)}
$$
where
\begin{align*}
\hypertarget{Eq:Stab_1_one}{(1)} =& \sum_{a\in A^{\alpha, \beta} } \frac{1}{|a|!} \sum_{\substack{T_1, ..., T_{|a|}\in \scF \\ \scG_{\alpha, \beta}\big[ \cE^a[T_1, ..., T_{|a|}]\big]\geq \gamma}}^{\gamma-, \alpha, \beta} \Bigg( \bE^{H^{\cE^a[T_1, ..., T_{|a|}]}}\Bigg[  \bE^{Z^a[T_1, ..., T_{|a|}]} \bigg[ \partial_a f\Big( \big\langle \bX_{s}, \rId \big\rangle (\omega_0), \cL^{\langle \bY_s, \rId\rangle}, ... \Big)
\\
&\qquad  \cdot
\bigotimes_{r=1}^{|a|} \Big\langle \big[ \bX, \bY \big]_s, T_r \Big\rangle (\omega_{\tilde{h}_{a_r}^T}, \omega_{H^{T_r}}) \bigg] \cdot \Big\langle \rw_{s, t}, \cE^a[T_1, ..., T_{|a|}] \Big\rangle(\omega_0, \omega_{H^{\cE^a[T_1, ..., T_{|a|}]}}) \Bigg]
\\
&-\hat{\bE}^{H^{\cE^a[T_1, ..., T_{|a|}]}}\Bigg[ \hat{\bE}^{Z^a[T_1, ..., T_{|a|}]} \bigg[ \partial_a f\Big( \big\langle \hat{\bX}_{s}, \rId \big\rangle (\omega_0), \cL^{\langle \hat{\bY}_s, \rId\rangle}, ... \Big) \cdot
\bigotimes_{r=1}^{|a|} \Big\langle \big[ \hat{\bX}, \hat{\bY} \big]_s, T_r \Big\rangle (\hat{\omega}_{\tilde{h}_{a_r}^T}, \hat{\omega}_{H^{T_r}}) \bigg] 
\\
&\qquad \cdot \Big\langle \hat{\rw}_{s, t}, \cE^a[T_1, ..., T_{|a|}] \Big\rangle(\omega_0, \hat{\omega}_{H^{\cE^a[T_1, ..., T_{|a|}]}}) \Bigg] \Bigg), 
\end{align*}
\begin{align*}
\hypertarget{Eq:Stab_1_two}{(2)} =& \sum_{a\in A^{\alpha, \beta}} \frac{1}{|a|!} \Bigg( \bE^{1, ..., m[a]}\Bigg[
\partial_a f\Big( \big\langle \bX_{s}, \rId \big\rangle (\omega_0), \cL^{\langle \bY_s, \rId\rangle}, ..., \big\langle \bY_{s}, \rId \big\rangle (\omega_{m[a]}) \Big) 
\\
&\qquad \cdot \sum_{k=1}^{|a|} \bigg( \bigotimes_{r=1}^{k-1} \Big\langle \fJ\big[ \bX, \bY\big]_{s, t}, \rId\Big\rangle(\omega_{a_r}) \bigg) \otimes \Big\langle \big[\bX, \bY\big]_{s, t}^{\sharp}, \rId\Big\rangle(\omega_{a_k}) \otimes \bigg( \bigotimes_{r=k+1}^{|a|} \Big\langle \big[ \bX, \bY\big]_{s, t}, \rId\Big\rangle(\omega_{a_r}) \bigg) \Bigg]
\\
&-\hat{\bE}^{1, ..., m[a]}\Bigg[
\partial_a f\Big( \big\langle \hat{\bX}_{s}, \rId \big\rangle (\omega_0), \cL^{\langle \hat{\bY}_s, \rId\rangle}, ..., \big\langle \hat{\bY}_{s}, \rId \big\rangle (\hat{\omega}_{m[a]}) \Big) 
\\
&\qquad \cdot \sum_{k=1}^{|a|} \bigg( \bigotimes_{r=1}^{k-1} \Big\langle \fJ\big[ \hat{\bX}, \hat{\bY} \big]_{s, t}, \rId\Big\rangle(\hat{\omega}_{a_r}) \bigg) \otimes \Big\langle \big[\hat{\bX}, \hat{\bY} \big]_{s, t}^{\sharp}, \rId\Big\rangle(\hat{\omega}_{a_k}) \otimes \bigg( \bigotimes_{r=k+1}^{|a|} \Big\langle \big[ \hat{\bX}, \hat{\bY} \big]_{s, t}, \rId\Big\rangle(\hat{\omega}_{a_r}) \bigg) \Bigg]
\end{align*}
and 
\begin{align*}
\hypertarget{Eq:Stab_1_three}{(3)} =&  \sum_{a\in A_\ast^{\alpha, \beta}} \frac{1}{|a|!} \Bigg( \bE^{1, ..., m[a]}\bigg[ \partial_a f\Big( \big\langle \bX_s, \rId \big\rangle (\omega_0), \cL^{\langle \bY_s, \rId\rangle}, ...\Big) \cdot \bigotimes_{r=1}^{|a|} \Big\langle \big[ \bX, \bY \big]_{s,t}, \rId\Big\rangle (\omega_{a_r}) \bigg]
\\
&\quad - \hat{\bE}^{1, ..., m[a]}\bigg[ \partial_a f\Big( \big\langle \hat{\bX}_s, \rId \big\rangle (\omega_0), \cL^{\langle \hat{\bY}_s, \rId\rangle}, ...\Big) \cdot \bigotimes_{r=1}^{|a|} \Big\langle \big[ \hat{\bX}, \hat{\bY} \big]_{s,t}, \rId\Big\rangle (\hat{\omega}_{a_r}) \bigg] \Bigg)
\\
+& \frac{1}{ (n -1)!} \sum_{a \in \A{n} } \Bigg( \bE^{1, ..., m[a]}\bigg[ f^a \Big[ \big\langle \bX_s, \rId \big\rangle (\omega_0), \big\langle \bX_t, \rId \big\rangle (\omega_0), \Pi^{\langle \bY_s, \rId\rangle, \langle \bY_t, \rId\rangle} \Big] \cdot \bigotimes_{r=1}^{n} \Big\langle \big[ \bX, \bY \big]_{s, t}(\omega_{a_r}), \rId\Big\rangle \bigg]
\\
&-\hat{\bE}^{1, ..., m[a]}\bigg[ f^a \Big[ \big\langle \hat{\bX}_s, \rId \big\rangle (\omega_0), \big\langle \hat{\bX}_t, \rId \big\rangle (\omega_0), \Pi^{\langle \hat{\bY}_s, \rId\rangle, \langle \hat{\bY}_t, \rId\rangle} \Big] \cdot \bigotimes_{r=1}^{n} \Big\langle \big[ \hat{\bX}, \hat{\bY} \big]_{s, t}(\hat{\omega}_{a_r}), \rId\Big\rangle \bigg] \Bigg). 
\end{align*}
We address each of these terms individually. By using Equation \eqref{eq:TechLem_Products2}, we get that
\begin{align}
\nonumber
\hyperlink{Eq:Stab_1_one}{(1)} &= \sum_{a\in A^{\alpha, \beta} } \frac{1}{|a|!} \sum_{\substack{T_1, ..., T_{|a|}\in \scF \\ \scG_{\alpha, \beta}\big[ \cE^a[T_1, ..., T_{|a|}]\big]\geq \gamma}}^{\gamma-, \alpha, \beta} (\bE\times \hat{\bE})^{H^{\cE^a[T_1, ..., T_{|a|}]}}\Bigg[ (\bE\times \hat{\bE})^{ Z^a[T_1, ..., T_{|a|}]} \bigg[ 
\\
\nonumber
&\Bigg( \partial_a f\Big( \big\langle \bX_{s}, \rId \big\rangle (\omega_0), \cL^{\langle \bY_s, \rId\rangle}, ... \Big)(...,\omega_{\tilde{h}_{m[a]}^T})  - \partial_a f\Big( \big\langle \hat{\bX}_{s}, \rId \big\rangle (\omega_0), \cL^{\langle \hat{\bY}_s, \rId\rangle}, ... \Big)(...,\hat{\omega}_{\tilde{h}_{m[a]}^T}) \Bigg)
\\
\label{eq:theorem:Stability-ContinImage:one.1}
&\qquad \cdot
\bigotimes_{r=1}^{|a|} \Big\langle \big[ \hat{\bX}, \hat{\bY} \big]_s, T_r \Big\rangle (\hat{\omega}_{\tilde{h}_{a_r}^T}, \hat{\omega}_{H^{T_r}} ) \cdot \Big\langle \hat{\rw}_{s, t}, \cE^a[T_1, ..., T_{|a|}] \Big\rangle(\omega_0, \hat{\omega}_{H^{\cE^a[T_1, ..., T_{|a|}]}}) 
\\
\nonumber
+& \partial_a f\Big( \big\langle \bX_{s}, \rId \big\rangle (\omega_0), \cL^{\langle \bY_s, \rId\rangle}, ... \Big)(...,\omega_{\tilde{h}_{m[a]}^T}) \cdot \Bigg( \sum_{k=1}^{|a|} \bigg( \bigotimes_{r=1}^{k-1} \Big\langle \big[ \bX, \bY \big]_s, T_r \Big\rangle (\omega_{\tilde{h}_{a_r}^T}, \omega_{H^{T_r}}) \Big) 
\\
\nonumber
&\quad \otimes \Big\langle \big[ \bX, \bY \big]_s - \big[ \hat{\bX}, \hat{\bY} \big]_s, T_k \Big\rangle (\omega_{\tilde{h}_{a_k}^T}, \omega_{H^{T_k}}, \hat{\omega}_{\tilde{h}_{a_k}^T}, \hat{\omega}_{H^{T_k}}) \otimes \bigg( \bigotimes_{r=k+1}^{|a|} \Big\langle \big[ \bX, \bY \big]_s, T_r \Big\rangle (\hat{\omega}_{\tilde{h}_{a_r}^T}, \hat{\omega}_{H^{T_r}}) \bigg) \Bigg)
\\
\label{eq:theorem:Stability-ContinImage:one.2}
&\quad \cdot \Big\langle \hat{\rw}_{s, t}, \cE^a[T_1, ..., T_{|a|}] \Big\rangle(\omega_0, \hat{\omega}_{H^{\cE^a[T_1, ..., T_{|a|}]}})
\\
\nonumber
+& \partial_a f\Big( \big\langle \bX_{s}, \rId \big\rangle (\omega_0), \cL^{\langle \bY_s, \rId\rangle}, ... \Big) \cdot \bigotimes_{r=1}^{|a|} \Big\langle \big[ \bX, \bY \big]_s, T_r \Big\rangle (\omega_{\tilde{h}_{a_r}^T}, \omega_{H^{T_r}})) 
\\
\label{eq:theorem:Stability-ContinImage:one.3}
&\quad \cdot \Big\langle \rw_{s, t} - \hat{\rw}_{s, t}, \cE^a[T_1, ..., T_{|a|}] \Big\rangle(\omega_0, \omega_{H^{\cE^a[T_1, ..., T_{|a|}]}}, \hat{\omega}_{H^{\cE^a[T_1, ..., T_{|a|}]}}) \bigg] \Bigg]
\end{align}
Addressing each of these terms in turn, we see that
\begin{align*}
\eqref{eq:theorem:Stability-ContinImage:one.1} 
\leq&
\sum_{a\in A^{\alpha, \beta} } \frac{1}{|a|!} \sum_{\substack{T_1, ..., T_{|a|}\in \scF \\ \scG_{\alpha, \beta}\big[ \cE^a[T_1, ..., T_{|a|}]\big]\geq \gamma}}^{\gamma-, \alpha, \beta}
\bigg( \Big\| \partial_a f\Big\|_{\lip, 0} \cdot \bW_{\Pi}^{p_x[\rId]}\big[\rId\big]\Big( \bX_s, \hat{\bX}_s\Big)(\omega_0)
\\
&+ \Big\| \partial_a f\Big\|_{\lip, \mu} \cdot \bW_{\Pi}^{p_y[\rId]}\big[\rId\big]\Big( \bY_s, \hat{\bY}_s\Big) + \sum_{r=1}^{m[a]} \Big\| \partial_a f\Big\|_{\lip, r} \cdot \bW_{\Pi}^{p_y[\rId]}\big[\rId\big]\Big( \bY_s, \hat{\bY}_s\Big) \bigg)
\\
&\cdot
\Big\| \big\langle \hat{\bX}_s, T_r \big\rangle(\omega_0) \Big\|_{p_x[T_r]}^{l[a]_0} \cdot\prod_{r=1}^{m[a]} \Big\| \big\langle \hat{\bY}_s, T_r \big\rangle\Big\|_{p_y[T_r]}^{l[a]_r} \cdot \Big\| \big\langle \hat{\rw}_{s, t}, \cE^a[T_1, ..., T_{|a|}] \big\rangle(\omega_0) \Big\|_{q\big[ \cE^a[T_1, ..., T_{|a|}] \big]}, 
\end{align*}
\begin{align*}
\eqref{eq:theorem:Stability-ContinImage:one.2} 
\leq& 
\sum_{a\in A^{\alpha, \beta} } \frac{1}{|a|!} \sum_{\substack{T_1, ..., T_{|a|}\in \scF \\ \scG_{\alpha, \beta}\big[ \cE^a[T_1, ..., T_{|a|}]\big]\geq \gamma}}^{\gamma-, \alpha, \beta} \Big\| \partial_a f\Big\|_\infty \cdot \Bigg( \sum_{k=1}^{|a|} \prod_{r=1}^{k-1} \Big\| \big\langle [\bX, \bY]_s, T_r \big\rangle(\omega_{a_r}) \Big\|_{[a],p_x, p_y}  
\\
&\cdot \bW_{\Pi}^{[a],p_x,p_y} \big[T_k\big]\Big( [\bX,\bY]_s, [ \hat{\bX},\hat{\bY}]_s\Big)(\omega_{a_r}) \cdot \prod_{r=k+1}^{|a|} \Big\| \big\langle [\hat{\bX}, \hat{\bY}]_s, T_r \big\rangle(\omega_{a_r}) \Big\|_{[a],p_x,p_y}
\\
&\quad \cdot \Big\| \big\langle \hat{\rw}_{s, t}, \cE^a[T_1, ..., T_{|a|}] \big\rangle(\omega_0) \Big\|_{q\big[ \cE^a[T_1, ..., T_{|a|}] \big]}, 
\end{align*}
and
\begin{align*}
\eqref{eq:theorem:Stability-ContinImage:one.3} 
\leq& 
\sum_{a\in A^{\alpha, \beta} } \frac{1}{|a|!} \sum_{\substack{T_1, ..., T_{|a|}\in \scF \\ \scG_{\alpha, \beta}\big[ \cE^a[T_1, ..., T_{|a|}]\big]\geq \gamma}}^{\gamma-, \alpha, \beta} \Big\| \partial_a f\Big\|_\infty 
\cdot
\Big\| \big\langle \bX_s, T_r \big\rangle(\omega_0) \Big\|_{p_x[T_r]}^{l[a]_0} \cdot\prod_{r=1}^{m[a]} \Big\| \big\langle \bY_s, T_r \big\rangle\Big\|_{p_y[T_r]}^{l[a]_r} 
\\
&\quad \cdot \bW_{\Pi}^{q\big[ \cE^a[T_1, ..., T_{|a|}] \big]} \big[ \cE^a[T_1, ..., T_{|a|}] \big] \Big( \rw_{s, t}, \hat{\rw}_{s, t} \Big)(\omega_0). 
\end{align*}

Secondly,
\begin{align}
\nonumber
\hyperlink{Eq:Stab_1_two}{(2)} &= \sum_{a\in A^{\alpha, \beta}} \frac{1}{|a|!} (\bE \times \hat{\bE})^{1, ..., m[a]}\Bigg[ 
\\
\nonumber
& \Bigg( \partial_a f\Big( \big\langle \bX_{s}, \rId \big\rangle (\omega_0), \cL^{\langle \bY_s, \rId\rangle}, ... \Big)(...,\omega_{m[a]})  - \partial_a f\Big( \big\langle \hat{\bX}_{s}, \rId \big\rangle (\omega_0), \cL^{\langle \hat{\bY}_s, \rId\rangle}, ... \Big)(...,\hat{\omega}_{m[a]}) \Bigg)
\\
\label{eq:theorem:Stability-ContinImage:two.1}
& \cdot \sum_{k=1}^{|a|} \bigg( \bigotimes_{r=1}^{k-1} \Big\langle \fJ\big[ \hat{\bX}, \hat{\bY} \big]_{s, t}, \rId\Big\rangle(\hat{\omega}_{a_r}) \bigg) \otimes \Big\langle \big[\hat{\bX}, \hat{\bY} \big]_{s, t}^{\sharp}, \rId\Big\rangle(\hat{\omega}_{a_k}) \otimes \bigg( \bigotimes_{r=k+1}^{|a|} \Big\langle \big[ \hat{\bX}, \hat{\bY} \big]_{s, t}, \rId\Big\rangle(\hat{\omega}_{a_r}) \bigg) 
\\
\nonumber
+&\partial_a f\Big( \big\langle \bX_{s}, \rId \big\rangle (\omega_0), \cL^{\langle \bY_s, \rId\rangle}, ... \Big)(...,\omega_{m[a]}) \cdot \sum_{k=1}^{|a|} \Bigg( \sum_{l=1}^{k-1} \bigg( \bigotimes_{r=1}^{l-1} \Big\langle \fJ[\bX, \bY]_{s, t}, \rId\Big\rangle(\omega_{a_r}) \bigg) 
\\
\nonumber
&\quad \otimes \Big\langle \fJ[\bX, \bY]_{s, t} - \fJ[\hat{\bX}, \hat{\bY}]_{s, t}, \rId\Big\rangle(\omega_{a_l}, \hat{\omega}_{a_l}) \otimes \bigg( \bigotimes_{r=l+1}^{k-1} \Big\langle \fJ[\hat{\bX}, \hat{\bY}]_{s, t}, \rId\Big\rangle(\omega_{a_r}) \Bigg)
\\
\label{eq:theorem:Stability-ContinImage:two.2}
&\quad \otimes \Big\langle \big[\hat{\bX}, \hat{\bY} \big]_{s, t}^{\sharp}, \rId\Big\rangle(\hat{\omega}_{a_k}) \otimes \bigg( \bigotimes_{r=k+1}^{|a|} \Big\langle \big[ \hat{\bX}, \hat{\bY} \big]_{s, t}, \rId\Big\rangle(\hat{\omega}_{a_r}) \bigg) \Bigg)
\\
\nonumber
+& \partial_a f\Big( \big\langle \bX_{s}, \rId \big\rangle (\omega_0), \cL^{\langle \bY_s, \rId\rangle}, ... \Big)(...,\omega_{m[a]}) 
\cdot 
\sum_{k=1}^{|a|} \bigg( \bigotimes_{r=1}^{k-1} \Big\langle \fJ\big[ \bX, \bY \big]_{s, t}, \rId\Big\rangle(\omega_{a_r}) \bigg) 
\\
\label{eq:theorem:Stability-ContinImage:two.3}
&\quad \otimes 
\Big\langle [\bX, \bY]_{s, t}^{\sharp} - [\hat{\bX}, \hat{\bY} ]_{s, t}^{\sharp}, \rId\Big\rangle(\omega_{a_k}, \hat{\omega}_{a_k}) 
\otimes 
\bigg( \bigotimes_{r=k+1}^{|a|} \Big\langle \big[ \hat{\bX}, \hat{\bY} \big]_{s, t}, \rId\Big\rangle(\hat{\omega}_{a_r}) \bigg) 
\\
\nonumber
+& \partial_a f\Big( \big\langle \bX_{s}, \rId \big\rangle (\omega_0), \cL^{\langle \bY_s, \rId\rangle}, ... \Big)(...,\omega_{m[a]}) 
\cdot 
\sum_{k=1}^{|a|} \bigg( \bigotimes_{r=1}^{k-1} \Big\langle \fJ\big[ \bX, \bY \big]_{s, t}, \rId\Big\rangle(\omega_{a_r}) \bigg) 
\\
\nonumber
&\quad \otimes 
\Big\langle [\bX, \bY]_{s, t}^{\sharp}, \rId\Big\rangle(\omega_{a_k}) 
\otimes \Bigg( \sum_{l=k+1}^{|a|} \bigg( \bigotimes_{r=k+1}^{l-1} \Big\langle [\bX, \bY]_{s, t}, \rId\Big\rangle(\omega_{a_r}) \bigg) 
\\
\label{eq:theorem:Stability-ContinImage:two.4}
&\quad \otimes \Big\langle [\bX, \bY]_{s, t} - [\hat{\bX}, \hat{\bY}]_{s, t}, \rId\Big\rangle(\omega_{a_r}, \hat{\omega}_{a_r}) \otimes \bigg( \bigotimes_{r=l+1}^{|a|} \Big\langle [\hat{\bX}, \hat{\bY}]_{s, t}, \rId\Big\rangle(\hat{\omega}_{a_r}) \bigg) \Bigg) \Bigg]. 
\end{align}

Addressing each these terms in turn, we see that
\begin{align*}
\eqref{eq:theorem:Stability-ContinImage:two.1} \leq& \sum_{a\in A^{\alpha, \beta}} \frac{1}{|a|!}
\bigg( \Big\| \partial_a f\Big\|_{\lip, 0} \cdot \bW_{\Pi}^{p_x[\rId]}\big[\rId\big]\Big( \bX_s, \hat{\bX}_s\Big)(\omega_0) + \Big\| \partial_a f\Big\|_{\lip, \mu} \cdot \bW_{\Pi}^{p_y[\rId]}\big[\rId\big]\Big( \bY_s, \hat{\bY}_s\Big)
\\
& + \sum_{r=1}^{m[a]} \Big\| \partial_a f\Big\|_{\lip, r} \cdot \bW_{\Pi}^{p_y[\rId]}\big[\rId\big]\Big( \bY_s, \hat{\bY}_s\Big) \bigg)
\cdot
\sum_{k=1}^{|a|} \bigg( \prod_{r=1}^{k-1} \Big\| \big\langle \fJ[ \hat{\bX}, \hat{\bY} ]_{s, t}, \rId\big\rangle(\hat{\omega}_{a_r})\Big\|_{[a],p_x,p_y} \bigg) 
\\
&\quad \cdot \Big\| \big\langle [\hat{\bX}, \hat{\bY} ]_{s, t}^{\sharp}, \rId\big\rangle(\hat{\omega}_{a_k}) \Big\|_{[a],p_x,p_y} \cdot \bigg( \prod_{r=k+1}^{|a|} \Big\| \big\langle [ \hat{\bX},\hat{\bY}]_{s, t}, \rId\big\rangle(\hat{\omega}_{a_r})\Big\|_{[a],p_x,p_y} \bigg), 
\end{align*}
\begin{align*}
\eqref{eq:theorem:Stability-ContinImage:two.2} \leq&
\sum_{a\in A^{\alpha, \beta}} \frac{1}{|a|!}
\Big\| \partial_a f \Big\|_{\infty} \cdot \sum_{k=1}^{|a|} \Bigg( \sum_{l=1}^{k-1} \bigg( \prod_{r=1}^{l-1} \Big\| \big\langle \fJ[\bX, \bY]_{s, t}, \rId\big\rangle(\omega_{a_r}) \Big\|_{[a], p_x,p_y} \bigg) 
\\
&\quad \cdot \bW_{\Pi}^{[a],p_x,p_y}\big[\rId\big]\Big( \fJ[\bX, \bY]_{s, t}, \fJ[\hat{\bX}, \hat{\bY}]_{s, t}\Big)(\omega_{a_l}) \cdot \bigg( \prod_{r=l+1}^{k-1} \Big\| \big\langle \fJ[\hat{\bX}, \hat{\bY}]_{s, t}, \rId\big\rangle(\omega_{a_r}) \Big\|_{[a],p_x,p_y} \bigg)
\\
&\quad \cdot \Big\| \big\langle [\hat{\bX}, \hat{\bY}]_{s, t}^{\sharp}, \rId\big\rangle(\hat{\omega}_{a_k})\Big\|_{[a], p_x,p_y} \cdot \bigg( \bigotimes_{r=k+1}^{|a|} \Big\| \big\langle [\hat{\bX},\hat{\bY}]_{s, t}, \rId\big\rangle(\hat{\omega}_{a_r})\Big\|_{[a],p_x,p_y} \bigg) \Bigg), 
\end{align*}
\begin{align*}
\eqref{eq:theorem:Stability-ContinImage:two.3} \leq&
\sum_{a\in A^{\alpha, \beta}} \frac{1}{|a|!} 
\Big\| \partial_a f\Big\|_\infty \cdot 
\sum_{k=1}^{|a|} \bigg( \prod_{r=1}^{k-1} \Big\| \big\langle \fJ[\bX, \bY]_{s, t}, \rId\big\rangle(\omega_{a_r})\Big\|_{[a], p_x,p_y} 
\\
&\quad \cdot \bW_{\Pi}^{[a],p_x,p_y} \big[ \rId\big] \Big(
[\bX, \bY]_{s, t}^{\sharp}, [\hat{\bX}, \hat{\bY} ]_{s, t}^{\sharp}\Big)
\cdot \prod_{r=k+1}^{|a|} \Big\| \big\langle [\hat{\bX},\hat{\bY}]_{s, t}, \rId\big\rangle(\hat{\omega}_{a_r})\Big\|_{[a],p_x,p_y} \bigg) 
\end{align*}
and
\begin{align*}
\eqref{eq:theorem:Stability-ContinImage:two.4} \leq&
\sum_{a\in A^{\alpha, \beta}} \frac{1}{|a|!}
\Big\| \partial_a f \Big\|_{\infty} \cdot \sum_{k=1}^{|a|} \Bigg(  \bigg(\prod_{r=1}^{k-1} \Big\| \big\langle \fJ[\bX, \bY]_{s, t}, \rId\big\rangle(\omega_{a_r}) \Big\|_{[a], p_x,p_y} \bigg) 
\\
&\quad \cdot \Big\| \big\langle [\bX, \bY]_{s, t}^{\sharp}, \rId \big\rangle(\omega_{a_k})\Big\|_{[a], p_x,p_y} \cdot \bigg( \sum_{l=k+1}^{|a|} \bigg( \prod_{r=1}^{l-1} \Big\| \big\langle [\bX, \bY]_{s, t}, \rId\big\rangle(\omega_{a_r}) \Big\|_{[a], p_x,p_y} \bigg) 
\\
&\quad \cdot \bW_{\Pi}^{[a],p_x,p_y}\big[\rId\big]\Big( [\bX, \bY]_{s, t}, [\hat{\bX}, \hat{\bY}]_{s, t}\Big)(\omega_{a_l}) \cdot \bigg( \prod_{r=l+1}^{k-1} \Big\| \big\langle [\hat{\bX}, \hat{\bY}]_{s, t}, \rId\big\rangle(\omega_{a_r}) \Big\|_{[a],p_x,p_y} \bigg). 
\end{align*}

Finally, 
\begin{align}
\nonumber
\hyperlink{Eq:Stab_1_three}{(3)} =&  \sum_{a\in A_\ast^{\alpha, \beta}} \frac{1}{|a|!} \Bigg( (\bE\times \hat{\bE})^{1, ..., m[a]}\Bigg[ \bigg( \partial_a f\Big( \big\langle \bX_s, \rId \big\rangle (\omega_0), \cL^{\langle \bY_s, \rId\rangle}, ...\Big)(...,\omega_{m[a]})
\\
\label{eq:theorem:Stability-ContinImage:three.1}
&\quad - \partial_a f\Big( \big\langle \hat{\bX}_s, \rId \big\rangle (\omega_0), \cL^{\langle \hat{\bY}_s, \rId\rangle}, ...\Big)(...,\hat{\omega}_{m[a]}) \bigg) \cdot \bigg( \bigotimes_{r=1}^{|a|} \Big\langle \big[ \hat{\bX}, \hat{\bY} \big]_{s,t}, \rId\Big\rangle (\hat{\omega}_{a_r}) \bigg)
\\
\nonumber
&+\partial_a f\Big( \big\langle \bX_s, \rId \big\rangle (\omega_0), \cL^{\langle \bY_s, \rId\rangle}, ...\Big)(...,\omega_{m[a]})
\cdot
\sum_{k=1}^{|a|} \bigg( \bigotimes_{r=1}^{k-1} \Big\langle \big[ \bX, \bY \big]_{s,t}, \rId\Big\rangle (\omega_{a_r}) \bigg)
\\
\label{eq:theorem:Stability-ContinImage:three.2}
&\quad \otimes \Big\langle \big[ \bX, \bY \big]_{s,t} - \big[ \hat{\bX}, \hat{\bY} \big]_{s,t}, \rId\Big\rangle (\omega_{a_r}, \hat{\omega}_{a_r}) \otimes \bigg( \bigotimes_{r=k+1}^{|a|} \Big\langle \big[ \hat{\bX}, \hat{\bY} \big]_{s,t}, \rId\Big\rangle (\hat{\omega}_{a_r}) \bigg) \Bigg]
\\
\nonumber
+& \tfrac{1}{ (n -1)!} \sum_{a \in \A{n} } (\bE\times \hat{\bE})^{1, ..., m[a]}\Bigg[ \bigg( f^a \Big[ \big\langle \bX_s, \rId \big\rangle (\omega_0), \big\langle \bX_t, \rId \big\rangle (\omega_0), \Pi^{\langle \bY_s, \rId\rangle, \langle \bY_t, \rId\rangle} \Big] (...,\omega_{m[a]})
\\
\label{eq:theorem:Stability-ContinImage:three.3}
&\quad - f^a \Big[ \big\langle \hat{\bX}_s, \rId \big\rangle (\omega_0), \big\langle \hat{\bX}_t, \rId \big\rangle (\omega_0), \Pi^{\langle \hat{\bY}_s, \rId\rangle, \langle \hat{\bY}_t, \rId\rangle} \Big](...,\hat{\omega}_{m[a]}) \bigg) \cdot \bigotimes_{r=1}^{n} \Big\langle \big[ \hat{\bX}, \hat{\bY} \big]_{s, t}(\hat{\omega}_{a_r}), \rId\Big\rangle \bigg]
\\
\nonumber
&+ f^a \Big[ \big\langle \bX_s, \rId \big\rangle (\omega_0), \big\langle \bX_t, \rId \big\rangle (\omega_0), \Pi^{\langle \bY_s, \rId\rangle, \langle \bY_t, \rId\rangle} \Big] (...,\omega_{m[a]}) \cdot \sum_{k=1}^{|a|} \bigg( \bigotimes_{r=1}^{k-1} \Big\langle \big[ \bX, \bY \big]_{s,t}, \rId\Big\rangle (\omega_{a_r}) \bigg)
\\
\label{eq:theorem:Stability-ContinImage:three.4}
&\quad \otimes \Big\langle \big[ \bX, \bY \big]_{s,t} - \big[ \hat{\bX}, \hat{\bY} \big]_{s,t}, \rId\Big\rangle (\omega_{a_r}, \hat{\omega}_{a_r}) \otimes \bigg( \bigotimes_{r=k+1}^{|a|} \Big\langle \big[ \hat{\bX}, \hat{\bY} \big]_{s,t}, \rId\Big\rangle (\hat{\omega}_{a_r}) \bigg) \Bigg].  
\end{align}

Addressing each these terms in turn, we see that
\begin{align*}
\eqref{eq:theorem:Stability-ContinImage:three.1} \leq& 
\sum_{a\in A_\ast^{\alpha, \beta}} \frac{1}{|a|!}
\bigg( \Big\| \partial_a f\Big\|_{\lip, 0} \cdot \bW_{\Pi}^{p_x[\rId]}\big[\rId\big]\Big( \bX_s, \hat{\bX}_s\Big)(\omega_0) + \Big\| \partial_a f\Big\|_{\lip, \mu} \cdot \bW_{\Pi}^{p_y[\rId]}\big[\rId\big]\Big( \bY_s, \hat{\bY}_s\Big)
\\
&\quad+ \sum_{r=1}^{m[a]} \Big\| \partial_a f\Big\|_{\lip, r} \cdot \bW_{\Pi}^{p_y[\rId]}\big[\rId\big]\Big( \bY_s, \hat{\bY}_s\Big) \bigg)
\cdot
\prod_{r=1}^{|a|} \Big\| \big\langle [\hat{\bX},\hat{\bY}]_{s,t}, \rId \big\rangle (\omega_{a_r}) \Big\|_{[a],p_x,p_y}
\end{align*}
and
\begin{align*}
\eqref{eq:theorem:Stability-ContinImage:three.2} \leq&
\sum_{a\in A_\ast^{\alpha, \beta}} \frac{1}{|a|!}
\Big\| \partial_a f \Big\|_{\infty}
\cdot
\sum_{k=1}^{|a|} \bigg( \prod_{r=1}^{k-1} \Big\| \big\langle [\bX,\bY]_{s,t}, \rId\big\rangle(\omega_{a_r})\Big\|_{[a],p_x,p_y} 
\\
&\quad \cdot \bW_{\Pi}^{[a],p_x,p_y}\big[ \rId\big]\Big( [\bX,\bY]_{s,t},  [\hat{\bX},\hat{\bY}]_{s,t}\Big)(\omega_{a_r}) \cdot \prod_{r=k+1}^{|a|} \Big\| \big\langle [\hat{\bX},\hat{\bY}]_{s,t}, \rId \big\rangle (\omega_{a_r})\Big\|_{[a], p_x,p_y} \bigg)
\end{align*}
Thanks to Equation \eqref{eq:theorem:ContinIm-RCRPs1.2_rem1} and the assumption that $f\in C_b^{n+1, (n+1)}\big( \bR^d \times \cP_2(\bR^d)\big)$, we can additionally say that 
\begin{align*}
\eqref{eq:theorem:Stability-ContinImage:three.3}
=&
\frac{1}{n!} \sum_{a\in \A{n+1}} (\bE \times \hat{\bE})^{1, ..., m[a]} \Bigg[ \int_0^1 \partial_a f\Big( \big\langle \bX_s + \xi \cdot \bX_{s, t}, \rId\big\rangle(\omega_0), \Pi_{\xi}^{\langle \bY_s, \rId\rangle, \langle \bY_t, \rId\rangle}, ...\Big)\cdot (1-\xi)^{n} d\xi
\\
&\qquad \cdot \Big\langle [\bX, \bY]_{s, t}, \rId\Big\rangle(\omega_{a_{n+1}}) \otimes \bigotimes_{r=1}^n \Big\langle [\hat{\bX}, \hat{\bY}]_{s,t}, \rId\Big\rangle(\hat{\omega}_{a_r})
\\
&- \int_0^1 \partial_a f\Big( \big\langle \hat{\bX}_s + \xi \cdot \hat{\bX}_{s, t}, \rId\big\rangle(\omega_0), \Pi_{\xi}^{\langle \hat{\bY}_s, \rId\rangle, \langle \hat{\bY}_t, \rId\rangle}, ...\Big)\cdot (1-\xi)^{n} d\xi
\\
&\qquad \cdot \Big\langle [\hat{\bX}, \hat{\bY}]_{s, t}, \rId\Big\rangle(\hat{\omega}_{a_{n+1}}) \otimes \bigotimes_{r=1}^n \Big\langle [\hat{\bX}, \hat{\bY}]_{s, t}, \rId\Big\rangle(\hat{\omega}_{a_r})\Bigg]
\\
\leq& \frac{1}{n!} \sum_{a\in \A{n+1}} \Bigg( \Big\| \partial_a f\Big\|_{\lip, 0} \cdot \bW_{\Pi}^{p_x[\rId]}\big[\rId\big]\Big( \bX_s, \hat{\bX}_s\Big)(\omega_0) + \Big\| \partial_a f\Big\|_{\lip, \mu} \cdot \bW_{\Pi}^{p_y[\rId]}\big[\rId\big]\Big( \bY_s, \hat{\bY}_s\Big)
\\
&+ \sum_{r=1}^{m[a]} \Big\| \partial_a f\Big\|_{\lip, r} \cdot \bW_{\Pi}^{p_y[\rId]}\big[\rId\big]\Big( \bY_s, \hat{\bY}_s\Big) \Bigg)
\cdot 
\Big\| \big\langle [\bX, \bY]_{s, t}, \rId\big\rangle(\omega_{a_{n+1}})\Big\|_{[a],p_x,p_y} 
\\
&\qquad \cdot \prod_{r=1}^n \Big\| \big\langle [\hat{\bX}, \hat{\bY}]_{s,t}, \rId\big\rangle(\hat{\omega}_{a_r}) \Big\|_{[a],p_x,p_y}
\\
+&\frac{1}{n!} \sum_{a\in \A{n+1}} \Big\| \partial_a f \Big\|_{\infty} \cdot \bW_{\Pi}^{[a], p_x,p_y}\big[\rId\big]\Big( [\bX, \bY]_{s, t}, [\hat{\bX}, \hat{\bY}]_{s, t}\Big)(\omega_{a_{n+1}}) \cdot \prod_{r=1}^n \Big\| \big\langle [\hat{\bX}, \hat{\bY}]_{s,t}, \rId\big\rangle(\hat{\omega}_{a_r}) \Big\|_{[a],p_x,p_y}
\end{align*}
and
\begin{align*}
\eqref{eq:theorem:Stability-ContinImage:three.4} \leq& \frac{1}{n!} \sum_{a\in \A{n+1}} \Big\| \partial_a f \Big\|_\infty \cdot \sum_{k=1}^n \bigg( \prod_{r=1}^{k-1} \Big\| \big\langle [\bX,\bY]_{s,t}, \rId \big\rangle (\omega_{a_r}) \Big\|_{[a],p_x,p_y} \bigg) 
\\
&\cdot \bW_{\Pi}^{[a],p_x,p_y}\big[\rId\big]\Big( [\bX, \bY]_{s,t}, [\hat{\bX}, \hat{\bY}]_{s, t}\Big)(\omega_{a_k}) 
\cdot 
\bigg( \prod_{r=k+1}^n \Big\| \big\langle [\hat{\bX},\hat{\bY}]_{s,t}, \rId \big\rangle (\omega_{a_r}) \Big\|_{[a],p_x,p_y} \bigg) 
\\
&\qquad \cdot \Big\| \big\langle [\bX,\bY]_{s,t}, \rId \big\rangle (\omega_{a_{n+1}}) \Big\|_{[a],p_x,p_y}
\end{align*}

By combining all of this together and applying Proposition \ref{proposition:CompareRCRP} many times over, we can construct a polynomial $\fP_{\rId}$ such that
\begin{align*}
\sup_{s, t\in[u,v]}& \frac{\bW_{\Pi}^{p[\rId]}\big[ \rId \big] \Big( \bZ_{s, t}^{\sharp}, \hat{\bZ}_{s, t}^{\sharp}\Big)(\omega_0) }{|t-s|^{\gamma}} \leq \| f\|_{C_b^{n+1, (n+1)}}
\\
& \cdot \Bigg( \rho_{(\alpha, \beta, p, q)}\Big( \rw, \hat{\rw}\Big)(\omega_0) \cdot \fP_{\rId, 1}
+ \bigg( d_{\rw, \hat{\rw}, \gamma, 0} \Big( \bX, \hat{\bX} \Big)(\omega_0)
+ d_{\rw, \hat{\rw}, \gamma} \Big( \bY, \hat{\bY} \Big)\bigg) \cdot \fP_{\rId,2} \Bigg)
\end{align*}
where (recalling Equation \eqref{eq:Max_2norms})
\begin{align*}
\fP_{\rId, i}=& \fP_{\rId, i}\Bigg( \sum_{T\in\scF_0}^{\gamma-, \alpha, \beta} \Big\| \big\langle \{\bX^{\sharp}\vee \hat{\bX}^{\sharp}, T \big\rangle(\omega_0) \Big\|_{p_x[T], \gamma - \scG_{\alpha, \beta}[T]}
, 
\sum_{T\in \scF}^{\gamma-, \alpha, \beta} \Big\| \big\langle \{\bX_u\vee\hat{\bX}_u\}, T \big\rangle(\omega_0) \Big\|_{p_x[T]}, 
\\
&\qquad 
\sum_{T\in\scF_0}^{\gamma-, \alpha, \beta} \Big\| \big\langle \{ \bY^{\sharp} \vee \hat{\bY}^{\sharp}, T \big\rangle \Big\|_{p_y[T], \gamma - \scG_{\alpha, \beta}[T]}
, 
\sum_{T\in \scF}^{\gamma-, \alpha, \beta} \Big\| \big\langle \{ \bY_u \vee \bY_u\}, T\big\rangle \Big\|_{p_y[T]}, 
\\
&\qquad \sum_{T\in \scF}^{\gamma, \alpha, \beta} \Big\| \big\langle \{\rw\vee\hat{\rw} \}, T\big\rangle(\omega_0) \Big\|_{q[T], \scG_{\alpha, \beta}[T]}
, 
|v-u|^\alpha
, 
|v-u|^{\beta} \Bigg). 
\end{align*}

By detailed evaluation of the upper bound established for $\hyperlink{Eq:Stab_1_one}{(1)}$ $\hyperlink{Eq:Stab_1_two}{(2)}$ and $\hyperlink{Eq:Stab_1_three}{(3)}$, we observe that every time we have a term associated to the function $f$, this is included in the set
\begin{align*}
\bigg\{& \Big\| \partial_a f \Big\|_\infty : a\in \bigcup_{k=1}^{n+1} \A{k} \bigg\} \cup \bigg\{ \Big\| \partial_a f \Big\|_{\lip, 0}, \quad \Big\| \partial_a f \Big\|_{\lip, \mu}, \quad \Big\| \partial_a f \Big\|_{\lip, j}: j=1, ..., m[a], \quad a\in \A{n+1} \bigg\}
\\
&\cup \bigg\{ \Big\| \partial_a f \Big\|_{\lip, 0}, \quad \Big\| \partial_a f \Big\|_{\lip, \mu}, \quad \Big\| \partial_a f \Big\|_{\lip, j}: j=1, ..., m[a], \quad a\in \bigcup_{k=1}^{n} \A{k} \bigg\}. 
\end{align*}
Using the standard estimates
$$
\Big\| \partial_a f \Big\|_{\lip, j} \leq \Big\| \partial_{a, j} f \Big\|_{\infty}, \quad
\Big\| \partial_a f \Big\|_{\lip, \mu} \leq \Big\| \partial_{a, m[a]+1} f \Big\|_{\infty}
$$
we conclude that we can use $\| f\|_{C_b^{n+1, (n+1)}}$ as an upper bound. Using the same intuition as in Remark, we would expect to see polynomials of order $n+2$, but we must also take care to accommodate the comparative terms. Thus, we divide our upper bound into two collections of terms: those containing any terms from the set
\begin{equation}
\label{eq:theorem:Stability-ContinImage:set.1}
\bigg\{ \bW_{\Pi}^{q[T]} \big[ T\big] \Big( \rw, \hat{\rw}\Big) : \quad T\in \scF^{\gamma, \alpha, \beta} \bigg\}. 
\end{equation}
and those containing any of the terms from the set
\begin{equation}
\label{eq:theorem:Stability-ContinImage:set.2}
\bigg\{ \bW_{\Pi}^{p_x[T]} \big[ T\big]\Big( \bX, \hat{\bX}\Big)(\omega_0), \quad \bW_{\Pi}^{p_y[T]} \big[ T\big]\Big( \bY, \hat{\bY}\Big) : T\in \scF_0^{\gamma-, \alpha, \beta} \bigg\}
\end{equation}
We note there are no product of terms in our upper bound that contains a term from both set \eqref{eq:theorem:Stability-ContinImage:set.1} and set \eqref{eq:theorem:Stability-ContinImage:set.2}. Terms that contain an element of \eqref{eq:theorem:Stability-ContinImage:set.1} will contribute to the polynomial $\fP_{\rId, 1}$ and terms that contain an element of \eqref{eq:theorem:Stability-ContinImage:set.1} will contribute to the polynomial $\fP_{\rId, 2}$. 

We can apply Proposition \ref{proposition:CompareRCRP} and Proposition \ref{prop:Regu-RCRP} to all terms in the upper bound to conclude that any terms that contain an element from \eqref{eq:theorem:Stability-ContinImage:set.1} will contain at most $n+2$ products from the set
\begin{align}
\nonumber
\bigg\{& \Big\| \big\langle \{\bX_u\vee \hat{\bX}_u \}, T\rangle(\omega_0)\Big\|_{p_x[T]}, \quad 
\Big\| \big\langle \{ \bX^{\sharp} \vee \hat{\bX}^{\sharp}\}, T\rangle(\omega_0)\Big\|_{p_x[T], \gamma - \scG_{\alpha, \beta}[T]}, 
\\
\label{eq:theorem:Stability-ContinImage:set.3}
&\Big\| \big\langle \{ \bY_u \vee \hat{\bY}_u\}, T\rangle\Big\|_{p_y[T]}, \quad 
\Big\| \big\langle \{ \bY^{\sharp} \vee \hat{\bY}\}, T\rangle\Big\|_{p_y[T], \gamma - \scG_{\alpha, \beta}[T]} 
: \quad T\in \scF_0^{\gamma-, \alpha, \beta}
\bigg\}
\end{align}
and a product of one less terms from the set
\begin{equation}
\label{eq:theorem:Stability-ContinImage:set.4}
\bigg\{ 1, \quad \Big\| \big\langle\{ \rw\vee \hat{\rw} \}, T\big\rangle(\omega_0)\Big\|_{q[T], \scG_{\alpha, \beta}[T]}: T\in \scF^{\gamma, \alpha, \beta} \bigg\}. 
\end{equation}

By contrast, any terms that contain an element from \eqref{eq:theorem:Stability-ContinImage:set.2} will contain at most $n+1$ products from the set \eqref{eq:theorem:Stability-ContinImage:set.3} and at most n+1 terms from set \eqref{eq:theorem:Stability-ContinImage:set.4}. 

Finally, whenever a term of the form $|v-u|$ occurs, it is always to a positive power expressible of the form $\alpha \cdot i_6 + \beta \cdot i_7$. However, as before it should be clear that there are constructions where the integers $i_6$ and $i_7$ may not be positive. 

Thus, the polynomials $\fP_{\rId,1}, \fP_{\rId, 2}: \big(\bR^+\big)^{\times 7} \to \bR^+$ expressed as 
\begin{align*}
\fP_{\rId,1} \Big( x_1, x_2, y_1, y_2, w, t_1, t_2 \Big) =& \sum_{i\in I_{\rId, 1}} C_i \cdot x_1^{i_1} \cdot x_2^{i_2} \cdot y_1^{i_3} \cdot y_2^{i_4} \cdot w^{i_5} \cdot t_1^{i_6} \cdot t_2^{i_7}, 
\\
\fP_{\rId,1} \Big( x_1, x_2, y_1, y_2, w, t_1, t_2 \Big) =& \sum_{j \in J_{\rId, 2}} C_j \cdot x_1^{j_1} \cdot x_2^{j_2} \cdot y_1^{j_3} \cdot y_2^{j_4} \cdot w^{j_5} \cdot t_1^{j_6} \cdot t_2^{j_7}. 
\end{align*}
satisfies that $i_1, ..., i_5, j_1, ..., j_5 \in \bN_0$, $i_1+i_2+i_3+i_4 \leq n+2$,$j_1+j_2+j_3+j_4 \leq n+1$, and $i_5+1 \leq i_1 + i_2 + i_3 + i_4$ and $j_5 \leq j_1 + j_2 + j_3 + j_4$. This verifies the additional claims made in Remark \ref{remark:Polynomial-comments2}. 

To conclude, we need to prove a similar result for a general Lions tree $Y\in \scF^{\gamma-, \alpha, \beta}$: To do this, we start with Equation \eqref{eq:Remainder_Y} and proceed in a similar fashion as above. There are no additional issues and Equation \eqref{eq:theorem:Stability-ContinImage} follows. 
\end{proof}

\end{document}